\documentclass[11pt]{article}
\usepackage{amsmath,amssymb,amsthm}
\usepackage{latexsym}
\usepackage{graphicx}
\usepackage{tcolorbox}
\usepackage{enumitem}
\usepackage{subfigure}
\usepackage{authblk}

\usepackage[letterpaper,margin=2.5cm]{geometry}

\def\PP{\mathbb{P}}
\def\RR{\mathbb{R}}

\def\bftau{{\boldsymbol \tau}}

\def\bfxi{{\boldsymbol \xi}}
\def\bfi{{\boldsymbol i}}

\def\bfp{{\boldsymbol p}}
\def\bfn{{\boldsymbol n}}

\def\cardB{{\cal N}}
\def\bfu{\textbf{u}}

\def\d{\mathrm{d}}
\def\nknots{{n_\mathrm{el}}}

\def\opt{\mathrm{opt}}

\def\indeigk{l}

\def\indeigkone{l_1}
\def\indeigktwo{l_2}

\def\prec{{\bar{p}}}
\def\bfprec{{\boldsymbol \prec}}

\newtheorem{theorem}{Theorem}[section]
\newtheorem{lemma}{Lemma}[section]
\newtheorem{proposition}{Proposition}[section]

\newtheorem{corollary}{Corollary}[section]

\theoremstyle{remark}\newtheorem{remark}{Remark}[section]
\theoremstyle{remark}\newtheorem{example}{Example}[section]

\numberwithin{equation}{section}

\usepackage[para]{footmisc}

\begin{document}

\title{Application of optimal spline subspaces for the removal of spurious outliers in isogeometric discretizations}

\author[1]{Carla Manni\thanks{manni@mat.uniroma2.it}}
\author[2]{Espen Sande\thanks{espen.sande@epfl.ch}}
\author[1]{Hendrik Speleers\thanks{speleers@mat.uniroma2.it}}

\affil[1]{\small Department of Mathematics, University of Rome Tor Vergata, Italy}
\affil[2]{\small Institute of Mathematics, EPFL, Lausanne, Switzerland}

\maketitle

\begin{abstract}
We show that isogeometric Galerkin discretizations of eigenvalue problems related to the Laplace operator subject to any standard type of homogeneous boundary conditions have no outliers in certain optimal spline subspaces. Roughly speaking, these optimal subspaces are obtained from the full spline space defined on certain uniform knot sequences by imposing specific additional boundary conditions.
The spline subspaces of interest have been introduced in the literature some years ago when proving their optimality with respect to Kolmogorov $n$-widths in $L^2$-norm for some function classes. The eigenfunctions of the Laplacian --- with any standard type of homogeneous boundary conditions --- belong to such classes. Here we complete the analysis of the approximation properties of these optimal spline subspaces. In particular, we provide explicit $L^2$ and $H^1$ error estimates with full approximation order for Ritz projectors in the univariate and in the multivariate tensor-product setting. 
Besides their intrinsic interest, these estimates imply that, for a fixed number of degrees of freedom, all the eigenfunctions and the corresponding eigenvalues are well approximated, without loss of accuracy in the whole spectrum when compared to the full spline space. Moreover, there are no spurious values in the approximated spectrum. In other words, the considered subspaces provide accurate outlier-free discretizations in the univariate and in the multivariate tensor-product case.
This main contribution is complemented by an explicit construction of B-spline-like bases for the considered spline subspaces. The role of such spaces as accurate discretization spaces for addressing general problems with non-homogeneous boundary behavior is discussed as well.

\medskip
\emph{Keywords:} isogeometric analysis; outlier-free discretizations; optimal spline subspaces; error estimates; eigenvalue problems
\end{abstract}


\section{Introduction}
Isogeometric analysis (IgA) was introduced in \cite{Hughes:2005} with the aim of unifying computer aided design (CAD) and finite element analysis (FEA). It has rapidly become a mainstream analysis methodology within computational engineering and also stimulated new research in geometric design.
The core idea in IgA is to use the same discretization and representation tools for the design as well as for the analysis (in an isoparametric environment), providing a true design-through-analysis methodology~\cite{Cottrell:09,Hughes:2005}.

The isogeometric approach based on B-splines/NURBS shows important advantages over classical $C^0$ FEA. In particular, the inherently high smoothness of B-splines and NURBS allows for a higher accuracy per degree of freedom. This behavior has been numerically observed in a wide range of applications, and recently a mathematical explanation has been given thanks to error estimates in spline spaces with constants that are explicit in the polynomial degree $p$ and the global smoothness $C^k$, the parameters defining the spline spaces; see \cite{BeiraoDaVeiga:2012,Bressan:2019,Sande:2019,Sande:2020}.

Spectral analysis can be used to study the error in each eigenvalue and eigenfunction of a numerical discretization of an eigenvalue problem.
For a large class of boundary and initial-value problems the total discretization error on a given mesh can be recovered from its spectral error \cite{Hughes:2014,Hughes:2008}. It is argued in \cite{Garoni:symbol} that this is of primary interest in engineering applications since practical computations are not performed in the limit of mesh refinement. Usually the computation is performed on a few, or even just a single mesh, and then the asymptotic information deduced from classical error analysis is insufficient. It is more relevant to understand which eigenvalues/eigenfunctions are well approximated for a given mesh size.

The isogeometric approach for eigenvalue problems has been investigated in \cite{Cottrell:2006} and several follow-up papers; see, e.g., \cite{Garoni:symbol,Hughes:2014,Hughes:2008}. Therein, extensive comparisons of the spectral approximation properties between classical finite elements and isogeometric elements have been performed. It turns out that
the isogeometric elements improve the overall accuracy of the spectral approximation significantly.
In FEA, the upper part of the discrete spectrum is inaccurate: high-order $p$ elements produce so-called \emph{optical branches}, which cause deteriorating accuracy of the higher modes. On the contrary, it has been observed that maximally smooth spline discretizations of degree $p$ on uniform grids remove the optical branches from the discrete
spectrum, which almost entirely converges for increasing $p$. More generally, the spectral discretization using uniform B-splines of degree $p$ and smoothness $C^k$, $0\leq k\leq p-1$, presents roughly $p-k$ branches and only a single branch (the so-called \emph{acoustical branch}) converges to the true spectrum \cite{Garoni:symbol,Hughes:2014}.
Convergence in $p$ for the $L^2$-projection of the eigenfunctions onto spline spaces has been analyzed in \cite{Sande:2019}. For general smoothness $C^k$ and for a fixed number of degrees of freedom, the error estimates in \cite{Sande:2019,Sande:2020} ensure convergence of the $L^2$-projection for increasing $p$ only for a part of the eigenfunctions. The number of those well-approximated eigenfunctions decreases as the maximal grid spacing increases (for a fixed number of degrees of freedom), so favoring maximal smoothness on uniform grids, in complete agreement with the numerical observations. 

Maximally smooth spline spaces on uniform grids are hence an excellent choice for addressing eigenvalue problems. Yet, they still present a flaw: 
a very small portion of the frequencies are poorly approximated and the corresponding computed values are much larger than the exact ones. These spurious values are usually referred to as \emph{outliers} \cite{Cottrell:2006}.
The number of outliers increases with the degree $p$. However, for fixed $p$, it is independent of the degrees of freedom for univariate problems, while a ``thin layer'' of outliers is observed in the multivariate setting; see \cite{Cottrell:09,Garoni:symbol} and references therein. We refer the reader to \cite{Gallistl:2017} for an analysis of the asymptotic growth of the largest outliers.
Outliers persist, in the same amount although with mitigated magnitude, when considering isogeometric methods based on problem-dependent spline spaces like generalized B-splines which allow for exact representations of some trigonometric functions \cite{Roman:2017}.

Outlier-free discretizations are appealing, not only for their superior description of the spectrum of the continuous operator, but also for their beneficial effects in various contexts, such as an efficient selection of time-steps in (explicit) dynamics and robust treatment of wave propagation. For a fixed degree, the challenge is to remove outliers without loss of accuracy in the approximation of all eigenfunctions.

We also note that there is a widespread awareness that outliers are related to the treatment of boundary conditions; these introduce small-rank perturbations in the matrices arising in the considered discretization process \cite{Garoni:2020,Garoni:symbol}. 

Eigenfunction approximation in case of the Laplacian with periodic boundary conditions in the space of periodic splines has been theoretically addressed in \cite[Section~4]{Sande:2019}. It has been proved that
for maximal smoothness $C^{p-1}$ and uniform grids, the Galerkin approximations, in the periodic $n$-dimensional spline space of degree $p$, converge in $p$ to the first $n$ or $n-1$ eigenfunctions (depending on the parity of $n$) of the Laplacian with periodic boundary conditions. This implies that the periodic case incurs at most one outlier. 
It was actually conjectured in \cite{Sande:2019} that no outliers can appear in such a context.

In this paper we complete the theory started in \cite{Sande:2019} and prove that for the optimal spline spaces described in \cite{Floater:2017,Floater:2018}, as suggested in \cite{Sande:2019}, there are no outliers in isogeometric Galerkin discretizations of the eigenvalue problem related to the Laplace operator with classical boundary conditions (Dirichlet/Neumann/mixed) in the univariate and in the multivariate tensor-product case. Roughly speaking, these optimal spline spaces are obtained from the standard spline space by imposing specific homogeneous boundary conditions and using certain uniform knot sequences.

There are few empirical proposals for outlier removal in the recent literature. In \cite{Cottrell:2006} it was observed that a suitable non-linear parameterization of the domain, obtained through a uniform
distribution of the control points, seems to eliminate outliers
for any $p$. This is actually related to a special treatment of the boundary because a uniform spacing of the control points (i.e., of the Greville abscissae) in the context of open knots implies large boundary knot intervals. However, some drawbacks of this approach have recently been identified in \cite{Hiemstra:2021}, for example, that this non-linear parameterization can lead to worse approximation of the lower modes. A similar method was also proposed in \cite{Chan:2018} where the authors numerically tested the use of so-called \emph{smoothed knots} (i.e., approximations to non-uniform knots that give rise to certain $n$-width optimal spline spaces) for explicit time-stepping and outlier-removal. These smoothed knots are very similar to those in \cite{Cottrell:2006} (they also have larger boundary knot intervals) and it is reported in \cite{Chan:2018} that these knots lead to worse approximation of lower modes as well. We remark that a slightly better approximation of the mentioned non-uniform $n$-width optimal knots can be obtained using the algorithm in \cite{Bressan:2020}, however, one would expect them to suffer from the same problems when used in outlier-removal.

More recently, other interesting contributions have been presented in \cite{Deng:2021,Hiemstra:2021}, where the authors exploit the imposition of suitable additional boundary conditions in a similar manner to the optimal spline spaces from \cite{Floater:2017,Floater:2018,Sande:2019}.
In \cite{Deng:2021} a penalization of specific high-order derivatives near the
boundary is proposed to remove the outliers from the isogeometric approximation of the spectrum of the Laplace operator. This penalization approach deeply mitigates the spurious frequencies. In \cite{Hiemstra:2021} the same high-order derivatives at the boundary are strongly set equal to zero, by using suitable spline subspaces as trial spaces. The approach is also tested for fourth-order operators. There is a clear numerical evidence that the strong imposition of the additional boundary conditions removes the outliers without affecting the accuracy of the approximation for the global spectrum.

Both the approaches in \cite{Deng:2021,Hiemstra:2021} rely on the observation that the exact eigenfunctions satisfy additional homogeneous boundary conditions and on the intuition that adding such features of the exact solution in the discretization (in weak or strong form) would help in fixing the outlier issue. However, a theoretical foundation of the proposed procedures and a proper analysis of the approximation properties of the used spline subspaces and so of the accuracy of the whole process are missing.

The spline subspaces used in \cite{Hiemstra:2021} as trial spaces are not new in the literature. They are the ``reduced spline spaces'' considered in \cite[Section~5.2]{Sande:2020} (see also \cite{Sande:2019,Takacs:2016} for some special cases). Moreover, depending on the parity of the degree $p$, they coincide with optimal spline spaces (in the sense of Kolmogorov $n$-widths) investigated in \cite{Floater:2018}.

The theory of Kolmogorov $n$-widths is an interesting framework to examine approximation properties. It defines and gives a characterization of optimal $n$-dimensional spaces for approximating function classes and their associated norms \cite{Babuska:2002,Kolmogorov:36,Pinkus:85}.
Kolmogorov $n$-widths and optimal subspaces with respect to the $L^2$-norm were studied in \cite{Evans:2009} with the aim of (numerically) assessing the approximation properties of smooth splines in IgA. 
In a recent sequence of papers \cite{Floater:2017,Floater:per,Floater:2018}, it has been proved that subspaces of smooth splines of any degree on uniform grids, identified by suitable boundary conditions, are optimal subspaces for $L^2$ Kolmogorov $n$-width problems for certain function classes of importance in IgA and FEA. 
The results in \cite{Floater:2017,Floater:2018} were then applied in \cite{Bressan:2019} to show that, for uniform grids, $k$-refined spaces in IgA provide a better accuracy per degree of freedom than $C^0$ FEA and $C^{-1}$ discontinuous Galerkin spaces in almost all cases of practical~interest.

The theory of Kolmogorov $n$-widths and optimal subspaces is closely related to spectral analysis. Assume $A$ is a function class defined in terms of an integral operator $K$. Then, the space spanned by the first $n$ eigenfunctions of the self-adjoint operator $KK^*$ is an optimal subspace for $A$. This is naturally connected to a differential operator through the kernel of $KK^*$ being a Green's function. 
By using this general framework, in \cite[Section~7]{Sande:2019} we analyzed how well the eigenfunctions of a given differential operator are approximated in $n$-dimensional optimal subspaces. In particular, for fixed dimension $n$,
we identified the optimal spline subspaces that converge in $L^2$-norm to spaces spanned by the first $n$ eigenfunctions of the Laplacian subject to different types of boundary conditions, as their degree $p$ increases. 
Error estimates in $L^2$-norm for approximation in such (optimal) spline subspaces of functions belonging to $H^1$ and $H^1_0$ are provided in \cite{Sande:2020}. 

Here we continue the theoretical investigation in \cite{Sande:2019,Sande:2020} and 
we prove that the strategy of using optimal spline spaces leads to accurate outlier-free isogeometric Galerkin discretizations for the spectrum of the Laplacian with Dirichlet/Neumann/mixed boundary conditions in the univariate and in the multivariate tensor-product case. 
More precisely, 
\begin{itemize}
\item we discretize the eigenvalue problem in optimal spline subspaces identified in terms of vanishing high-order derivatives at the boundary, as suggested in \cite{Sande:2019,Sande:2020};
\item we provide error estimates for Ritz projectors in such optimal spline subspaces;
\item we exploit the above estimates to show that the considered Galerkin discretizations are outlier-free, without loss of accuracy in the whole spectrum when compared to the full spline space;
\item we produce explicit expressions of B-spline-like bases for the spline subspaces of interest to be used in practical simulations.
\end{itemize}
In \cite[Remark 7.1]{Sande:2019} these spaces were already identified as the right candidate spaces for outlier-free discretizations, but the necessary theoretical results were not provided. In that paper we only fully addressed the simpler case of periodic boundary conditions.
It turns out that our outlier-removal strategy is very similar (and actually identical in several cases) to the one proposed in \cite{Hiemstra:2021}. However, our path towards it is more theoretical
and allows us to equip the numerical process with a solid mathematical foundation.

A main question is to what extent the proposed outlier-free discretizations can be fruitfully used for addressing general problems with non-homogeneous boundary behavior. As stated above, the outlier-free discretizations are based on strong imposition of additional homogeneous boundary conditions for some derivatives up to a certain order which depends on the degree $p$. While these additional boundary conditions are intrinsically satisfied by the exact eigenfunctions we are dealing with, it is clear that this is not the case when considering the exact solution of a general (second-order) problem. A plain discretization in outlier-free spline subspaces in general leads to a substantial loss of approximation power compared to the corresponding full spline space and this discrepancy worsens as the spline degree increases.

To overcome this issue, for problems identified by sufficiently smooth data, we propose a suitable data-correction process for the missing boundary derivatives analogous to the classical reduction from non-homogeneous to homogeneous Dirichlet boundary conditions. When coupled with this boundary data correction, we prove that the discretization in outlier-free spline subspaces achieves full approximation order both in the univariate and in the multivariate tensor-product case. We note that a special case of this approach has also been suggested in \cite[Appendix~B]{Hiemstra:2021} for a one-dimensional problem.

The remainder of the paper is divided in seven sections and is organized as follows. In Section~\ref{sec:eigenvalue-problem} we summarize the necessary notation and preliminaries on the eigenvalue problems of interest and their Galerkin discretizations. Section~\ref{sec:counting-outliers} provides theoretical upper bounds, as a function of the degree $p$, for the number of outliers when the discretization process is performed in the usual (full) spline spaces; these upper bounds almost perfectly match the number of outliers numerically observed.
In Section~\ref{sec:outlier-free} we present suitable spline subspaces and we prove that they ensure outlier-free discretizations while enjoying full accuracy. It turns out that these outlier-free subspaces are optimal spline spaces in the sense of the $L^2$ Kolmogorov $n$-widths. 
Section~\ref{sec:general-BC} explains how to compensate the homogeneous high-order boundary derivatives --- which characterize the outlier-free spline subspaces --- for general problems identified by smooth data in order to maintain full accuracy in the complete discretization process.
An explicit construction of a B-spline-like basis for the considered outlier-free subspaces is described in Section~\ref{sec:bsplines} by exploiting the properties of cardinal B-splines.
Numerical tests validating the theoretical proposals are collected in Section~\ref{sec:numerics}. We conclude in Section~\ref{sec:conclusion} with some final remarks.
For a smoother reading of the paper, the technical details and the proofs of the error estimates used in Section~\ref{sec:outlier-free} are postponed to Appendix~\ref{Appendix:A}. The proofs of the error estimates used in Section~\ref{sec:general-BC} are postponed to Appendix~\ref{Appendix:B}.

Throughout the paper, for real-valued functions $f$ and $g$ we denote the norm and inner product on $L^2:=L^2(a,b)$ by
$$ \| f\|^2 := (f,f), \quad (f,g) := \int_a^b f(x) g(x)\,\d x, $$
and we consider the Sobolev spaces
$$ H^r:= H^r(a,b)=\{u\in L^2 : \, u^{(\alpha)}  \in L^2(a,b),\, \alpha=1,\ldots,r\}. $$
For the sake of simplicity, we set $(a,b)=(0,1)$.


\section {Second-order eigenvalue problems}\label{sec:eigenvalue-problem}
We consider the second-order eigenvalue problem related to the Laplace operator in the univariate and in the multivariate tensor-product case.

\subsection{Univariate case}\label{sec:eigenvalue-problem-1D}
We consider the second-order equation
\begin{equation}
\label{eq:eigenvalue-equation-1D}
-u''= \omega^2 u, \quad \text{in } (0,1),
\end{equation}
and the following standard boundary conditions:
\begin{itemize}
\item Dirichlet boundary conditions (also referred to as fixed or type 0 boundary conditions),
\begin{equation}
\label{type-0-BC}
u(0)=u(1)=0;
\end{equation}
\item Neumann boundary conditions (also referred to as free or natural or type 1 boundary conditions),
	\begin{equation}
	\label{type-1-BC}
	u'(0)=u'(1)=0;
	\end{equation}
\item a combination of the previous ones (also referred to as mixed or type 2 boundary conditions),
\begin{equation}
\label{type-2-BC}
u(0)=u'(1)=0.
\end{equation}
\end{itemize}
The non-trivial exact solutions of \eqref{eq:eigenvalue-equation-1D} subject to one of the boundary conditions \eqref{type-0-BC}--\eqref{type-2-BC} form a numerable set of trigonometric functions, respectively,
\begin{alignat}{3}
\label{eq:eig-Laplace-type-0-BC}
u_\indeigk(x) &:= \sin(\omega_\indeigk x),\quad &\omega_\indeigk &:= \indeigk \pi, \quad &\indeigk &= 1,2,\ldots \\
\label{eq:eig-Laplace-type-1-BC}
u_\indeigk(x) &:= \cos(\omega_\indeigk x),\quad &\omega_\indeigk &:= \indeigk \pi, \quad &\indeigk &= 0,1,2,\ldots \\
\label{eq:eig-Laplace-type-2-BC}
u_\indeigk(x) &:= \sin(\omega_\indeigk x),\quad &\omega_\indeigk &:= (\indeigk-1/2)\pi, \quad &\indeigk &=1,2,\ldots
\end{alignat}

For the sake of brevity and simplicity of presentation, in the following we will focus on Dirichlet boundary conditions; the treatment of the other cases is completely analogous. Therefore, we will focus on the problem
\begin{equation}\label{eq:prob-eigenv-1D}
\left\{ \begin{aligned}
-u'' &= \omega^2 u, \quad \text{in } (0,1), \\
u(0) &=0, \quad  u(1) = 0,
\end{aligned} \right.
\end{equation}
whose non-trivial exact solutions are given in \eqref{eq:eig-Laplace-type-0-BC}.

The weak form of problem \eqref{eq:prob-eigenv-1D} reads as follows: find non-trivial $u\in H^1_0(0,1)$ and $\omega^2\in \mathbb{R}$ such that
$$ (u',v') = \omega^2 (u,v), \quad \forall v \in H^1_0(0,1). $$
According to the Galerkin approach, we choose a finite-dimensional subspace $\mathbb{V}_h$ of $H^1_0(0,1)$ spanned by the basis $\{\varphi_1,\ldots,\varphi_{n_h}\}$ and
we find approximate values $\omega_h$ to $\omega$ by solving
\begin{equation} \label{eq:eig-problem-galerkin}
S_h \bfu_h = (\omega_h)^2 M_h \bfu_h,
\end{equation}
where the matrices $S_h$ and $M_h$ consist of the elements
$$
S_{h,i,j}:= \int_0^1 \varphi'_j(s)\varphi'_i(s)\,\d s,\quad
M_{h,i,j} := \int_0^1 \varphi_j(s) \varphi_i(s)\,\d s,\quad
i,j=1,\ldots,n_h.
$$
This means that each $(\omega_h)^2$ is an eigenvalue of the matrix $ M_h^{-1} S_h$.
Then, for $\indeigk=1, \ldots, {n_h}$, an approximation of the frequency $\omega_\indeigk$ is obtained by considering
the square root of the $\indeigk$-th eigenvalue of $M_h^{-1} S_h$, denoted by $\omega_{h,\indeigk}$. Here we assume that those eigenvalues are given in ascending order.
Similarly, for $\indeigk=1, \ldots, {n_h}$, an approximation of the eigenfunction $u_\indeigk$ is obtained by considering
\begin{equation}
\label{eq:approx-eigenfun}
u_{h,l}(x):=\sum_{i=1}^{n_h} u_{h,\indeigk,i}\varphi_i(x),
\end{equation}
where
$\bfu_{h,\indeigk}:=(u_{h,\indeigk,1},\ldots, u_{h,\indeigk,{n_h}})$ is the $\indeigk$-th eigenvector of $M_h^{-1} S_h$.
Of course, a proper normalization is needed.
More information on this eigenvalue problem can be found in \cite{Boffi:2010}.

We are interested in selecting the discretization spaces such that all the first $n_h$ eigenfunctions and eigenvalues are approximated well.
In this perspective we will choose the approximation space as a proper subspace of maximally smooth spline spaces on the unit interval.

\subsection{Multivariate tensor-product case} 
We now consider the eigenvalue problem related to the Laplace operator in the unit cube of $\mathbb{R}^d$, namely
\begin{equation}
\label{eq:eigenvalue-equation}
-\Delta u= \omega^2 u, \quad \text{in } (0,1)^d,
\end{equation}
subject to homogeneous Dirichlet/Neumann/mixed boundary conditions similar to the ones discussed in the univariate case.
It is easy to verify that the $d$-variate eigenvalues and eigenfunctions are given by, respectively, the sum and the product of the corresponding univariate ones. 

Again, for the sake of brevity and simplicity of presentation, in the following we will focus on Dirichlet boundary conditions and $d=2$; the treatment of the other cases is completely analogous. Therefore, we will focus on the problem
\begin{equation}\label{eq:prob-eigenv}
\left\{ \begin{aligned}
-\Delta u&= \omega^2 u, \quad \text{in } (0,1)^2, \\
	u_{|\partial \Omega} &= 0,
\end{aligned} \right.
\end{equation}
whose non-trivial exact solutions are given by
\begin{equation}\label{eq:eig-Laplace2D-type-0-BC}
u_{\indeigkone,\indeigktwo}(x_1,x_2):=\sin(\indeigkone\pi x_1)\sin(\indeigktwo\pi x_2), \quad (\omega_{\indeigkone,\indeigktwo})^2:=(\indeigkone\pi)^2+(\indeigktwo\pi)^2, \quad \indeigkone,\indeigktwo=1,2, \ldots
\end{equation}
Note that 
\begin{equation}
\label{eq:eig-Laplace2D}
u_{\indeigkone,\indeigktwo}(x_1,x_2)=u_{\indeigkone}(x_1)u_{\indeigktwo}(x_2), \quad 
(\omega_{\indeigkone,\indeigktwo})^2=(\omega_{\indeigkone})^2+(\omega_{\indeigktwo})^2, \quad \indeigkone,\indeigktwo=1,2, \ldots,
\end{equation}
where $u_{\indeigkone},u_{\indeigktwo}$ and $\omega_{\indeigkone},\omega_{\indeigktwo}$ are the univariate counterparts defined in \eqref{eq:eig-Laplace-type-0-BC}.

In order to discretize problem \eqref{eq:prob-eigenv}, it is natural to consider a finite-dimensional tensor-product discretization space 
$\mathbb{V}_{h_1,h_2}:=\mathbb{V}_{h_1}\otimes\mathbb{V}_{h_2}$.
Then, the Galerkin method amounts to solve the following problem: find $\bfu_{h_1,h_2}$ and $\lambda_{h_1,h_2}$ such that
$$
(S_{h_1}\otimes M_{h_2}+M_{h_1}\otimes S_{h_2})\bfu_{h_1,h_2}=\lambda_{h_1,h_2} (M_{h_1}\otimes M_{h_2})\bfu_{h_1,h_2},
$$
where $S_h$ and $M_h$ are univariate stiffness and mass matrices, respectively, defined as in \eqref{eq:eig-problem-galerkin}.
It is known that, using the so-called \emph{fast diagonalization method} \cite{Horn:2013,Sangalli:2016},
the corresponding eigenvectors and eigenvalues can be represented in matrix form as 
$$
U_{h_1}\otimes U_{h_2}, 
\quad
\Lambda_{h_1}\otimes I_{h_2} + I_{h_1}\otimes \Lambda_{h_2},
$$
respectively, where $U_h$ is the matrix of the eigenvectors associated with the univariate discretization \eqref{eq:eig-problem-galerkin}, $\Lambda_h$ is the diagonal matrix of the corresponding eigenvalues, and $I_h$ is the identity matrix of the same size as $\Lambda_h$. Thus, very similar to \eqref{eq:eig-Laplace2D}, our approximations take the form
\begin{equation}\label{eq:approx-eig-Laplace2D}
u_{h_1,h_2,\indeigkone,\indeigktwo}(x_1,x_2) = u_{h_1,\indeigkone}(x_1)u_{h_2,\indeigktwo}(x_2), 
\quad
(\omega_{h_1,h_2,\indeigkone,\indeigktwo})^2=(\omega_{h_1,\indeigkone})^2+(\omega_{h_2,\indeigktwo})^2,
\end{equation}
where $u_{h_1,\indeigkone},u_{h_2,\indeigktwo}$ and $\omega_{h_1,\indeigkone},\omega_{h_2,\indeigktwo}$ are the univariate counterparts described in Section~\ref{sec:eigenvalue-problem-1D}.

The above decomposition does not only reduce the computational cost of the discrete spectrum but also provides a natural safe way to match the approximate eigenfunctions/eigenvalues with the exact ones. This was also observed in \cite{Hiemstra:2021}.


\section{Maximally smooth spline approximations and outliers}
\label{sec:counting-outliers}
Suppose $\bftau := (\tau_0,\ldots,\tau_{\nknots})$ is a sequence of (break) points that partition the interval $[0,1]$ in $\nknots$ elements, i.e.,
\begin{equation*}
0=:\tau_0 < \tau_1 < \cdots < \tau_{\nknots-1} < \tau_{\nknots}:= 1,
\end{equation*}
and let
$I_j := [\tau_{j-1},\tau_{j})$,
$j=1,\ldots,\nknots-1$, and $I_\nknots := [\tau_{\nknots-1},\tau_{\nknots}]$.
Let $\PP_p$ be the space of polynomials of
degree at most $p$. For $0\leq k\leq p-1$, we define the space $\mathbb{S}^k_{p,\bftau}$ of splines of degree $p$ and smoothness $k$ by
$$ \mathbb{S}^k_{p,\bftau } := \{s \in C^{k}[0,1] : s|_{I_j} \in \PP_p,\, j=1,\ldots,\nknots \}, $$
and we set
\begin{equation}
\label{eq:spline-max-smooth}
 \mathbb{S}_{p,\bftau} := \mathbb{S}^{p-1}_{p,\bftau}. 
\end{equation}

From classical spline approximation theory we know that for any $u\in H^{r}$ and any $\bftau$ there exists $s_p\in \mathbb{S}^k_{p,\bftau}$
 such that
 \begin{equation} \label{eq:classical-err-est}
 \|(u-s_p)^{(\ell)} \|\leq C(p,k,\ell,r)h^{r-\ell} \| u^{(r)} \|, \quad 0\leq \ell \leq r\leq p+1, \quad \ell\leq k+1\leq p,
 \end{equation}
where
\begin{equation*}
h:=\max_{j=1,\ldots,\nknots} h_j, \quad h_j:=\tau_{j}-\tau_{j-1}.
\end{equation*}
The above estimates can be generalized to any $L^q$-norm; see, e.g., \cite{Lyche:18,Schumaker2007}.

In the important case of maximally smooth splines ($k=p-1$), the constant in \eqref{eq:classical-err-est} admits simple explicit expressions. In particular, it has been proved in \cite[Theorems~1.1 and~3.1]{Sande:2019} that
 for any $u\in H^{r}$ and any $\bftau$ there exists $s_p\in \mathbb{S}_{p,\bftau}$ such that 
\begin{align*}
\| u-s_p \|&\leq \left(\frac {h}{\pi}\right)^{r} \| u^{(r)} \|,  
\\
\|(u-s_p)' \|&\leq \left(\frac {h}{\pi}\right)^{r-1} \|u^{(r)} \|, 
\end{align*}
for all $p\geq \max\{r-1,1\}$.
More generally, the results in \cite[Theorem~3 and Lemma~2]{Sande:2021} ensure that the above inequalities still hold if we want to approximate $u$ by a spline enjoying the same boundary conditions as $u$. This is summarized in the following theorem.

\begin{theorem}\label{thm:Qspline}
Let $u\in H^r(0,1)$ be given.
For any $\bftau$ and $q=0,\ldots,\min\{p,r\}$ there exists a projector $Q_p^{q}$ onto $\mathbb{S}_{p,\bftau}$ such that
\begin{equation*}
((Q_p^qu)^{(q)}, v^{(q)})=(u^{(q)}, v^{(q)}), \quad \forall v\in \mathbb{S}_{p,\bftau},
\end{equation*}
\begin{equation*}
(Q_p^q)^{(\ell)}(0)= u^{(\ell)}(0), \quad (Q_p^q)^{(\ell)}(1)= u^{(\ell)}(1), \quad \ell=0,\ldots,q-1,
\end{equation*}
and 
\begin{equation*}
\|(u-Q^{q}_pu)^{(\ell)}\| \leq \left(\frac {h}{\pi}\right)^{r-\ell}\|u^{(r)}\|, \quad \ell=0,\ldots,q,
\end{equation*}
for all $p\geq \max\{r-1,2q-1\}$. 
\end{theorem}
In our context we are interested in the case $q=1$ in the previous theorem. In this case, from \cite[Section~3]{Sande:2021} we also know the stability estimate
\begin{equation}
\label{eq:stab}
\|(Q_p^1u)'\|\leq \|u'\|, \quad u\in H^1(0,1).
\end{equation}

Natural discretization spaces for \eqref{eq:prob-eigenv-1D} are the subspaces
of $H_0^1(0,1)$ given by maximally smooth splines vanishing at the two ends of the interval, i.e., 
\begin{equation}
\label{eq:space-BC}
\mathbb{S}_{p,\bftau,0}^0:=\{s\in \mathbb{S}_{p,\bftau}: s(0)=s(1)=0 \},
\end{equation}
for $p\geq1$ and $\nknots>2-p$.
For such a space we have
$$n_h=\dim(\mathbb{S}_{p,\bftau,0}^0)=\nknots+p-2.$$
As mentioned in the introduction, it has been observed in \cite{Cottrell:2006} that the spline space \eqref{eq:space-BC} is an excellent choice for the Galerkin discretization of \eqref{eq:prob-eigenv-1D} when considering a uniformly distributed sequence of break points $\bftau$. However, it has also been observed that with such a choice there is still a very small portion of the frequencies poorly approximated and the corresponding computed values are much larger than the exact values. This is illustrated in Figure~\ref{fig:demo-outliers} for spline spaces of degree $p=5$ and different dimensions $n_h$. 
These spurious values are called \emph{outliers}.
Their number is independent of the grid size (and so of the dimension of the spline space for fixed $p$). Other examples can be found in Section~\ref{sec:numerics-1D}; see in particular Figures~\ref{fig:eigenvalues1D.odd}(b)--(c) and \ref{fig:eigenvalues1D.even}(c).
Furthermore, fixing the dimension, it has been observed that the number of outliers increases with $p$ and their approximations deteriorate in $p$, while the approximation of the rest of the spectrum improves.
The growth of the largest outlier in the case $p\to\infty$ has been analyzed in \cite{Gallistl:2017}.

\begin{figure}[t!]
\centering
\subfigure[$n_h=25$]{\includegraphics[height=4.1cm]{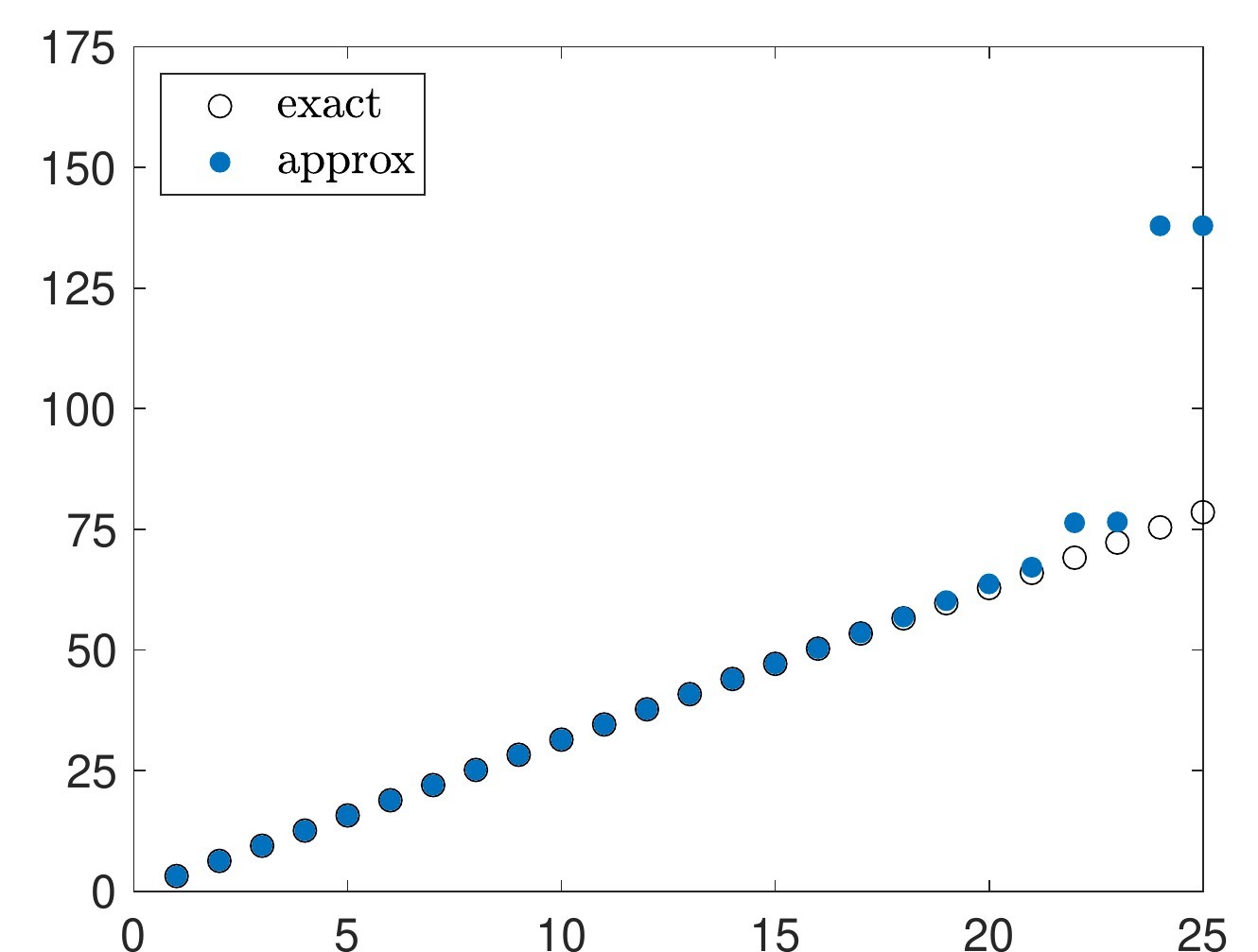}}\hspace*{0.1cm}
\subfigure[$n_h=50$]{\includegraphics[height=4.1cm]{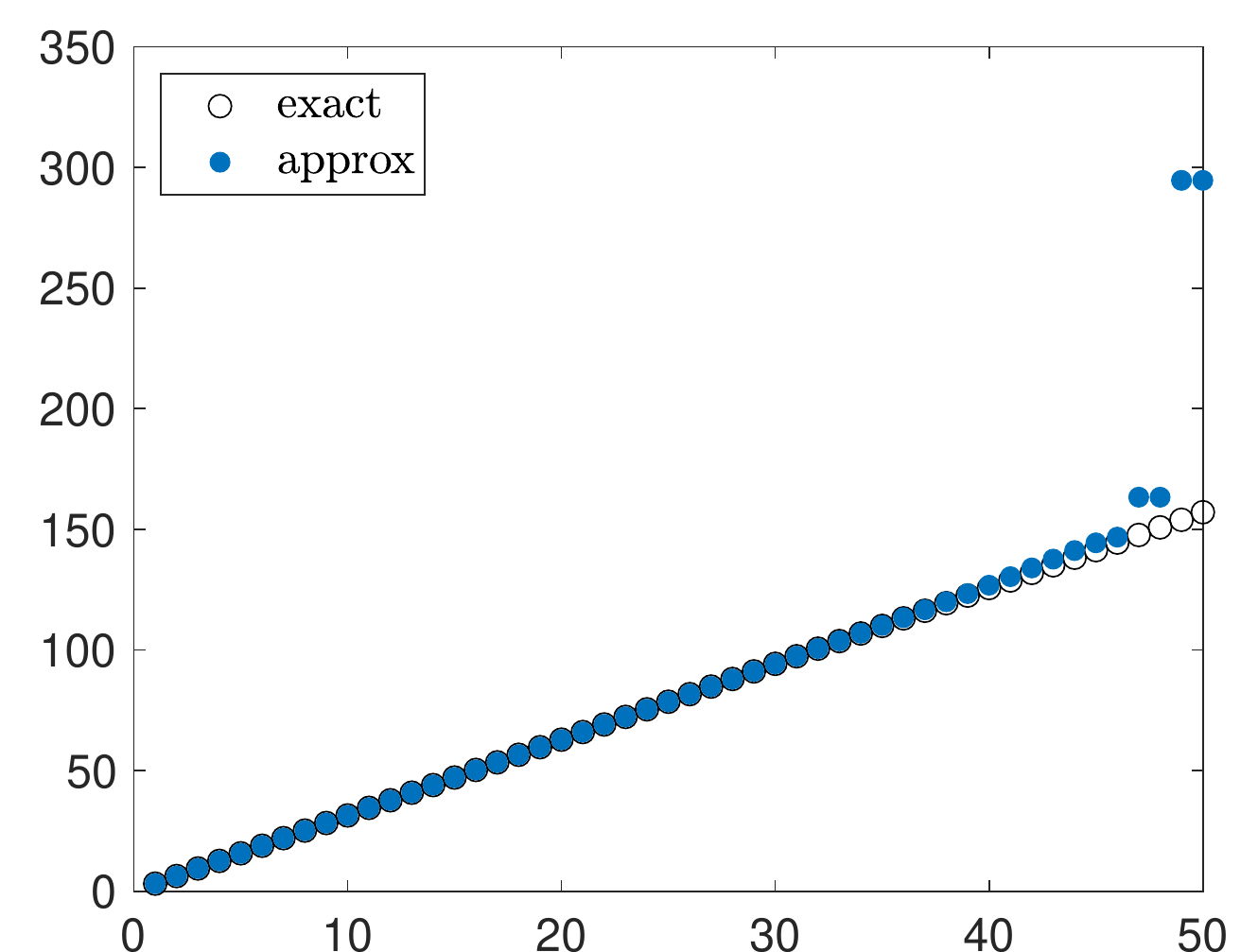}}\hspace*{0.1cm}
\subfigure[$n_h=100$]{\includegraphics[height=4.1cm]{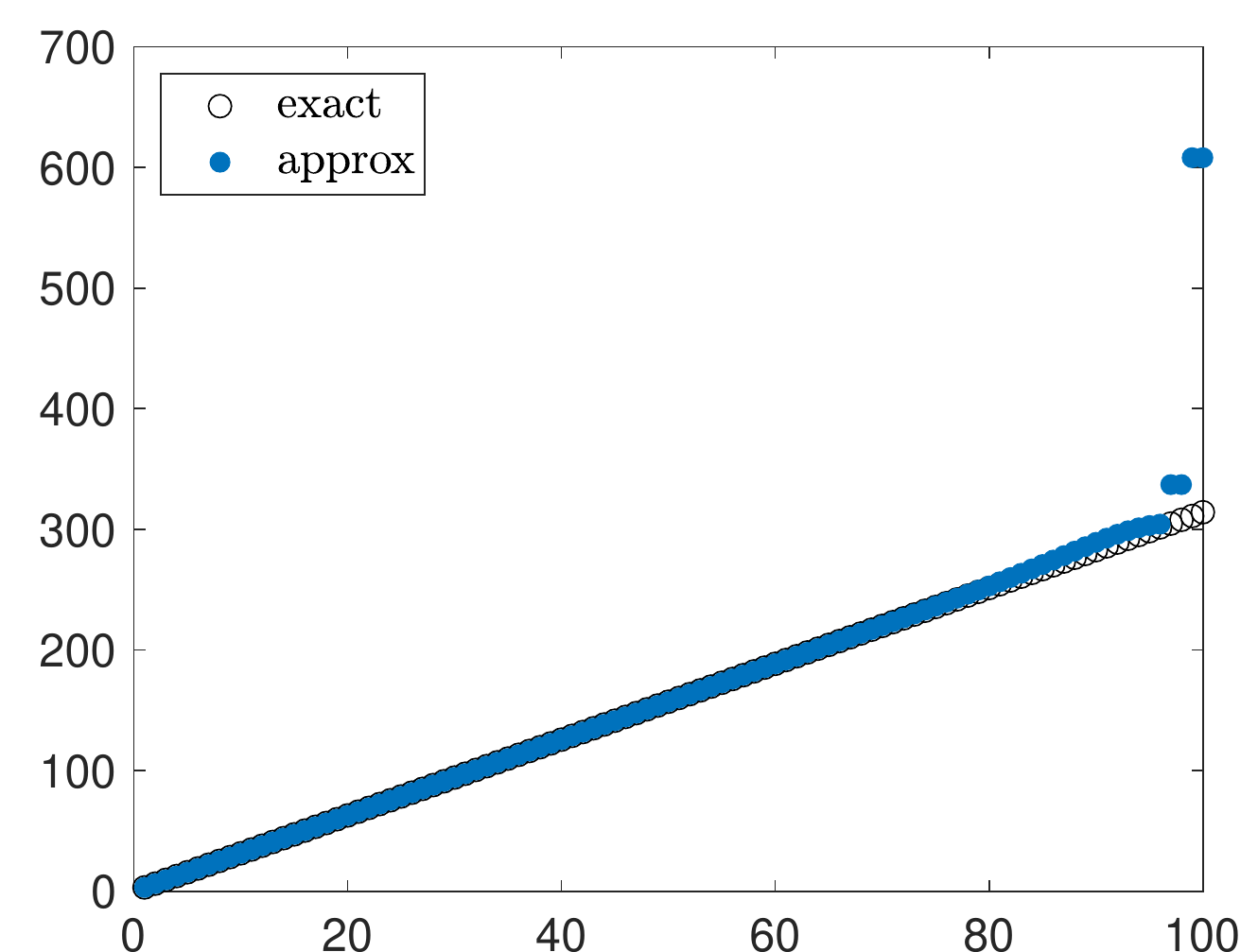}} \\
\subfigure[$n_h=25$]{\includegraphics[height=4.1cm]{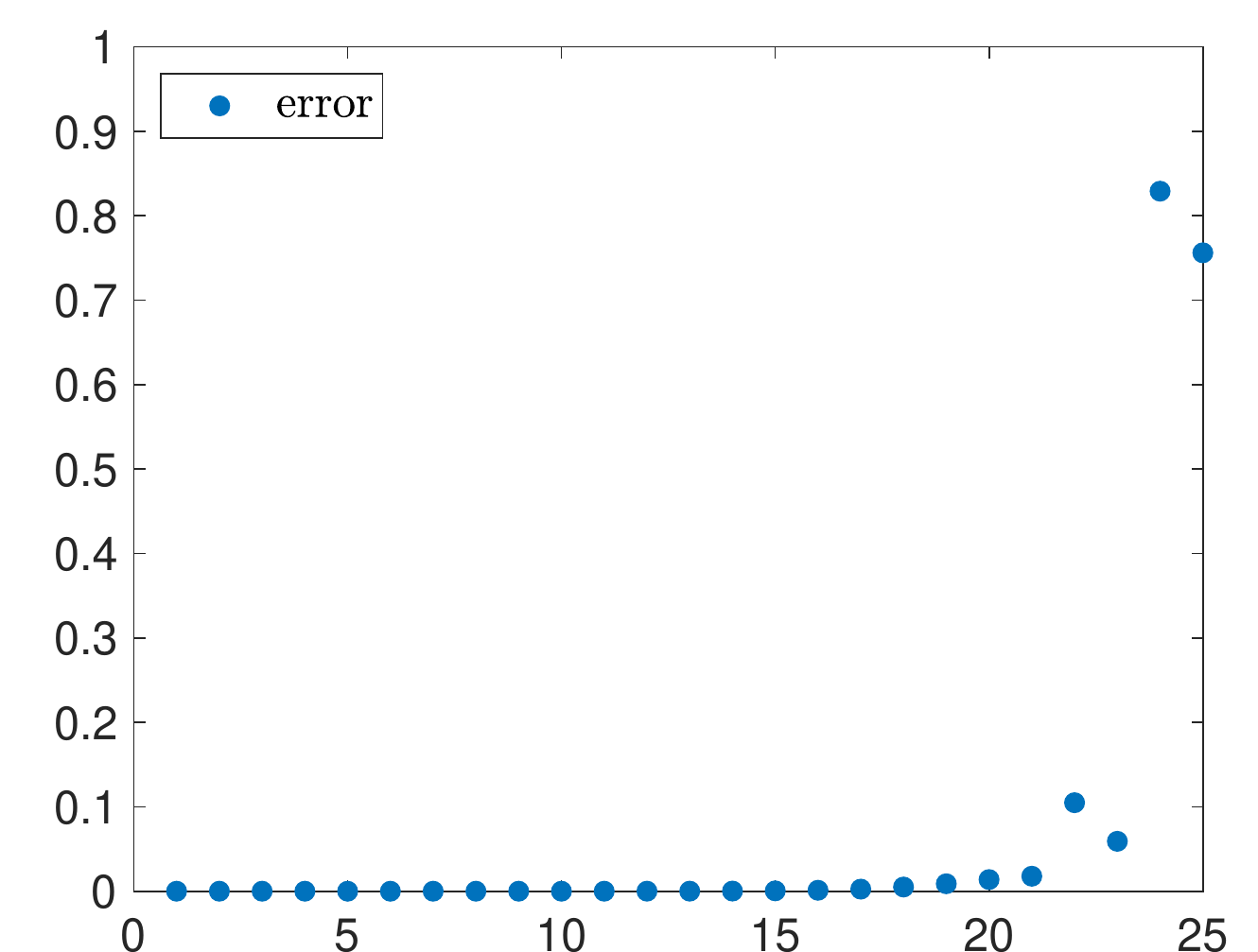}}\hspace*{0.1cm}
\subfigure[$n_h=50$]{\includegraphics[height=4.1cm]{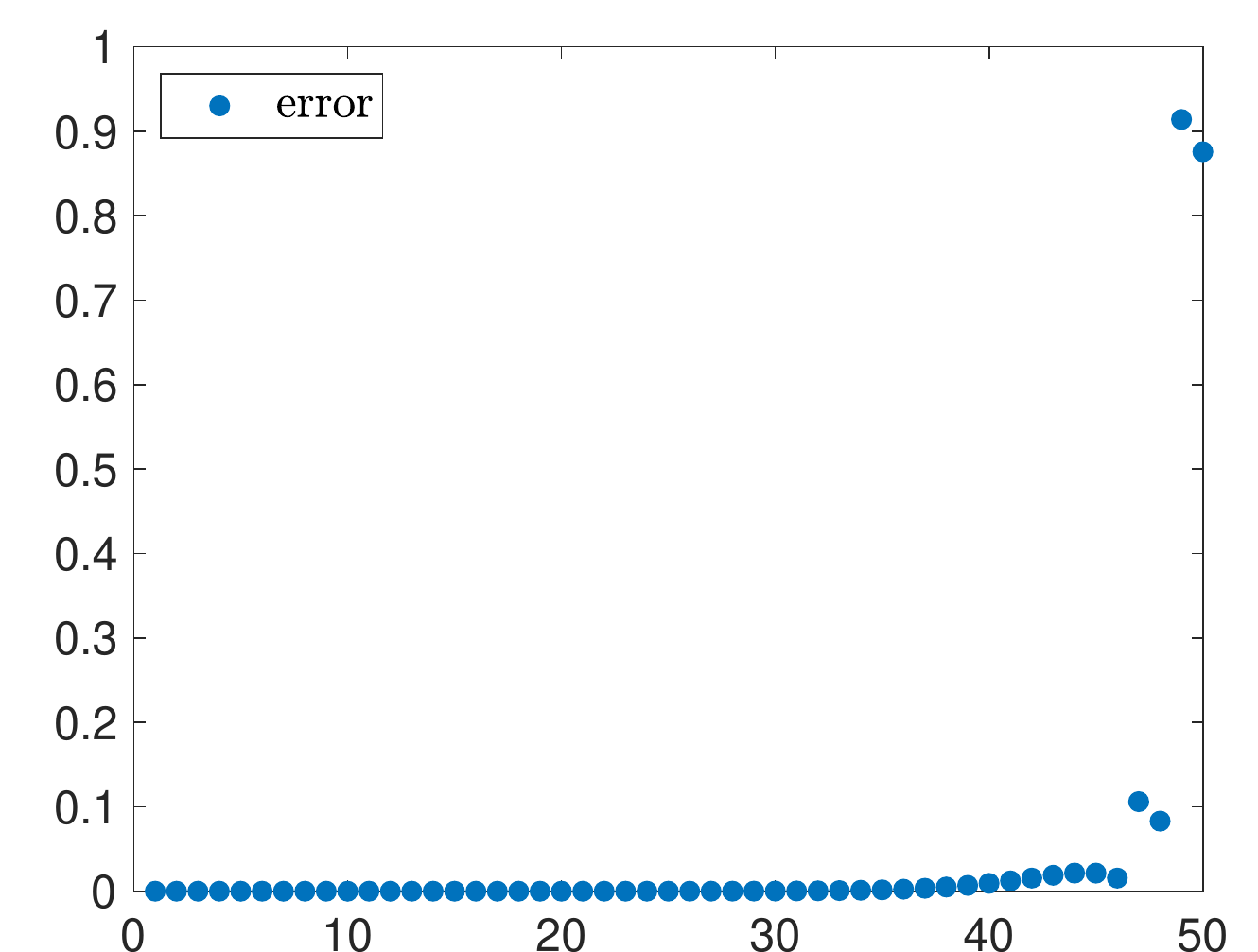}}\hspace*{0.1cm}
\subfigure[$n_h=100$]{\includegraphics[height=4.1cm]{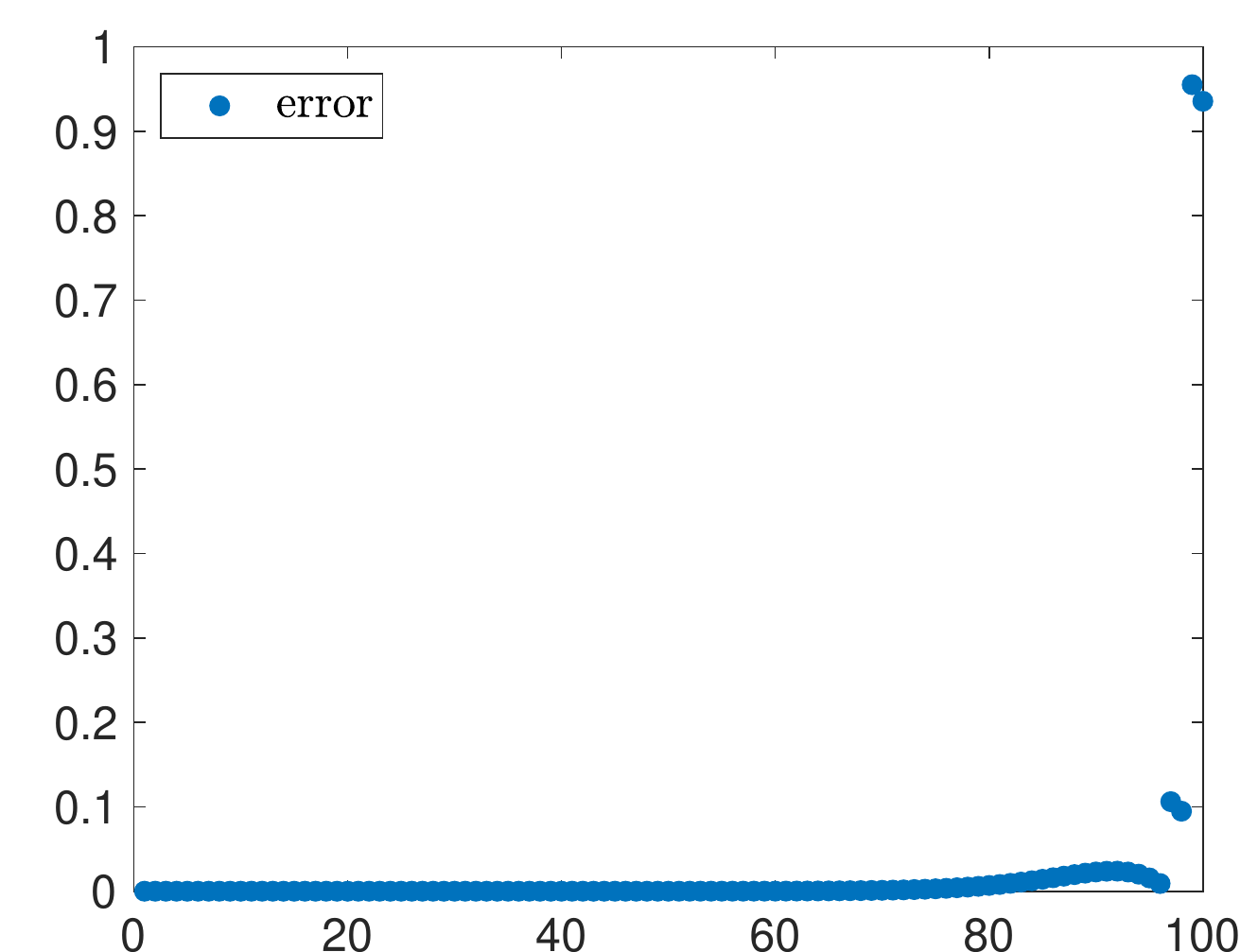}}
\caption{Top row (a)--(c): Exact frequencies and their approximations corresponding to the maximally smooth spline spaces $\mathbb{S}_{p,\bftau,0}^0$ of degree $p=5$ defined on uniform grids and dimensions $n_h=25,50,100$. Bottom row (d)--(f): Relative frequency errors for the same approximations.
In this case there are clearly four spurious (overestimated) values, called \emph{outliers}, independently of the dimension $n_h$.} \label{fig:demo-outliers}
\end{figure}

The presence of a small number of outliers for such spline spaces is in a close agreement with the available theoretical error estimates. More precisely, from \cite[Section~8]{Boffi:2010} or \cite[Chapter~6]{Strang:2008} we know that the error between the eigenfunction $u_\indeigk$ in \eqref{eq:eig-Laplace-type-0-BC} and its approximation $u_{h,\indeigk}$ given by the Galerkin method in \eqref{eq:approx-eigenfun} can be bounded by the error between $u_\indeigk$ and $Q_p^1u_\indeigk$ defined in Theorem~\ref{thm:Qspline}. Since $u_\indeigk\in C^\infty$, using Theorem~\ref{thm:Qspline} with $r=p+1$ we deduce
\begin{equation*}
\|u_\indeigk- u_{h,\indeigk}\| \leq C_\indeigk \left(\frac {h}{\pi}\right)^{p+1}\|u^{(p+1)}\|
=\frac{C_\indeigk}{\sqrt{2}} (h\indeigk)^{p+1},
\end{equation*}
for some constant $C_l$.
Therefore, convergence in $p$ of the approximate eigenfunction $u_{h,\indeigk}$ is ensured whenever
$$
h\indeigk<1.
$$
For fixed dimension of the approximation space, i.e., for fixed $\nknots$ and $p$, the value of $h$ is minimized when the break points are uniformly distributed, 
\begin{equation}\label{eq:knots-uniform}
  \tau_i=\frac{i}{\nknots}, \quad i=0,\ldots,\nknots,
\end{equation}
and so 
$$h=\frac{1}{\nknots}.$$
Thus, under the assumption of uniform grid spacing, convergence in $p$ of the approximated eigenfunction $u_{h,\indeigk}$ to the exact eigenfunction $u_\indeigk$ is ensured for
$$ \indeigk=1, \ldots, \nknots-1.$$
The arguments in \cite[Section~8]{Boffi:2010}, see also Remark~\ref{rmk:stab} (in Appendix~\ref{Appendix:A}) and \eqref{eq:stab}, show that convergence is ensured also for the corresponding eigenvalues.

A similar discussion can be carried out for natural and mixed boundary conditions. For the former the natural discretization space is the full spline space $\mathbb{S}_{p,\bftau}$,
while for the latter the space 
\begin{equation}
 \label{eq:space-left}
 \{s\in \mathbb{S}_{p,\bftau}: s(0)=0 \}
\end{equation}
has to be considered. Taking into account that the dimension of such spaces amounts to $\nknots+p$ and $\nknots+p-1$, respectively, and keeping in mind the expressions of the eigenfunctions in \eqref{eq:eig-Laplace-type-1-BC} and \eqref{eq:eig-Laplace-type-2-BC}, we can summarize the above discussion as follows.

\medskip
\begin{tcolorbox}
Consider problem \eqref{eq:eigenvalue-equation-1D} with boundary conditions \eqref{type-0-BC}, \eqref{type-1-BC} or \eqref{type-2-BC}.
Let $\mathbb{V}_{h}$ be the spline space \eqref{eq:space-BC}, \eqref{eq:spline-max-smooth} or \eqref{eq:space-left}, respectively, with $\bftau$ defined as in \eqref{eq:knots-uniform}.
The approximations of the eigenvalues obtained by
finding $u\in \mathbb{V}_{h}$ and $\omega^2\in\RR$ such that
\begin{equation*}
( u', v') = \omega^2 (u, v), \quad \forall v\in \mathbb{V}_{h},
\end{equation*}
incur at most the following number of outliers:
\begin{itemize}
   \item $p-1$ for Dirichlet boundary conditions;
   \item $p$ for Neumann boundary conditions;
   \item $p-1$ for mixed boundary conditions.
\end{itemize}
\end{tcolorbox}

\pagebreak
\begin{remark}
According to \cite[Table~3]{Hiemstra:2021} the number of outliers observed numerically is
\begin{itemize}
   \item $2\lfloor \frac {p-1}{2}\rfloor$ for Dirichlet boundary conditions;
   \item $2\lfloor \frac {p}{2}\rfloor$ for Neumann boundary conditions;
   \item $ p-1$ for mixed boundary conditions.
\end{itemize}
Our theoretical upper bounds on the number of outliers are closely related to the numerically observed ones and exhibit an exact match in the majority of the cases. 
\end{remark}

Outlier-free discretizations can be achieved by identifying spline subspaces of dimension $n_h$ that ensure $L^2$ and $H^1$ convergence in $p$ for the approximations of the first $n_h$ eigenfunctions of the problem we are dealing with.
In the next section we propose a solution to this problem based on optimal spline spaces.
 
\section{Optimal spline spaces have no outliers}
\label{sec:outlier-free}
In this section we extend the results of \cite[Section~7]{Sande:2019} to prove that there are no outliers in the Galerkin eigenvalue approximation for the Laplacian with various boundary conditions when using spline subspaces that are optimal with respect to the Kolmogorov $n$-width in $L^2$-norm.
 
We denote by $X$ the $L^2$-projector onto a finite-dimensional subspace $\mathbb{X}$ of $L^2$.
For a subset $A$ of $L^2$, let
$$ E(A, \mathbb{X}) := \sup_{u \in A} \|u-Xu\| $$
be the distance to $A$ from $\mathbb{X}$ relative to the $L^2$-norm.
Then, the Kolmogorov $L^2$ $n$-width
of $A$ is defined by
$$ d_n(A) := \inf_{\substack{\mathbb{X}\subset L^2\\ \dim \mathbb{X}=n}} E(A, \mathbb{X}). $$
If $\mathbb{X}$ has dimension at most $n$ and satisfies
\begin{equation*}
d_n(A) = E(A, \mathbb{X}),
\end{equation*}
then we call $\mathbb{X}$ an \emph{optimal} subspace for $d_n(A)$.

\begin{example}
	Let $A=\{u\in H^r : \|u^{(r)}\|\leq 1\}$.
	Then, by looking at $u/\|u^{(r)}\|$, for functions $u\in H^r$, we have for any subspace $\mathbb{X}$ of $L^2$, the sharp estimate
	\begin{equation*}
	\|u-Xu\|\leq E(A, \mathbb{X})\|u^{(r)}\|.
	\end{equation*}
	Here $E(A, \mathbb{X})$ is the smallest possible constant for the subspace $\mathbb{X}$.
	Moreover, if $\mathbb{X}$ is optimal for the $n$-width of $A$, then
	\begin{equation}\label{ineq:optimal}
	\|u-Xu\|\leq d_n(A)\|u^{(r)}\|,
	\end{equation}
	and $d_n(A)$ is the smallest possible constant over all $n$-dimensional subspaces $\mathbb{X}$.
\end{example}

For all $p\geq r-1$, let us consider the function classes
 \begin{equation}\label{eq:Hr}
 \begin{aligned}
 H^r_0&:=\{u\in H^r :\, u^{(\alpha)}(0)=u^{(\alpha)}(1)=0,\ \ 0\leq \alpha<r,\ \ \alpha \text{ even}\},
 \\
 H^r_1&:=\{u\in H^r :\, u^{(\alpha)}(0)=u^{(\alpha)}(1)=0,\ \ 0\leq \alpha<r,\ \ \alpha \text{ odd}\},
 \\
 H^r_2&:=\{u\in H^r :\, \partial^{\alpha_0} u(0)=\partial^{\alpha_1} u(1)=0,\ \ 0\leq \alpha_0,\alpha_1<r,\ \
 \alpha_0 \text{ even}, \ \ \alpha_1 \text{ odd}\},
 \end{aligned}
 \end{equation}
and
\begin{equation*}
\begin{aligned}
A^r_0&:=\{u\in H^r_0:\, \|u^{(r)}\|\leq 1\},
\\
A^r_1&:=\{u\in H^r_1:\, \|u^{(r)}\|\leq 1\},
\\
A^r_2&:=\{u\in H^r_2:\, \|u^{(r)}\|\leq 1\}.
\end{aligned}
\end{equation*}
By using the representation of these function classes in terms of repeated applications of suitable integral operators, it has been shown \cite{Pinkus:85} that the $n$-dimensional space
consisting of the first $n$ eigenfunctions of the Laplacian satisfying the boundary conditions \eqref{type-0-BC}--\eqref{type-2-BC}  (i.e., the first $n$ functions in each of the sequences \eqref{eq:eig-Laplace-type-0-BC}--\eqref{eq:eig-Laplace-type-2-BC}) are optimal for $A^r_i$, $i=0,1,2$, respectively.
Moreover,
 \begin{equation} \label{eq:n-width}
 d_n(A^r_0)=\left(\frac{1}{(n+1)\pi}\right)^r, 
 \quad
 d_n(A^r_1)=\left(\frac{1}{n\pi}\right)^r,
 \quad
 d_n(A^r_2)=\left(\frac{2}{(2n+1)\pi}\right)^r.
\end{equation}

For $0\leq \ell\leq p$ let us now consider
the following subspaces of $\mathbb{S}_{p,\bftau}$ identified by certain derivatives vanishing at the boundary:
\begin{equation} \label{eq:allS}
\begin{aligned}
\mathbb{S}_{p,\bftau,0}^\ell &:= \{s\in \mathbb{S}_{p,\bftau} :\, s^{(\alpha)}(0)=s^{(\alpha)}(1)=0,\ \ 0\leq \alpha\leq \ell,\ \ \alpha \text{ even}\}, \\
\mathbb{S}_{p,\bftau,1}^\ell &:= \{s\in \mathbb{S}_{p,\bftau} :\, s^{(\alpha)}(0)=s^{(\alpha)}(1)=0,\ \ 0\leq \alpha\leq\ell,\ \ \alpha \text{ odd}\}, \\
\mathbb{S}_{p,\bftau,2}^\ell&:= \{s\in \mathbb{S}_{p,\bftau} :\, s^{(\alpha_0)}(0)= s^{(\alpha_1)}(1)=0,\ \ 0\leq \alpha_0,\alpha_1\leq \ell, \ \ \alpha_0 \text{ even}, \ \ \alpha_1 \text{ odd}\}.
\end{aligned}
\end{equation}
With the aim of constructing optimal spline spaces we focus on the special (degree-dependent) sequences of break points $\bftau_{p,n,i}^\opt$, $i=0,1,2$, where
\begin{equation*} 
\begin{aligned}
\bftau_{p,n,0}^\opt &:= \begin{cases}
\left(0,\frac{1\vphantom{1/2}}{n+1},\frac{2}{n+1},\ldots,\frac{n}{n+1},1\right),\quad\ &p \text{ odd},\\[0.2cm]
\left(0,\frac{1/2}{n+1},\frac{3/2}{n+1},\ldots,\frac{n+1/2}{n+1},1\right),\quad\ &p \text{ even},
\end{cases}\\[0.2cm]
\bftau_{p,n,1}^\opt &:= \begin{cases}
\left(0,\frac{1/2}{n},\frac{3/2}{n},\ldots,\frac{n-1/2}{n},1\right),\qquad &p \text{ odd},\\[0.2cm]
\left(0,\frac{1\vphantom{1/2}}{n},\frac{2}{n},\ldots,\frac{n-1}{n},1\right),\qquad &p \text{ even},
\end{cases}\\[0.2cm]
\bftau_{p,n,2}^\opt &:= \begin{cases}
\left(0,\frac{2\vphantom{1/2}}{2n+1},\frac{4}{2n+1},\ldots,\frac{2n}{2n+1},1\right),\quad &p \text{ odd},\\[0.2cm]
\left(0,\frac{1\vphantom{1/2}}{2n+1},\frac{3}{2n+1},\ldots,\frac{2n-1}{2n+1},1\right),\quad &p \text{ even},
\end{cases}
\end{aligned}
\end{equation*}
and the spaces
\begin{equation}
\label{eq:opt-spaces}
\mathbb{S}_{p,n,i}^\opt:=\mathbb{S}_{p,\bftau_{p,n,i}^\opt,i}^p, \quad i=0,1,2.
\end{equation}
We also define the corresponding grid sizes
$$
h_{p,n,0}^\opt:=\frac{1}{n+1}, \quad  h_{p,n,1}^\opt:=\frac{1}{n}, \quad h_{p,n,2}^\opt:=\frac{2}{2n+1}.
$$
Note that all the spaces $\mathbb{S}_{p,n,i}^\opt$, $i=0,1,2$ have dimension $n$ and from \eqref{eq:n-width} we get 
\begin{equation}
\label{eq:h-nwidths}
d_n(A_i^r)=\left(\frac{ h_{p,n,i}^\opt}{\pi}\right)^r, \quad i=0,1,2.
\end{equation}
Actually, it was shown in \cite[Theorem~2]{Floater:2018} that for all $p\geq r-1$ and $r\geq 1$ the spline spaces 
$\mathbb{S}_{p,n,i}^\opt$ are optimal for the function classes $ A^r_i$, $ i=0,1,2$, respectively. 
Therefore, \eqref{ineq:optimal} and \eqref{eq:h-nwidths} immediately give
 the following result.
 \begin{theorem}
 \label{thm:L2-error}
Let $S_{p,n,i}^\opt$ be the $L^2$-projector onto
$\mathbb{S}_{p,n,i}^\opt$, $i=0,1,2$. Then, for $u\in H^r_i$ and $p\geq r-1$ we have
 \begin{equation*} 
 \|u-S_{p,n,i}^\opt u\| \leq \left(\frac{h_{p,n,i}^\opt}{\pi}\right)^{r}\|u^{(r)}\|.
 \end{equation*}
\end{theorem}
In particular, we have convergence in $p$ in $L^2$-norm of the $L^2$-projections of the functions
\begin{equation*} 
\begin{gathered}
\{\sin(\pi x), \sin(2\pi x), \ldots,\sin(n\pi x)\},\\
\{1,\cos(\pi x), \ldots,\cos((n-1)\pi x)\},\\
\{\sin((1/2)\pi x), \sin((3/2)\pi x), \ldots, \sin((n-1/2)\pi x)\}
\end{gathered}
\end{equation*}
onto $\mathbb{S}_{p,n,i}^\opt$, $i=0,1,2$, respectively; see also \cite{Sande:2019}.

To establish convergence of the standard Galerkin approximation we need the following theorem about error estimates for derivatives, analogous to Theorem~\ref{thm:L2-error}. It generalizes the results proved in \cite[Theorem~4.1]{Sande:2019} for the less technical periodic case towards the subspaces $ \mathbb{S}_{p,n,i}^\opt$, $i=0,1,2$.

\begin{theorem}
	\label{thm:error-der}
	Let $u\in H^r_i$, $i=0,1,2$ for $r\geq 1$ be given. Then, for all $p\geq \max\{r-1,1\}$ there exists $R_{p,n,i}^\opt u\in\mathbb{S}_{p,n,i}^\opt$ such that 
\begin{equation*} 
\|(u-R_{p,n,i}^\opt u)^{(\ell)}\| \leq \left(\frac{h_{p,n,i}^\opt}{\pi}\right)^{r-\ell}\|u^{(r)}\|, \quad \ell=0,1.
\end{equation*}
\end{theorem}
\begin{proof}
The result follows from Example~\ref{ex:our-classes} and Proposition~\ref{prop:err} (in Appendix~\ref{Appendix:A}) with
\begin{equation*}
\begin{aligned}
R_{p,n,0}^\opt :=&R_{\mathbb{Y}_p},  &&\mathbb{S}_{p,n,0}^\opt =\mathbb{Y}_p,
\\
R_{p,n,1}^\opt :=&P_0\oplus R_{\mathbb{X}_p},\ \  &&\mathbb{S}_{p,n,1}^\opt=\PP_0\oplus\mathbb{X}_p,
\\
R_{p,n,2}^\opt :=&R_{\mathbb{X}_p},  &&\mathbb{S}_{p,n,2}^\opt=\mathbb{X}_p,
\end{aligned}
\end{equation*}
taking into account \eqref{eq:h-nwidths} and \eqref{eq:nwidth-r}. 
Here $P_0$ stands for the $L^2$-projector onto $\PP_0$.
\end{proof}

Then, Corollaries~\ref{cor:Strang} and \ref{cor:galerkin-A} (in Appendix~\ref{Appendix:A}) imply the following result.

\begin{proposition}\label{pro:error-eigvals}
For any $\indeigk=1,\dots,n$, let $u_\indeigk, \omega_\indeigk$ be an exact solution of problem \eqref{eq:eigenvalue-equation-1D} with boundary conditions \eqref{type-0-BC}, \eqref{type-1-BC} or \eqref{type-2-BC},
and let $u_{h,\indeigk},\omega_{h,\indeigk}$ be their approximation 
obtained by
finding $u\in \mathbb{S}_{p,n,i}^\opt$ and $\omega^2\in\RR$ such that
\begin{equation*}
( u', v') = \omega^2 (u, v), \quad \forall v\in \mathbb{S}_{p,n,i}^\opt,
\end{equation*}
for $i=0,1,2$, respectively. Then,
\begin{equation}\label{ineq:eigvals}
\omega_\indeigk \leq \omega_{h,\indeigk} \leq \frac{\omega_\indeigk}{1-\left(\frac{\omega_\indeigk}{\omega_{n+1}}\right)^{p+1}}.
\end{equation}
Moreover,
\begin{equation*}
\frac{\|u_\indeigk-u_{h,\indeigk}\|}{\|u_\indeigk\|} \leq 2(1+\rho_\indeigk) \left(\frac{\omega_\indeigk}{\omega_{n+1}}\right)^{p+1},
\end{equation*}
where $\|u_\indeigk\|=\|u_{h,\indeigk}\|$, $(u_\indeigk,u_{h,\indeigk})>0$ and
\begin{equation*}
\rho_\indeigk := \max_{\substack{i=1,\ldots,n \\ i\neq \indeigk}} \frac{\omega^2_\indeigk}{|\omega^2_\indeigk - \omega^2_{h,i}|}
\end{equation*}
is the $\indeigk$-th separation constant.
\end{proposition}

Observe that the error estimate for the approximation of the exact frequencies $\omega_\indeigk$ by the $\omega_{h,\indeigk}$ in \eqref{ineq:eigvals} only depends on the value of the exact frequencies $\omega_\indeigk$ and $\omega_{n+1}$. Since $\omega_n<\omega_{n+1}$, we have that $\omega_{h,\indeigk}\to\omega_\indeigk$ for all $\indeigk =1,\ldots,n$ as the degree $p$ of the optimal spline spaces increases. Moreover, $\omega_{n+1}<\omega_{n+2}<\dots$, so for a fixed index $\indeigk$ we have that $\omega_{h,\indeigk}\to\omega_\indeigk$ as the  grid size $h$ decreases, since $h\to 0$ is equivalent to $n\to\infty$. See Remark~\ref{rmk:eigval-rate} for a further discussion on this.

The error of the eigenfunctions in $H^1$-seminorm can be deduced from Proposition~\ref{pro:error-eigvals} and the so-called \emph{Pythagorean eigenvalue error theorem} \cite[page~233]{Strang:2008}:
\begin{equation*}
\frac{\|(u_\indeigk-u_{h,\indeigk})'\|^2}{\| u'_\indeigk\|^2}= \frac{\|u_\indeigk-u_{h,\indeigk}\|^2}{\|u_\indeigk\|^2} + \frac{\omega^2_{h,\indeigk}-\omega^2_\indeigk}{\omega^2_{\indeigk}},
\end{equation*}
where $\|u_\indeigk\|=\|u_{h,\indeigk}\|$ and $(u_\indeigk,u_{h,\indeigk})>0$.

We can summarize the above results as follows.

\medskip
\begin{tcolorbox}
Consider problem \eqref{eq:eigenvalue-equation-1D} with boundary conditions \eqref{type-0-BC}, \eqref{type-1-BC} or \eqref{type-2-BC}.
The approximations of the eigenvalues obtained by finding $u\in \mathbb{S}_{p,n,i}^\opt$ and $\omega^2\in\RR$ such that
\begin{equation*}
( u', v') = \omega^2 (u, v), \quad \forall v\in \mathbb{S}_{p,n,i}^\opt,
\end{equation*}
for $i=0,1,2$, respectively, have no outliers.
\end{tcolorbox}

\begin{remark}\label{rmk:other-reduced-space}
The subspaces $\mathbb{S}_{p,\bftau,i}^{p-1}$, $i=0,1$, introduced for uniform knot sequences in \cite{Sogn:2018,Takacs:2016} and further analyzed in \cite[Section~5.2]{Sande:2020}, were considered for outlier removal in \cite{Hiemstra:2021}. Observe that $\mathbb{S}_{p,\bftau,0}^p \subseteq {\mathbb{S}}_{p,\bftau,0}^{p-1}$ where equality holds for $p$ odd and that $\mathbb{S}_{p,\bftau,1}^p \subseteq {\mathbb{S}}_{p,\bftau,1}^{p-1}$ where equality holds for $p$ even; see their definitions in \eqref{eq:allS}. Hence, for a fixed type of boundary condition (Dirichlet/Neumann/mixed) these subspaces are very similar to our optimal spaces but can slightly differ in the partition and in the maximum order of vanishing derivatives at the boundary, depending on the parity of the degree $p$.
 In Section~\ref{sec:numerics} we will also numerically illustrate their performance with respect to outliers.
\end{remark}

\begin{remark}\label{rmk:eigval-rate}
While the error estimate in \eqref{ineq:eigvals} is sufficient to deduce a good approximation of all $n$ eigenvalues of the Laplacian with any of the boundary conditions \eqref{type-0-BC}, \eqref{type-1-BC} or \eqref{type-2-BC}, the asymptotic rate of convergence is not sharp. From Corollary \ref{Cor:RT} we can get the error estimate
	\begin{equation}\label{ineq:eigvals-RT}
	\omega_\indeigk \leq \omega_{\prec,\indeigk} \leq \frac{\omega_\indeigk}{\left(1-2\sqrt{\indeigk}\left(\frac{\omega_\indeigk}{\omega_1}\right)^{2}\left(\frac{\omega_\indeigk}{\omega_{n+1}}\right)^{2p}\right)^{1/2}},
	\end{equation}
for all choices of $\indeigk$ and $p$ such that
\begin{align*}
\sqrt{\indeigk}\left(\frac{\omega_\indeigk}{\omega_1}\right)^{2}\left(\frac{\omega_\indeigk}{\omega_{n+1}}\right)^{2p}<\frac{1}{2}.
\end{align*}
The estimate in \eqref{ineq:eigvals-RT} gives a sharper asymptotic rate of convergence as $h\to0$ (i.e., as $n\to\infty$) or as $p\to\infty$. However, since we are in this paper mainly concerned with the outliers, it is more relevant to consider the case when $\indeigk$ is large (i.e., close to $n$), and in this case estimate \eqref{ineq:eigvals-RT} is only applicable for very large degrees $p$.
\end{remark}

We end the section by extending the univariate results towards higher dimensions.
Let us consider tensor-product spline spaces of the form
\begin{equation}\label{eq:spline-tensor}
\mathbb{S}_{\bfp,\bfn,\bfi}^\opt:=\mathbb{S}_{p_1,n_1,i_1}^\opt\otimes\mathbb{S}_{p_2,n_2,i_2}^\opt \otimes\cdots\otimes\mathbb{S}_{p_d,n_d,i_d}^\opt, \quad i_1,\ldots,i_d=0,1,2,
\end{equation}
according to the type of boundary conditions. Thanks to the decomposition of the approximate eigenfunctions/eigenvalues like in \eqref{eq:approx-eig-Laplace2D} and the exact ones like in \eqref{eq:eig-Laplace2D}, both in terms of their univariate counterparts, we immediately arrive at the following multivariate result.

\medskip
\begin{tcolorbox}
Consider problem \eqref{eq:eigenvalue-equation} with homogeneous Dirichlet/Neumann/mixed boundary conditions.
The approximations of the eigenvalues obtained by
finding $u\in \mathbb{S}_{\bfp,\bfn,\bfi}^\opt$ and $\omega^2\in\RR$ such that
\begin{equation*}
( \nabla u, \nabla v) = \omega^2 (u, v), \quad \forall v\in \mathbb{S}_{\bfp,\bfn,\bfi}^\opt,
\end{equation*}
for appropriate choices of $i_1,\ldots,i_d$ according to the type of boundary conditions, have no outliers.
\end{tcolorbox}


\section{Approximations with non-homogeneous boundary} 
\label{sec:general-BC}
The optimal spline subspaces in \eqref{eq:opt-spaces} 
provide outlier-free approximations for the eigenfunctions of the corresponding eigenvalue problems still maintaining the accuracy of the full spline space because the additional boundary conditions identifying such subspaces are satisfied by the exact eigenfunctions. However, it is clear that their direct use will result in a loss of accuracy for approximating solutions that do not satisfy those boundary conditions; see also Examples~\ref{ex:convergence1D-boundary} and \ref{ex:convergence2D-boundary}.

Let us focus on the problem 
\begin{equation}
\label{eq:Laplace}
	\left\{ \begin{aligned}
	-\Delta u &= f, \quad \text{in } \Omega:=(0,1)^d, \\
	u_{|\partial \Omega} &= 0,
	\end{aligned} \right.
\end{equation}
and on approximations of its solution obtained by the Galerkin method using (tensor products of) the outlier-free subspace $\mathbb{S}_{p,n,0}^\opt$. 
Given a smooth right-hand side $f$, in this section we propose and analyze a possible strategy to recover the full approximation order of the usual spline subspace of $H^1_0(\Omega)$, i.e., (tensor products of) $\mathbb{S}_{p,\bftau,0}^0$ for some $\bftau$.
A special case of this approach has also been suggested in \cite[Appendix~B]{Hiemstra:2021} for the one-dimensional problem with $f=1$. Here we consider any dimension and any sufficiently smooth $f$.
The treatment of other types of boundary conditions is similar.

\subsection{Univariate case}\label{sec:general-BC-1D}
We first consider the case $d=1$. 
We recall that the optimal spline space $\mathbb{S}_{p,n,i}^\opt$ defined in \eqref{eq:opt-spaces} has full approximation power for functions taken from the space $H^r_i$ for $i=0,1,2$; see Theorem~\ref{thm:error-der} (and also Theorem~\ref{thm:L2-error}). In the following we are addressing functions $u$ that do not satisfy the boundary conditions on (even) derivatives of the space $H^r_0$. As mentioned before, the other types of boundary conditions can be treated in a similar way.

For sufficiently smooth $f$, 
let $s_u\in \mathbb{S}_{p,\bftau,0}^0$ be such that for even values of $\alpha$, $2\leq \alpha\leq p,$  
 \begin{align*}
 (s_u)^{(\alpha)}(0) &= (u)^{(\alpha)}(0)=-f^{(\alpha-2)}(0), 
 \\
 (s_u)^{(\alpha)}(1) &= (u)^{(\alpha)}(1)=-f^{(\alpha-2)}(1). 
\end{align*}
We can then write the solution of \eqref{eq:Laplace} as
 $$
 u=u_0+s_u,
 $$
where $u_0$ solves the problem
\begin{equation*}
 \left\{ \begin{aligned}
 - u''_0  &= f+s''_u, \quad \text{in } (0,1), \\
 u_{0} (0) &= u_0(1)=0,
 \end{aligned} \right.
\end{equation*}
and
$$
 (u_0)^{(\alpha)}(0)=(u_0)^{(\alpha)}(1)=0, \quad 0\leq \alpha\leq p,\ \ \alpha  \text{ even}.
$$
Thus, it suffices to construct the Galerkin approximation of $u_0$ in $\mathbb{S}_{p,n,0}^\opt$. The approximation power of the space $\mathbb{S}_{p,n,i}^\opt$ (see Theorem~\ref{thm:error-der}) ensures no loss of approximation order with respect to the Galerkin approximation of $u$ in the usual space $\mathbb{S}_{p,\bftau,0}^0$.
 
For sequences of break points with $\nknots> p+1$, the construction of $s_u$ is straightforward. As an example, it can be obtained by solving the following Hermite interpolation problem: find
$$
   s_u=\sum_{i=-p}^0 c_iN_{i,\bfxi}^p+\sum_{i=\nknots-p-1}^{\nknots-1} c_iN_{i,\bfxi}^p, \\
$$
such that
\begin{alignat}{3}
   \label{eq:su-even}
    (s_u)^{\alpha}(z)&=u^{(\alpha)}(z), \quad &\alpha&=0,2,4,\ldots, 2\lfloor \tfrac {p}{2}\rfloor, \quad &z&=0,1,
    \\
    \label{eq:su-odd}
    (s_u)^{\alpha}(z)&=0,\quad &\alpha&=1,3,5,\ldots, 2\lfloor \tfrac {p-1}{2}\rfloor+1,\quad  &z&=0,1,
\end{alignat}
where the B-splines $N_{i,\bfxi}^p$ are defined in Section~\ref{sec:basis-full}. Due to the support property of B-splines, for $\nknots> p+1$ the above problem decouples in two independent linear systems of size $p+1$ each. These two linear systems are unisolvent because each of them corresponds to Taylor interpolation in the polynomial space $\PP_p$. Moreover, when using an open knot sequence $\bfxi$, the above systems are triangular. Although of limited practical interest, we remark that a function $s_u\in \mathbb{S}_{p,\bftau}$ satisfying \eqref{eq:su-even} can be obtained also in case $\nknots\leq p+1$ by removing some of the conditions on the odd derivatives in order to match the dimension of the space.

\subsection{Multivariate tensor-product case}\label{sec:general-BC-2D}
In the multivariate setting, we start by showing that the tensor-product spline space $\mathbb{S}_{\bfp,\bfn,\bfi}^\opt$ defined in \eqref{eq:spline-tensor} has full approximation power for functions taken from the tensor-product space
\begin{equation}\label{eq:H-tensor}
H^r_{\bfi}:=H^r_{i_1} \otimes H^r_{i_2}\otimes\cdots\otimes H^r_{i_d}.
\end{equation}
We let $\Omega:=(0,1)^d$ and denote the $L^2$-norm on $\Omega$ by $\|\cdot\|_{\Omega}$. We define the grid size as 
$$
h_{\bfp,\bfn,\bfi}^\opt:=\max\left\{h_{p_1,n_1,i_1}^\opt,h_{p_2,n_2,i_2}^\opt,\ldots,h_{p_d,n_d,i_d}^\opt\right\}.
$$
For simplicity we state the following theorem in the case $d=2$.
From Proposition~\ref{prop:tensorRitz} (in Appendix~\ref{Appendix:B}) we can deduce the following generalization of Theorem~\ref{thm:error-der} to the tensor-product case. 
\begin{theorem}\label{thm:error-der-2d}
Let $u\in H^r_{i_1}\otimes H^r_{i_2}$, $i_1,i_2=0,1,2$ for $r\geq 1$ be given. Then, for all $p_1,p_2\geq \max\{r-1,1\}$ there exists $R_{p_1,n_1,i_1}^\opt\otimes R_{p_2,n_2,i_2}^\opt u\in\mathbb{S}_{p_1,n_1,i_1}^\opt\otimes \mathbb{S}_{p_2,n_2,i_2}^\opt$ such that 
\begin{equation*} 
\|u-R_{p_1,n_1,i_1}^\opt\otimes R_{p_2,n_2,i_2}^\opt u\|_{\Omega} \leq \left(\frac{h_{\bfp,\bfn,\bfi}^\opt}{\pi}\right)^{r}\left(\|\partial_1^ru\|_{\Omega}+\|\partial_2^ru\|_{\Omega} 
 +\min\left\{\|\partial_1\partial_2^{r-1}u\|_{\Omega}, \,\|\partial_1^{r-1}\partial_2u\|_{\Omega}\right\}\right),
\end{equation*}
and
\begin{equation*}
\begin{aligned}
\|\partial_1(u-R_{p_1,n_1,i_1}^\opt\otimes R_{p_2,n_2,i_2}^\opt u)\|_{\Omega}
&\leq \left(\frac{h_{\bfp,\bfn,\bfi}^\opt}{\pi}\right)^{r-1}\left(\|\partial_1^ru\|_{\Omega}+\|\partial_1\partial_2^{r-1}u\|_{\Omega}\right), 
\\
\|\partial_2(u-R_{p_1,n_1,i_1}^\opt\otimes R_{p_2,n_2,i_2}^\opt u)\|_{\Omega}
&\leq \left(\frac{h_{\bfp,\bfn,\bfi}^\opt}{\pi}\right)^{r-1}\left(\|\partial_1^{r-1}\partial_2u\|_{\Omega}+\|\partial_2^ru\|_{\Omega}\right).
\end{aligned}
\end{equation*}
\end{theorem}
\begin{proof}
For simplicity of notation, we will only consider the case $p_1=p_2=p$, $n_1=n_2=n$ and $i_1=i_2=i$.
In this case the result follows from Example~\ref{ex:our-classes} (in Appendix~\ref{Appendix:A}) and Proposition~\ref{prop:tensorRitz} (in Appendix~\ref{Appendix:B}) with
\begin{equation*}
\begin{aligned}
R_{p,n,0}^\opt :=&R_{\mathbb{Y}_p},  &&\mathbb{S}_{p,n,0}^\opt =\mathbb{Y}_p,
\\
R_{p,n,1}^\opt :=&P_0\oplus R_{\mathbb{X}_p},\quad  &&\mathbb{S}_{p,n,1}^\opt=\PP_0\oplus\mathbb{X}_p,
\\
R_{p,n,2}^\opt :=&R_{\mathbb{X}_p},  &&\mathbb{S}_{p,n,2}^\opt=\mathbb{X}_p,
\end{aligned}
\end{equation*}
taking into account \eqref{eq:h-nwidths} and \eqref{eq:nwidth-r}. 
\end{proof}
Using standard arguments we can conclude that the Galerkin method applied to problem \eqref{eq:Laplace} in the reduced spline space \eqref{eq:spline-tensor} has full approximation order for any solution belonging to the space \eqref{eq:H-tensor}.
Note that the Ritz projection $R u\in\mathbb{S}_{\bfp,\bfn,\bfi}^\opt$ of $u\in H^1_{\bfi}$ related to the Laplace operator, i.e., given by
\begin{equation*}
(\nabla R u,\nabla v) = (\nabla u, \nabla v), \quad \forall v\in \mathbb{S}_{\bfp,\bfn,\bfi}^\opt,
\end{equation*}
is the best approximation to $u$ in the $H^1$-seminorm and
potentially different from the one considered in Theorem~\ref{thm:error-der-2d}.

The strategy in Section \ref{sec:general-BC-1D} can now be extended to the tensor-product setting. Let us outline the strategy in the case $d=2$ and Dirichlet boundary conditions. In this case we are interested in discretizing problem \eqref{eq:Laplace} in $\mathbb{S}_{p_1,n_1,0}^\opt\otimes \mathbb{S}_{p_2,n_2,0}^\opt$. For $s\in \mathbb{S}_{p_1,n_1,0}^\opt\otimes \mathbb{S}_{p_2,n_2,0}^\opt$ we have
 \begin{align*}
 \partial_{x_1}^{\alpha_1}\partial_{x_2}^{\alpha_2}s(0,x_2) &= \partial_{x_1}^{\alpha_1}\partial_{x_2}^{\alpha_2}s(1,x_2)=0, \quad 0\leq \alpha_1\leq p_1, \ \ \alpha_1 \ \text{even}, \quad 0\leq \alpha_2\leq p_2,
 \\
 \partial_{x_1}^{\alpha_1}\partial_{x_2}^{\alpha_2}s(x_1,0) &= \partial_{x_1}^{\alpha_1}\partial_{x_2}^{\alpha_2}s(x_1,1)=0, \quad 0\leq \alpha_2\leq p_2, \ \ \alpha_2 \ \text{even}, \quad 0\leq \alpha_1\leq p_1.
\end{align*}
We can compensate such behavior by approximating $u$ by $u_{0,h}+s_u$ where $u_{0,h}$ is the Galerkin approximation in $\mathbb{S}_{p_1,n_1,0}^\opt\otimes \mathbb{S}_{p_2,n_2,0}^\opt$ of the solution of
\begin{equation*}
\left\{ \begin{aligned}
-\Delta u_0 &= f+\Delta s_u, \quad \text{in } \Omega:=(0,1)^2, \\
u_{0|\partial \Omega} &= 0,
\end{aligned} \right.
\end{equation*}
and $s_u$ is a function, possibly belonging to 
$\mathbb{S}_{p_1,\bftau_{p_1,n_1,0}^\opt}\otimes \mathbb{S}_{p_2,\bftau_{p_2,n_2,0}^\opt}$, 
that approximates with the same approximation order of the space the corresponding boundary derivatives of $u$, i.e.,
\begin{align*}
 \partial_{x_1}^{\alpha_1}\partial_{x_2}^{\alpha_2}u(0,x_2),\quad & \partial_{x_1}^{\alpha_1}\partial_{x_2}^{\alpha_2}u(1,x_2), \quad 0\leq \alpha_1\leq p_1, \ \ \alpha_1 \ \text{even}, \quad 0\leq \alpha_2\leq p_2,
 \\
 \partial_{x_1}^{\alpha_1}\partial_{x_2}^{\alpha_2}u(x_1,0),\quad &  \partial_{x_1}^{\alpha_1}\partial_{x_2}^{\alpha_2}u(x_1,1), \quad 0\leq \alpha_2\leq p_2, \ \ \alpha_2 \ \text{even}, \quad 0\leq \alpha_1\leq p_1.
\end{align*}
For smooth data $f$, the above derivatives can be easily deduced from \eqref{eq:Laplace} by repeated differentiation. As an example, for $0\leq \alpha_1\leq p_1$, $\alpha_1 \text{ even}$, $0\leq \alpha_2\leq p_2$, we have
$$
  \partial_{x_1}^{\alpha_1}\partial_{x_2}^{\alpha_2}u(z,x_2)=\sum_{r=1}^{\alpha_1/2}(-1)^r\partial_{x_1}^{\alpha_1-2r}\partial_{x_2}^{\alpha_2+2(r-1)} f(z,x_2)+(-1)^{\alpha_1/2}\partial_{x_2}^{\alpha_1+\alpha_2}u(z,x_2), \quad z=0,1.
$$
The last term in the above expression is zero due to the imposed homogeneous Dirichlet boundary conditions in \eqref{eq:Laplace}; this is not the case for general non-homogeneous Dirichlet boundary conditions.
Once the necessary derivatives of $u$ on the boundary of $\Omega$ are computed, a discrete least squares approach in $\mathbb{S}_{p_1,\bftau_{p_1,n_1,0}^\opt}\otimes \mathbb{S}_{p_2,\bftau_{p_2,n_2,0}^\opt}$ can be used to obtain $s_u$.
Note that only the derivatives of the form
 $$
  \partial_{x_1}^{\alpha_1}u(z,x_2), \quad \partial_{x_2}^{\alpha_2}u(x_1,z), \quad z=0,1,\quad 0\leq\alpha_1\leq p_1,\quad 0\leq\alpha_2\leq p_2, \quad \alpha_1,\alpha_2\ \text{even},
 $$
have to be deduced from \eqref{eq:Laplace} and used to identify $s_u$ because they determine all the remaining ones by direct differentiation. 

The above error estimates for the reduced tensor-product spline space ensure no loss of approximation order with respect to the Galerkin approximation of $u$ in the non-reduced space. 
 

\section{B-spline-like bases for outlier-free spline spaces}
\label{sec:bsplines}
In this section we construct a B-spline-like basis for the outlier-free spline spaces we are interested in. 
As explained in Section~\ref{sec:counting-outliers}, uniform grids are the most relevant in the outlier-free context.
For the space $\mathbb{S}_{p,\bftau,1}^{p-1}$, on a uniform knot sequence $\bftau$, such a basis was first presented in \cite{Takacs:2016} for sufficiently small values of $h$ compared to the degree $p$. Subsequently, plots of B-spline-like bases for $\mathbb{S}_{p,n,i}^{\rm opt}$ for all $i=0,1,2$, together with a very brief explanation on how to construct them, were presented in \cite{Floater:2018}. An algorithmic construction in the spirit of MDB-spline extraction was later described in \cite{Hiemstra:2021} for a general knot sequence, but again for sufficiently small values of $h$ compared to the degree $p$. Here we remove the restriction between $h$ and $p$, and consider uniform knot sequences to obtain explicit extraction formulas in terms of cardinal B-splines.
For the sake of brevity, we just focus on the subspaces $\mathbb{S}_{p,\bftau,0}^p$ and 
$ \mathbb{S}_{p,\bftau,0}^{p-1}$ where the interior break points are equally spaced; the other types of subspaces can be treated similarly. 

\subsection{B-spline bases for full spline spaces}\label{sec:basis-full}
The full spline space $\mathbb{S}_{p,\bftau}$ in \eqref{eq:spline-max-smooth} is usually represented in terms of the classical B-spline basis which is defined through a knot sequence.
For $p\ge0$ and $\nknots\ge1$, consider the knot sequence
\begin{equation}\label{eq:knots}
\bfxi:=\{\xi_{-p}\leq\cdots\leq\xi_{0}<\xi_{1}<\cdots<\xi_{\nknots-1}<\xi_{\nknots}\leq\cdots\leq\xi_{\nknots+p}\},
\end{equation}
such that $\xi_{0}\leq 0<\xi_{1}$ and $\xi_{\nknots-1}<1\leq \xi_{\nknots}$.
This knot sequence allows us to define $\nknots+p$ B-splines of degree $p$,
\begin{equation}
\label{eq:B-splines}
N_{i,\bfxi}^p:\RR\rightarrow \RR, \quad i=-p,\ldots,\nknots-1,
\end{equation}
defined recursively as follows: for $-p \le i\le \nknots+p-1$,
	\begin{equation*}
	N_{i,\bfxi}^0(x):=\begin{cases}
	1, & x \in [\xi_i,\xi_{i+1}), \\
	0, & \text{otherwise};
	\end{cases}
	\end{equation*}
	for $1\le k\le p$ and $-p\le i\le \nknots+p-1-k$,
	\begin{equation*}
	N_{i,\bfxi}^k(x):=\frac{x-\xi_i}{\xi_{i+k}-\xi_i}N_{i,\bfxi}^{k-1}(x)+\frac{\xi_{i+k+1}-x}{\xi_{i+k+1}-\xi_{i+1}}N_{i+1,\bfxi}^{k-1}(x),
	\end{equation*}
	where a fraction with zero denominator is assumed to be zero.
It is well known that the B-splines $N_{i,\bfxi}^p$, $i=-p,\ldots,\nknots-1$, are linearly independent and they enjoy the following properties (see, e.g.,~\cite{deBoor2001,Lyche:18}).
\begin{itemize}
	\item Local support:
	\begin{equation*}
	\text{supp}(N_{i,\bfxi}^p)=[\xi_i,\xi_{i+p+1}], \quad i=-p,\ldots,\nknots-1.
	\end{equation*}
	\item Smoothness: 
	\begin{equation*}
	N_{i,\bfxi}^p \in C^{p-1}(0,1), \quad i=-p,\ldots,\nknots-1.
	\end{equation*}
	\item Differentiation: 
	\begin{equation*}
	\left(N_{i,\bfxi}^p(x)\right)' = p\left(\frac{N_{i,\bfxi}^{p-1}(x)}{\xi_{i+p}-\xi_i}-
	\frac{N_{i+1,\bfxi}^{p-1}(x)}{\xi_{i+p+1}-\xi_{i+1}}\right), \quad i=-p,\ldots,\nknots-1, \quad p \geq 1.
	\end{equation*}
	\item Non-negative partition of unity:
	\begin{equation*}
	N_{i,\bfxi}^p(x)\ge0, \quad i=-p,\ldots,\nknots-1, \qquad \sum_{i=-p}^{\nknots-1}N_{i,\bfxi}^p(x)=1,\quad x\in[0,1).
	\end{equation*}
\end{itemize}
Usually, open knots are employed to identify the B-splines in \eqref{eq:B-splines}, i.e.,
$$\xi_{-p}=\dots=\xi_0=0,\quad 1=\xi_\nknots=\dots=\xi_{\nknots+p},$$
because with such a configuration it is straightforward to identify a basis for the subspace $\mathbb{S}_{p,\bftau,0}^0$. However, in our context only the case of uniform grid spacing is of interest because it minimizes the number of outliers, as discussed in Section~\ref{sec:counting-outliers}. Therefore, it is natural to consider B-splines with uniform knots, i.e., the B-splines $N_{i,\bfxi}^p$, $i=-p,\ldots,\nknots-1$, are the restriction to the interval $[0,1]$ of uniformly shifted and scaled versions of a single shape function, the so-called \emph{cardinal B-spline} $\cardB_p:\RR\rightarrow \RR$, where
\begin{equation*}
\cardB_0(t) := \begin{cases}
1, & t \in [0, 1), \\
0, & \text{otherwise},
\end{cases}
\end{equation*}
and
\begin{equation*}
\cardB_p (t) := \frac{t}{p} \cardB_{p-1}(t) + \frac{p+1-t}{p} \cardB_{p-1}(t-1), \quad p \geq 1.
\end{equation*}
The cardinal B-spline $\cardB_p$ belongs to $C^{p-1}(\RR)$ and is supported on the interval $[0,p+1]$. It is a symmetric function with respect to the midpoint of its support, i.e.,
\begin{equation}\label{eq:symmetry}
\cardB_p\biggl(\frac{p+1}{2}+t\biggr)=\cardB_p\biggl(\frac{p+1}{2}-t\biggr).
\end{equation}
For other common properties of cardinal B-splines, we refer the reader to \cite{Lyche:18}.

From now on, we consider the knot sequence \eqref{eq:knots} equally spaced, so that we have
\begin{equation}\label{eq:BcardB}
N^{p}_{i,\bfxi}(x) =\cardB_{p}\biggl(\frac{\nknots(x-\xi_i)}{\xi_\nknots-\xi_0}\biggr), \quad i=-p,\ldots,\nknots-1,
\end{equation}
and we explicitly construct a basis for the spaces $\mathbb{S}_{p,\bftau,0}^p$ and $\mathbb{S}_{p,\bftau,0}^{p-1}$ in case their interior break points are equally spaced. The basis elements are expressed as linear combinations of the B-splines in \eqref{eq:BcardB}.

\subsection{B-spline-like bases for optimal spline spaces} \label{sec:basis-optimal}
In this subsection we construct a basis for optimal spline spaces of the form $\mathbb{S}_{p,\bftau,0}^p$ defined in \eqref{eq:opt-spaces}.
Let us first address the case $p$ odd. We note that in this case $\mathbb{S}_{p,\bftau,0}^p=\mathbb{S}_{p,\bftau,0}^{p-1}$ and $\dim(\mathbb{S}_{p,\bftau,0}^p)=\nknots-1$.
We select the knots in \eqref{eq:knots} as
\begin{equation}
\label{eq:knots:0:odd}
\xi_i=\frac{i}{\nknots}, \quad  i=-p,\ldots, \nknots+p,
\end{equation}
and observe that with such a choice we have
$$
\bftau=\bftau_{p,\nknots-1,0}^\opt,
$$
so that $\mathbb{S}_{p,\bftau,0}^p=\mathbb{S}_{p,\nknots-1,0}^\opt$.
Then, we consider the set of B-spline-like spline functions
 \begin{equation}
 \label{eq:B-spline-0-odd}
\{ N^{p}_{i,\bfxi,0}, \  i=1,\ldots, \nknots-1\},
 \end{equation}
 defined by
\begin{equation}
\label{eq:basis-0-odd}
\begin{bmatrix}
N^{p}_{1,\bfxi,0}\\
N^{p}_{2,\bfxi,0} \\
\vdots \\
N^{p}_{\nknots-1,\bfxi,0}
\end{bmatrix}:=
\begin{bmatrix}
\underbrace{\overbrace{\cdots\bigg|\,L_{\nknots-1}\,\bigg|\,L_{\nknots-1}\,\bigg|}^{\frac{p+1}{2}}I_{\nknots-1}\overbrace{\bigg|\,R_{\nknots-1}\,\bigg|\,R_{\nknots-1}\,\bigg|\cdots}^{\frac{p+1}{2}}}_{\nknots+p}
\end{bmatrix}
\begin{bmatrix}
N^{p}_{-p,\bfxi}\\
\vdots \\
N^{p}_{0,\bfxi}\\
N^{p}_{1,\bfxi}\\
\vdots \\
N^{p}_{\nknots-1,\bfxi}
\end{bmatrix},
\end{equation}
where
\begin{equation}
\label{eq:LR-odd=even}
	L_m:=
	\begin{bmatrix}
	\,I_m\,\bigg|\,\boldsymbol{0}_m\,\bigg|\,-J_m\,\bigg|\,\boldsymbol{0}_m\,
	\end{bmatrix},
\quad
		R_m:=
	\begin{bmatrix}
	\boldsymbol{0}_m\,\bigg|\,-J_m\,\bigg|\,\boldsymbol{0}_m\,\bigg|\,I_m\,
	\end{bmatrix}.
	\end{equation}
Here 
$I_m$ denotes the identity matrix of size $m$, $J_m$ denotes the exchange matrix of size $m$, i.e., the matrix with $1$ along the anti-diagonal, and $\boldsymbol{0}_m$ is the zero (column) vector of length $m$.
The notation in the matrix in \eqref{eq:basis-0-odd} has to be interpreted as follows (see also Example~\ref{ex:basis-0-odd}):
\begin{itemize}
	\item take the identity matrix of size $\nknots-1$;
	\item construct the matrices $L_{\nknots-1}$ and $R_{\nknots-1}$ as in \eqref{eq:LR-odd=even};
	\item add copies of the matrix $L_{\nknots-1}$ to the left of the central identity matrix and take the first $\frac{p+1}{2}$ columns to the left of the identity matrix;
	\item add copies of the matrix $R_{\nknots-1}$ to the right of the central identity matrix and take the first $\frac{p+1}{2}$ columns to the right of the identity matrix.
\end{itemize}

\begin{figure}[t!]
\centering
\subfigure[B-spline-like functions]{\includegraphics[height=4.1cm]{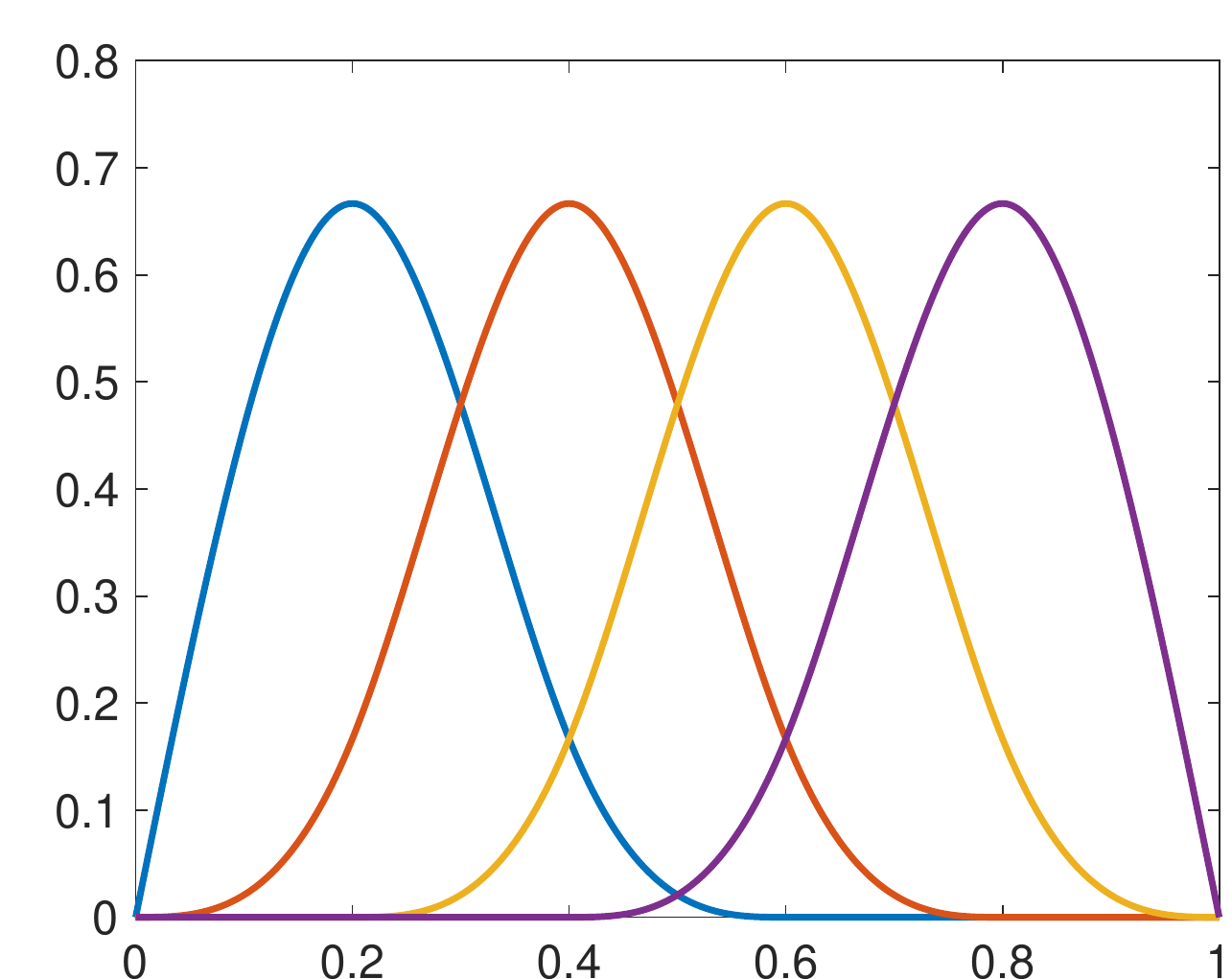}}\hspace*{0.1cm}
\subfigure[second derivatives]{\includegraphics[height=4.1cm]{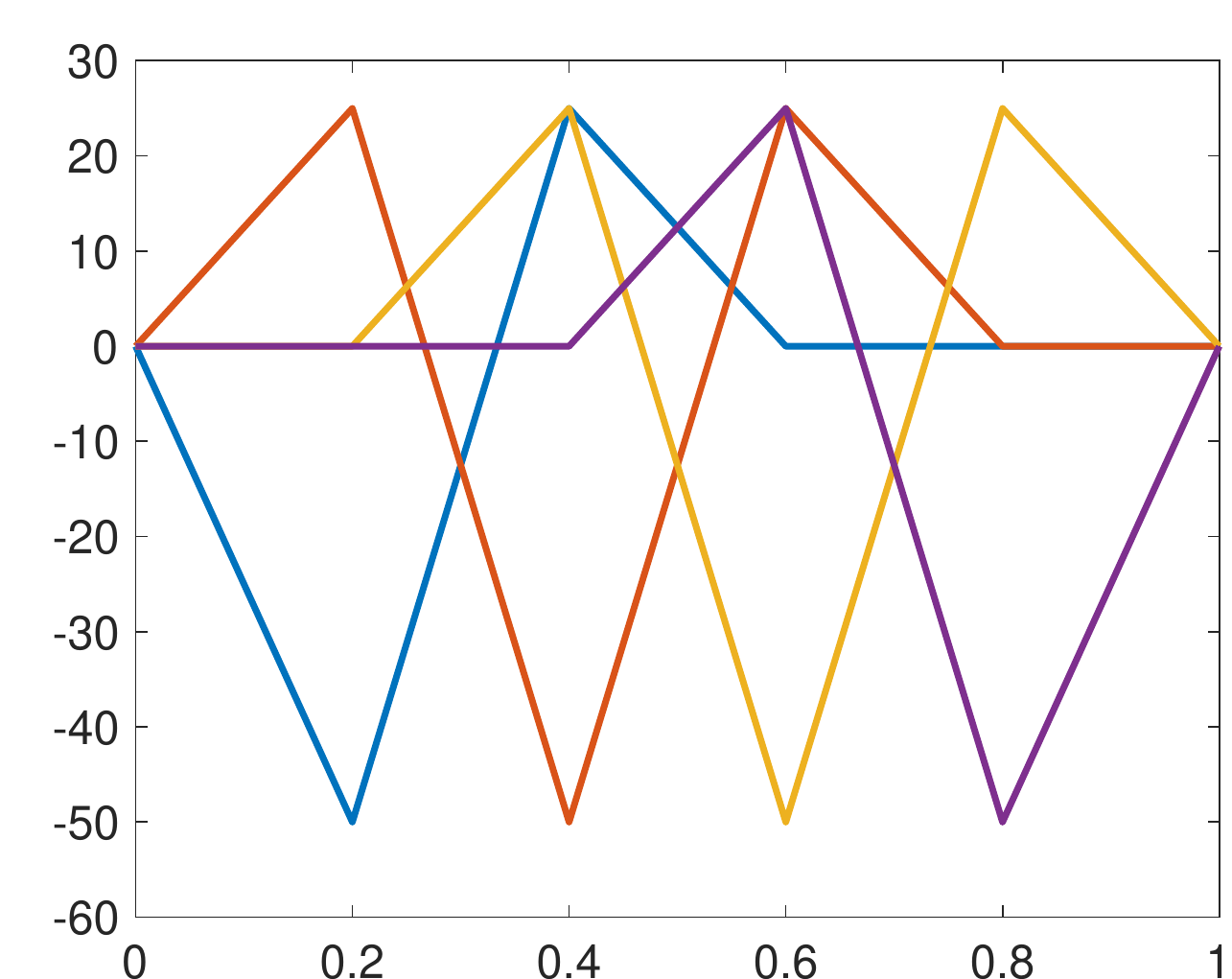}}
\caption{Example~\ref{ex:basis-0-odd}: B-spline-like functions and their second derivatives for the space $\mathbb{S}_{p,\nknots-1,0}^\opt$ with $p=3$ and $\nknots=5$.} \label{fig:basis.0.odd:a}
\bigskip
\centering
\subfigure[B-spline-like functions]{\includegraphics[height=4.1cm]{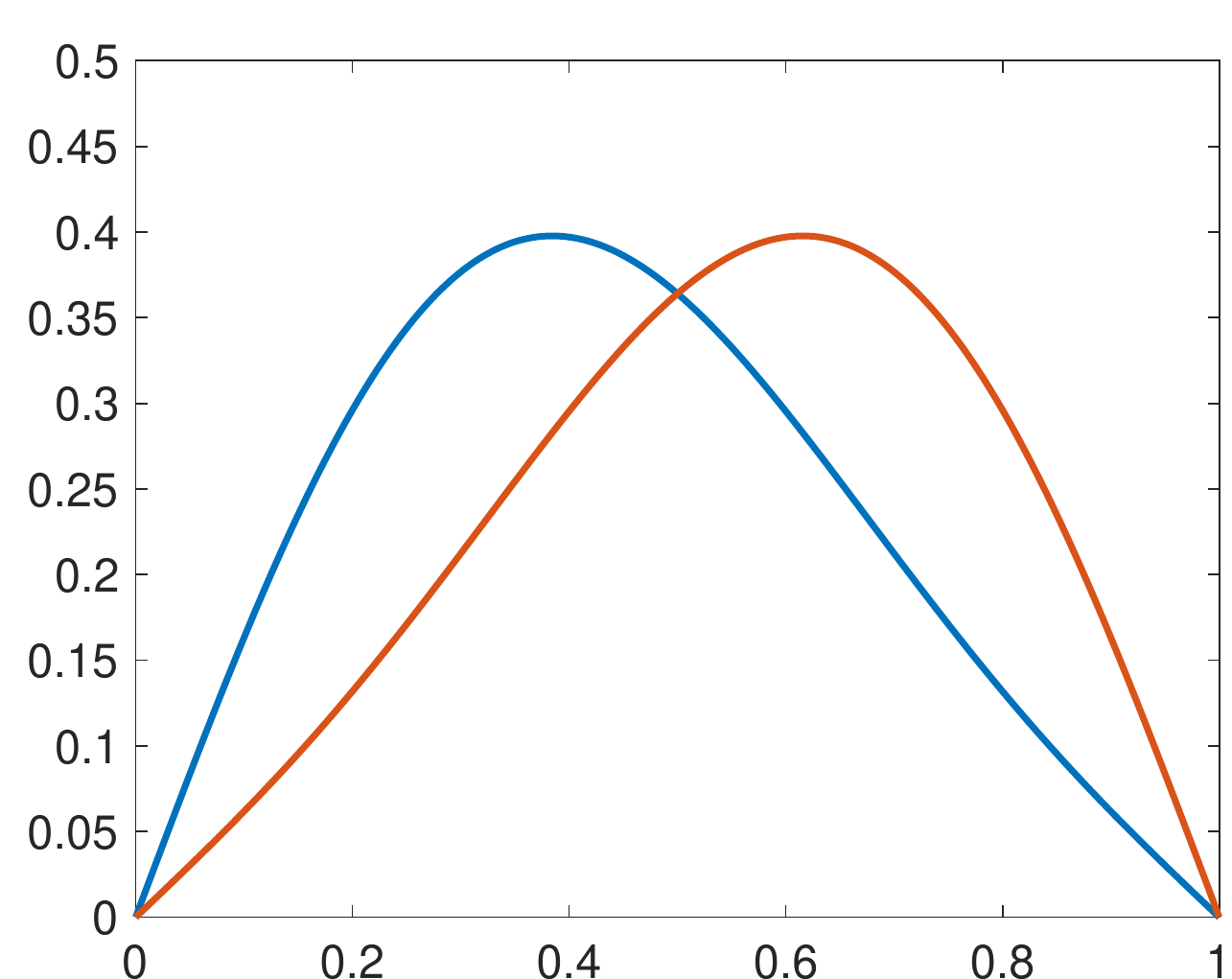}}\hspace*{0.1cm}
\subfigure[second derivatives]{\includegraphics[height=4.1cm]{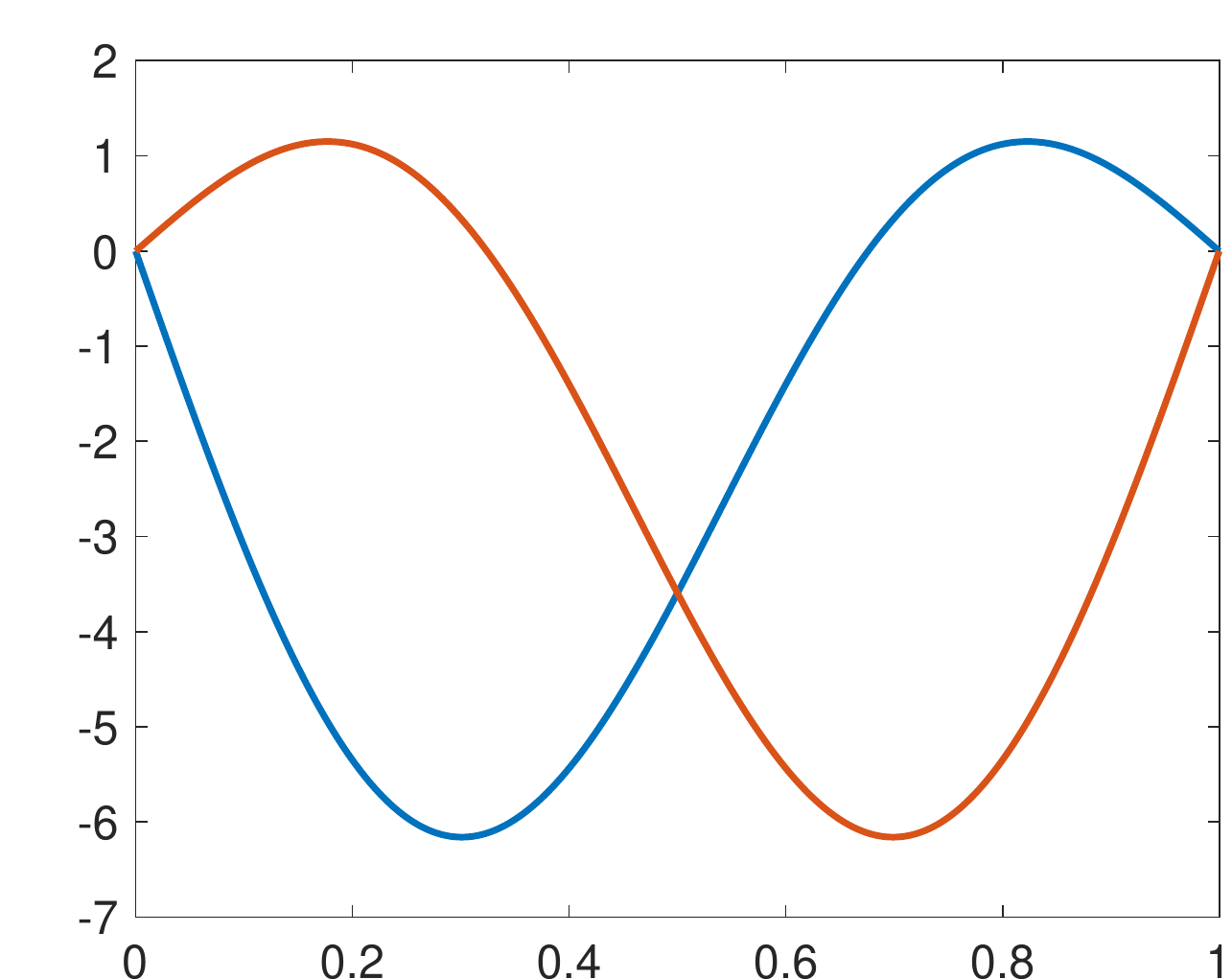}}\hspace*{0.1cm}
\subfigure[fourth derivatives]{\includegraphics[height=4.1cm]{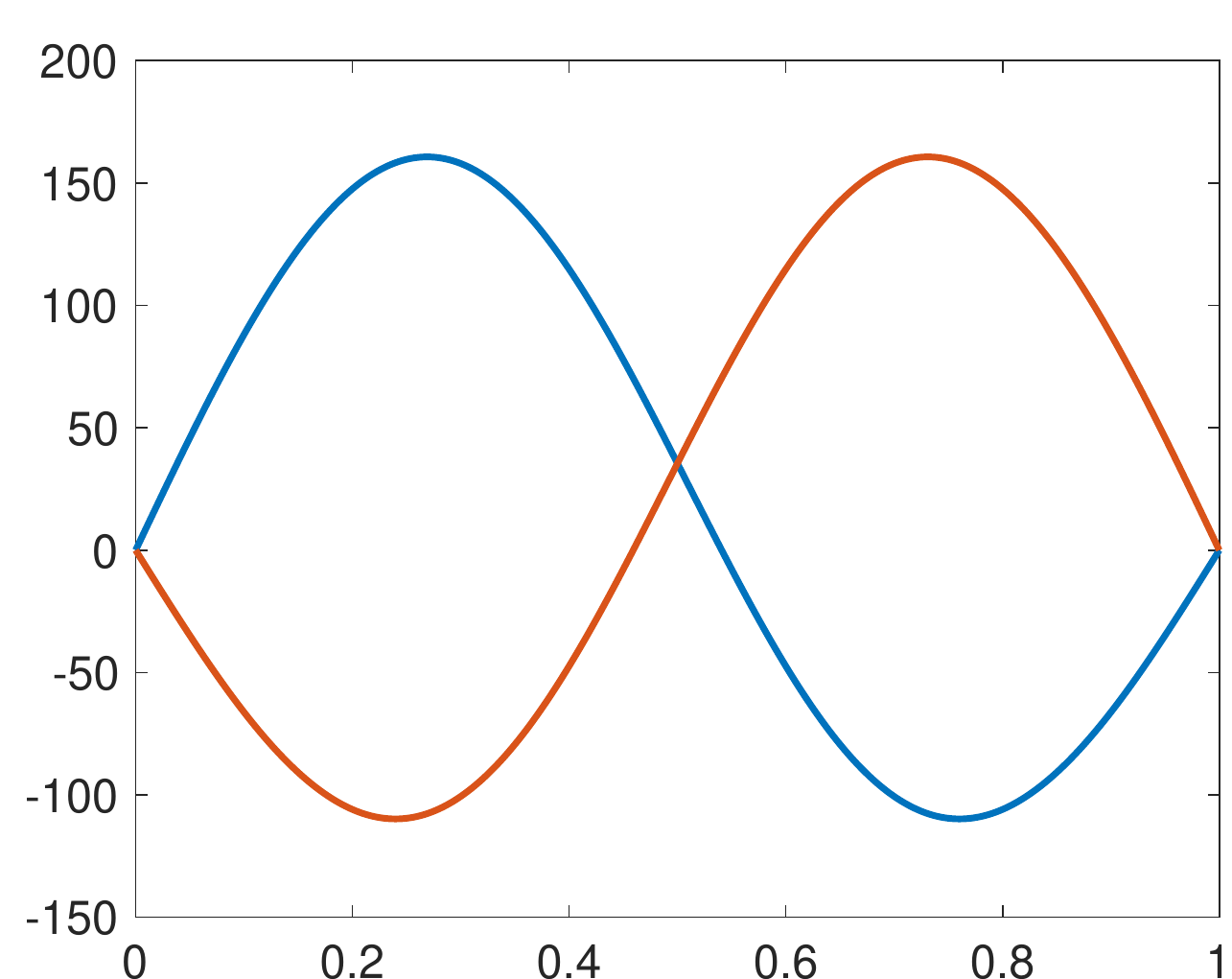}}\\
\subfigure[sixth derivatives]{\includegraphics[height=4.1cm]{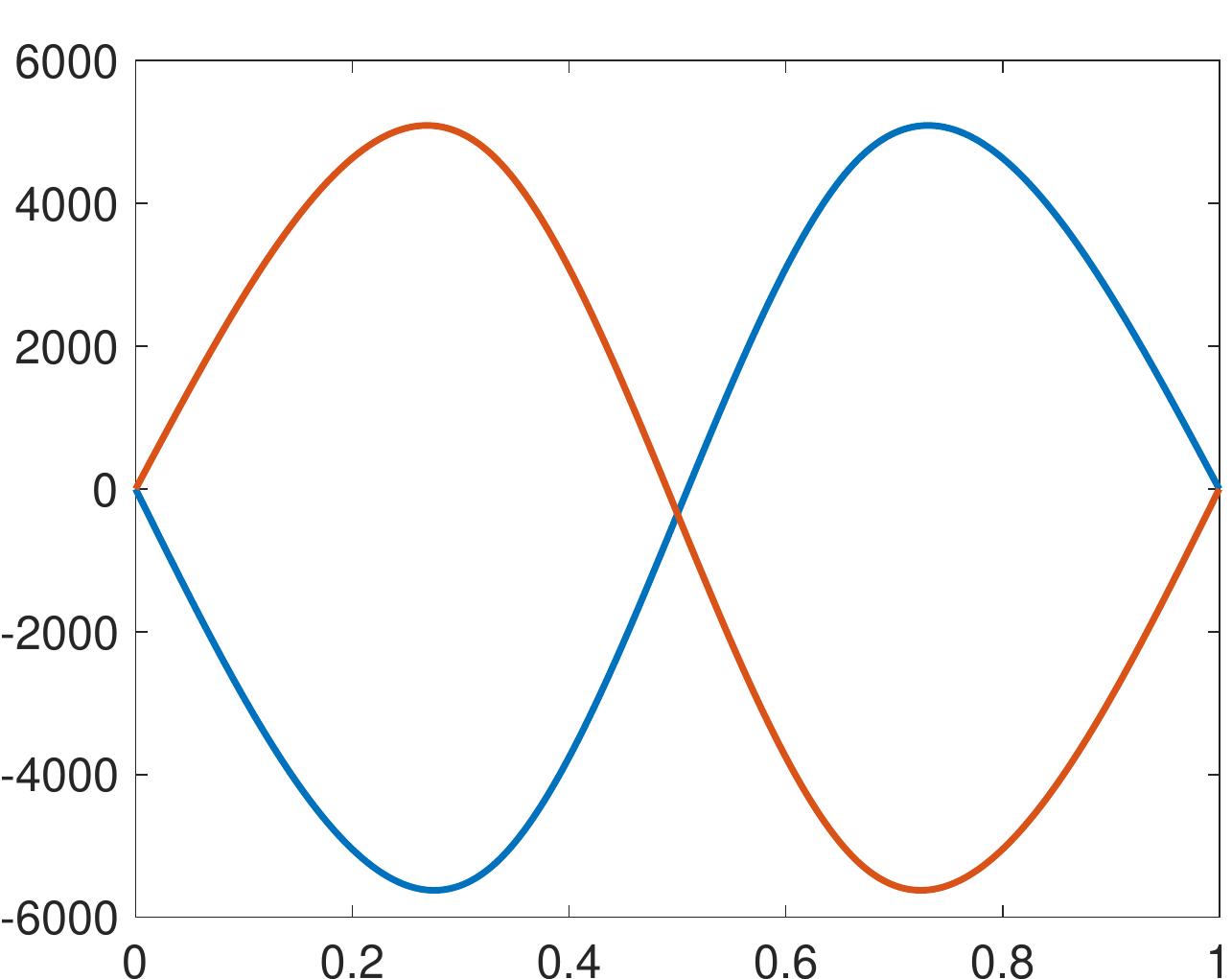}}\hspace*{0.1cm}
\subfigure[eighth derivatives]{\includegraphics[height=4.1cm]{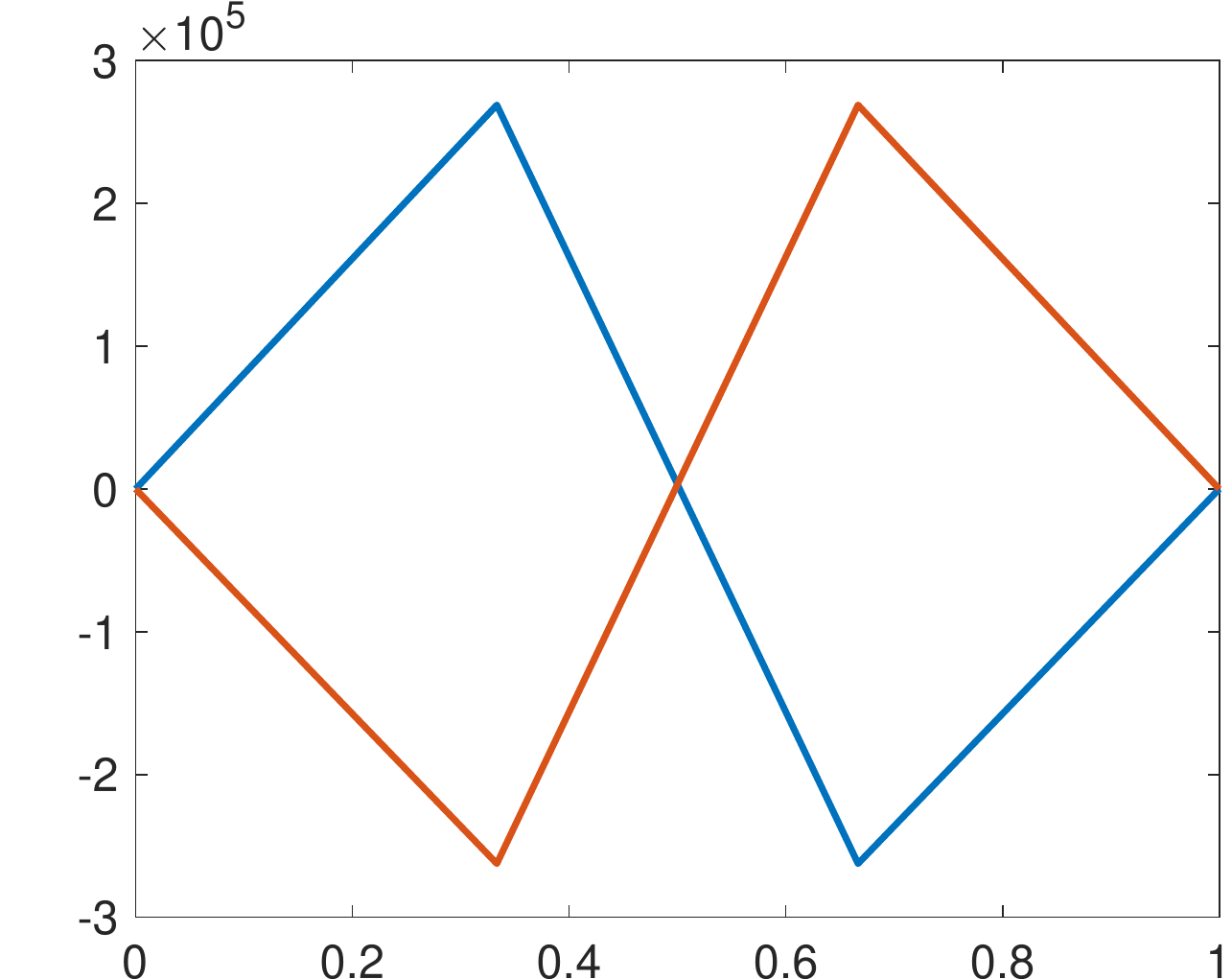}}
\caption{Example~\ref{ex:basis-0-odd}: B-spline-like functions and their even order derivatives for the space $\mathbb{S}_{p,\nknots-1,0}^\opt$ with $p=9$ and $\nknots=3$.} \label{fig:basis.0.odd:b}
\end{figure}
\begin{example}
	\label{ex:basis-0-odd}
	For $p=3$ and $\nknots=5$, the matrix in \eqref{eq:basis-0-odd} has $4$ rows and $8$ columns, and it takes the form
	\begin{equation}
	\label{eq:ex-matrix-1}
	\begin{bmatrix}
		-1&0&1&0&0&0&0&0
		\\
		 0&0&0&1&0&0&0&0
		 \\
		 0&0&0&0&1&0&0&0
		 \\
		 0&0&0&0&0&1&0&-1
	\end{bmatrix}.
	\end{equation}
	For $p=9$ and $\nknots=3$, the matrix in \eqref{eq:basis-0-odd} has $2$ rows and $12$ columns, and it takes the form
	\begin{equation}
	\label{eq:ex-matrix-2}
	\begin{bmatrix}
	0&0&0&-1&0&1&0&0&0&-1&0&1
\\
	1&0&-1&0&0&0&1&0&-1&0&0&0
	\end{bmatrix}.
\end{equation}
	The graph of the corresponding B-spline-like functions and their even order derivatives are depicted in Figures~\ref{fig:basis.0.odd:a} and~\ref{fig:basis.0.odd:b}. One clearly notices that the functions satisfy the boundary conditions of the space $\mathbb{S}_{p,\nknots-1,0}^\opt$.
\end{example}

We now address the case $p$ even. We note that here $\dim(\mathbb{S}_{p,\bftau,0}^p)=\nknots-2$.
We select the knots in \eqref{eq:knots} as
\begin{equation}
\label{eq:knots:0:even}
\xi_i=\frac{i-1/2}{\nknots-1}, \quad  i=-p,\ldots,\nknots+p,
\end{equation}
and observe that with such a choice we have
$$
\bftau=\bftau_{p,\nknots-2,0}^\opt,
$$
so that $\mathbb{S}_{p,\bftau,0}^p=\mathbb{S}_{p,\nknots-2,0}^\opt$.
Then, we consider the set of B-spline-like functions
\begin{equation}
\label{eq:B-spline-0-even}
\{ N^{p}_{i,\bfxi,0}, \  i=1,\ldots, \nknots-2\},
\end{equation}
defined by
\begin{equation}
\label{eq:basis-0-even}
\begin{bmatrix}
N^{p}_{1,\bfxi,0}\\
N^{p}_{2,\bfxi,0} \\
\vdots \\
N^{p}_{\nknots-2,\bfxi,0}
\end{bmatrix}:=
\begin{bmatrix}
\underbrace{\overbrace{\cdots\,\bigg|\,L_{\nknots-2}\,\bigg|\,L_{\nknots-2}\,\bigg|}^{\frac{p}{2}+1} I_{\nknots-2}\overbrace{\bigg|\,R_{\nknots-2}\bigg|\,R_{\nknots-2}\,\bigg|\,\dots}^{\frac{p}{2}+1}}_{\nknots+p}
\end{bmatrix}
\begin{bmatrix}
N^{p}_{-p,\bfxi}\\
\vdots \\
N^{p}_{0,\bfxi}\\
N^{p}_{1,\bfxi}\\
\vdots \\
N^{p}_{\nknots-1,\bfxi}
\end{bmatrix}.
\end{equation}
Further details for the above constructions can be found in \cite{DiVona:2019}.

\begin{figure}[t!]
\centering
\subfigure[B-spline-like functions]{\includegraphics[height=4.1cm]{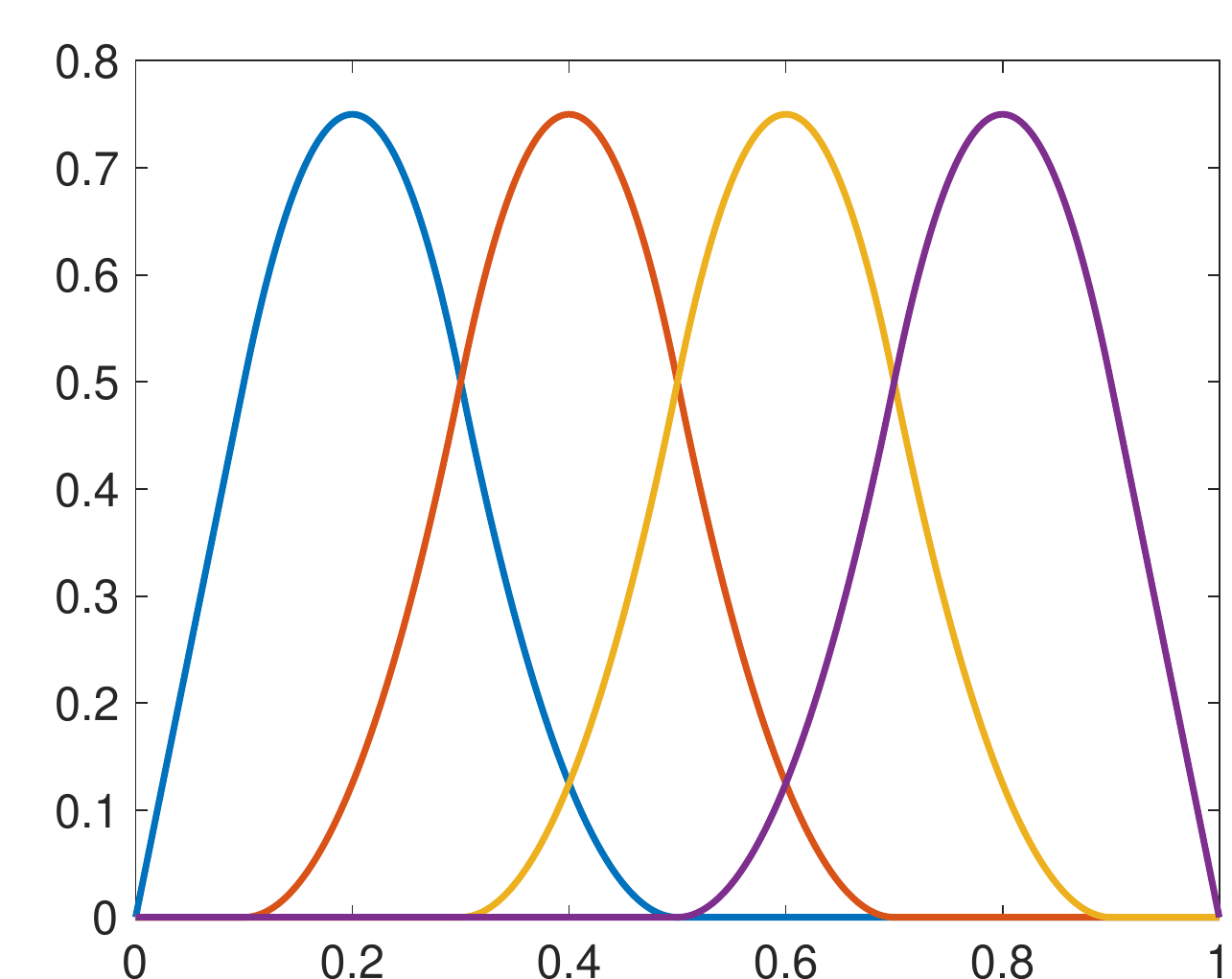}}\hspace*{0.1cm}
\subfigure[second derivatives]{\includegraphics[height=4.1cm]{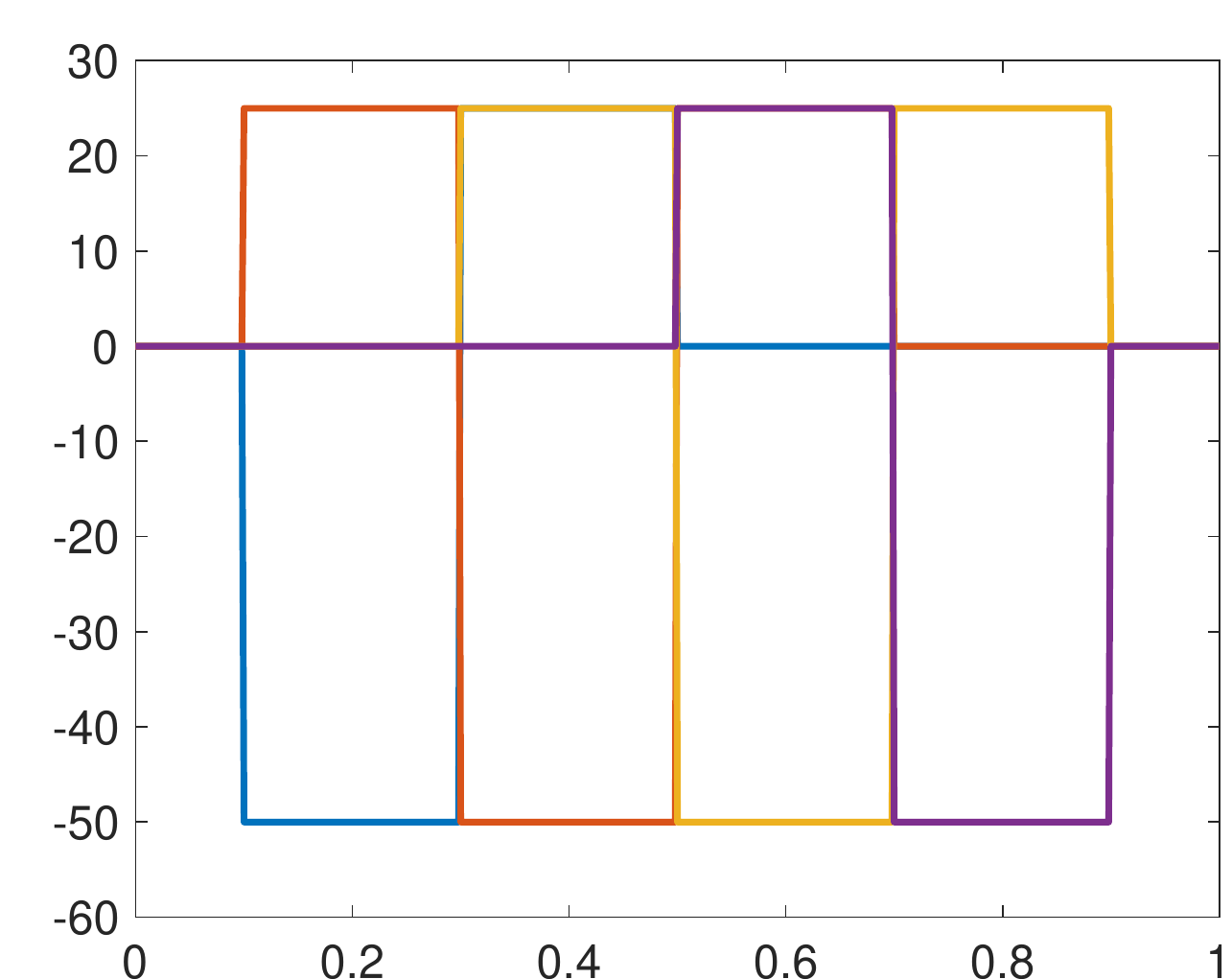}}
\caption{Example~\ref{ex:basis-0-even}: B-spline-like functions and their second derivatives for the space $\mathbb{S}_{p,\nknots-2,0}^\opt$ with $p=2$ and $\nknots=6$.} \label{fig:basis.0.even:a}
\bigskip
\centering
\subfigure[B-spline-like functions]{\includegraphics[height=4.1cm]{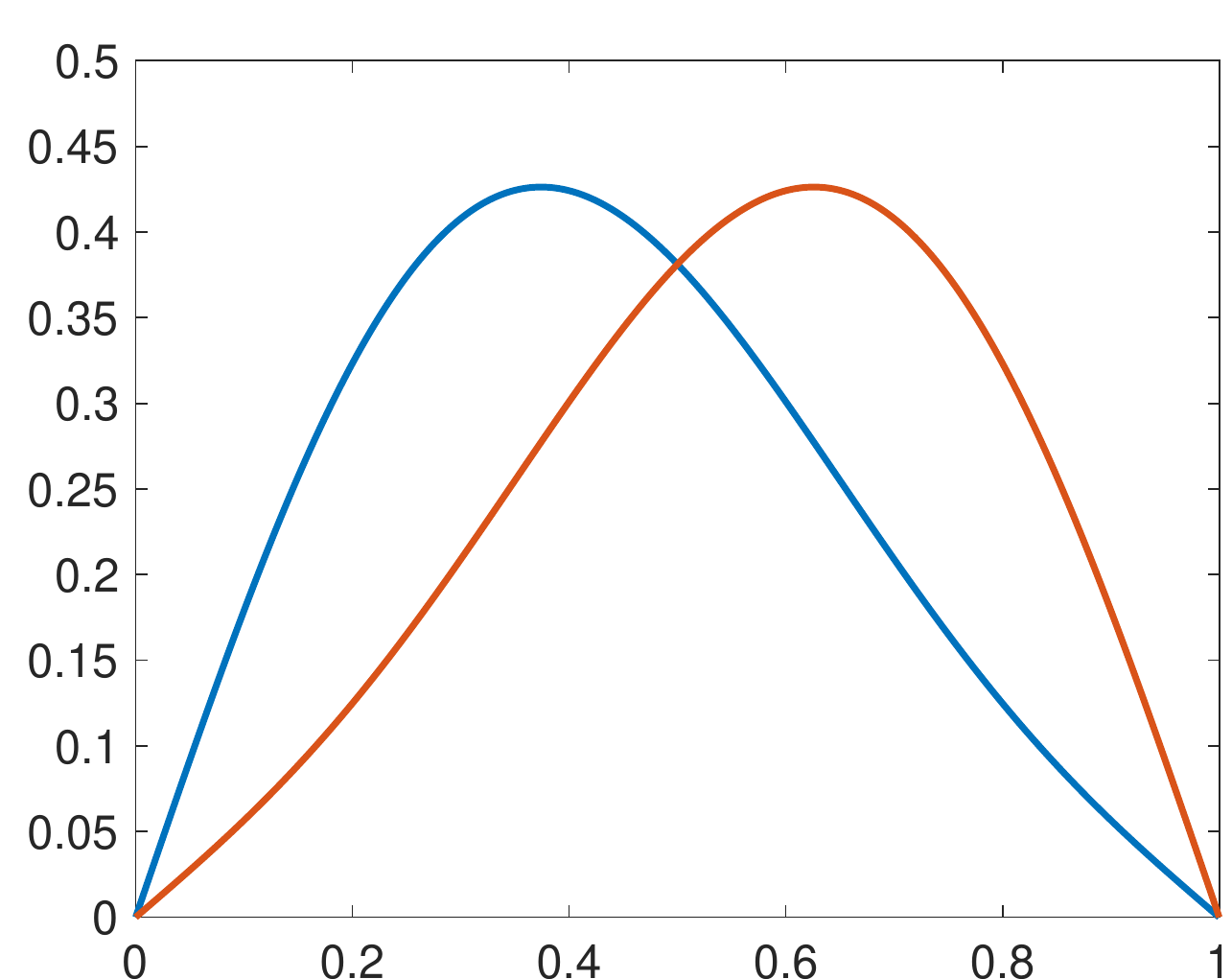}}\hspace*{0.1cm}
\subfigure[second derivatives]{\includegraphics[height=4.1cm]{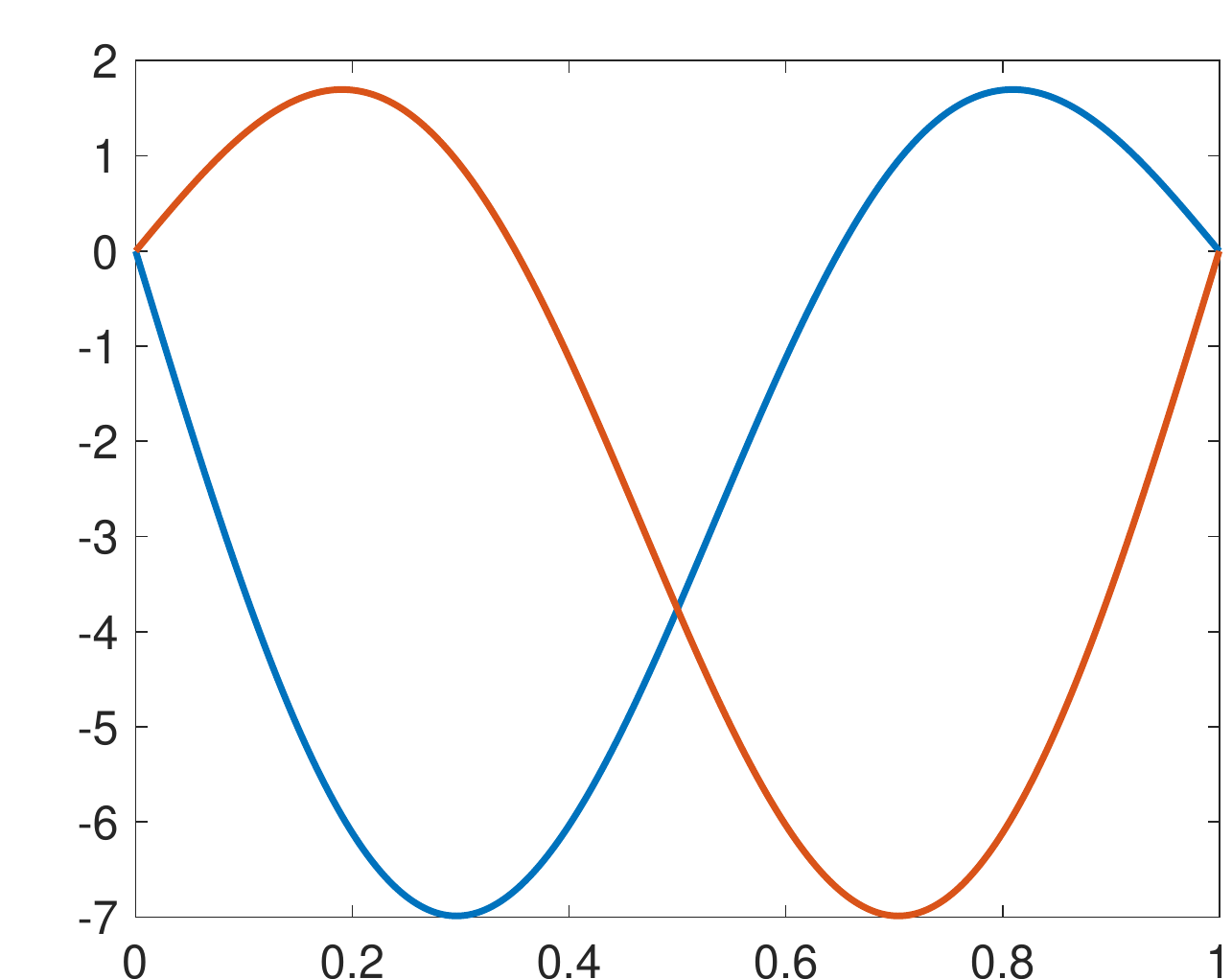}}\hspace*{0.1cm}
\subfigure[fourth derivatives]{\includegraphics[height=4.1cm]{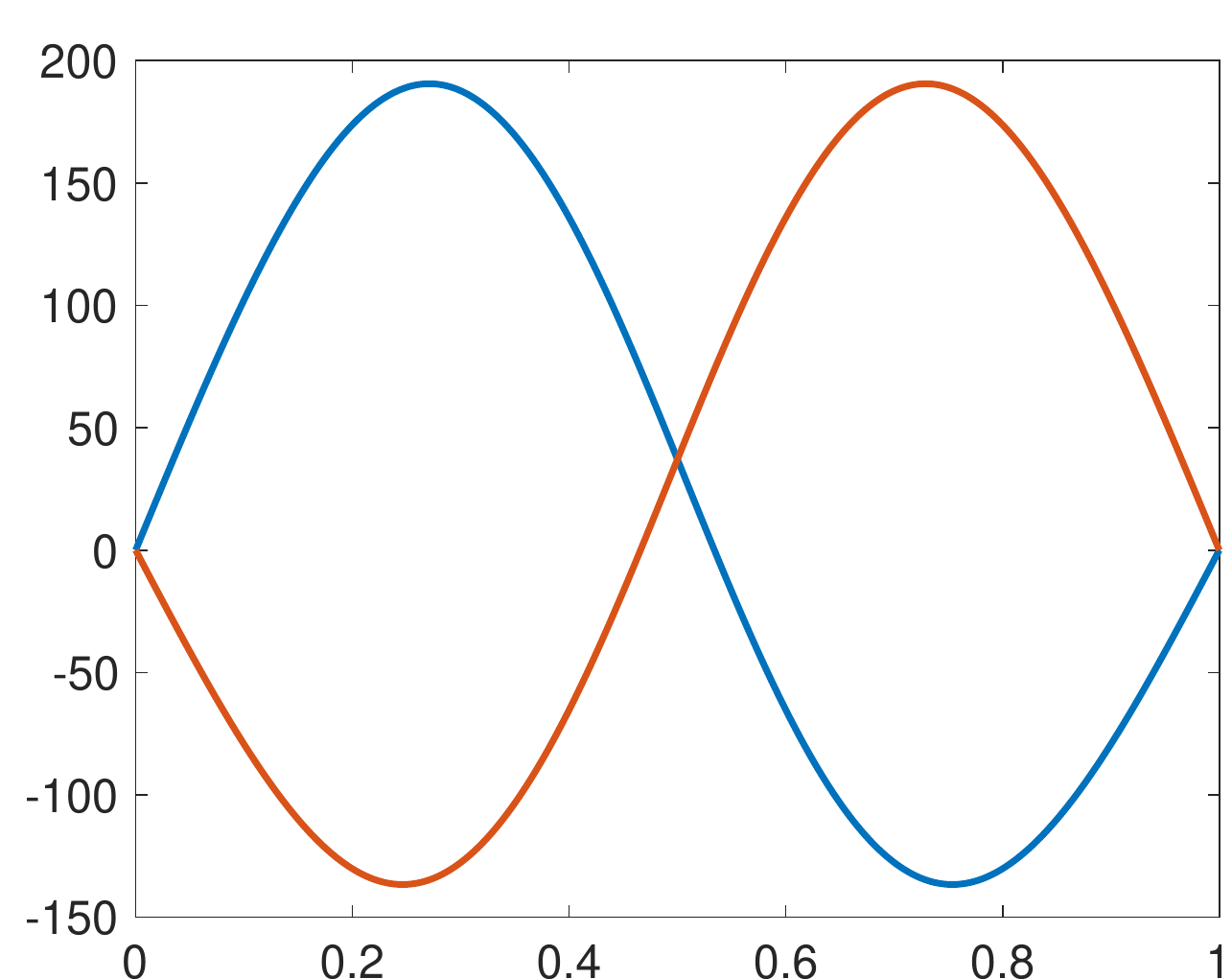}}\\
\subfigure[sixth derivatives]{\includegraphics[height=4.1cm]{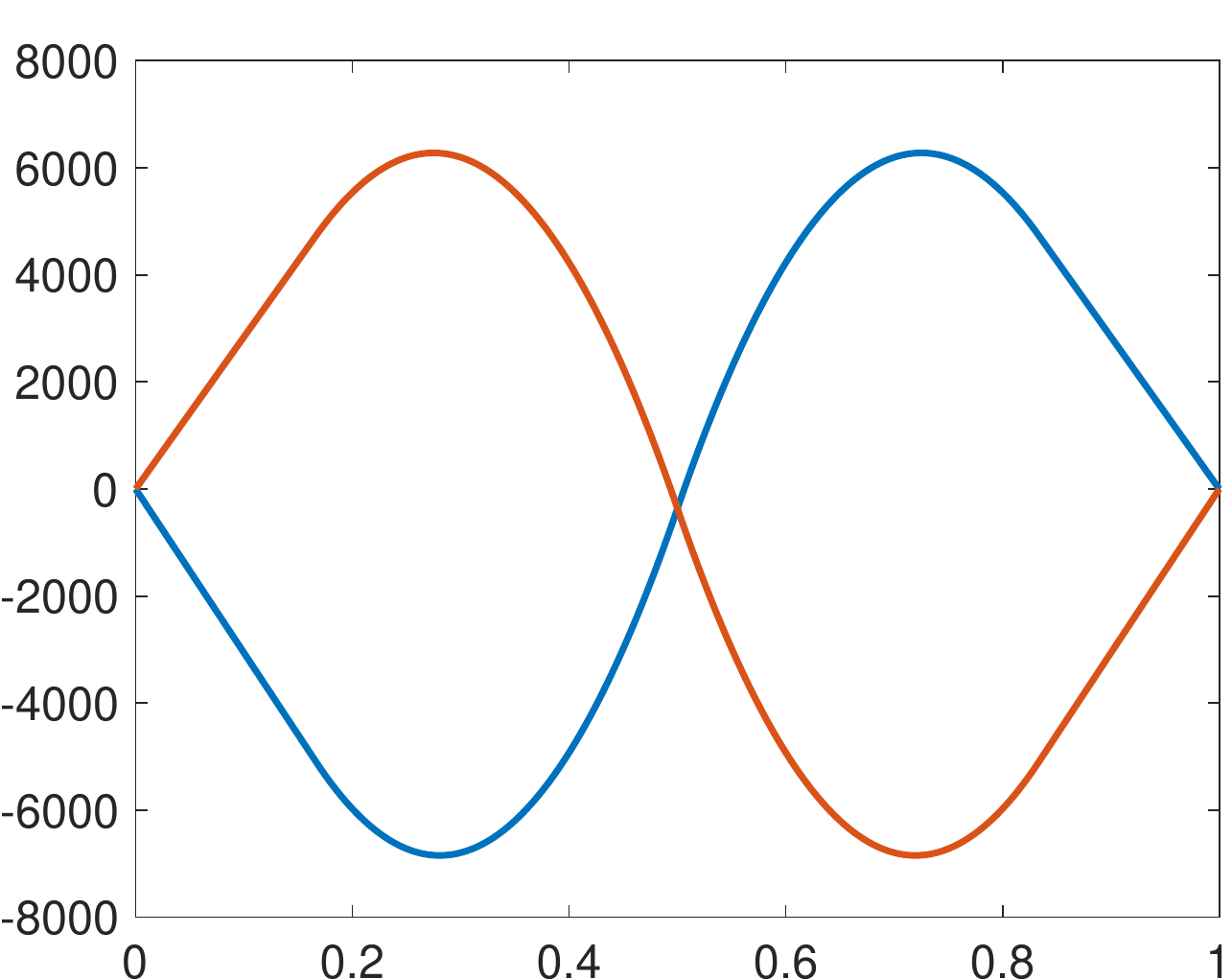}}\hspace*{0.1cm}
\subfigure[eighth derivatives]{\includegraphics[height=4.1cm]{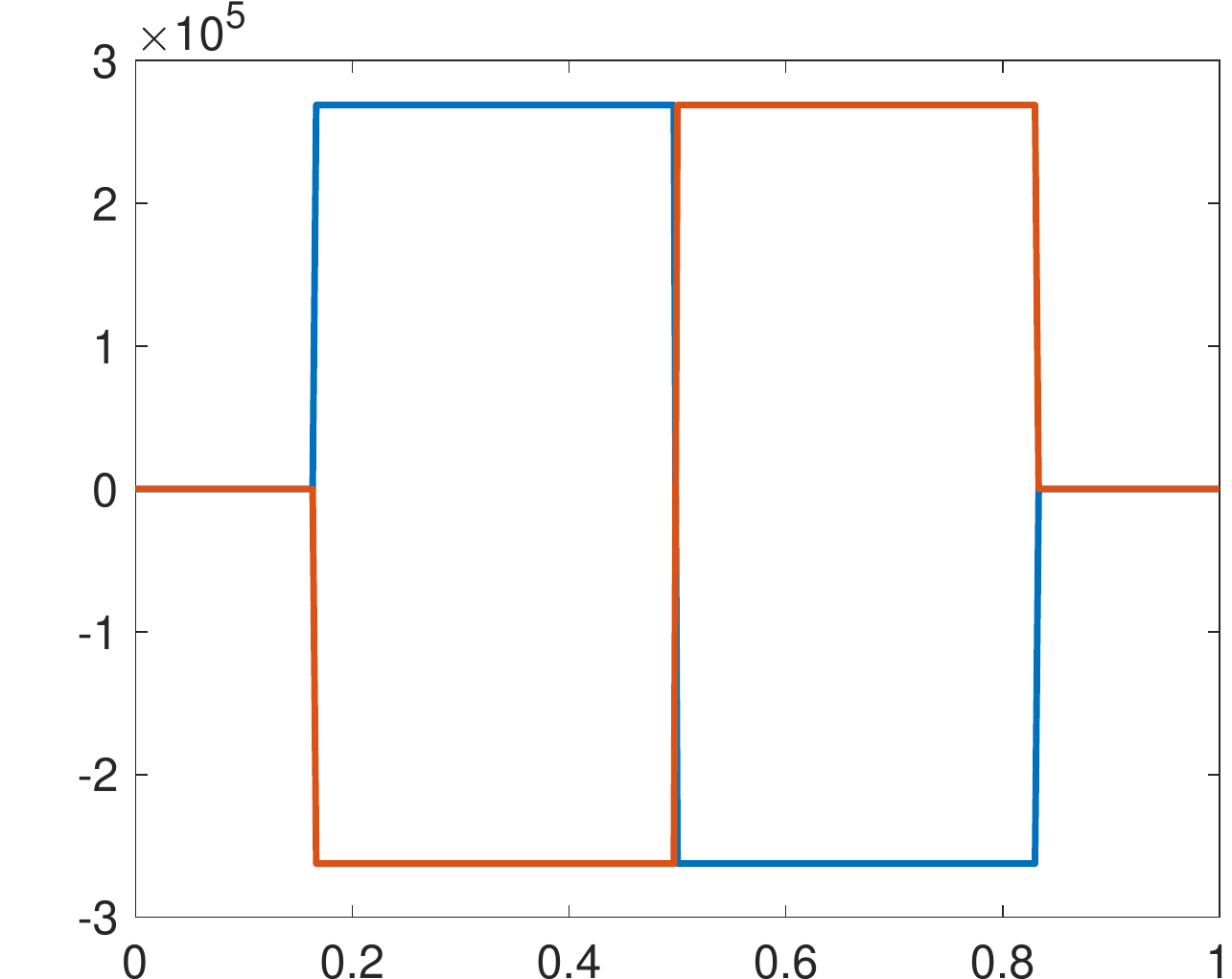}}
\caption{Example~\ref{ex:basis-0-even}: B-spline-like functions and their even order derivatives for the space $\mathbb{S}_{p,\nknots-2,0}^\opt$ with $p=8$ and $\nknots=4$.} \label{fig:basis.0.even:b}
\end{figure}
\begin{example}
	\label{ex:basis-0-even}
	For $p=2$ and $\nknots=6$ the matrix in \eqref{eq:basis-0-even} has $4$ rows and $8$ columns, and it is equal to the matrix \eqref{eq:ex-matrix-1}.
	For $p=8$ and $\nknots=4$ the matrix in \eqref{eq:basis-0-even} has $2$ rows and $12$ columns, and it is equal to the matrix \eqref{eq:ex-matrix-2}.
	The graph of the corresponding B-spline-like functions and their even order derivatives are depicted in Figures~\ref{fig:basis.0.even:a} and~\ref{fig:basis.0.even:b}. One clearly notices that the functions satisfy the boundary conditions of the space $\mathbb{S}_{p,\nknots-2,0}^\opt$.
\end{example}

Finally, we show that the above sets of B-spline-like functions form a basis of our optimal spline spaces.

\begin{proposition}
\label{prop:basis-0}
The functions defined in \eqref{eq:B-spline-0-odd} are a basis of the optimal space $\mathbb{S}_{p,\nknots-1,0}^\opt$ for $p$ odd.
Likewise, the functions defined in \eqref{eq:B-spline-0-even} are a basis of the optimal space $\mathbb{S}_{p,\nknots-2,0}^\opt$ for $p$ even.
\end{proposition}
\begin{proof}
Let us start by recalling that $\mathbb{S}_{p,\bftau,0}^p=\mathbb{S}_{p,\nknots-1,0}^\opt$ for $p$ odd and $\mathbb{S}_{p,\bftau,0}^p=\mathbb{S}_{p,\nknots-2,0}^\opt$ for $p$ even,
where
$\bftau$ is such that 
$$
\tau_i=\xi_i,\quad  i=1,\ldots,\nknots-1,
$$
and $\bfxi$ is defined in \eqref{eq:knots:0:odd} and \eqref{eq:knots:0:even}, respectively.
The functions in \eqref{eq:B-spline-0-odd} and \eqref{eq:B-spline-0-even} clearly belong to $\mathbb{S}_{p,\bftau}$ for the considered sequences of break points. From the symmetry property of cardinal B-splines, see \eqref{eq:symmetry}, a direct check shows that the functions $ \{ N^{p}_{i,\bfxi,0}\}$  satisfy the additional boundary conditions identifying the subspace $\mathbb{S}_{p,\bftau,0}^p$. The considered functions are linearly independent because the translates of a cardinal B-spline are linearly independent and the matrices in \eqref{eq:basis-0-odd} and \eqref{eq:basis-0-even} have maximum rank. Therefore, the considered functions are a basis of $\mathbb{S}_{p,\bftau,0}^p$ because their number equals the dimension of the space. 
\end{proof}

\subsection{B-spline-like bases for other reduced spline spaces} 
We now construct a basis for the space $\mathbb{S}_{p,\bftau,0}^{p-1}$ defined by uniformly spaced break points $\bftau$, i.e.,
\begin{equation*}
\tau_i=\frac{i}{\nknots}, \quad i=0,\ldots, \nknots.
\end{equation*}
It has been numerically observed in \cite{Hiemstra:2021} that these spaces are outlier-free; see also Remark~\ref{rmk:other-reduced-space}.
The only case to be treated is $p$ even since the odd degree case is the same as before in Section~\ref{sec:basis-optimal}. In the even degree case we have $\dim(\mathbb{S}_{p,\bftau,0}^{p-1})=\nknots$, and denote the corresponding space by $\overline{\mathbb{S}}_{p,\nknots,0}$.
We select the knots in \eqref{eq:knots} as
\begin{equation*}
 \xi_i=\frac{i}{\nknots}, \quad  i=-p,\ldots, \nknots+p.
\end{equation*}
Then, we consider the set of B-spline-like functions
\begin{equation}
 \label{eq:B-spline-uniform-even}
 \{ \overline{N}^{p}_{i,\bfxi,0}, \ i=1,\ldots, \nknots\},
\end{equation}
defined by
\begin{equation}
 \label{eq:basis-uniform-even}
 \begin{bmatrix}
\overline{N}^{p}_{1,\bfxi,0}\\
\overline{N}^{p}_{2,\bfxi,0} \\
 \vdots \\
\overline{N}^{p}_{\nknots,\bfxi,0}
 \end{bmatrix}:=
 \begin{bmatrix}
 \underbrace{\overbrace{\cdots\,\bigg|\,\overline{L}_{\nknots}\,\bigg|\,\overline{L}_{\nknots}\,\bigg|}^{\frac{p}{2}} I_{\nknots}\overbrace{\bigg|\,\overline{R}_{\nknots}\,\bigg|\,\overline{R}_{\nknots}\,\bigg|\,\dots}^{\frac{p}{2}}}_{\nknots+p}
 \end{bmatrix}
 \begin{bmatrix}
 N^{p}_{-p,\bfxi}\\
 \vdots \\
 N^{p}_{0,\bfxi}\\
 N^{p}_{1,\bfxi}\\
 \vdots \\
 N^{p}_{\nknots-1,\bfxi}
 \end{bmatrix},
\end{equation}
where 
\begin{equation*}
\overline{L}_m:=
 \begin{bmatrix}
 \,I_m\,\big|\,-J_m\,
 \end{bmatrix},
 \quad
 \overline{R}_m:=
 \begin{bmatrix}
 -J_m\,\big|\,I_m\,
 \end{bmatrix}.
\end{equation*}
 
\begin{figure}[t!]
\centering
\subfigure[B-spline-like functions]{\includegraphics[height=4.1cm]{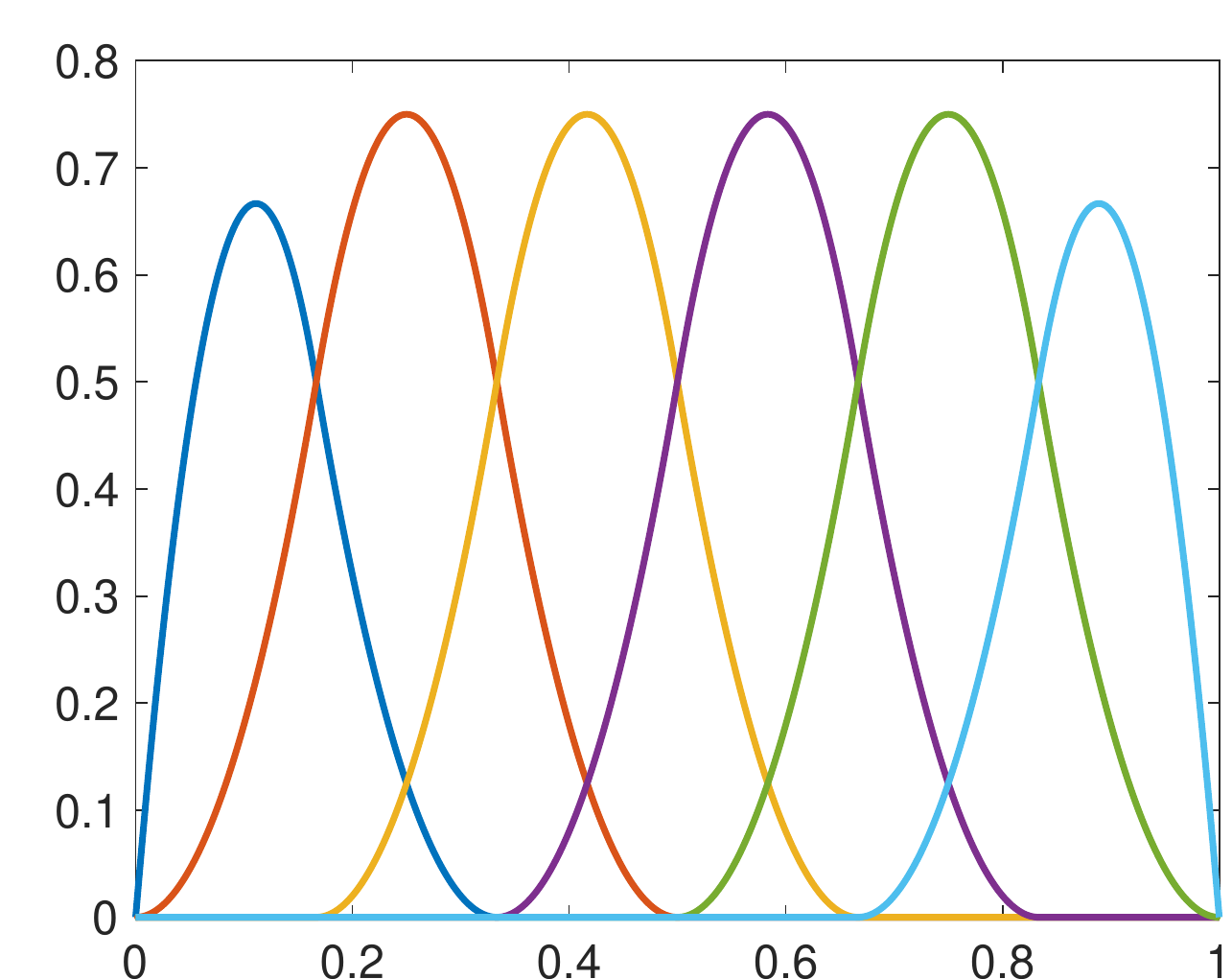}}\hspace*{0.1cm}
\subfigure[second derivatives]{\includegraphics[height=4.1cm]{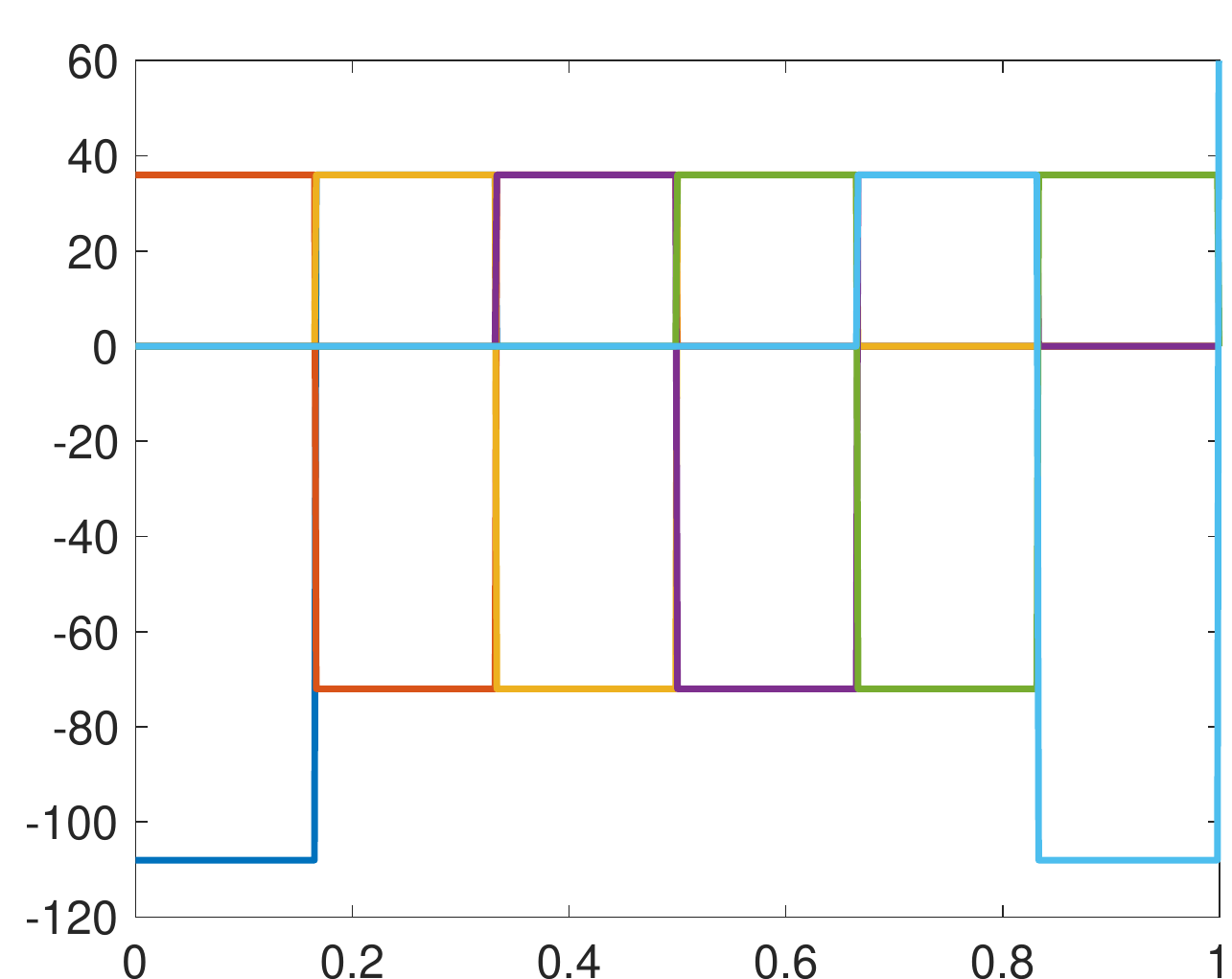}}
\caption{Example~\ref{ex:basis-uniform-even}: B-spline-like functions and their second derivatives for the space $\overline{\mathbb{S}}_{p,\nknots,0}$ with $p=2$ and $\nknots=6$.} \label{fig:basis.uniform.even:a}
\bigskip
\centering
\subfigure[B-spline-like functions]{\includegraphics[height=4.1cm]{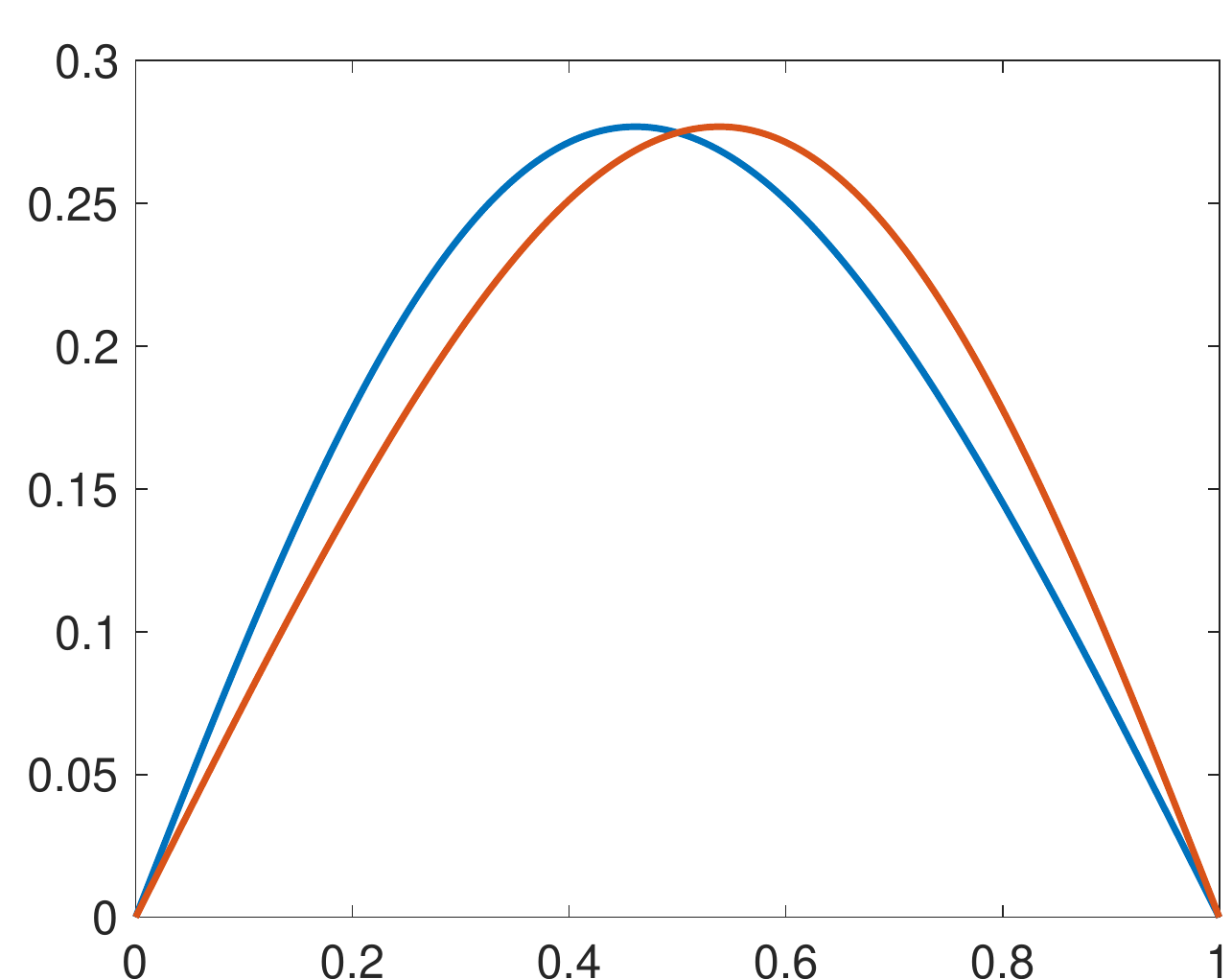}}\hspace*{0.1cm}
\subfigure[second derivatives]{\includegraphics[height=4.1cm]{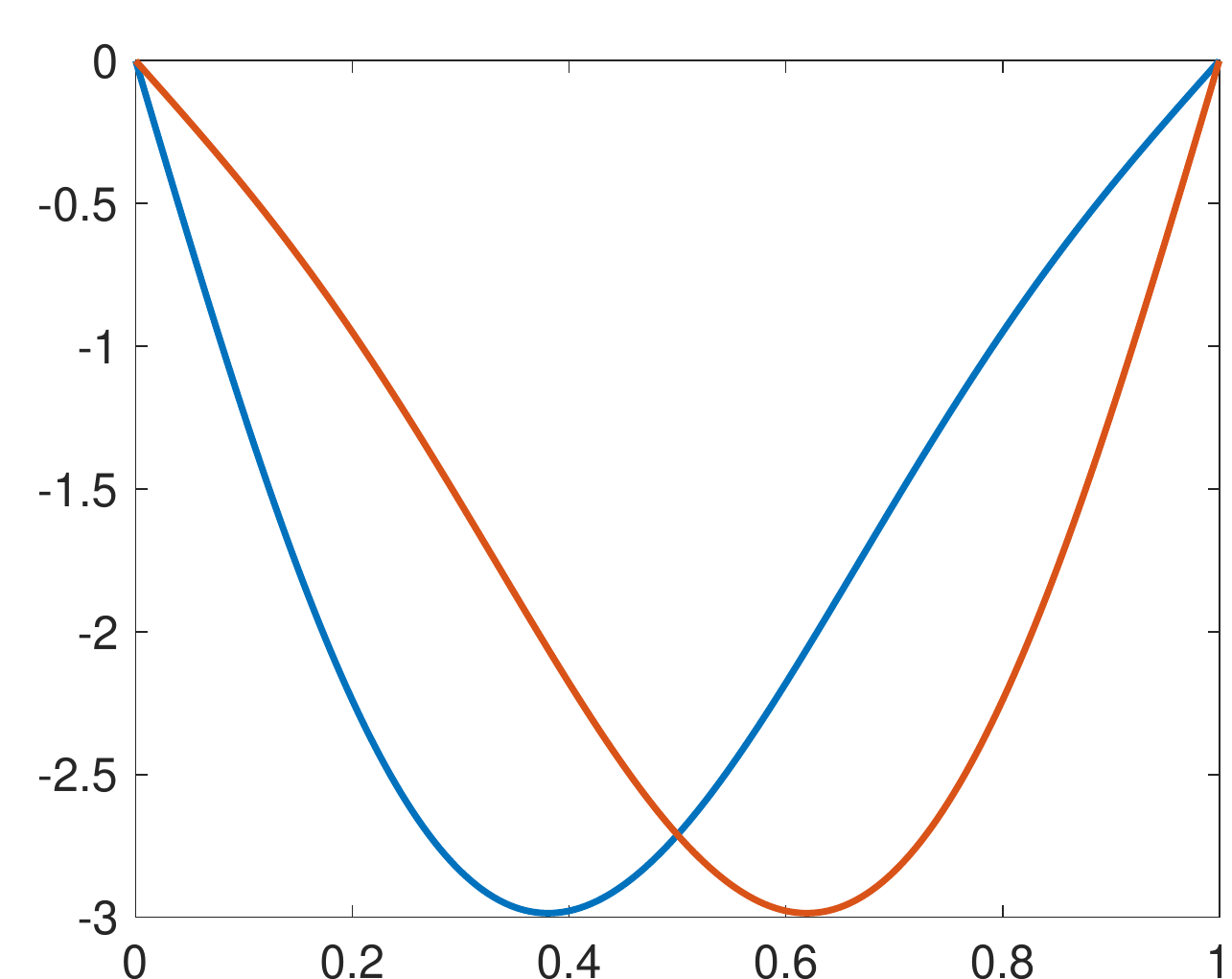}}\hspace*{0.1cm}
\subfigure[fourth derivatives]{\includegraphics[height=4.1cm]{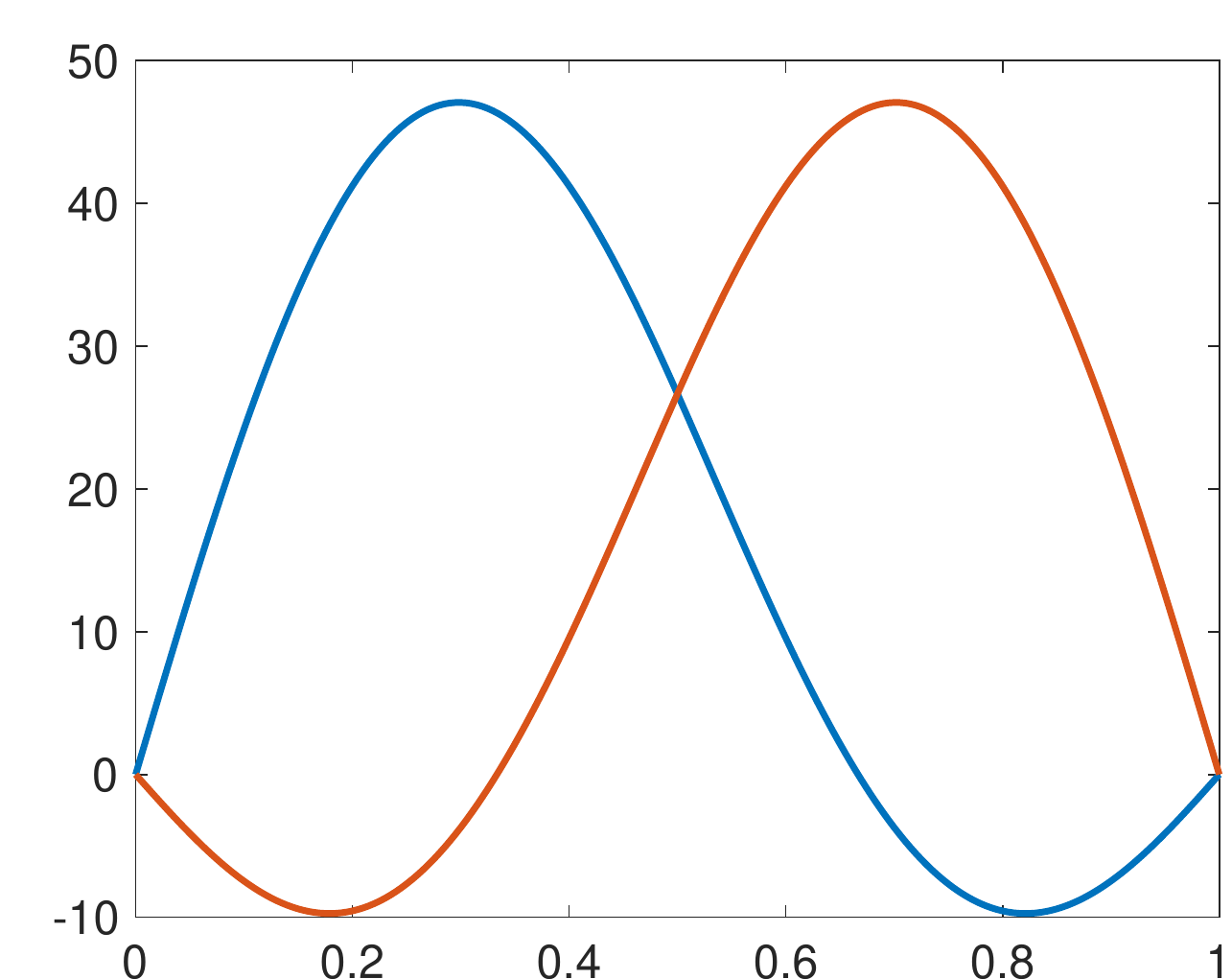}}\\
\subfigure[sixth derivatives]{\includegraphics[height=4.1cm]{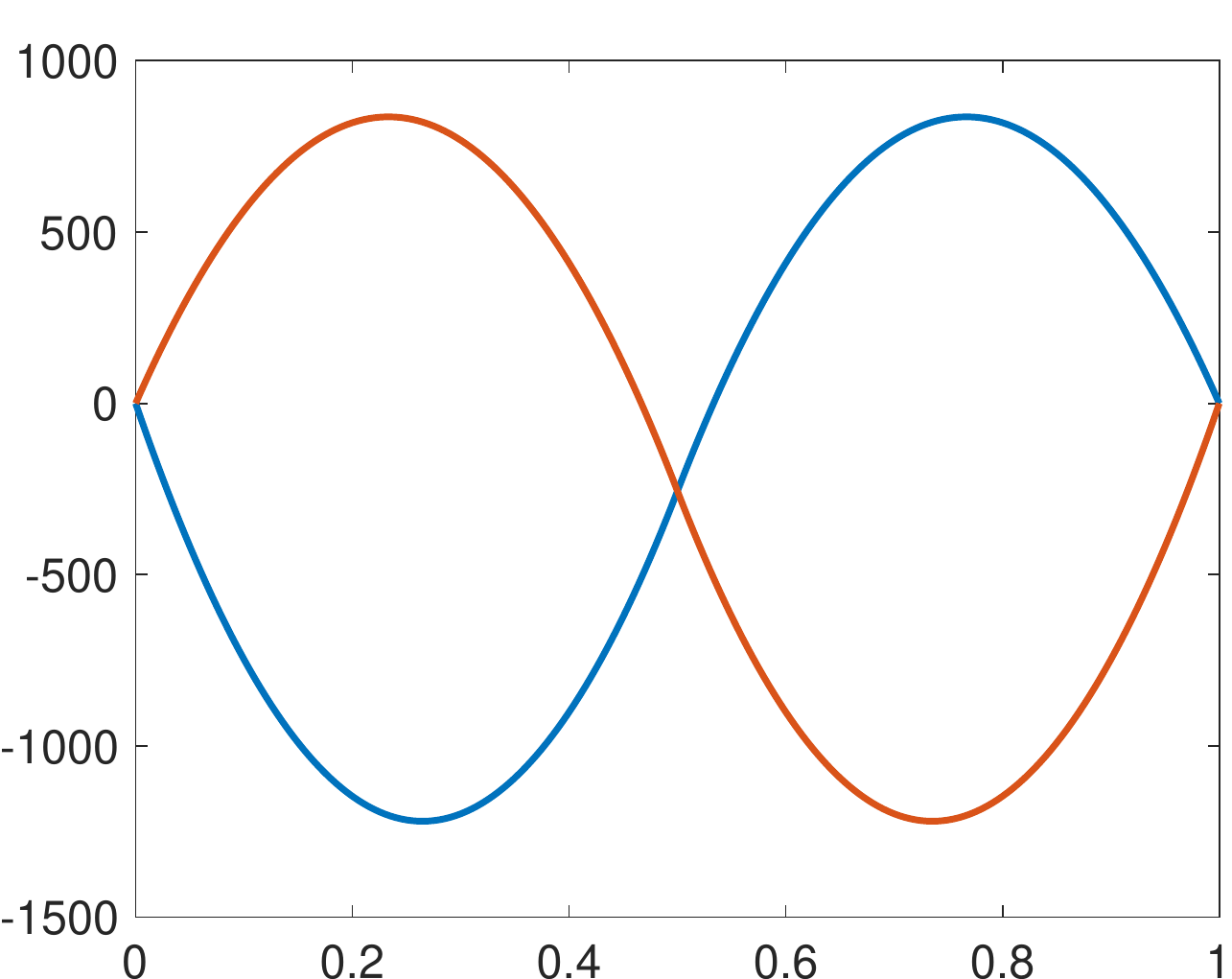}}\hspace*{0.1cm}
\subfigure[eighth derivatives]{\includegraphics[height=4.1cm]{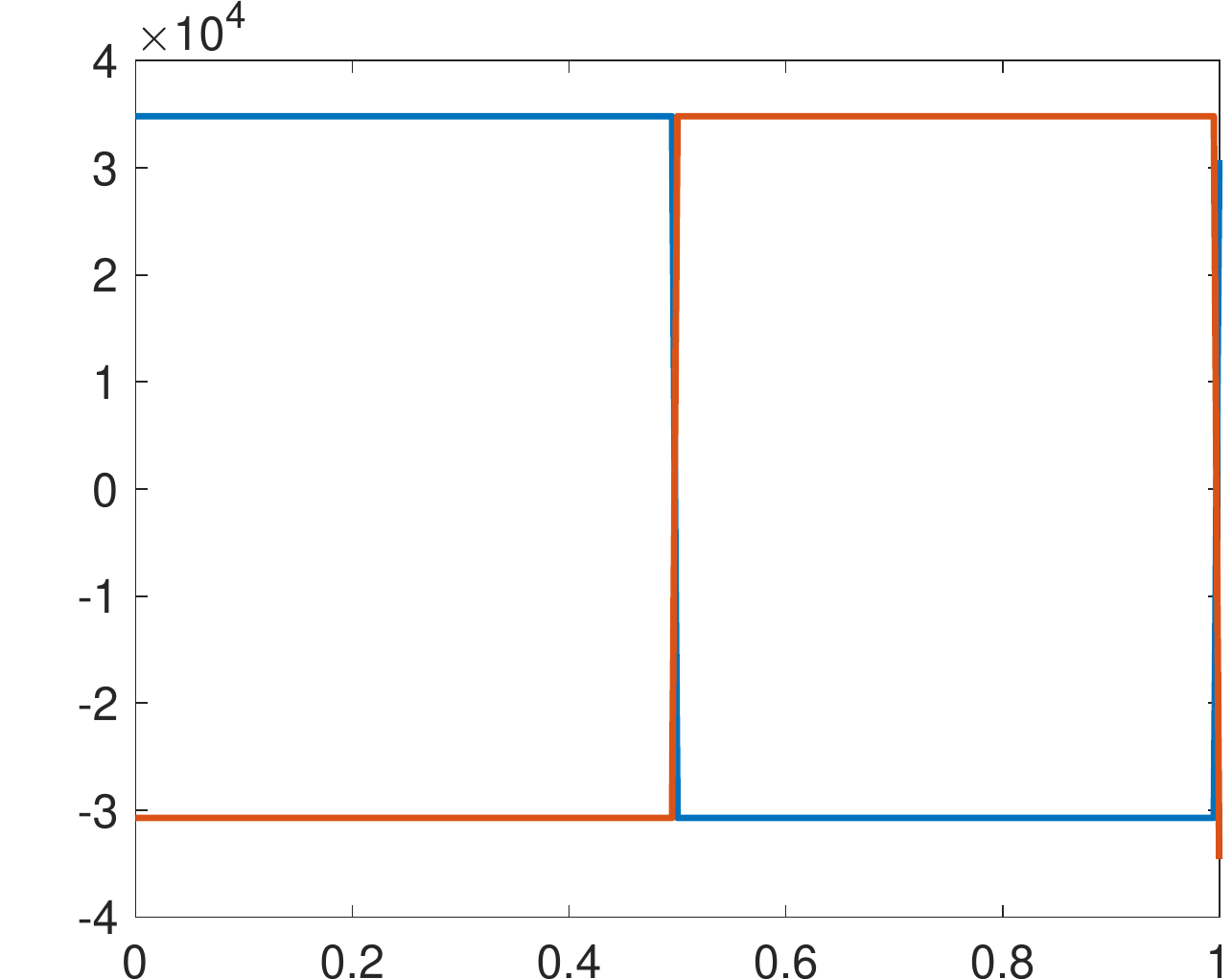}}
\caption{Example~\ref{ex:basis-uniform-even}: B-spline-like functions and their even order derivatives for the space $\overline{\mathbb{S}}_{p,\nknots,0}$ with $p=8$ and $\nknots=2$.} \label{fig:basis.uniform.even:b}
\end{figure}
\begin{example}
 	\label{ex:basis-uniform-even}
 	For $p=2$ and $\nknots=6$ the matrix in \eqref{eq:basis-uniform-even} has $6$ rows and $8$ columns, and it takes the form
 	\begin{equation*}
 	\begin{bmatrix}
 	-1&1&0&0&0&0&0&0
 	\\
 	0&0&1&0&0&0&0&0
 	\\
 	0&0&0&1&0&0&0&0
 	\\
 	0&0&0&0&1&0&0&0
 	\\
 	0&0&0&0&0&1&0&0
 	\\
 	0&0&0&0&0&0&1&-1
 	\end{bmatrix}.
 	\end{equation*}
 	For $p=8$ and $\nknots=2$ the matrix in \eqref{eq:basis-uniform-even} has $2$ rows and $10$ columns, and it takes the form
 	\begin{equation*}
 	\begin{bmatrix}
 	1&0&0&-1&1&0&0&-1&1&0
 	\\
 	0&1&-1&0&0&1&-1&0&0&1
 	\end{bmatrix}.
 	\end{equation*}
 	The graph of the corresponding B-spline-like functions and their even order derivatives are depicted in Figures~\ref{fig:basis.uniform.even:a} and~\ref{fig:basis.uniform.even:b}. One clearly notices that the functions satisfy the boundary conditions of the space $\overline{\mathbb{S}}_{p,\nknots,0}$.
\end{example}

With the same line of arguments as the proof of Proposition~\ref{prop:basis-0} we arrive at the following result.
\begin{proposition}
\label{prp:basis-uniform}
The functions defined in \eqref{eq:B-spline-uniform-even} are a basis of the space $\overline{\mathbb{S}}_{p,\nknots,0}$ for $p$ even.
\end{proposition}

\begin{remark}
The basis functions in \eqref{eq:B-spline-0-odd}, \eqref{eq:B-spline-0-even} and \eqref{eq:B-spline-uniform-even} are defined for any number of elements $\nknots>2$ and any degree $ p$. While the case $p\gg \nknots$ has a theoretical interest for analyzing the convergence in $p$, the most interesting practical case is $\nknots\gg p$. When $\nknots$ is large with respect to $p$ the boundary constraints at the two ends of the interval involve disjoint sets of (scaled cardinal) B-splines and the construction of the bases in \eqref{eq:B-spline-0-odd}, \eqref{eq:B-spline-0-even} and \eqref{eq:B-spline-uniform-even} is particularly easy. Only few basis functions near the ends have to be modified and the matrices in \eqref{eq:basis-0-odd}, \eqref{eq:basis-0-even} and \eqref{eq:basis-uniform-even} are basically identity matrices with the addition of very few columns to the left and to the right. The added columns have either zero entries or come from exchange matrices; see Examples~\ref{ex:basis-0-odd}--\ref{ex:basis-uniform-even}. 
\end{remark}

\begin{remark}
A similar construction of B-spline-like bases was proposed in \cite{Hiemstra:2021} for the reduced spline spaces considered in that paper. It is also an extraction procedure, but in terms of open-knot B-splines instead of cardinal B-splines. As a consequence, the corresponding extraction matrices are not known in explicit form and need to be computed algorithmically (in the spirit of the MDB-spline construction \cite{Speleers:2019,Toshniwal:2020}). There is also a restriction on the minimum number of elements so that the boundary constraints at the two ends of the interval are well separated.
\end{remark}

 
\section{Numerical examples}
\label{sec:numerics}
In this section we consider some numerical tests to show the potential of outlier-free spline Galerkin discretizations. For the sake of brevity, we will just focus on problems with Dirichlet boundary conditions; the remaining cases are completely analogous.

\subsection{Univariate problems}\label{sec:numerics-1D}
We first show the numerical performance of the presented strategies in the univariate setting. We consider both the eigenvalue problem \eqref{eq:prob-eigenv-1D} and second-order problems of the form
\begin{equation}
\label{eq:second-order-prob}
\left\{ \begin{aligned}
- u'' &= f, \quad \text{in } (0,1), \\
u(0)&=u(1)=0,
\end{aligned} \right.
\end{equation}
and we approximately solve them by means of Galerkin discretizations in the outlier-free optimal spline space $\mathbb{S}_{p,n,0}^\opt$ and the alternative reduced spline space $\overline{\mathbb{S}}_{p,n,0}$. 
We also compare them with the full spline space $\mathbb{S}_{p,\bftau,0}^0$
of the same dimension $n$, defined by
$$
\tau_i:=\frac{i}{\nknots}, \quad i=0,\ldots, \nknots, \quad \nknots:=n-p+2,
$$
and denote this space by $\mathbb{S}_{p,n,0}$.
Observe that the grid size in the considered discretization spaces of dimension $n$ is different, ranging from
$\frac{1}{n+1}$ to $\frac{1}{n-p+2}$.

\begin{figure}[t!]
\centering
\subfigure[$e_{\omega,\indeigk}$ in $\mathbb{S}_{p,200,0}^\opt$]{\includegraphics[height=4.1cm]{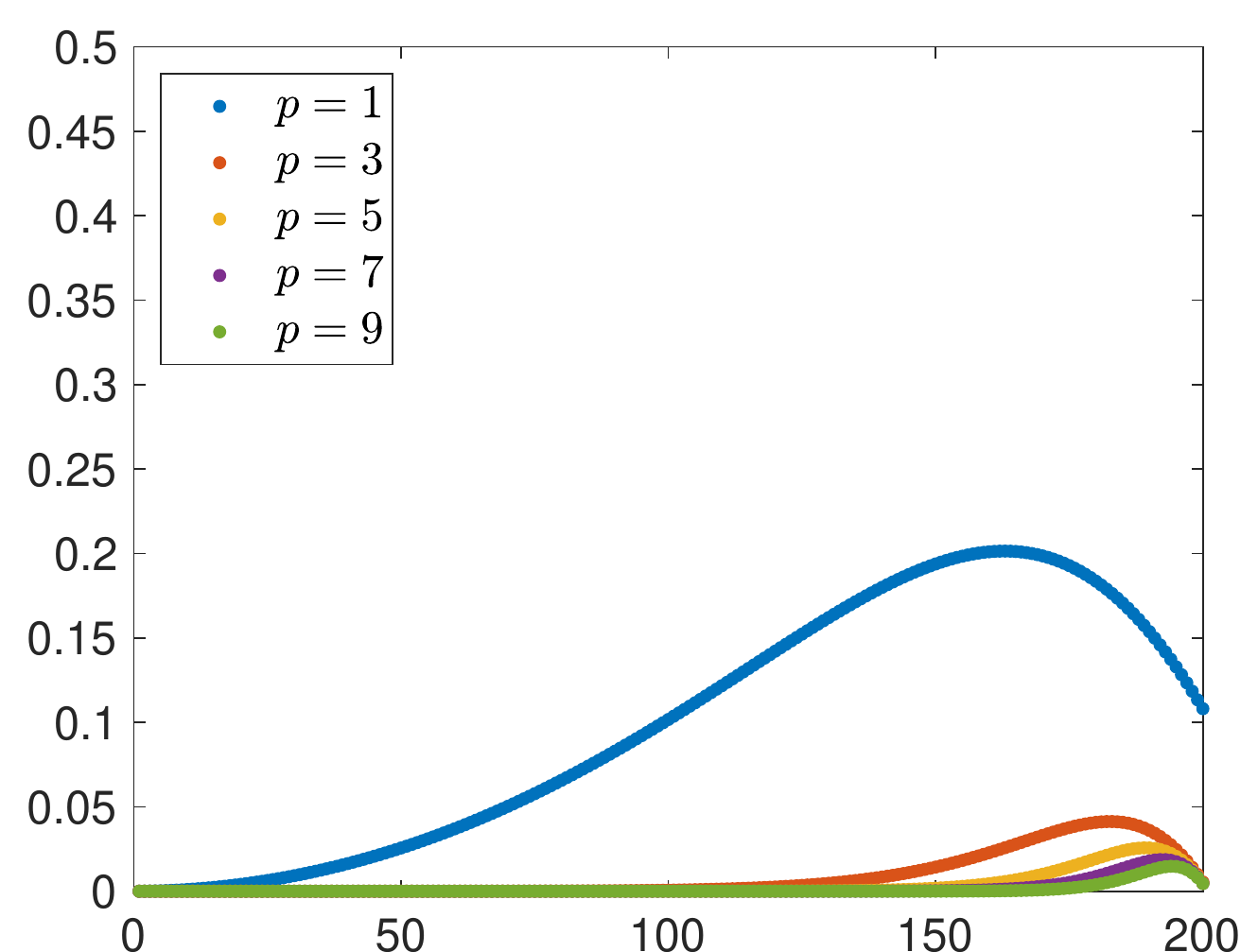}}\hspace*{0.1cm}
\subfigure[$e_{\omega,\indeigk}$ in $\mathbb{S}_{p,200,0}$]{\includegraphics[height=4.1cm]{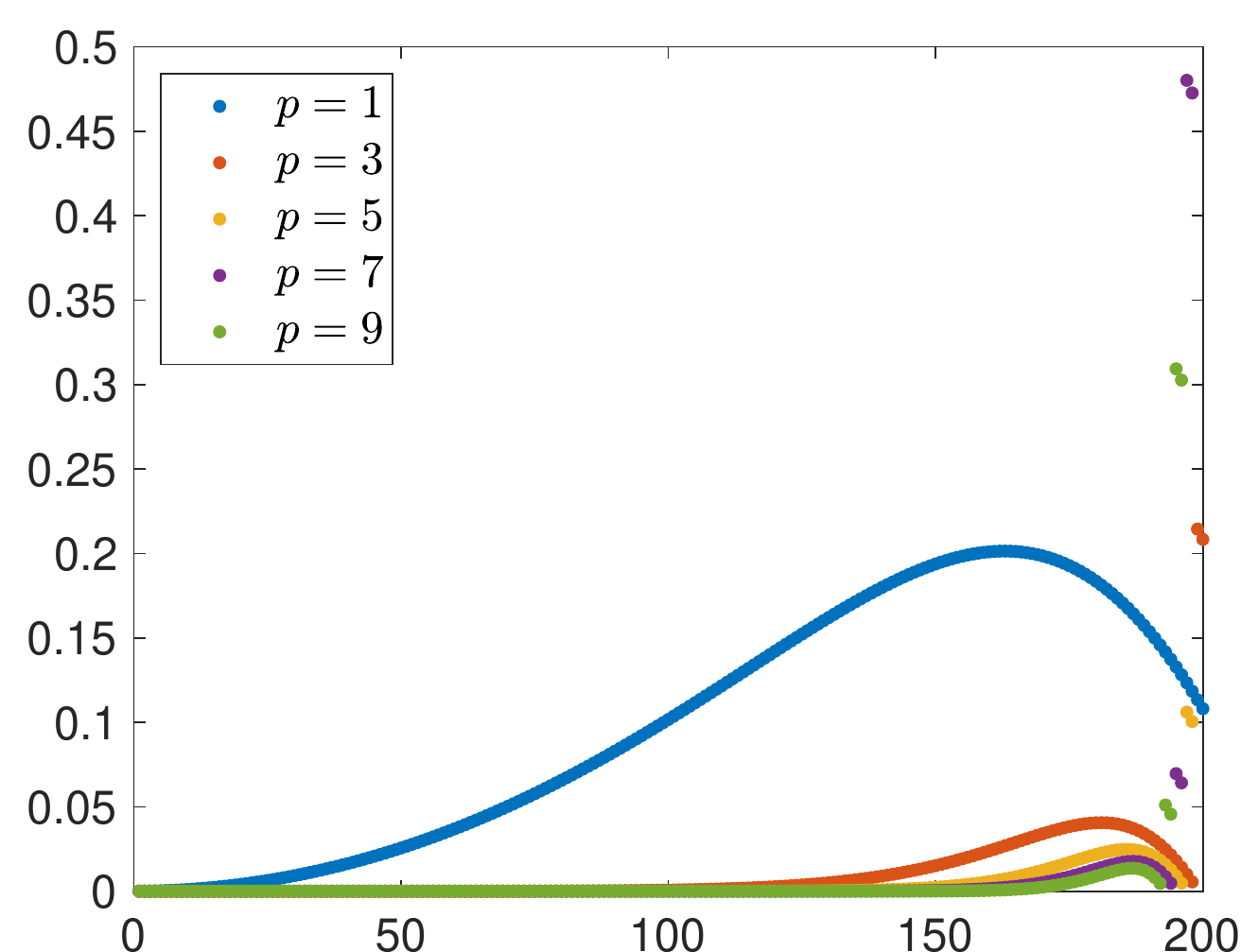}}\hspace*{0.1cm}
\subfigure[zoom out for $\mathbb{S}_{p,200,0}$]{\includegraphics[height=4.1cm]{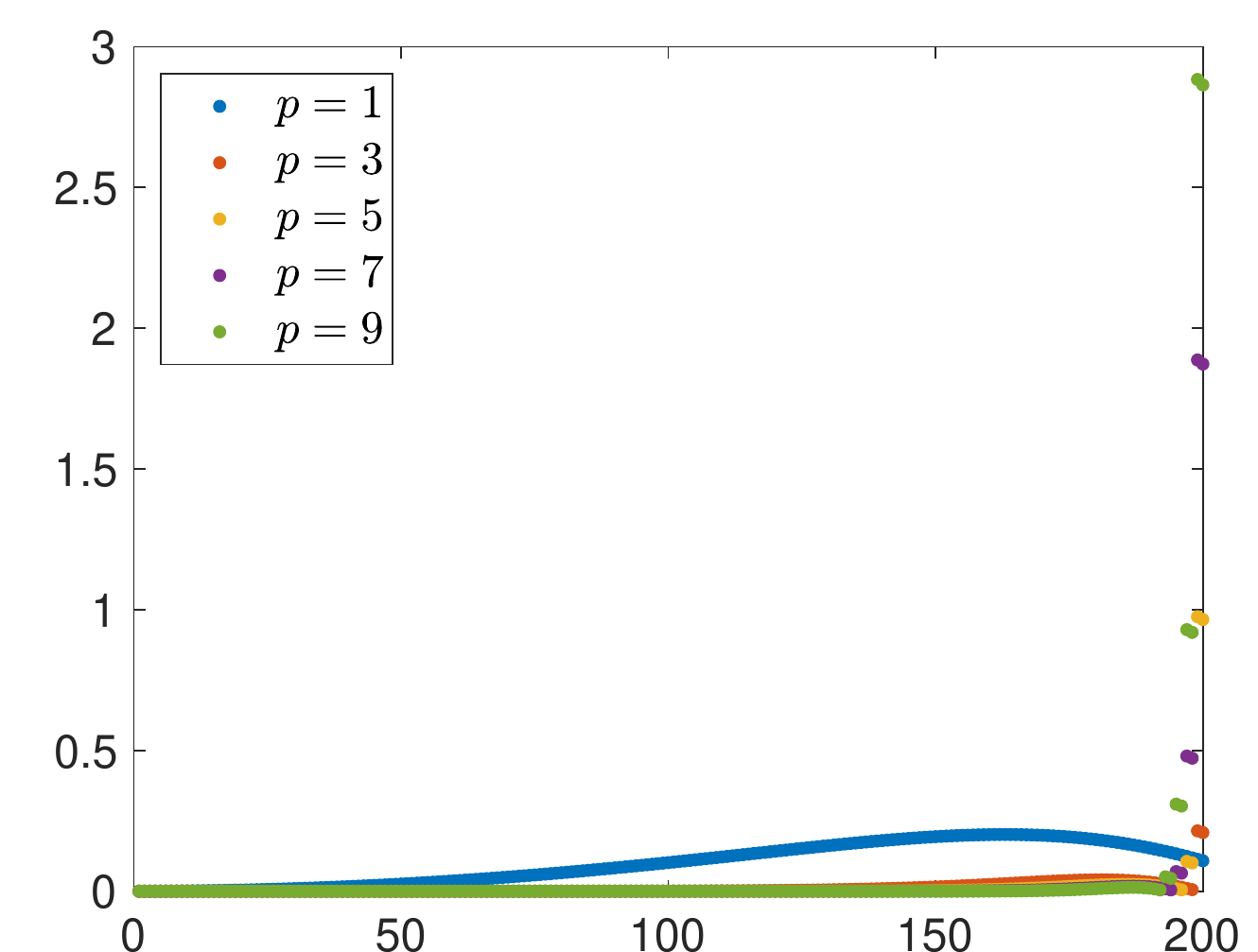}}\\
\subfigure[$e_{u,\indeigk}$ in $\mathbb{S}_{p,200,0}^\opt$]{\includegraphics[height=4.1cm]{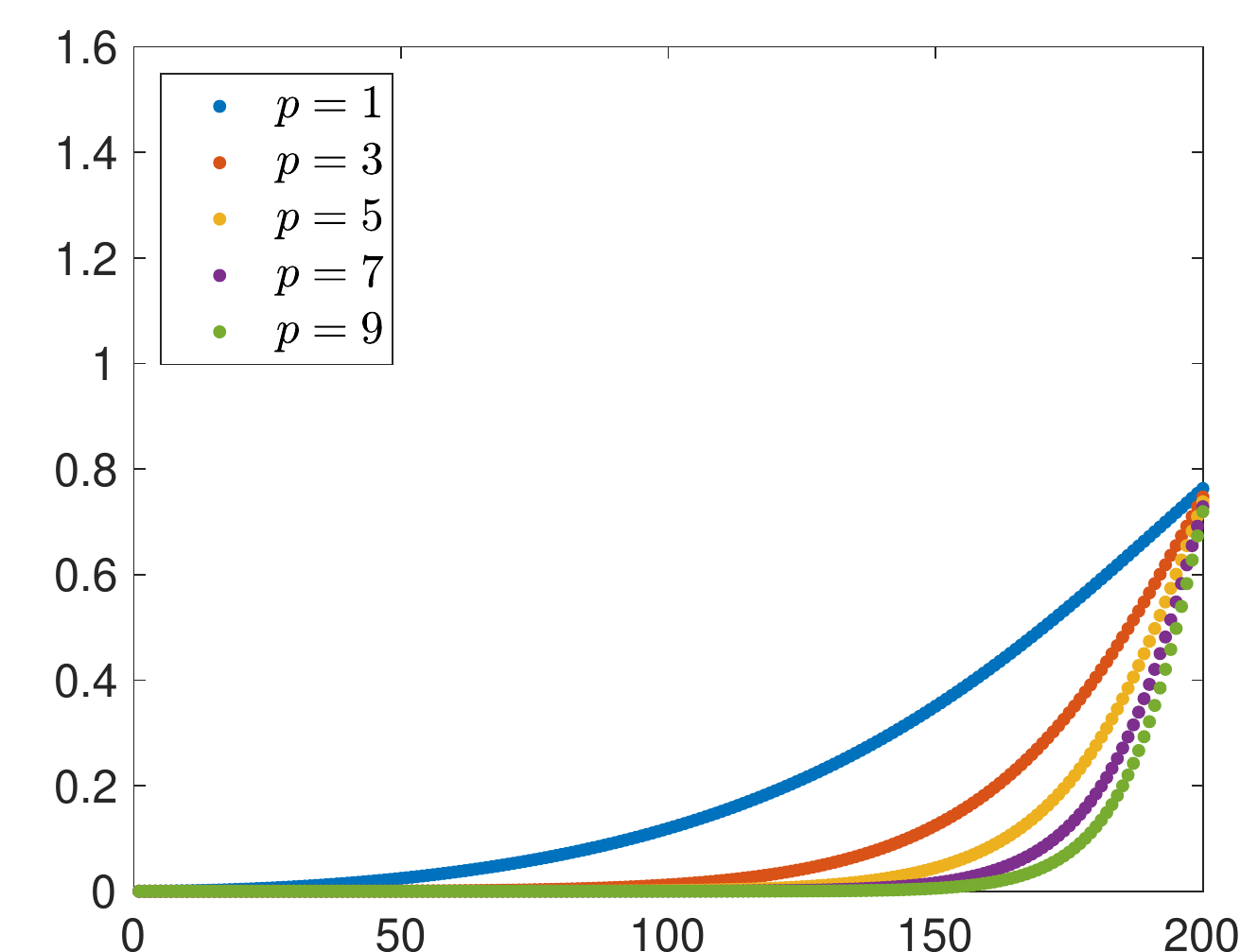}}\hspace*{0.1cm}
\subfigure[$e_{u,\indeigk}$ in $\mathbb{S}_{p,200,0}$]{\includegraphics[height=4.1cm]{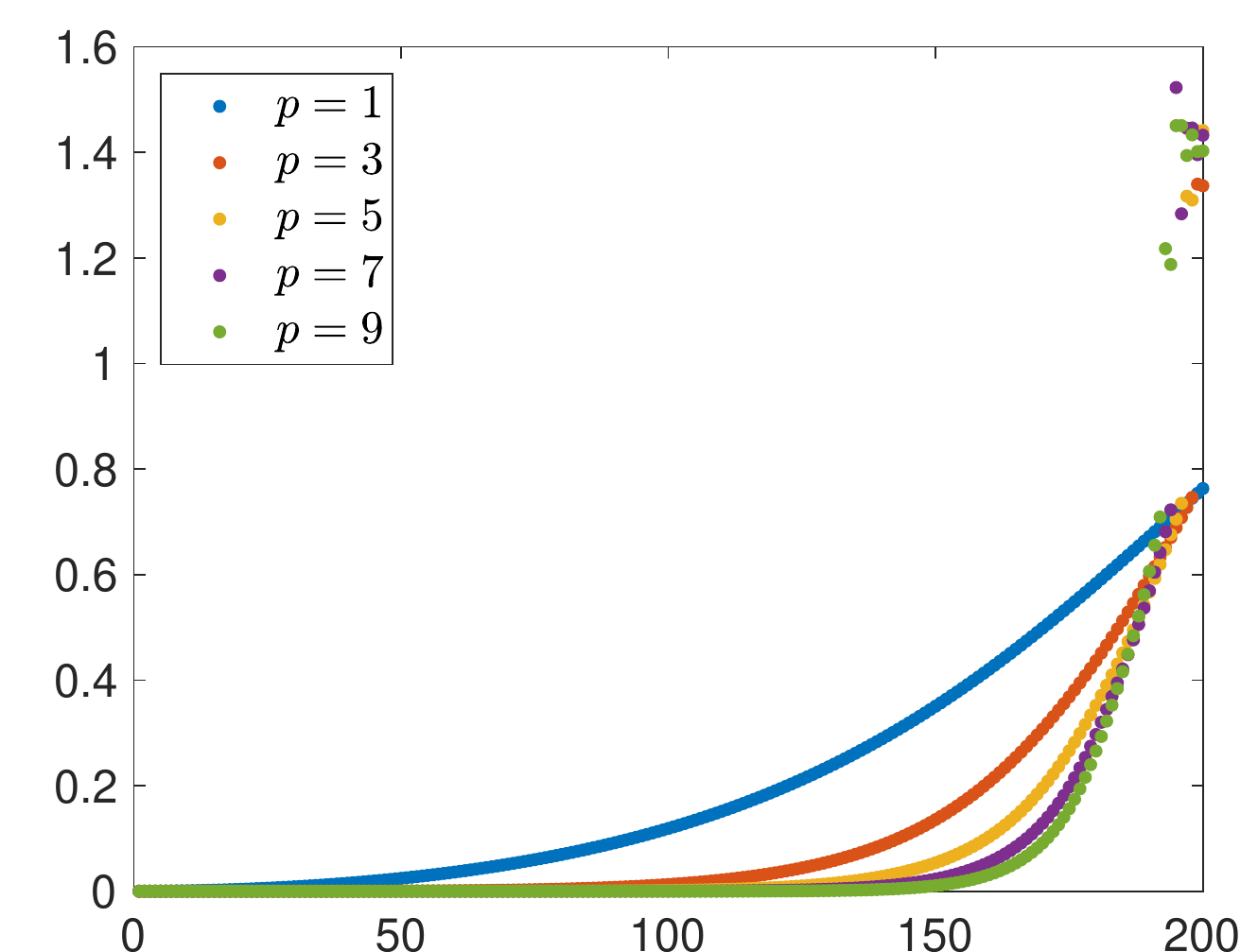}}
\caption{Example~\ref{ex:eigenvalues1D}: Relative frequency errors $e_{\omega,\indeigk}$ in \eqref{eq:error-eigenvalues1D} and $L^2$ relative eigenfunction errors $e_{u,\indeigk}$ in \eqref{eq:error-eigenfunctions1D} corresponding to the spline spaces $\mathbb{S}_{p,n,0}^\opt$ and $\mathbb{S}_{p,n,0}$ for odd degrees $p$ and $n=200$. The errors are ordered according to increasing exact frequencies. No outliers are observed for $\mathbb{S}_{p,n,0}^\opt$. Some outliers of $\mathbb{S}_{p,n,0}$ are outside the visible range in (b) as illustrated in (c).} \label{fig:eigenvalues1D.odd}
\end{figure}
\begin{figure}[t!]
\centering
\subfigure[$e_{\omega,\indeigk}$ in $\mathbb{S}_{p,200,0}^\opt$]{\includegraphics[height=4.1cm]{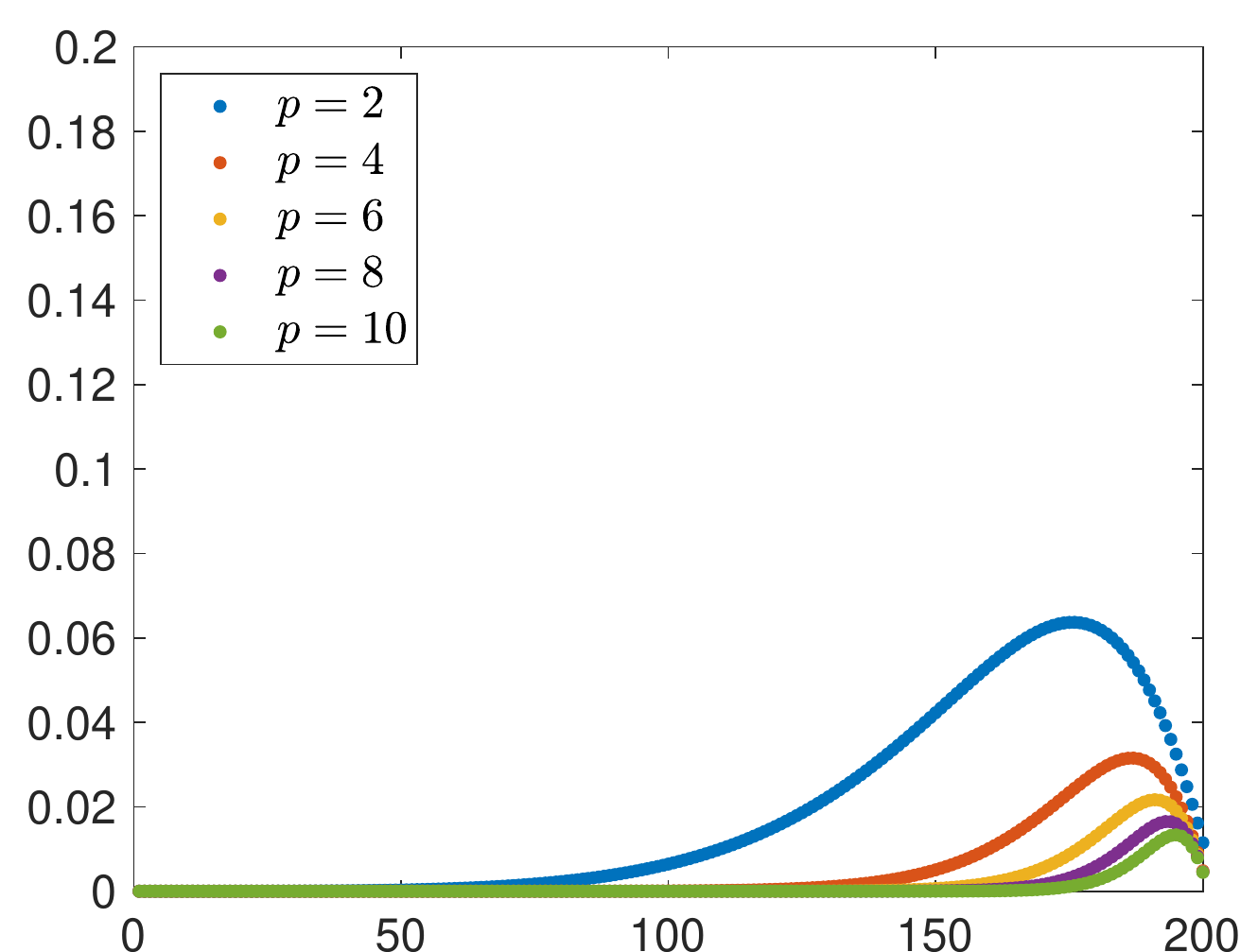}}\hspace*{0.1cm}
\subfigure[$e_{\omega,\indeigk}$ in $\overline{\mathbb{S}}_{p,200,0}$]{\includegraphics[height=4.1cm]{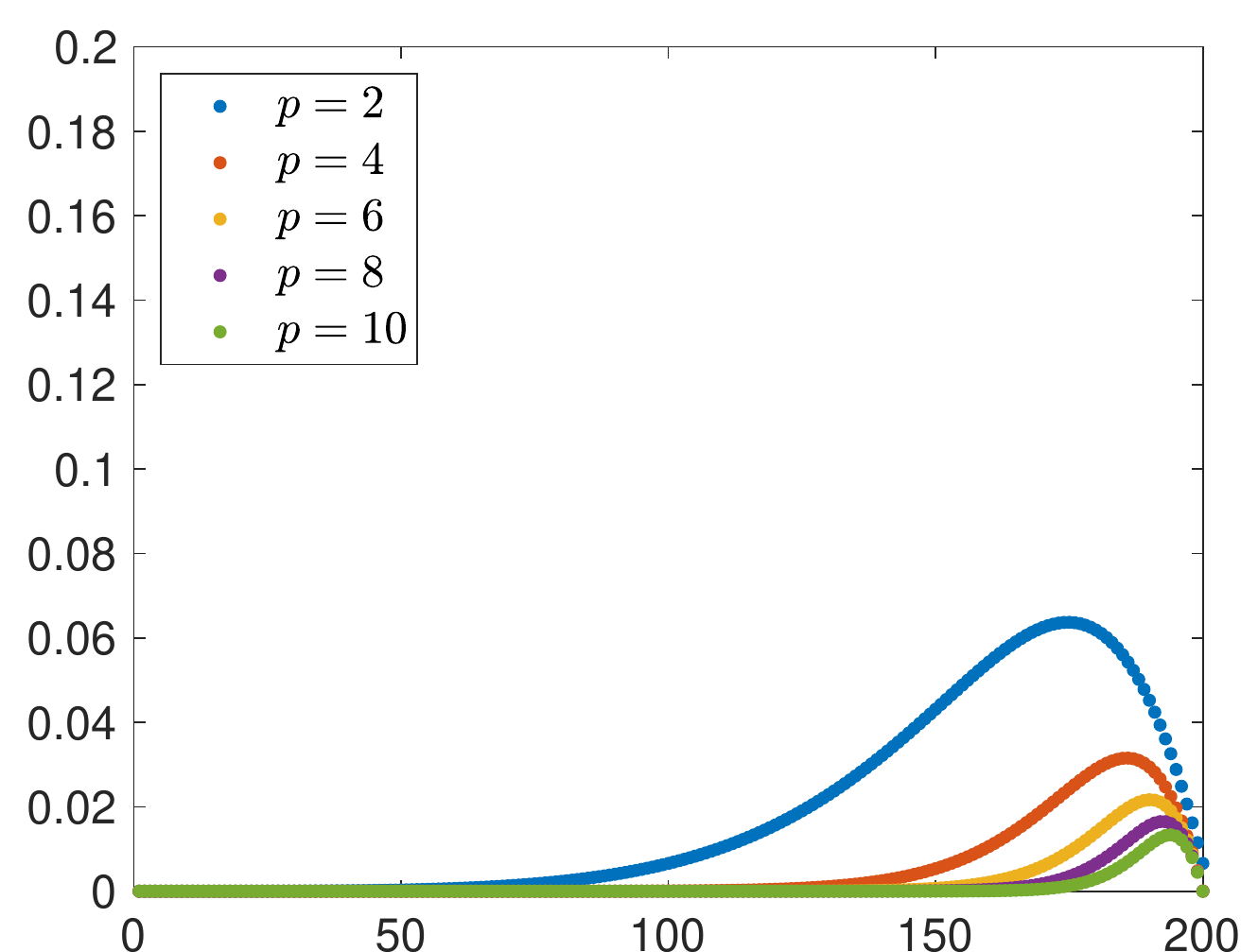}}\hspace*{0.1cm}
\subfigure[$e_{\omega,\indeigk}$ in $\mathbb{S}_{p,200,0}$]{\includegraphics[height=4.1cm]{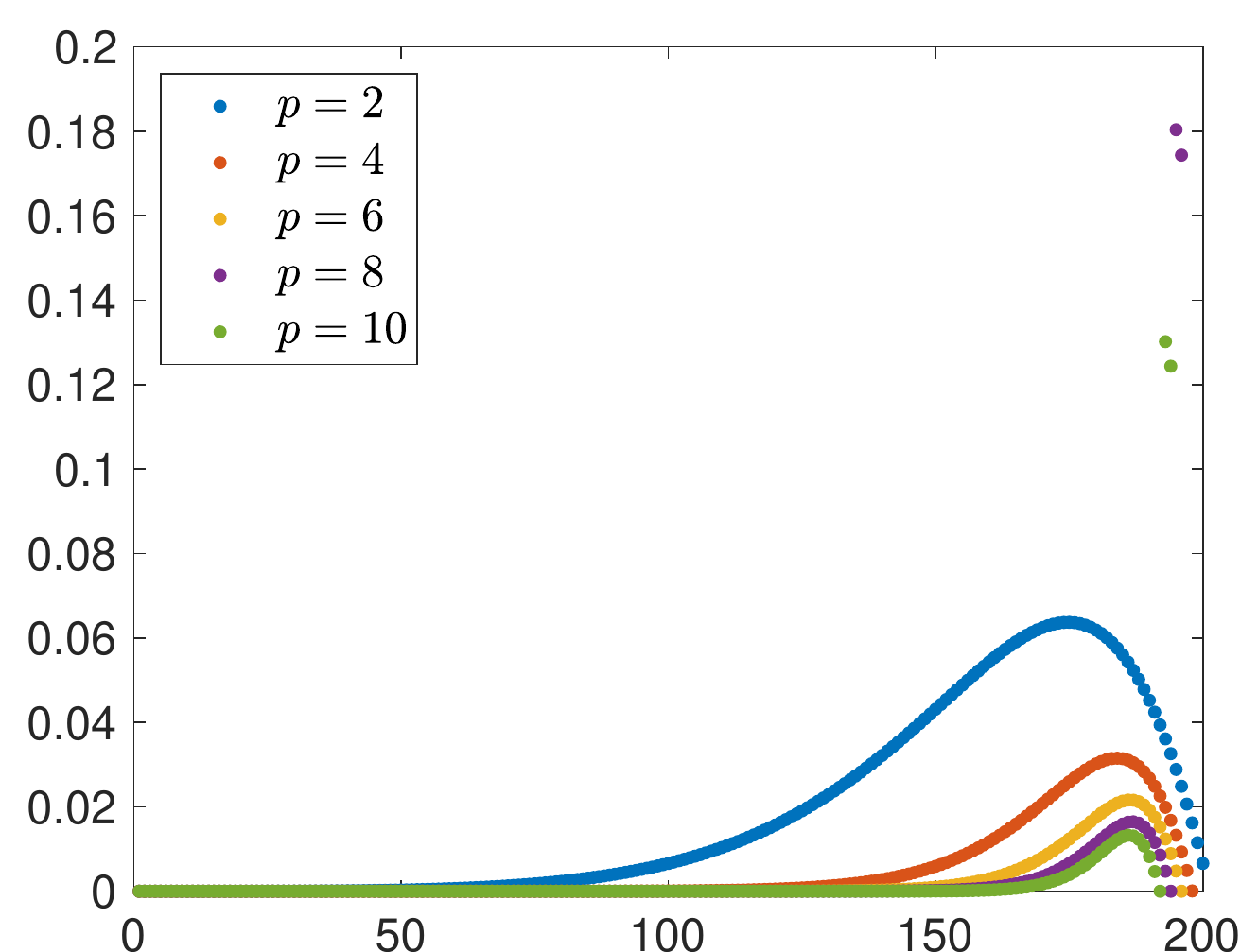}}\\
\subfigure[$e_{u,\indeigk}$ in $\mathbb{S}_{p,200,0}^\opt$]{\includegraphics[height=4.1cm]{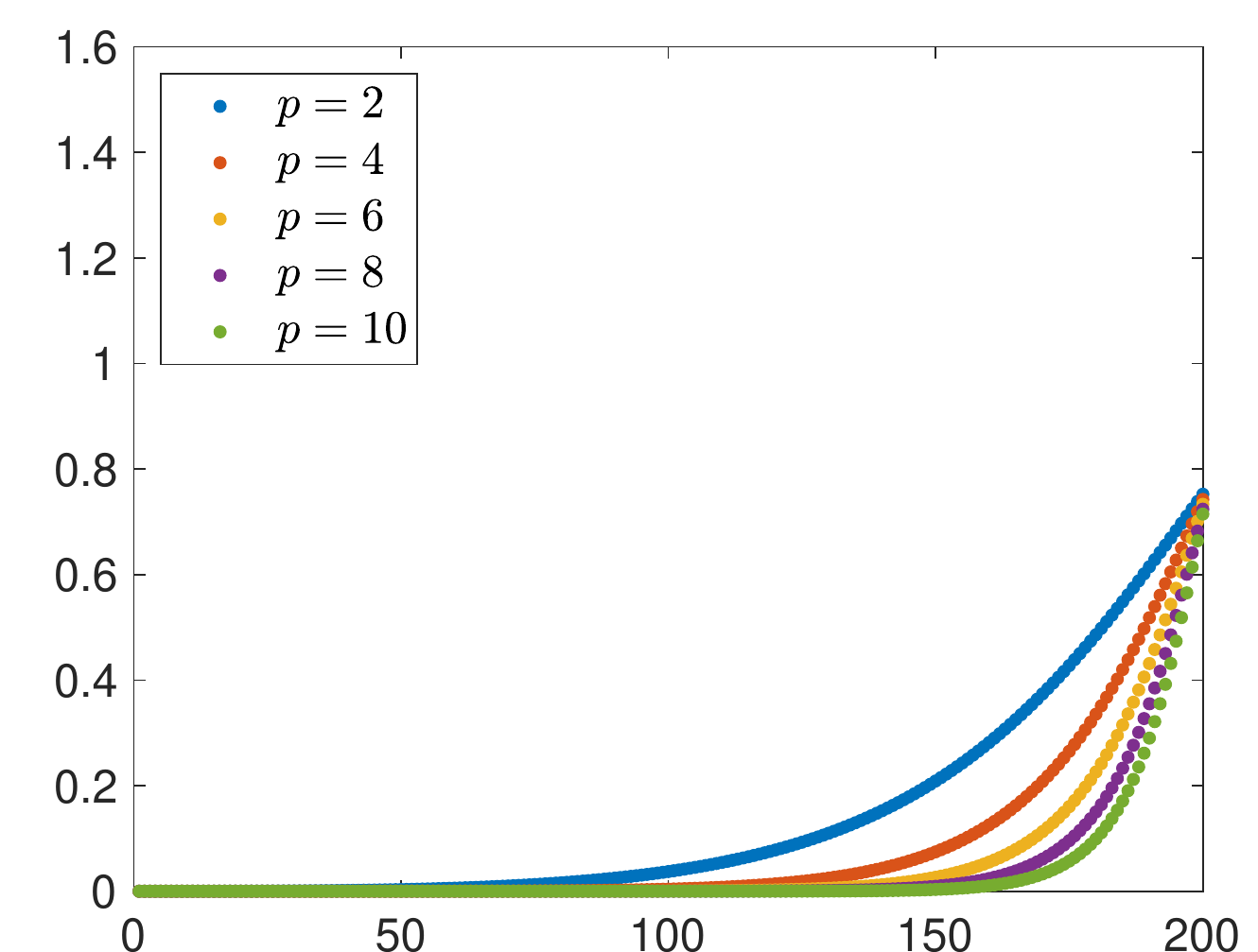}}\hspace*{0.1cm}
\subfigure[$e_{u,\indeigk}$ in $\overline{\mathbb{S}}_{p,200,0}$]{\includegraphics[height=4.1cm]{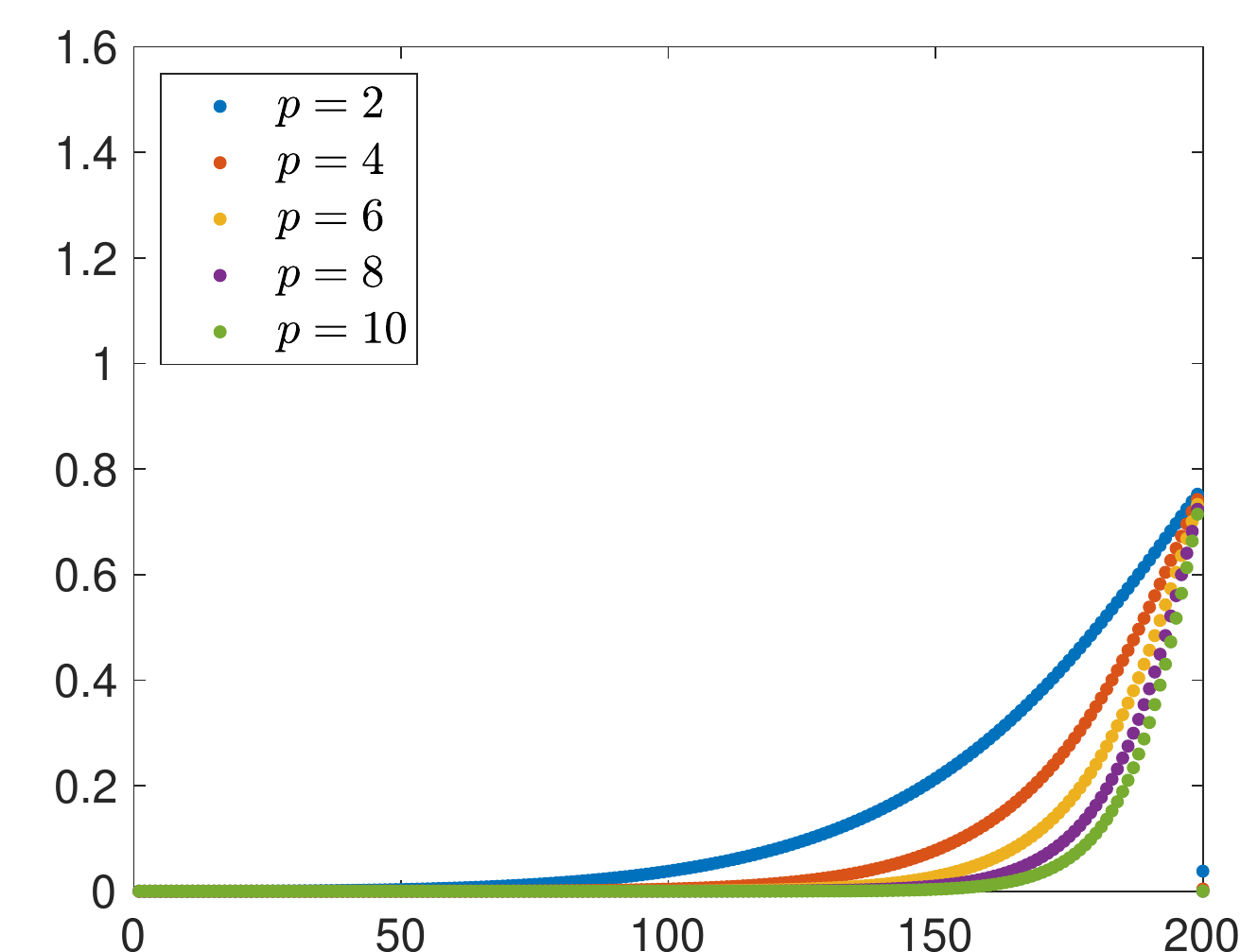}}\hspace*{0.1cm}
\subfigure[$e_{u,\indeigk}$ in $\mathbb{S}_{p,200,0}$]{\includegraphics[height=4.1cm]{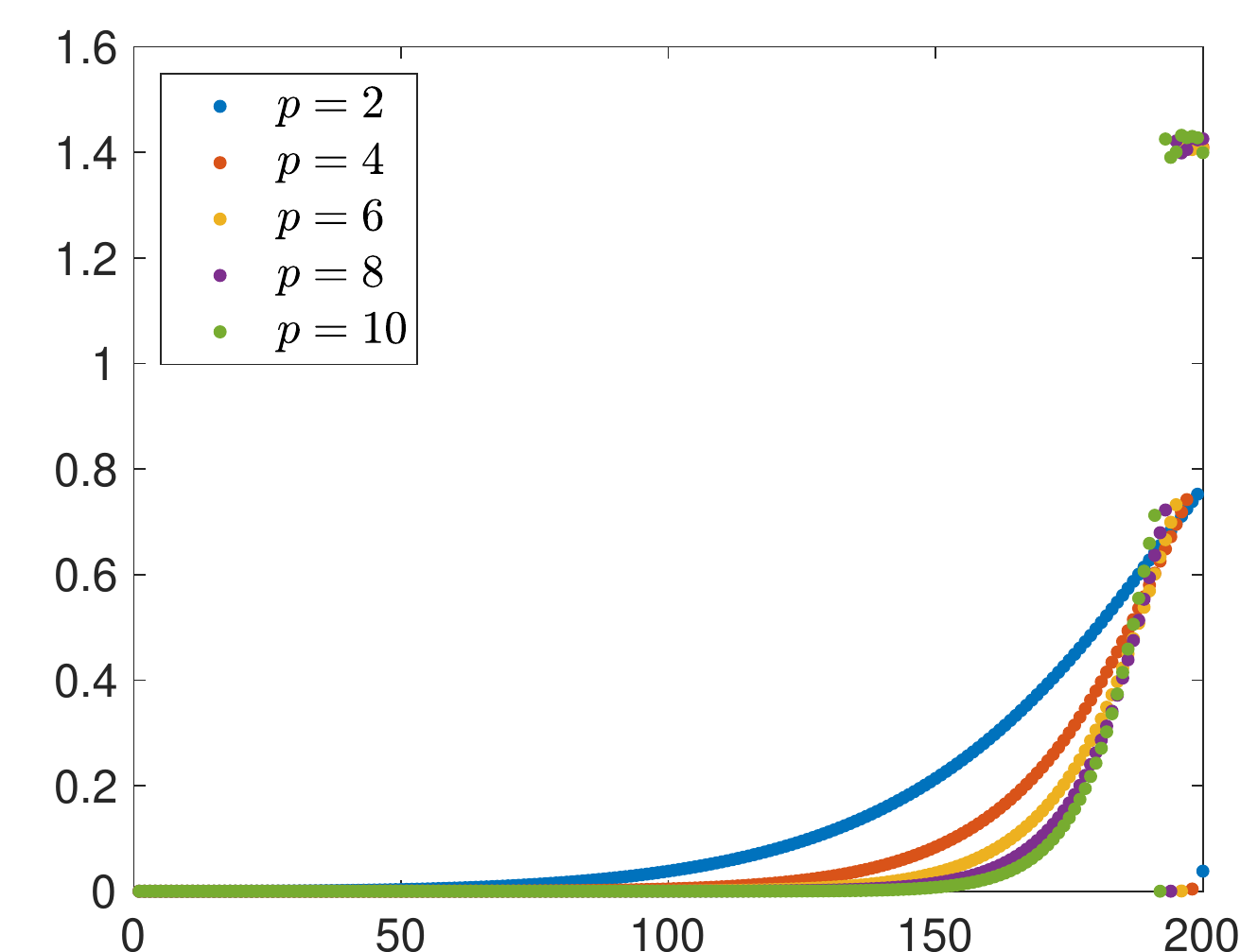}}
\caption{Example~\ref{ex:eigenvalues1D}: Relative frequency errors $e_{\omega,\indeigk}$ in \eqref{eq:error-eigenvalues1D} and $L^2$ relative eigenfunction errors $e_{u,\indeigk}$ in \eqref{eq:error-eigenfunctions1D} corresponding to the spline spaces $\mathbb{S}_{p,n,0}^\opt$, $\overline{\mathbb{S}}_{p,n,0}$ and $\mathbb{S}_{p,n,0}$ for even degrees $p$ and $n=200$. The errors are ordered according to increasing exact frequencies. No outliers are observed for $\mathbb{S}_{p,n,0}^\opt$ and $\overline{\mathbb{S}}_{p,n,0}$. Some outliers of $\mathbb{S}_{p,n,0}$ are outside the visible range in (c); they are not shown for clarity of the figure.} \label{fig:eigenvalues1D.even}
\end{figure}

\begin{example}\label{ex:eigenvalues1D}
In this example we show the performance of the discretizations for approximating the eigenvalue problem \eqref{eq:prob-eigenv-1D}. Let $\omega_{h,\indeigk}$ be the approximate value of the frequency $\omega_\indeigk$; see \eqref{eq:eig-Laplace-type-0-BC}. Here we assume that the exact frequencies and the approximate ones are sorted in ascending order to retrieve their matching.
In Figures~\ref{fig:eigenvalues1D.odd} and~\ref{fig:eigenvalues1D.even} we depict the relative frequency errors 
\begin{equation}\label{eq:error-eigenvalues1D}
e_{\omega,\indeigk}:=\frac{\omega_{h,\indeigk}-\omega_\indeigk}{\omega_\indeigk}, \quad \indeigk=1,\ldots,n,
\end{equation}
obtained by the Galerkin approximations in the spline spaces $\mathbb{S}_{p,n,0}^\opt$, $\overline{\mathbb{S}}_{p,n,0}$ and $\mathbb{S}_{p,n,0}$ for various degrees and $n=200$. We clearly notice that the optimal spline space $\mathbb{S}_{p,n,0}^\opt$ captures all the eigenvalues without any outlier, still maintaining the accuracy of the full spline space. A similar behavior is observed for the reduced spline space $\overline{\mathbb{S}}_{p,n,0}$.
In the same figures we also report the corresponding $L^2$ relative errors for the eigenfunction approximations
\begin{equation}\label{eq:error-eigenfunctions1D}
e_{u,\indeigk}:=\frac{\|u_\indeigk-u_{h,\indeigk}\|}{\|u_\indeigk\|}, \quad \indeigk=1,\ldots,n.
\end{equation}
\end{example}

\begin{figure}[t!]
\centering
\subfigure[$\mathbb{S}_{p,n,0}^\opt$]{\includegraphics[height=4.1cm]{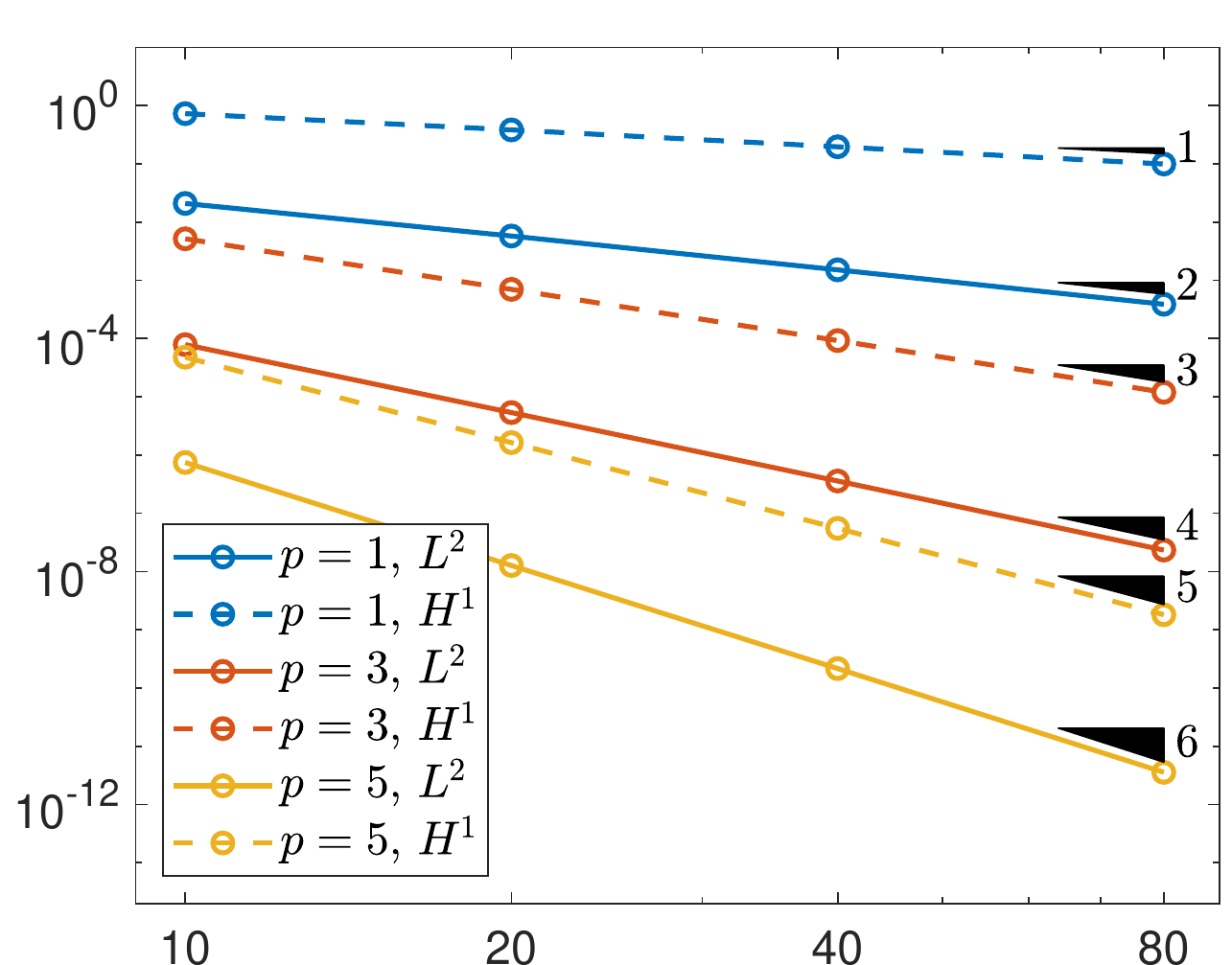}}\hspace*{0.1cm}
\subfigure[$\mathbb{S}_{p,n,0}$]{\includegraphics[height=4.1cm]{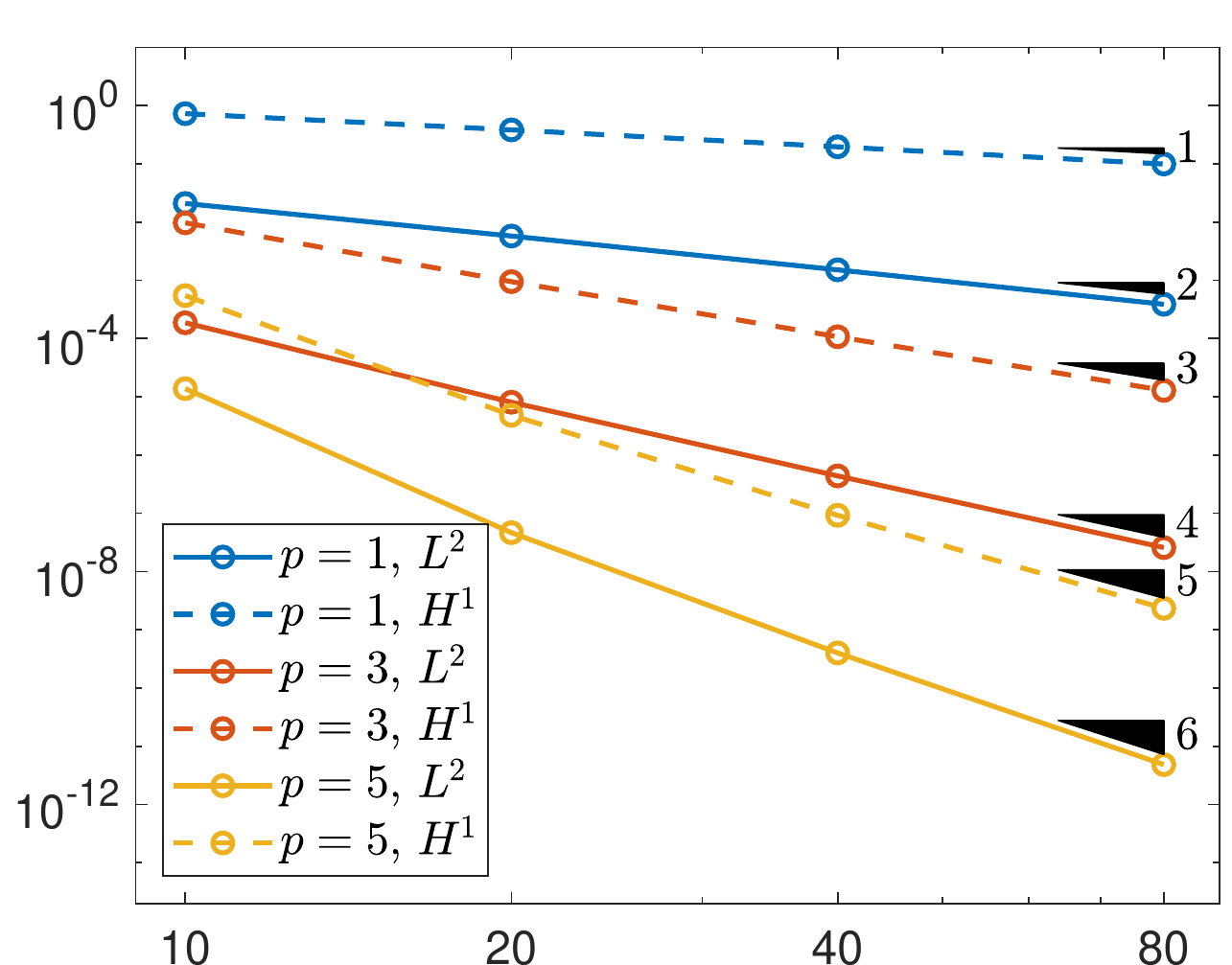}}
\caption{Example~\ref{ex:convergence1D-sin}: $L^2$ and $H^1$ error convergence in the spline spaces $\mathbb{S}_{p,n,0}^\opt$ and $\mathbb{S}_{p,n,0}$ in terms of $n$, for odd degrees $p$. The reference convergence order in $n$ is indicated by black triangles.} \label{ex:convergence1D-sin:a}
\bigskip
\centering
\subfigure[$\mathbb{S}_{p,n,0}^\opt$]{\includegraphics[height=4.1cm]{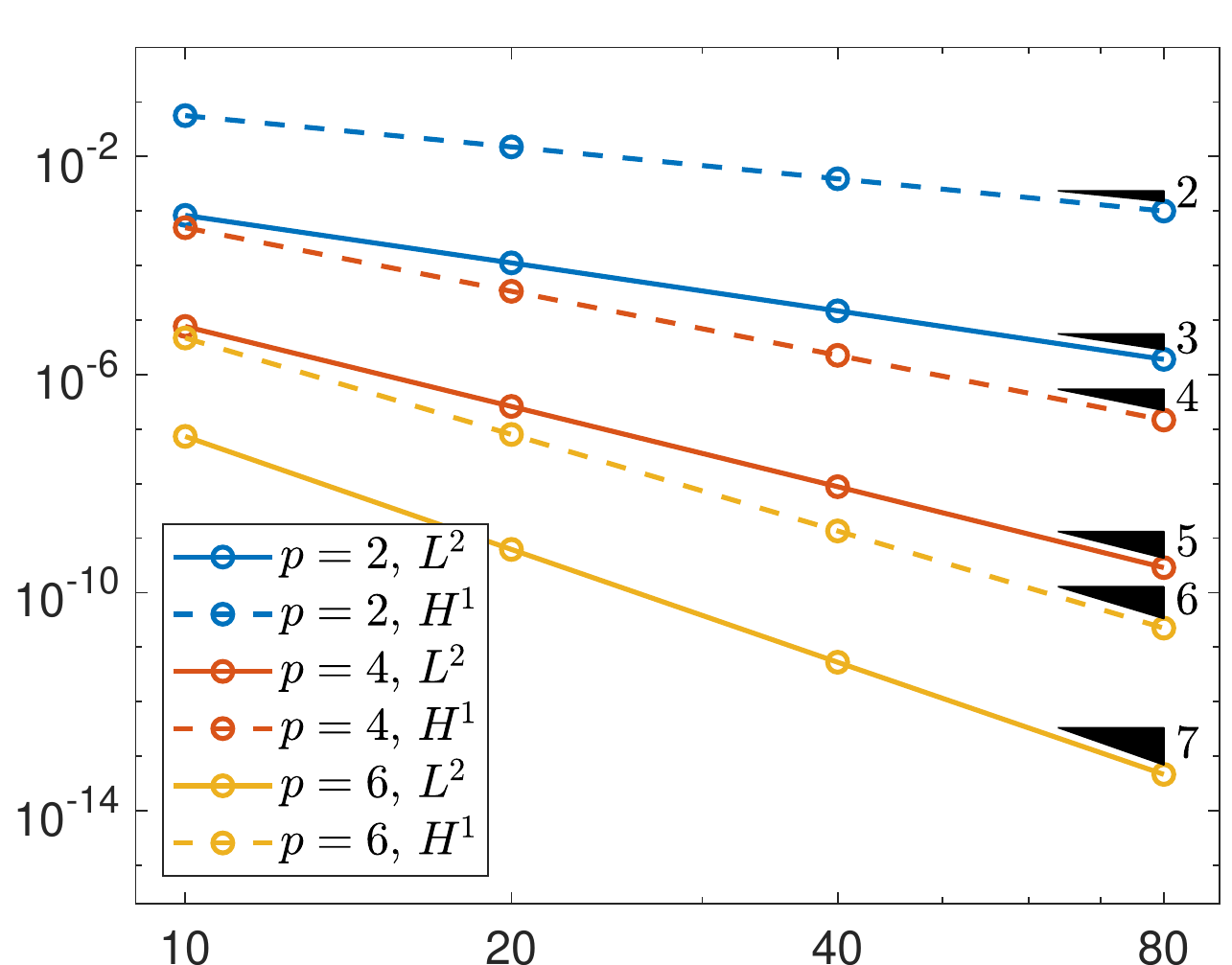}}\hspace*{0.1cm}
\subfigure[$\overline{\mathbb{S}}_{p,n,0}$]{\includegraphics[height=4.1cm]{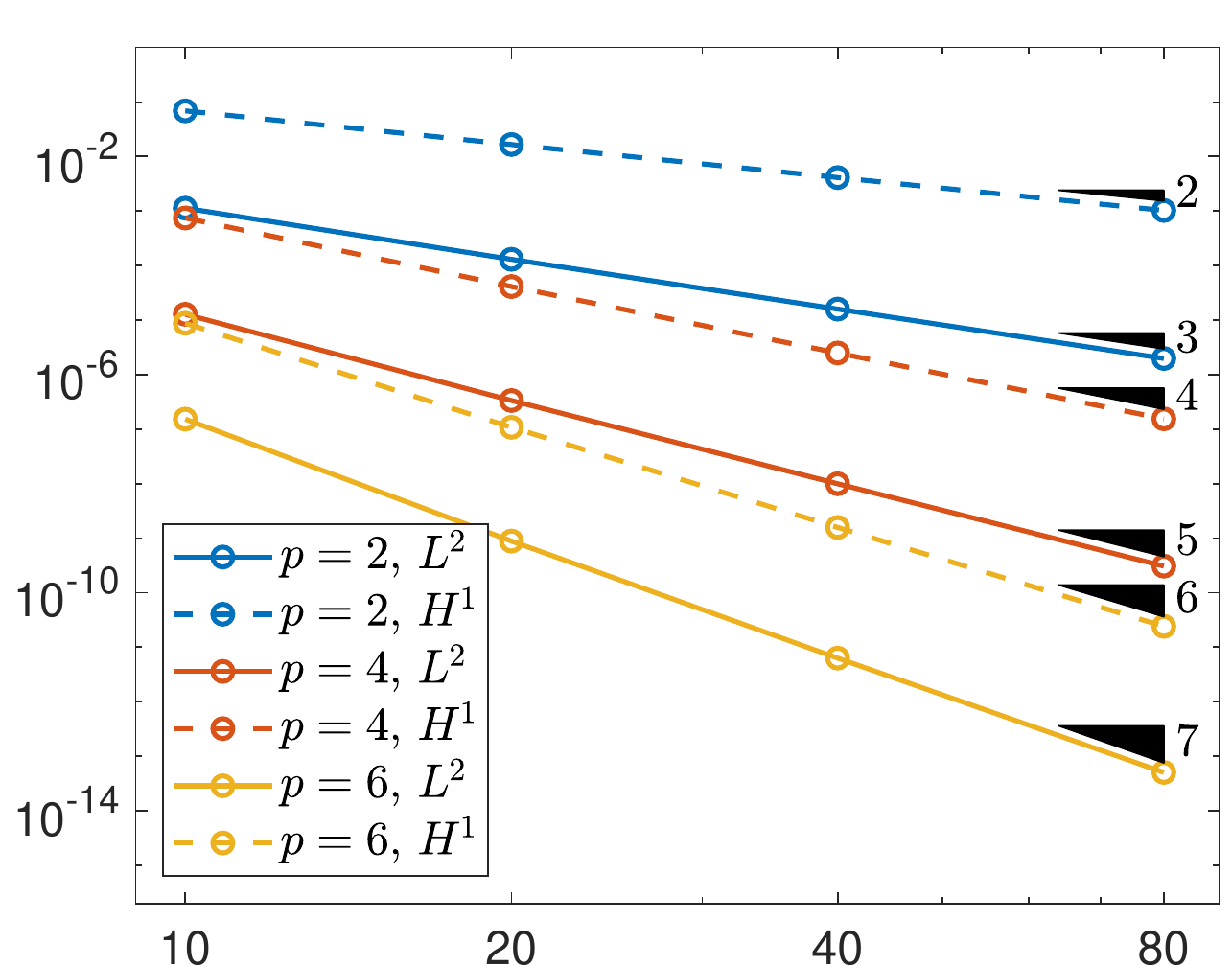}}\hspace*{0.1cm}
\subfigure[$\mathbb{S}_{p,n,0}$]{\includegraphics[height=4.1cm]{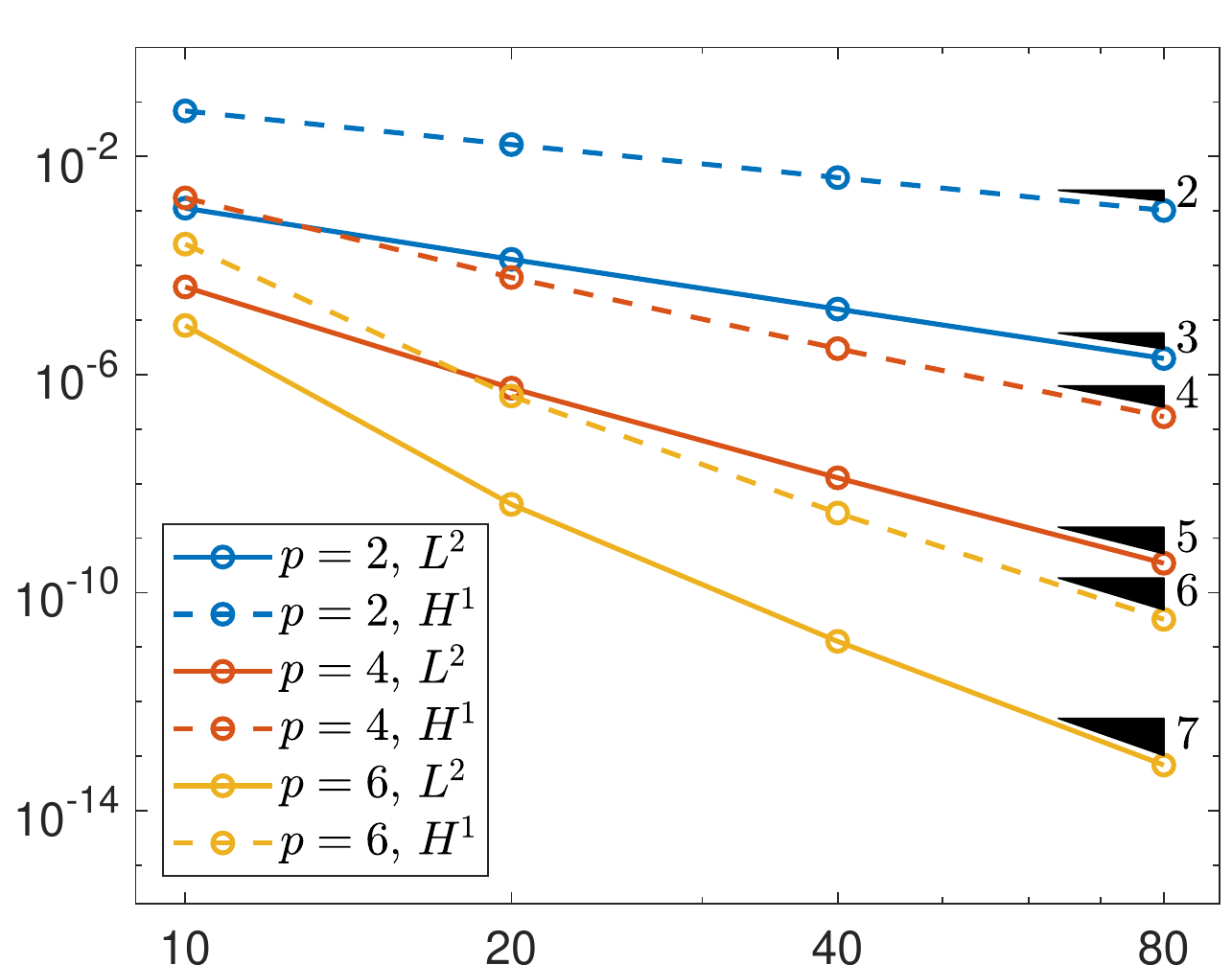}}
\caption{Example~\ref{ex:convergence1D-sin}: $L^2$ and $H^1$ error convergence in the spline spaces $\mathbb{S}_{p,n,0}^\opt$, $\overline{\mathbb{S}}_{p,n,0}$ and $\mathbb{S}_{p,n,0}$ in terms of $n$, for even degrees $p$. The reference convergence order in $n$ is indicated by black triangles.} \label{ex:convergence1D-sin:b}
\end{figure}

\begin{example}\label{ex:convergence1D-sin}
To test the approximation properties of the reduced spline spaces $\mathbb{S}_{p,n,0}^\opt$ and $\overline{\mathbb{S}}_{p,n,0}$ we consider problem \eqref{eq:second-order-prob} with the manufactured solution 
$$u(x)=\sin(2\pi x).$$
The exact solution satisfies all the additional conditions on high-order derivatives defining both the spaces $\mathbb{S}_{p,n,0}^\opt$ and $\overline{\mathbb{S}}_{p,n,0}$. In Figures~\ref{ex:convergence1D-sin:a} and~\ref{ex:convergence1D-sin:b} we depict the $L^2$ and $H^1$ error of the approximate solutions in the spline spaces $\mathbb{S}_{p,n,0}^\opt$, $\overline{\mathbb{S}}_{p,n,0}$ and $\mathbb{S}_{p,n,0}$ in terms of $n$, for various values of $p$. For fixed $p$, the three spaces achieve exactly the same error convergence orders as expected: $p+1$ in the $L^2$-norm and $p$ in the $H^1$-norm. Note that in this example, per degree of freedom, the error obtained in $\mathbb{S}_{p,n,0}^\opt$ is noticeably better than in $\mathbb{S}_{p,n,0}$ and also slightly better than in $\overline{\mathbb{S}}_{p,n,0}$ ($p$ even), especially for smaller values of $n$. This can be (partially) attributed to the finer grid spacing.
\end{example}

\begin{figure}[t!]
\centering
\subfigure[$\mathbb{S}_{p,n,0}^\opt$]{\includegraphics[height=4.1cm]{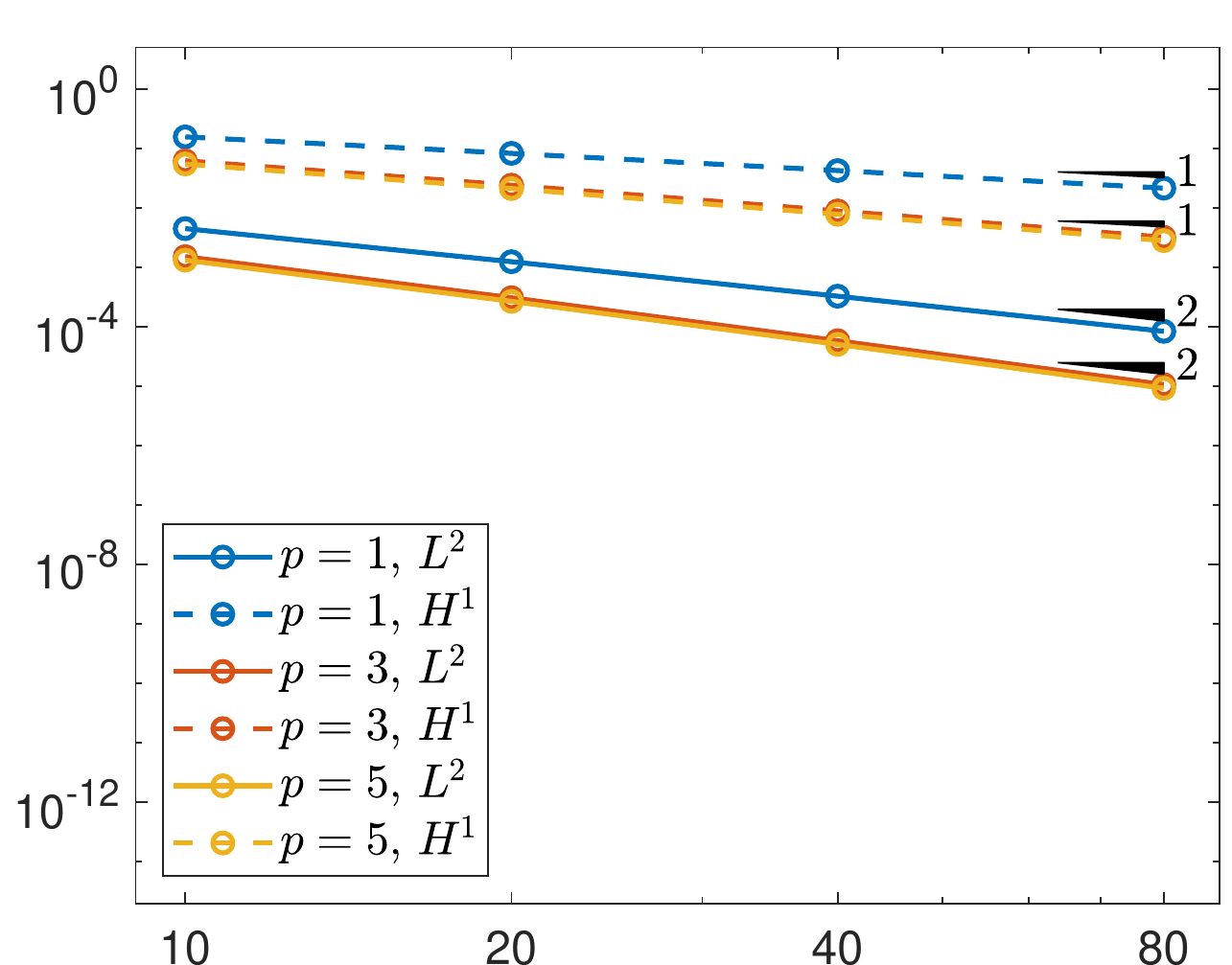}}\hspace*{0.1cm}
\subfigure[$\mathbb{S}_{p,n,0}^\opt$ with correction]{\includegraphics[height=4.1cm]{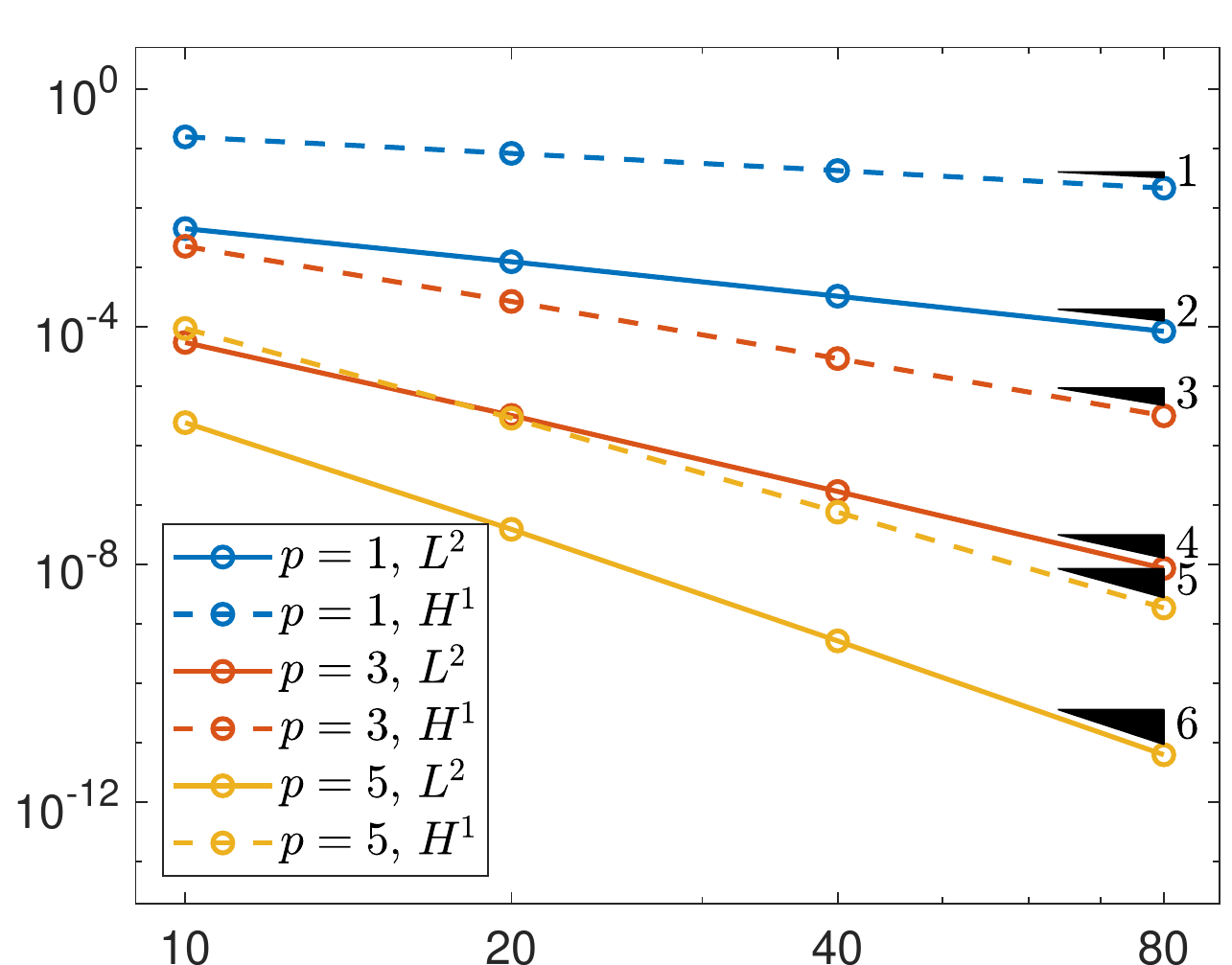}}\hspace*{0.1cm}
\subfigure[$\mathbb{S}_{p,n,0}$]{\includegraphics[height=4.1cm]{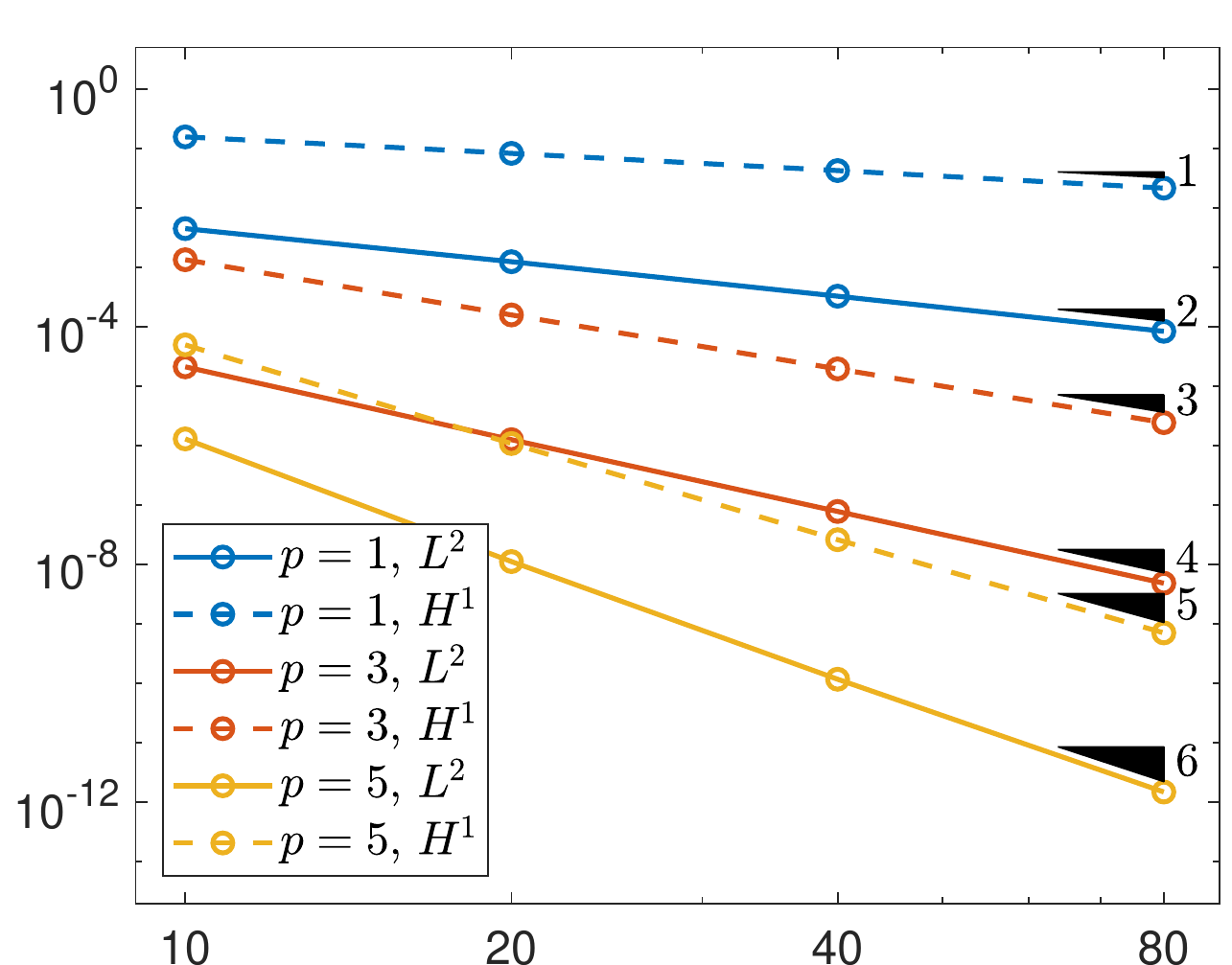}}
\caption{Example~\ref{ex:convergence1D-boundary}: $L^2$ and $H^1$ error convergence in the spline spaces $\mathbb{S}_{p,n,0}^\opt$ and $\mathbb{S}_{p,n,0}$ in terms of $n$, for odd degrees $p$. Both without and with boundary data correction are considered for the space $\mathbb{S}_{p,n,0}^\opt$. The reference convergence order in $n$ is indicated by black triangles.} \label{ex:convergence1D-boundary:a}
\bigskip
\centering
\subfigure[$\mathbb{S}_{p,n,0}^\opt$]{\includegraphics[height=4.1cm]{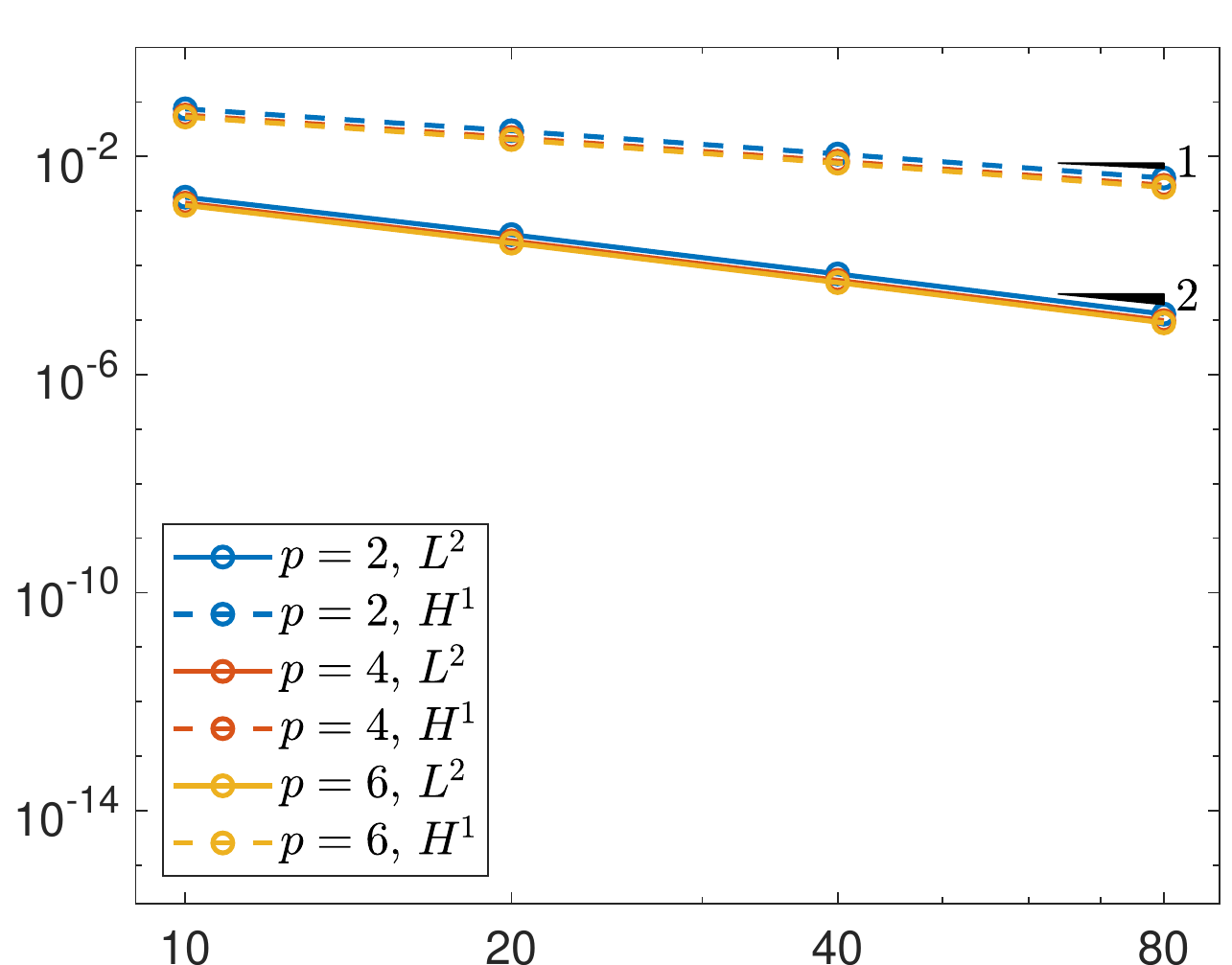}}\hspace*{0.1cm}
\subfigure[$\overline{\mathbb{S}}_{p,n,0}$]{\includegraphics[height=4.1cm]{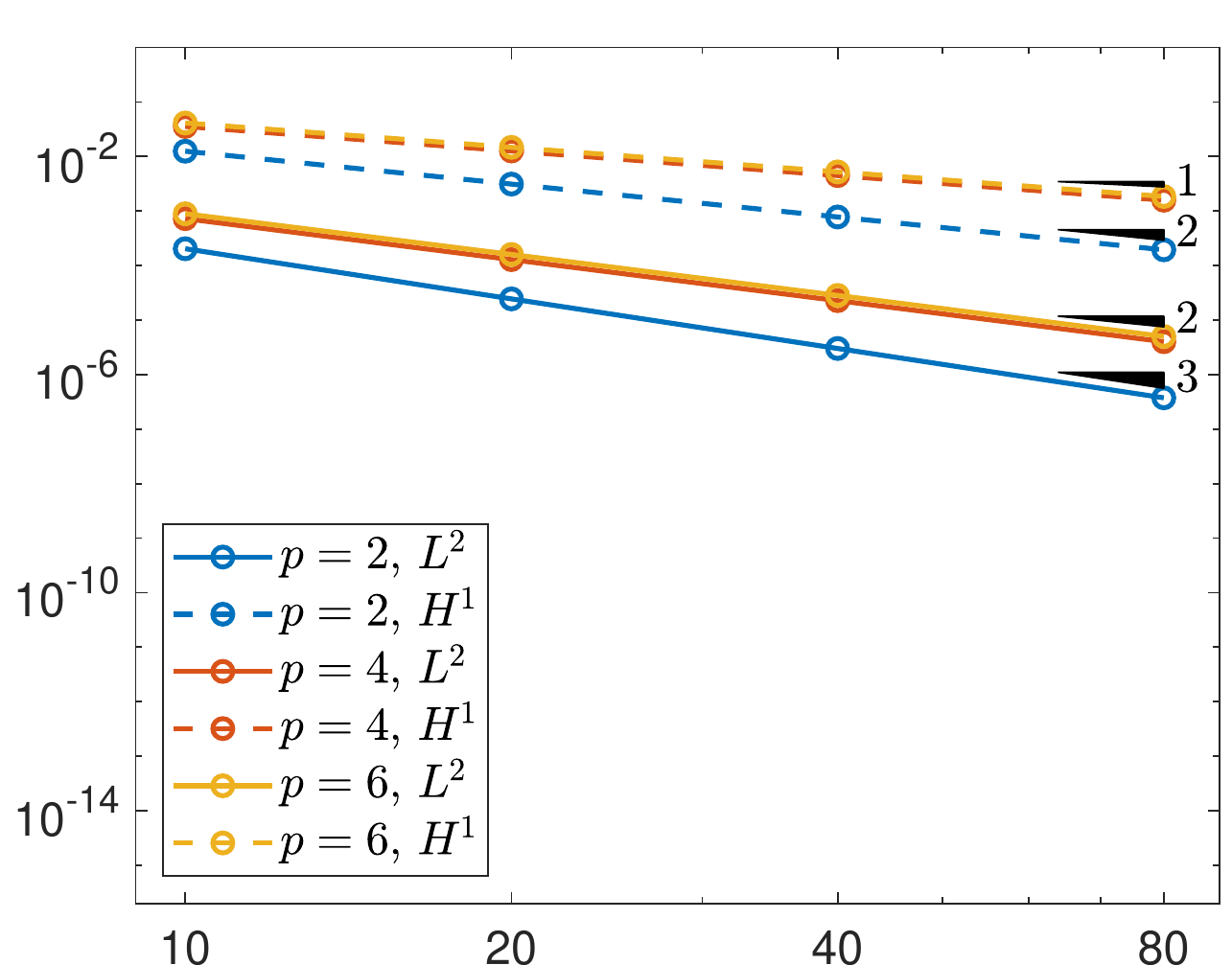}} \\
\subfigure[$\mathbb{S}_{p,n,0}^\opt$ with correction]{\includegraphics[height=4.1cm]{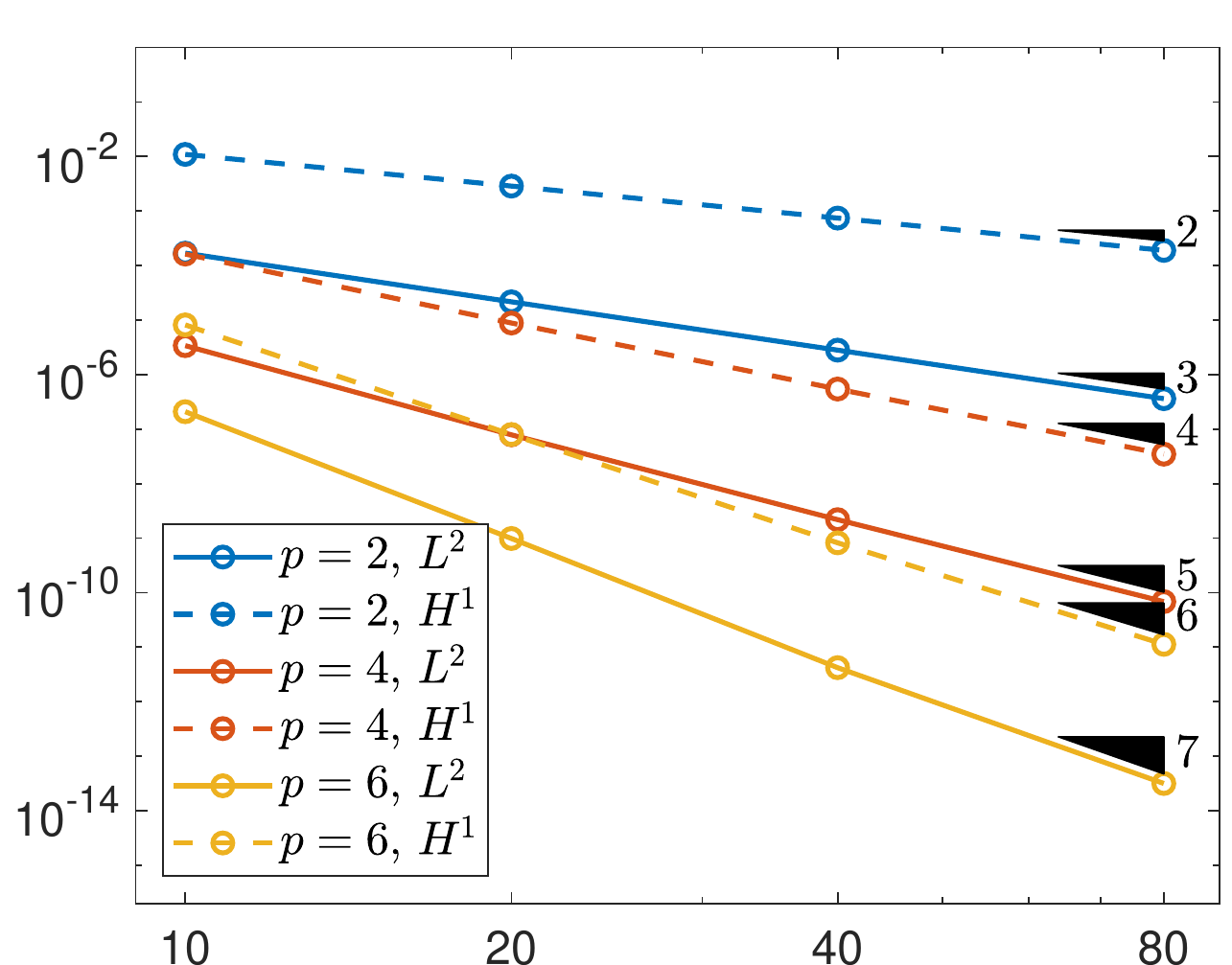}}\hspace*{0.1cm}
\subfigure[$\overline{\mathbb{S}}_{p,n,0}$ with correction]{\includegraphics[height=4.1cm]{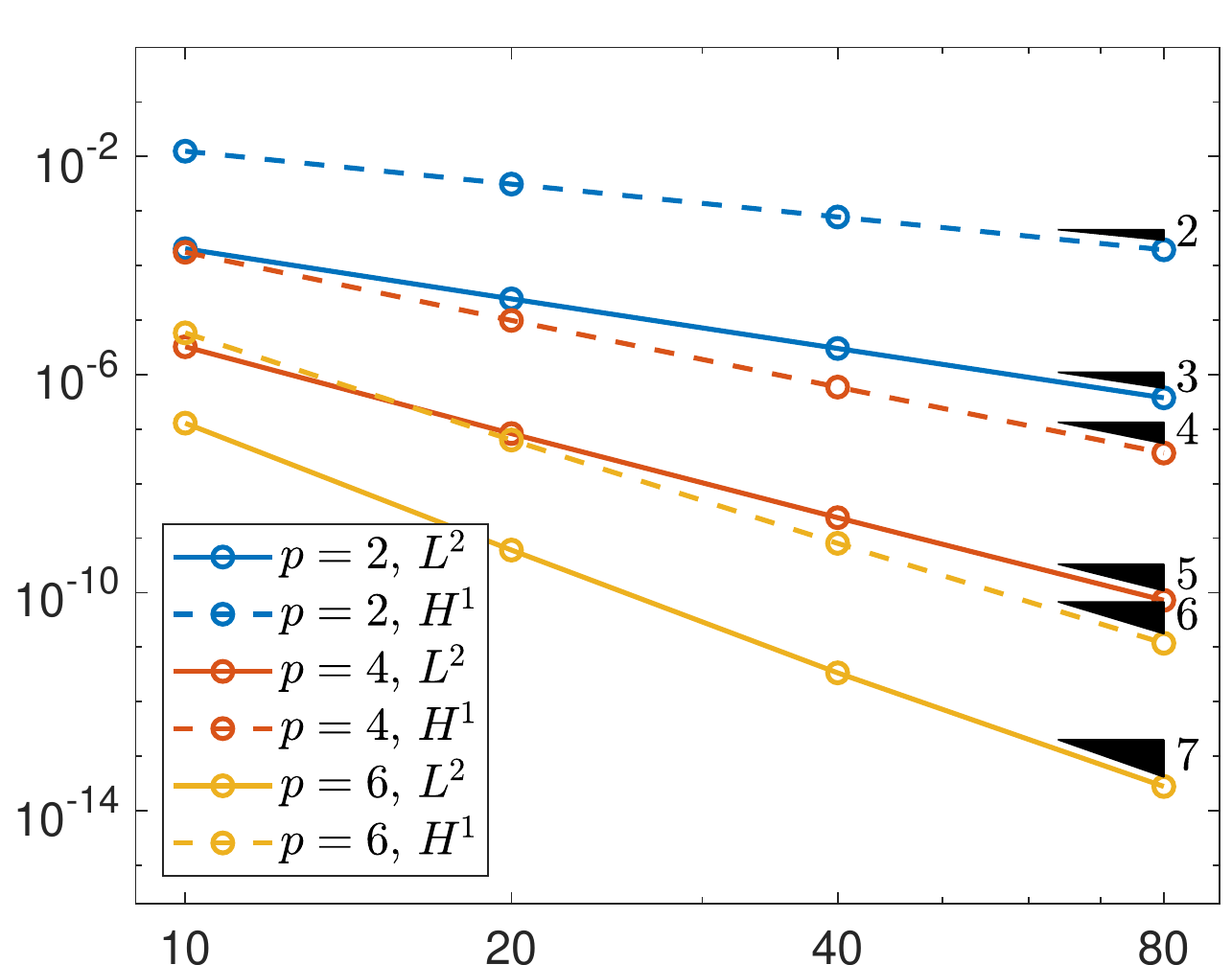}}\hspace*{0.1cm}
\subfigure[$\mathbb{S}_{p,n,0}$]{\includegraphics[height=4.1cm]{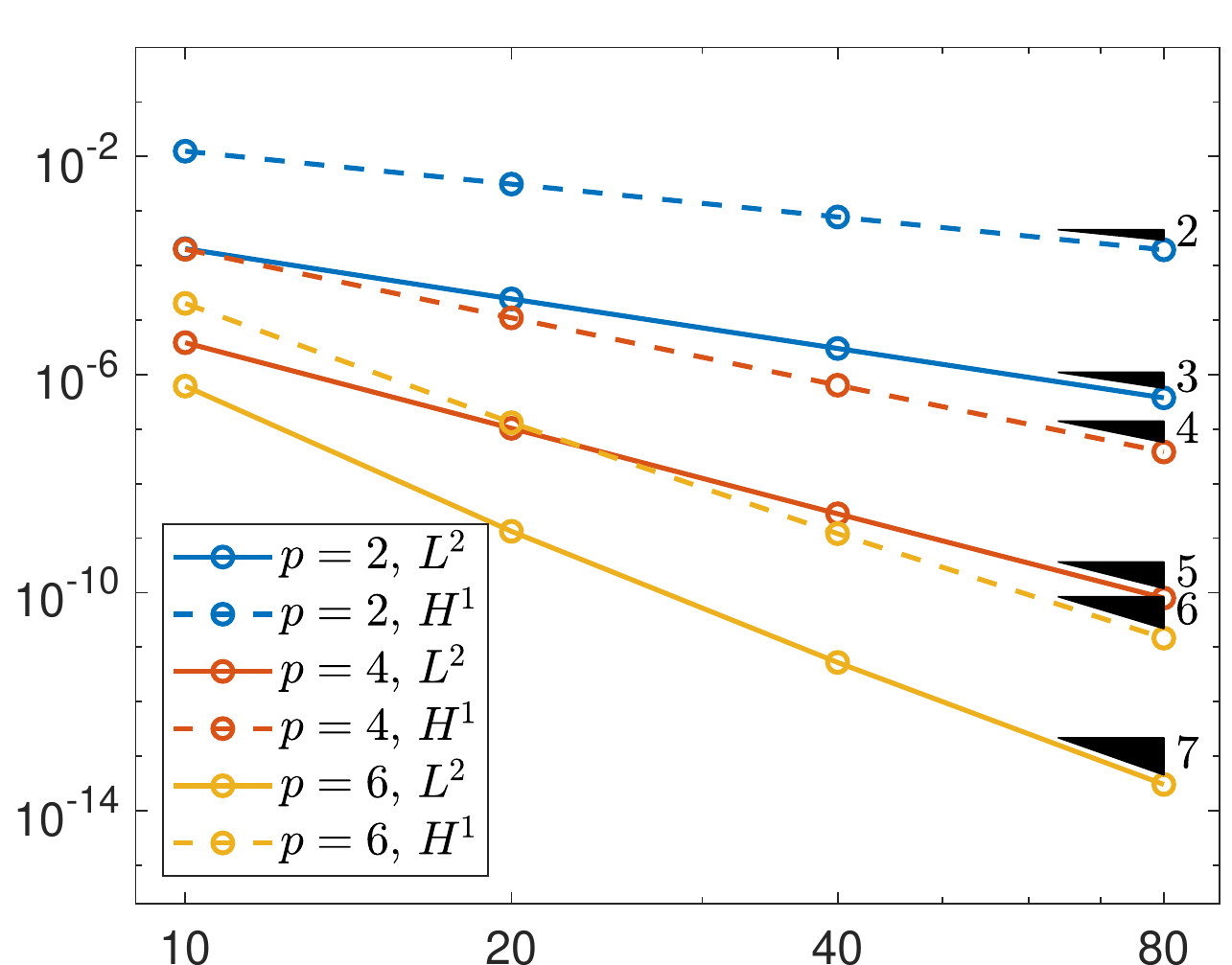}}
\caption{Example~\ref{ex:convergence1D-boundary}: $L^2$ and $H^1$ error convergence in the spline spaces $\mathbb{S}_{p,n,0}^\opt$, $\overline{\mathbb{S}}_{p,n,0}$ and $\mathbb{S}_{p,n,0}$ in terms of $n$, for even degrees $p$. Both without and with boundary data correction are considered for the spaces $\mathbb{S}_{p,n,0}^\opt$ and $\overline{\mathbb{S}}_{p,n,0}$. The reference convergence order in $n$ is indicated by black triangles.} \label{ex:convergence1D-boundary:b}
\end{figure}

\begin{example}\label{ex:convergence1D-boundary}
As a test for the strategy presented in Section~\ref{sec:general-BC-1D}, we consider problem \eqref{eq:second-order-prob} with the manufactured solution 
$$u(x)=1-\frac{15}{16}x-\frac{1}{(x+1)^4}.$$
The exact solution does not satisfy the additional conditions on high-order derivatives defining the spaces $\mathbb{S}_{p,n,0}^\opt$ and $\overline{\mathbb{S}}_{p,n,0}$. In Figures~\ref{ex:convergence1D-boundary:a} and~\ref{ex:convergence1D-boundary:b} we depict the convergence of the approximate solutions in the spline spaces $\mathbb{S}_{p,n,0}^\opt$, $\overline{\mathbb{S}}_{p,n,0}$ and $\mathbb{S}_{p,n,0}$ in terms of $n$, for various values of $p$. There is a substantial loss of accuracy for $p>2$ when approximating the solution in the reduced spline spaces. However, the full convergence order ($p+1$ in the $L^2$-norm and $p$ in the $H^1$-norm) is recovered by applying the boundary data correction described in Section~\ref{sec:general-BC-1D}. Note that in this example, per degree of freedom, the error obtained in $\mathbb{S}_{p,n,0}$ is slightly worse for $p$ even, but better for $p$ odd. 
\end{example}

\subsection{Multivariate problems}
We now test the numerical performance of the presented strategies in the bivariate setting. We consider both the eigenvalue problem \eqref{eq:prob-eigenv} and second-order problems of the form \eqref{eq:Laplace} with $d=2$.
We approximately solve them by means of Galerkin discretizations in the reduced tensor-product spline spaces $\mathbb{S}_{p,n,0}^\opt\otimes\mathbb{S}_{p,n,0}^\opt$ and $\overline{\mathbb{S}}_{p,n,0}\otimes \overline{\mathbb{S}}_{p,n,0}$. Just like in Section~\ref{sec:numerics-1D}, we also compare them with the full tensor-product spline space $\mathbb{S}_{p,n,0}\otimes\mathbb{S}_{p,n,0}$. All these spaces have the same dimension~$n^2$.

\begin{figure}[t!]
\centering
\subfigure[$e_{\omega,\indeigkone,\indeigktwo}$ in $\mathbb{S}_{p,50,0}^\opt\otimes\mathbb{S}_{p,50,0}^\opt$]{\includegraphics[height=4.1cm]{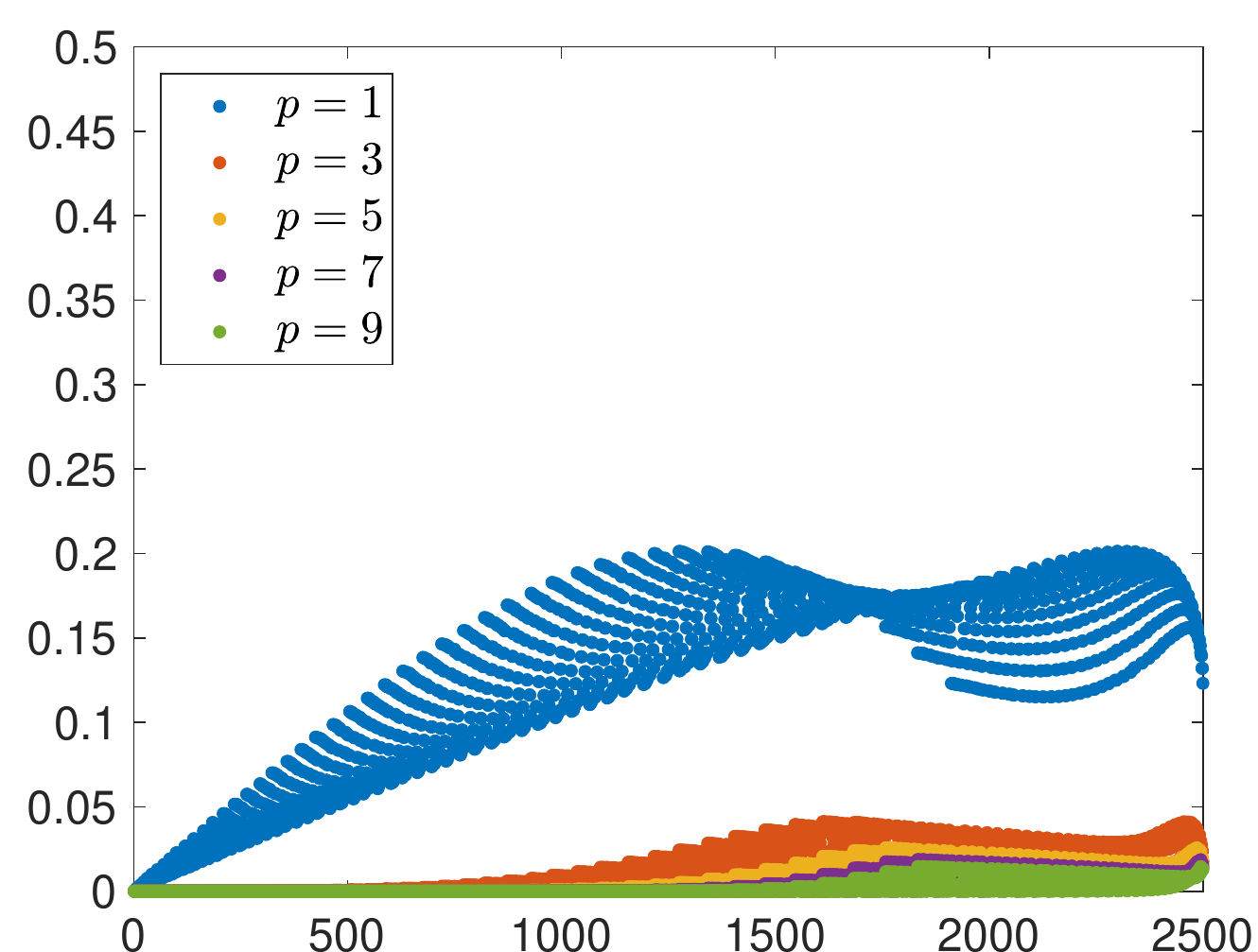}}\hspace*{0.1cm}
\subfigure[$e_{\omega,\indeigkone,\indeigktwo}$ in $\mathbb{S}_{p,50,0}\otimes\mathbb{S}_{p,50,0}$]{\includegraphics[height=4.1cm]{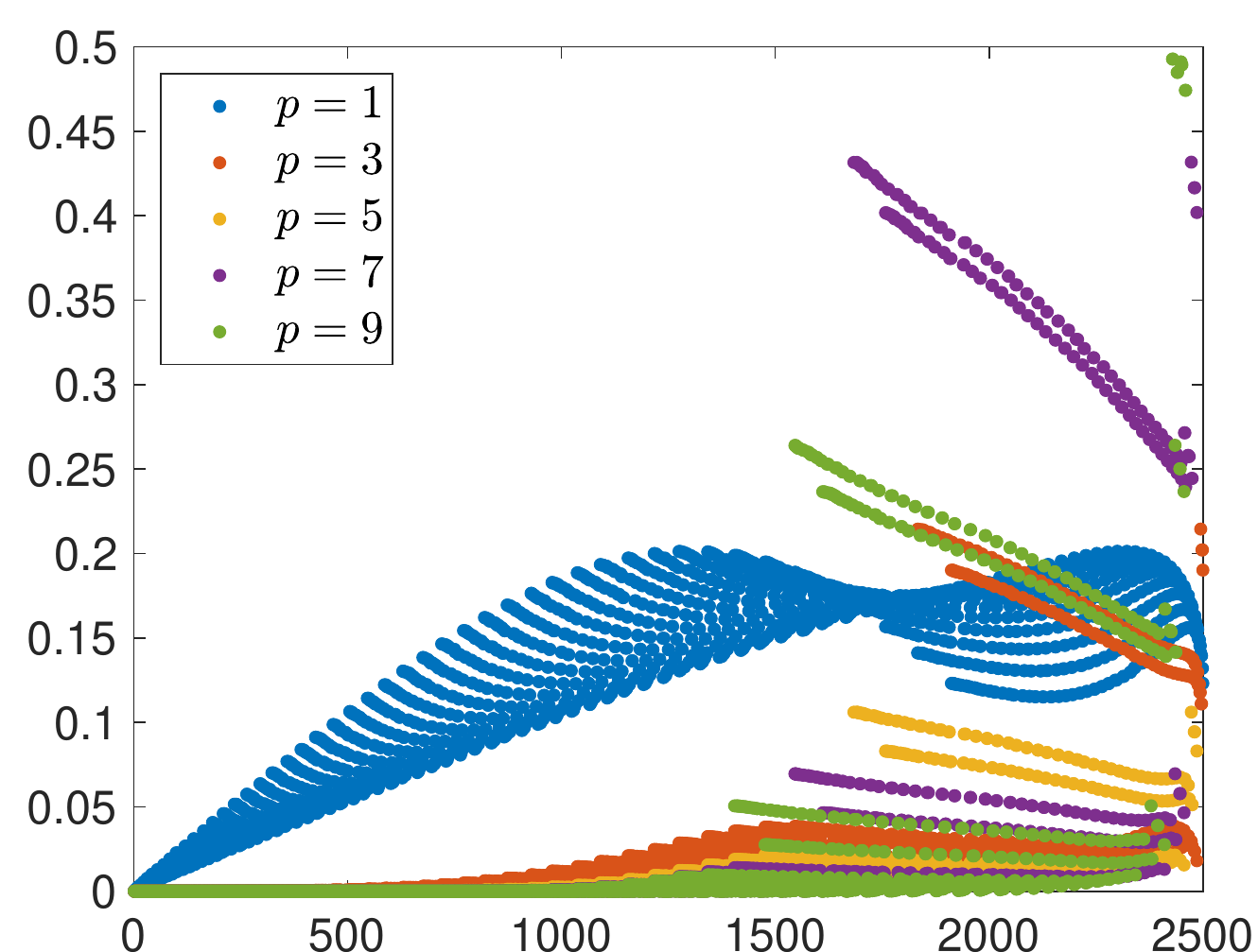}}\hspace*{0.1cm}
\subfigure[zoom out for $\mathbb{S}_{p,50,0}\otimes\mathbb{S}_{p,50,0}$]{\includegraphics[height=4.1cm]{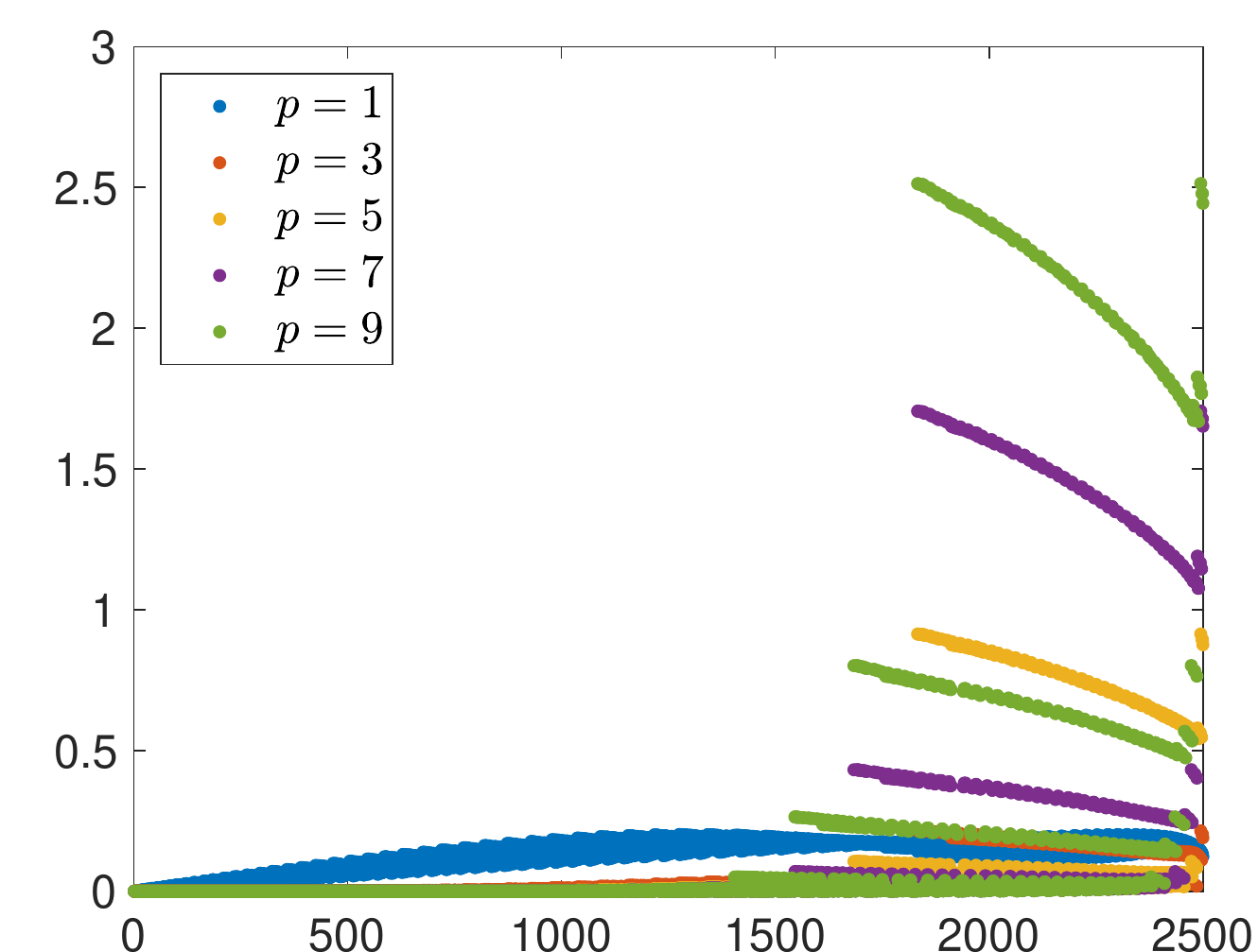}}\\
\subfigure[$e_{u,\indeigkone,\indeigktwo}$ in $\mathbb{S}_{p,50,0}^\opt\otimes\mathbb{S}_{p,50,0}^\opt$]{\includegraphics[height=4.1cm]{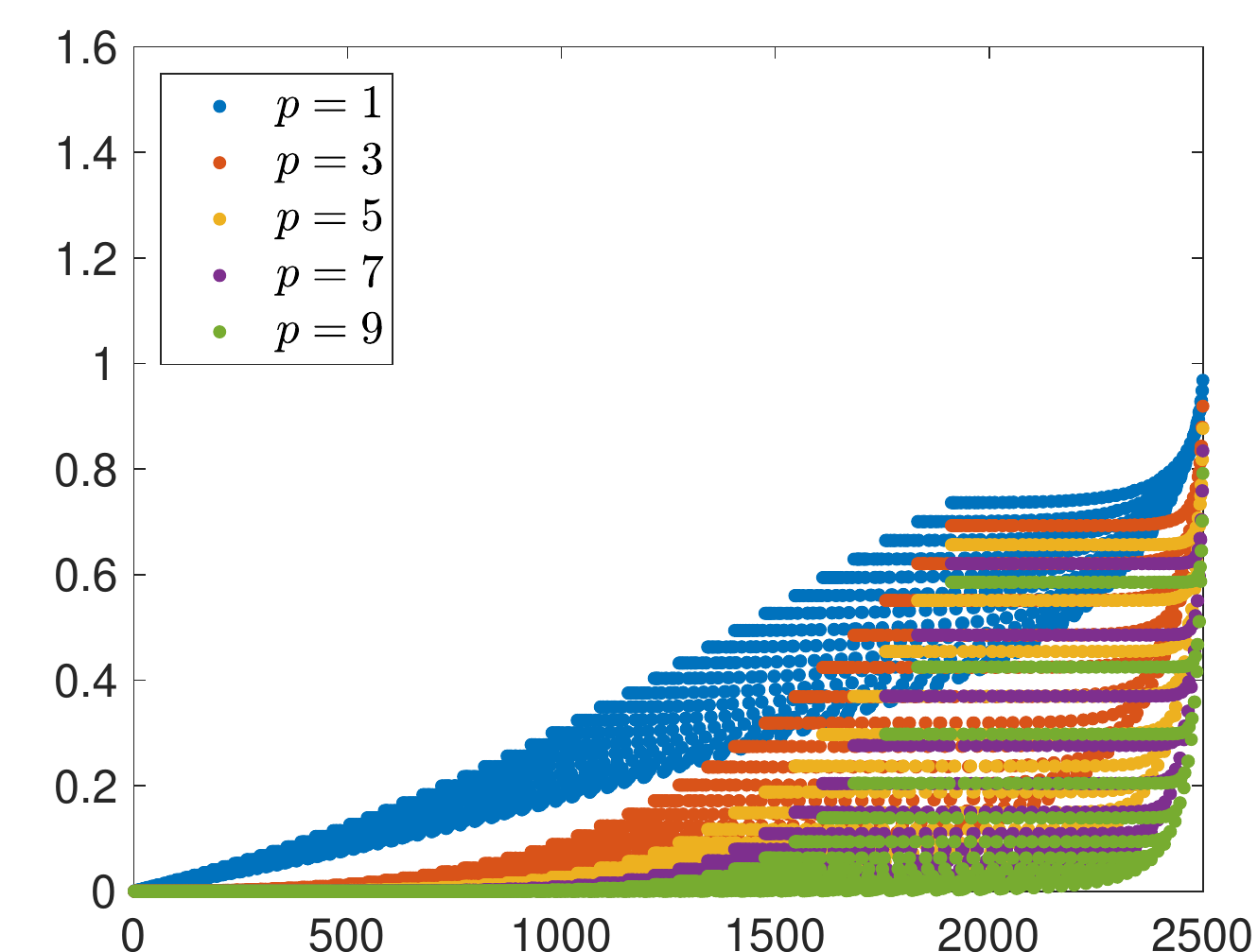}}\hspace*{0.1cm}
\subfigure[$e_{u,\indeigkone,\indeigktwo}$ in $\mathbb{S}_{p,50,0}\otimes\mathbb{S}_{p,50,0}$]{\includegraphics[height=4.1cm]{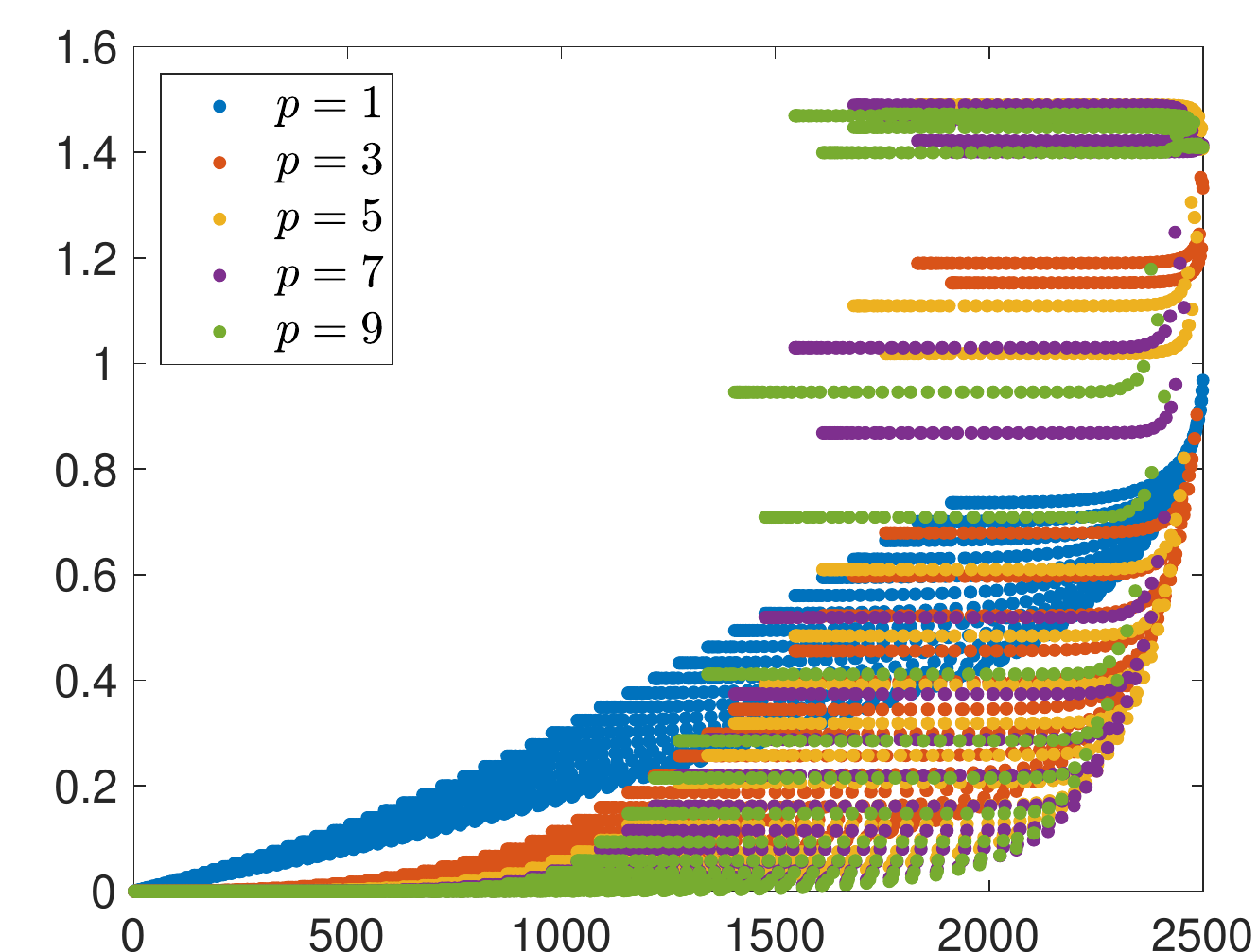}}
\caption{Example~\ref{ex:eigenvalues2D}: Relative frequency errors $e_{\omega,\indeigkone,\indeigktwo}$ in \eqref{eq:error-eigenvalues2D} and $L^2$ relative eigenfunction errors $e_{u,\indeigkone,\indeigktwo}$ in \eqref{eq:error-eigenfunctions2D} corresponding to the spline spaces $\mathbb{S}_{p,n,0}^\opt\otimes\mathbb{S}_{p,n,0}^\opt$ and $\mathbb{S}_{p,n,0}\otimes\mathbb{S}_{p,n,0}$ for odd degrees $p$ and $n=50$. The errors are ordered according to increasing exact frequencies. No outliers are observed for $\mathbb{S}_{p,n,0}^\opt\otimes\mathbb{S}_{p,n,0}^\opt$. Several outliers of $\mathbb{S}_{p,n,0}\otimes\mathbb{S}_{p,n,0}$ are outside the visible range in (b) as illustrated in (c).} \label{fig:eigenvalues2D.odd}
\end{figure}
\begin{figure}[t!]
\centering
\subfigure[$e_{\omega,\indeigkone,\indeigktwo}$ in $\mathbb{S}_{p,50,0}^\opt\otimes\mathbb{S}_{p,50,0}^\opt$]{\includegraphics[height=4.1cm]{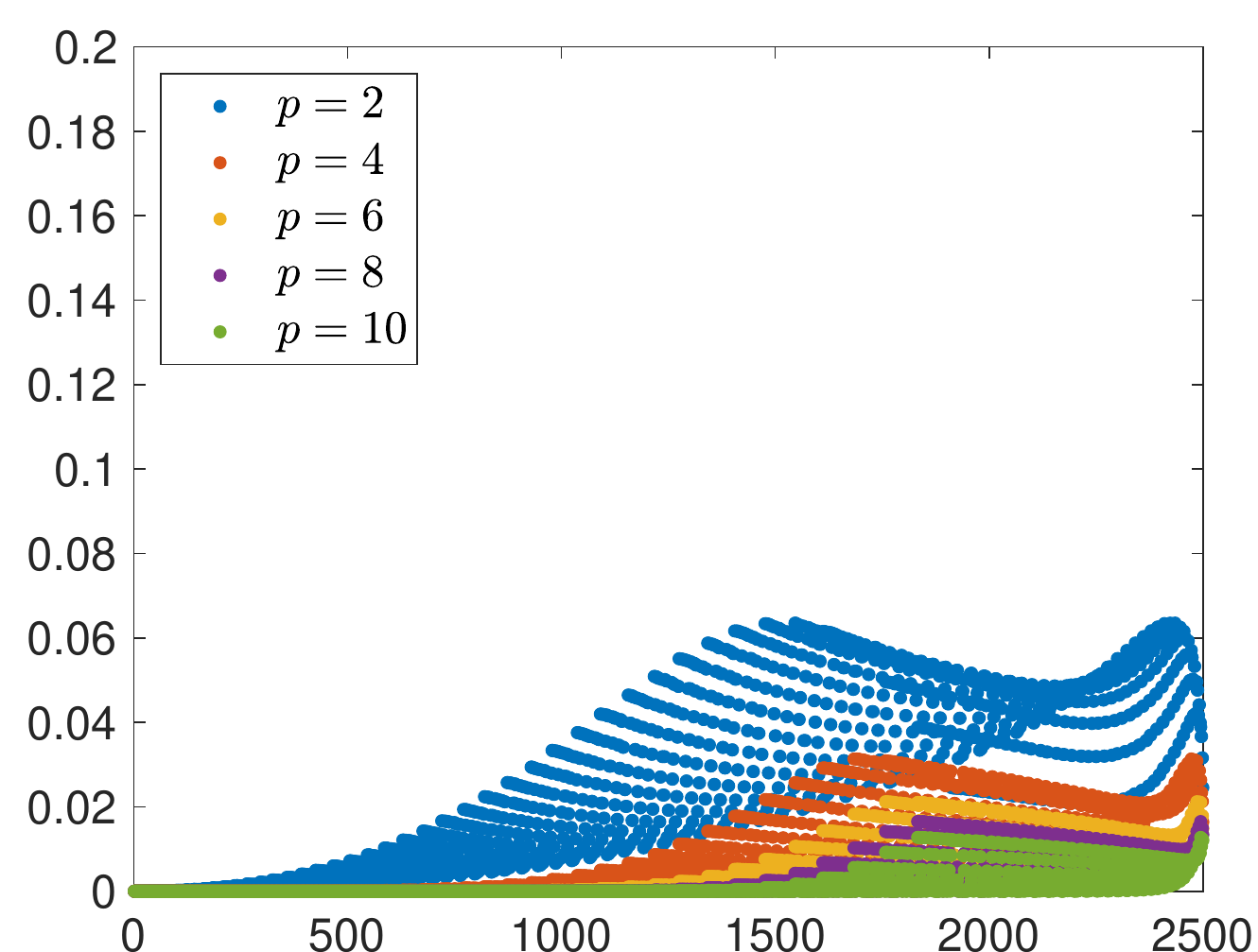}}\hspace*{0.1cm}
\subfigure[$e_{\omega,\indeigkone,\indeigktwo}$ in $\overline{\mathbb{S}}_{p,50,0}\otimes\overline{\mathbb{S}}_{p,50,0}$]{\includegraphics[height=4.1cm]{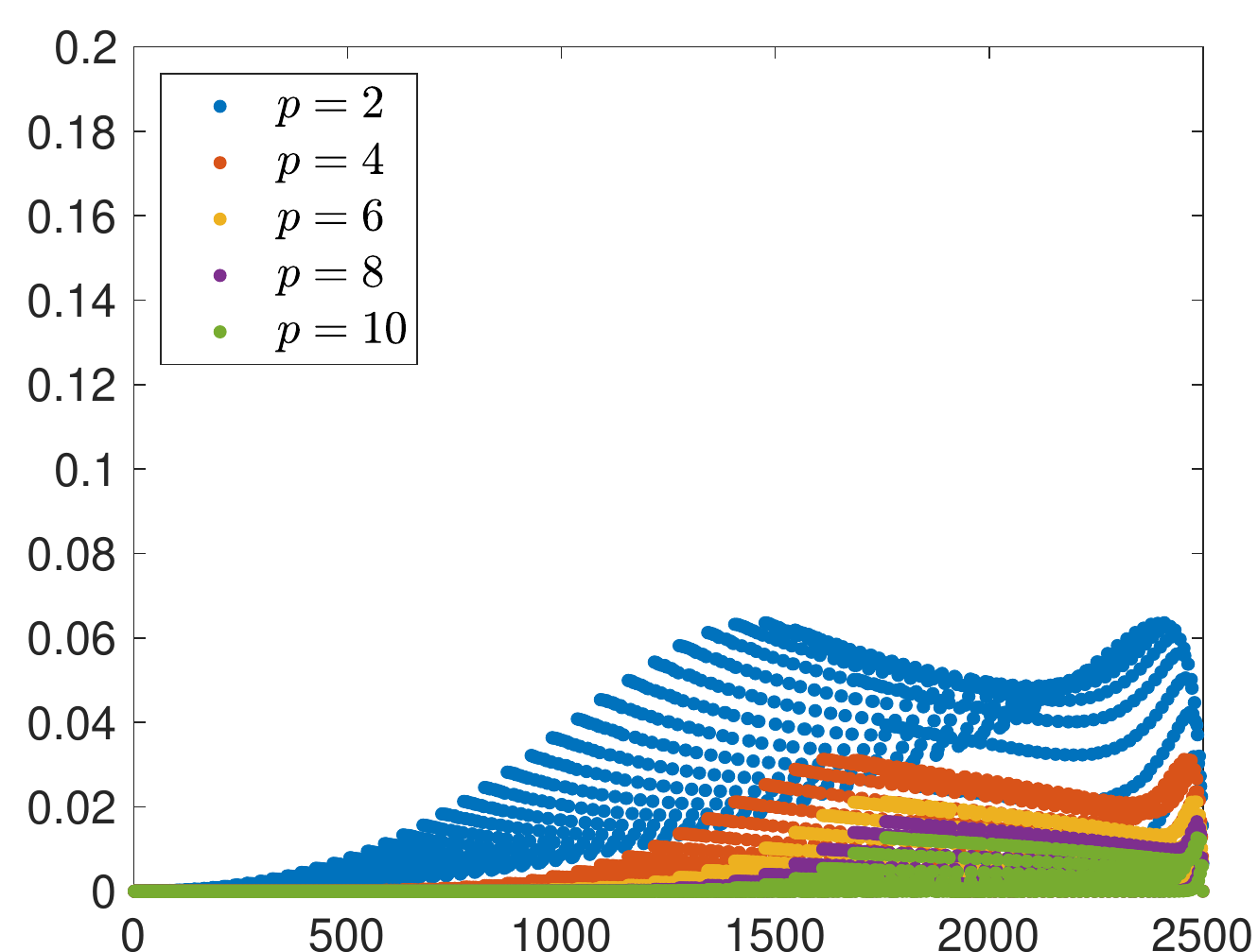}}\hspace*{0.1cm}
\subfigure[$e_{\omega,\indeigkone,\indeigktwo}$ in $\mathbb{S}_{p,50,0}\otimes\mathbb{S}_{p,50,0}$]{\includegraphics[height=4.1cm]{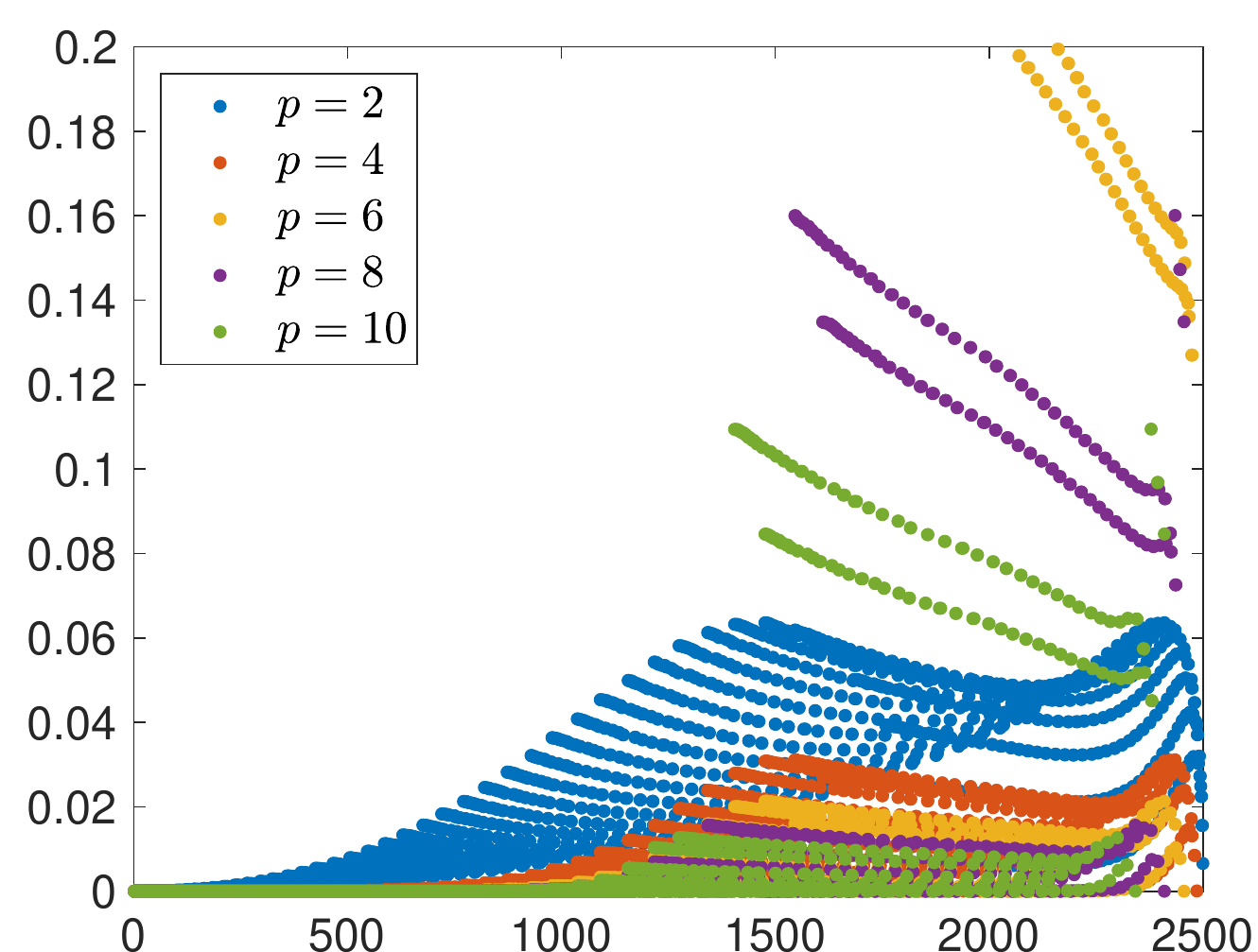}}\\
\subfigure[$e_{u,\indeigkone,\indeigktwo}$ in $\mathbb{S}_{p,50,0}^\opt\otimes\mathbb{S}_{p,50,0}^\opt$]{\includegraphics[height=4.1cm]{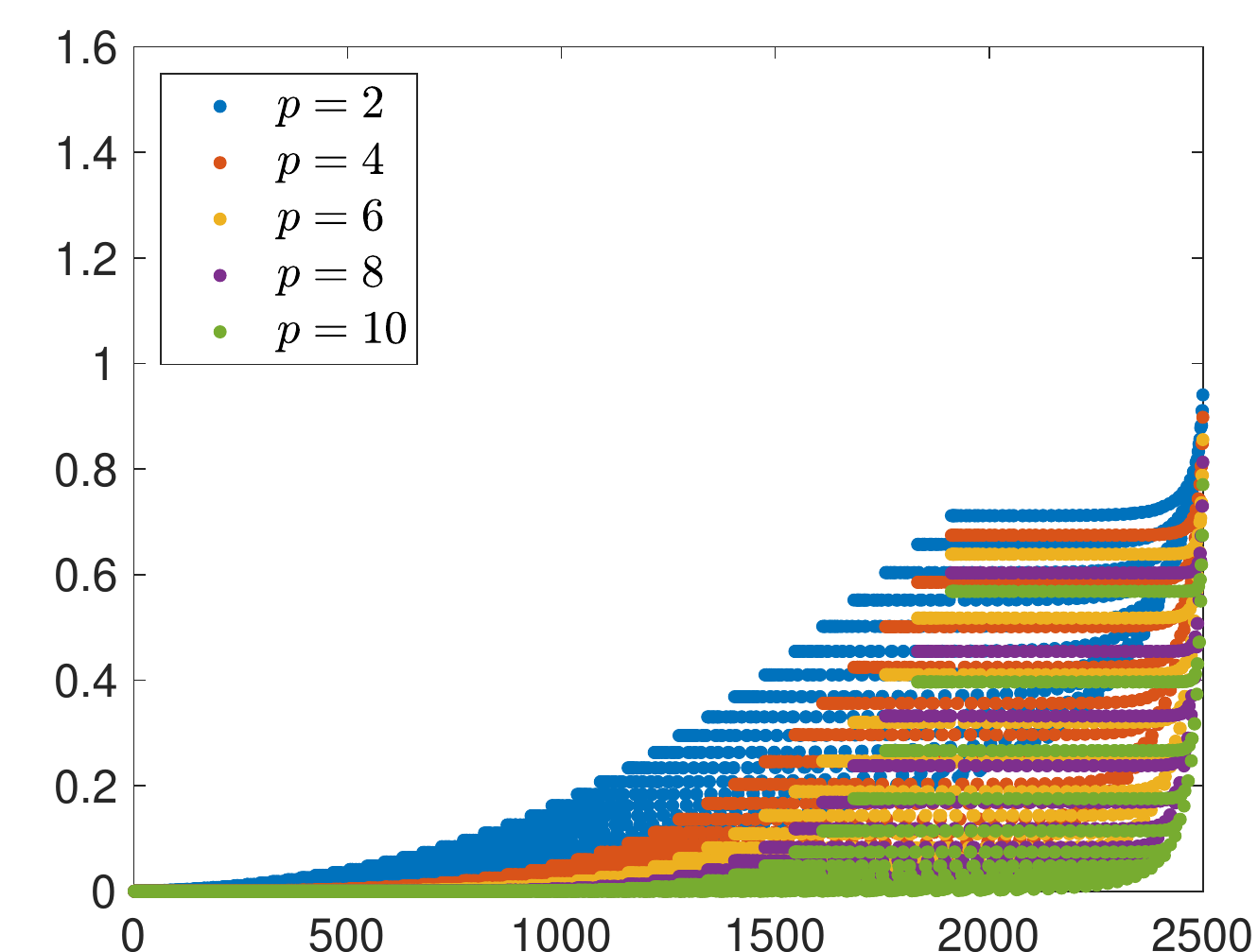}}\hspace*{0.1cm}
\subfigure[$e_{u,\indeigkone,\indeigktwo}$ in $\overline{\mathbb{S}}_{p,50,0}\otimes\overline{\mathbb{S}}_{p,50,0}$]{\includegraphics[height=4.1cm]{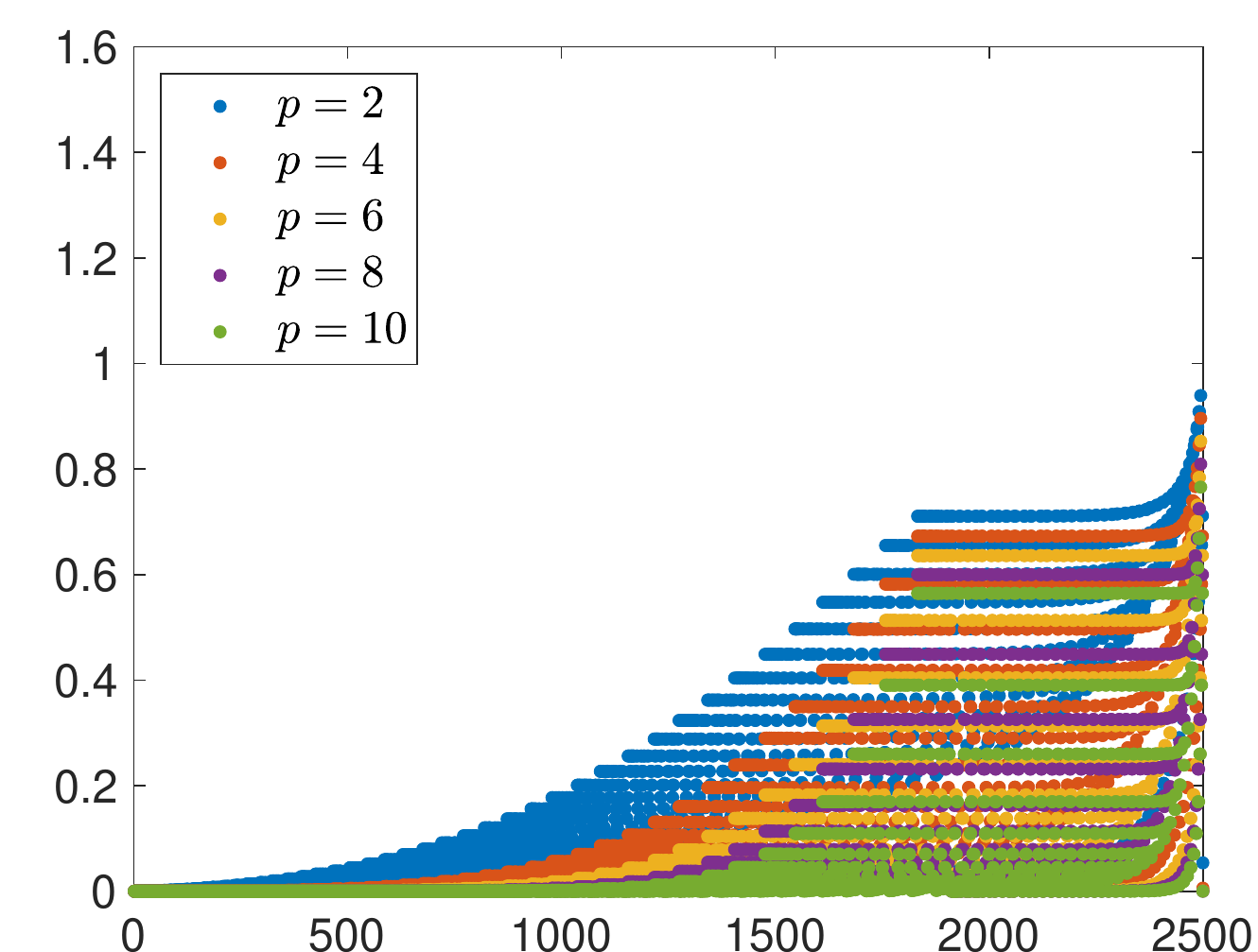}}\hspace*{0.1cm}
\subfigure[$e_{u,\indeigkone,\indeigktwo}$ in $\mathbb{S}_{p,50,0}\otimes\mathbb{S}_{p,50,0}$]{\includegraphics[height=4.1cm]{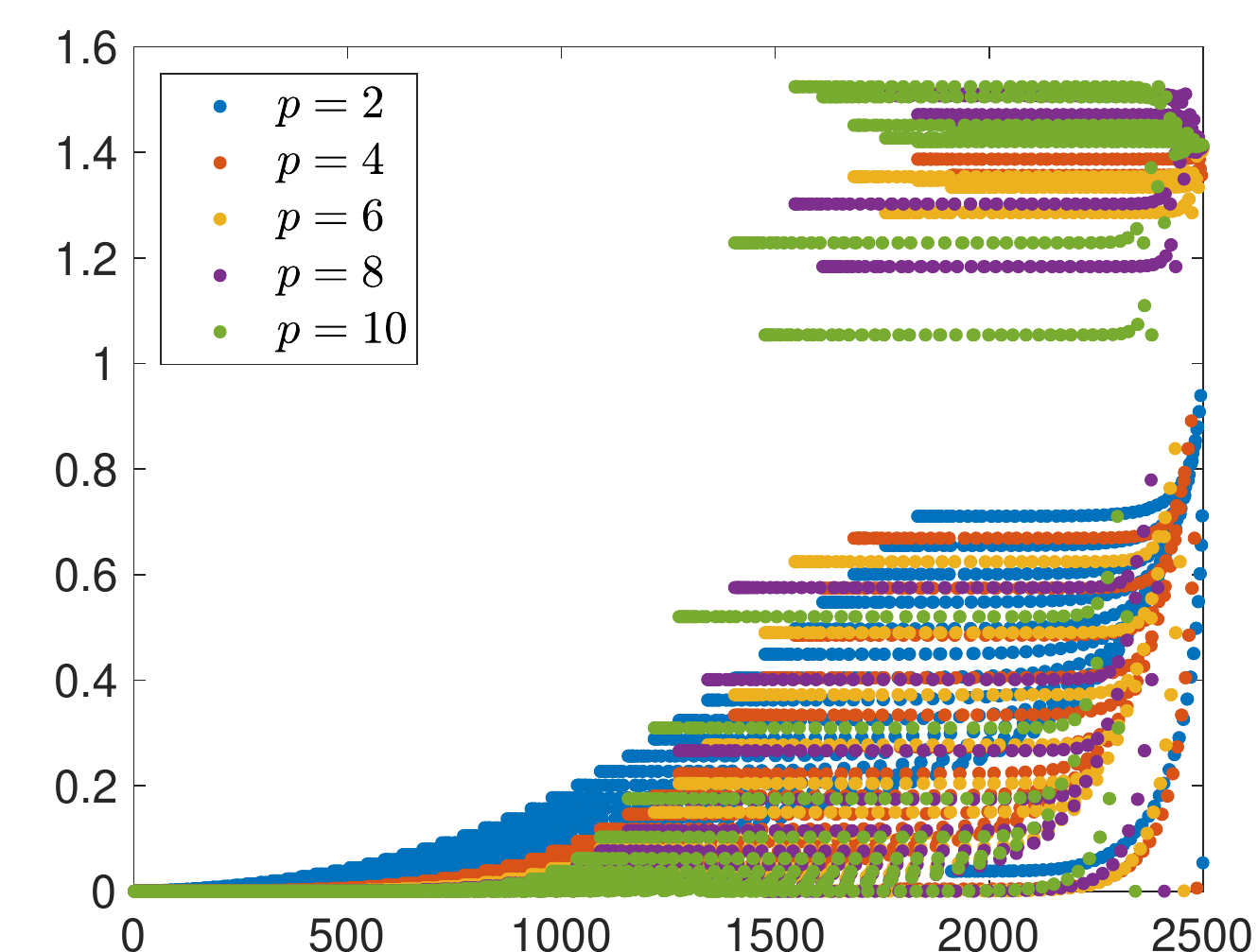}}
\caption{Example~\ref{ex:eigenvalues2D}: Relative frequency errors $e_{\omega,\indeigkone,\indeigktwo}$ in \eqref{eq:error-eigenvalues2D} and $L^2$ relative eigenfunction errors $e_{u,\indeigkone,\indeigktwo}$ in \eqref{eq:error-eigenfunctions2D} corresponding to the spline spaces $\mathbb{S}_{p,n,0}^\opt\otimes\mathbb{S}_{p,n,0}^\opt$, $\overline{\mathbb{S}}_{p,n,0}\otimes\overline{\mathbb{S}}_{p,n,0}$ and $\mathbb{S}_{p,n,0}\otimes\mathbb{S}_{p,n,0}$ for even degrees $p$ and $n=50$. The errors are ordered according to increasing exact frequencies. No outliers are observed for $\mathbb{S}_{p,n,0}^\opt\otimes\mathbb{S}_{p,n,0}^\opt$ and $\overline{\mathbb{S}}_{p,n,0}\otimes\overline{\mathbb{S}}_{p,n,0}$. Several outliers of $\mathbb{S}_{p,n,0}\otimes\mathbb{S}_{p,n,0}$ are outside the visible range in (c); they are not shown for clarity of the figure.} \label{fig:eigenvalues2D.even}
\end{figure}
\begin{figure}[t!]
\centering
\subfigure[$e_{\omega,\indeigkone,\indeigktwo}$ in $\mathbb{S}_{1,50,0}^\opt\otimes\mathbb{S}_{1,50,0}^\opt$]{\includegraphics[height=4.1cm]{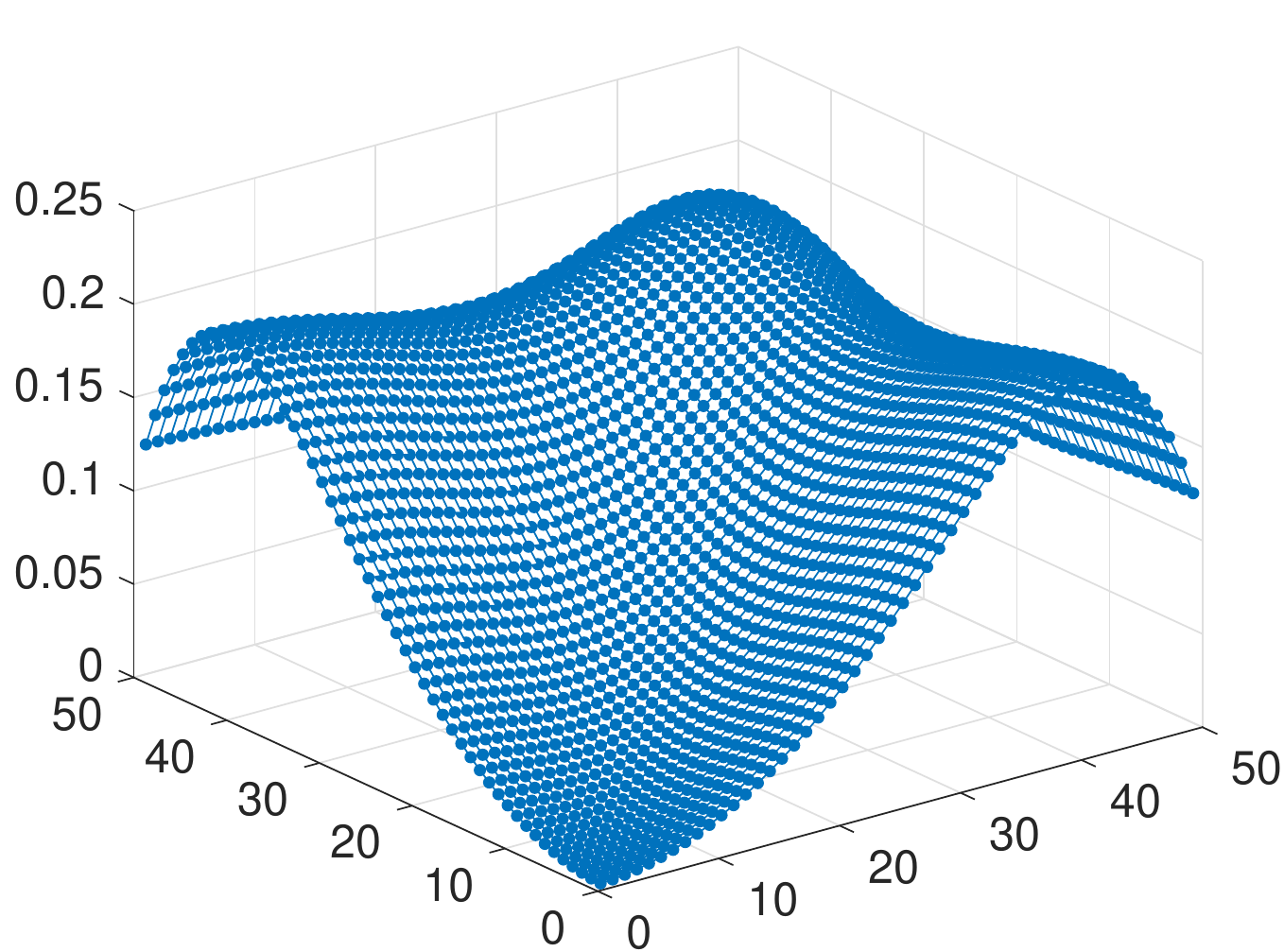}}\hspace*{0.1cm}
\subfigure[$e_{\omega,\indeigkone,\indeigktwo}$ in $\mathbb{S}_{3,50,0}^\opt\otimes\mathbb{S}_{3,50,0}^\opt$]{\includegraphics[height=4.1cm]{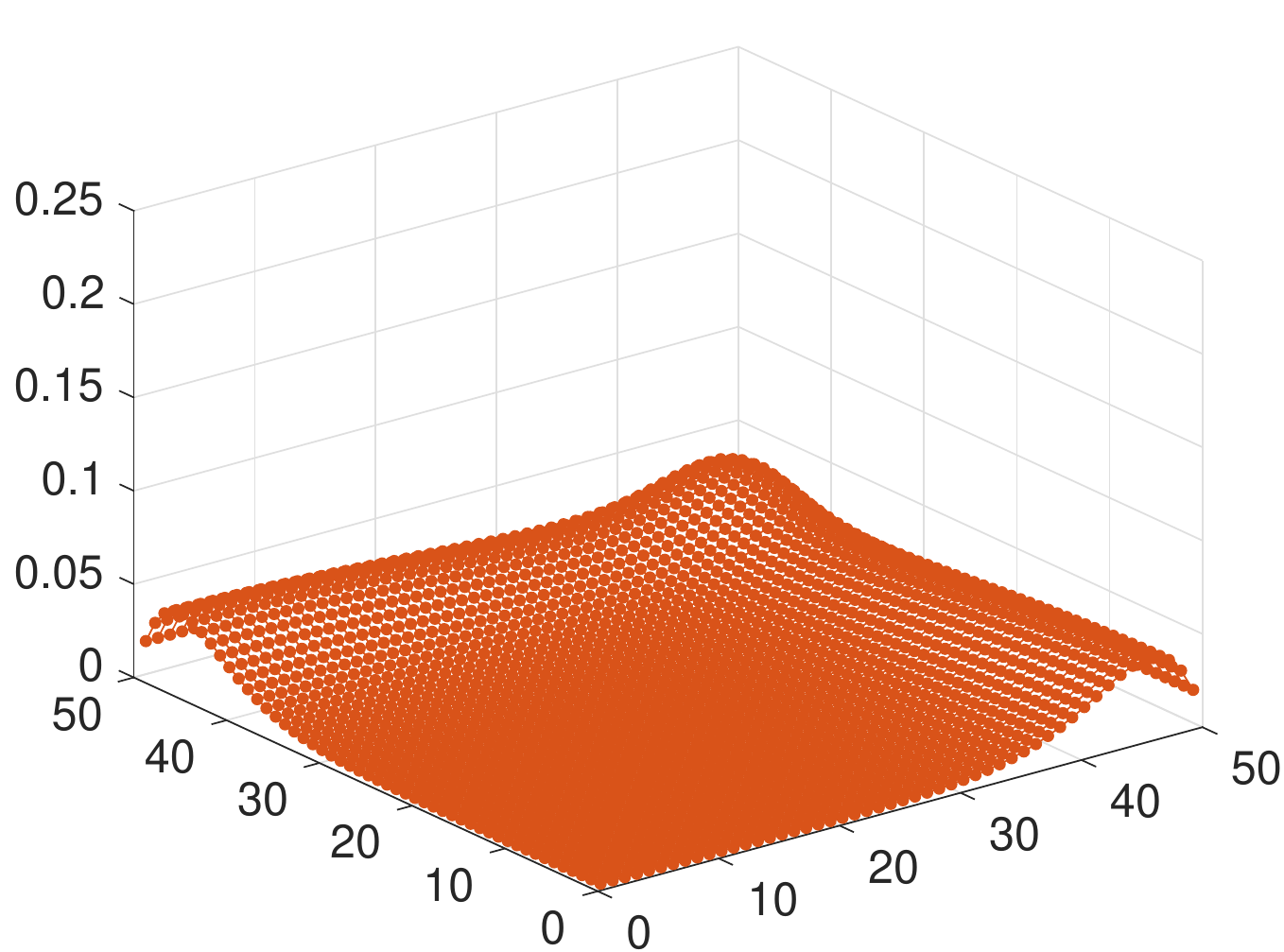}}\hspace*{0.1cm}
\subfigure[$e_{\omega,\indeigkone,\indeigktwo}$ in $\mathbb{S}_{5,50,0}^\opt\otimes\mathbb{S}_{5,50,0}^\opt$]{\includegraphics[height=4.1cm]{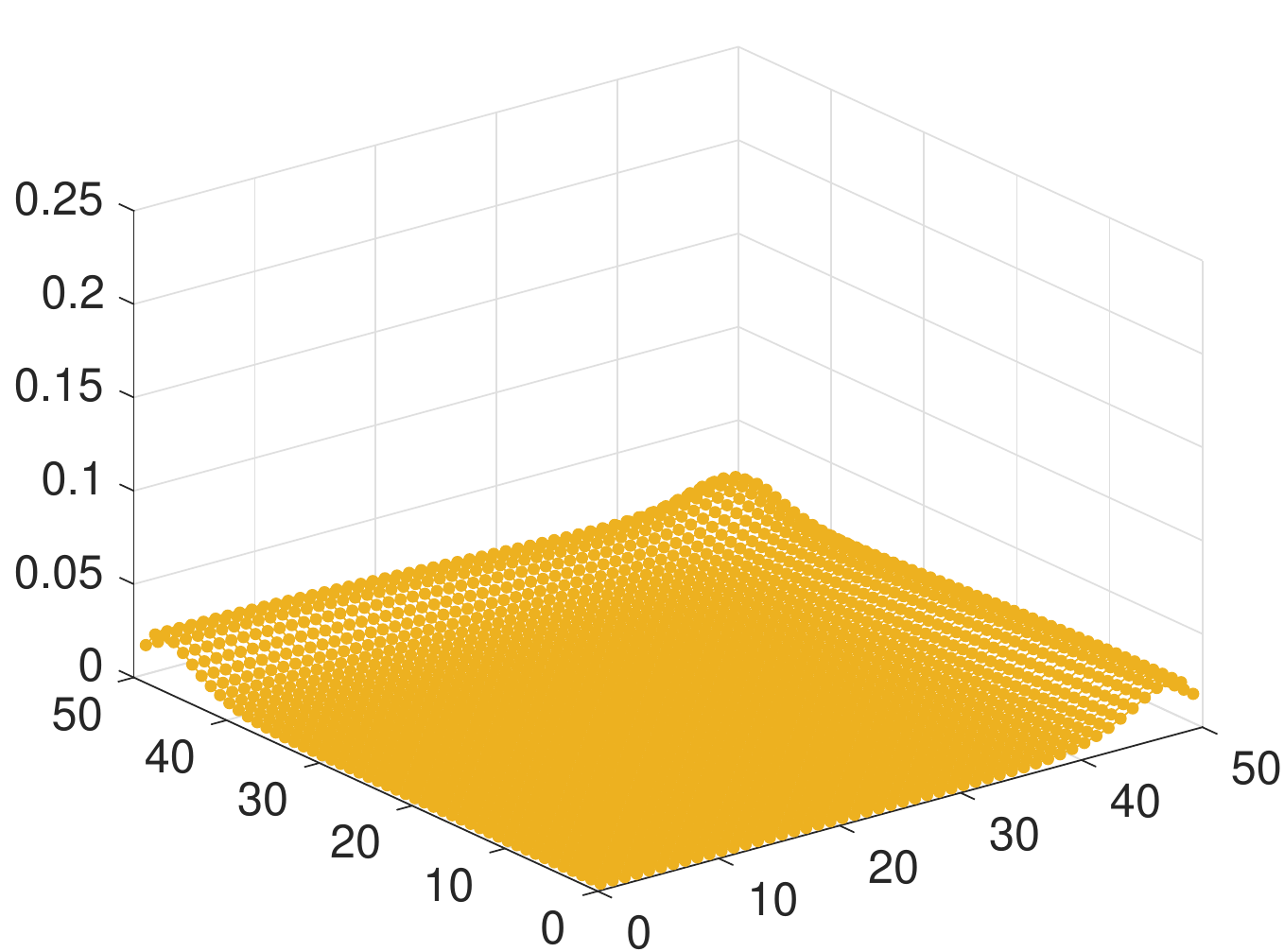}}\\
\subfigure[$e_{u,\indeigkone,\indeigktwo}$ in $\mathbb{S}_{1,50,0}^\opt\otimes\mathbb{S}_{1,50,0}^\opt$]{\includegraphics[height=4.1cm]{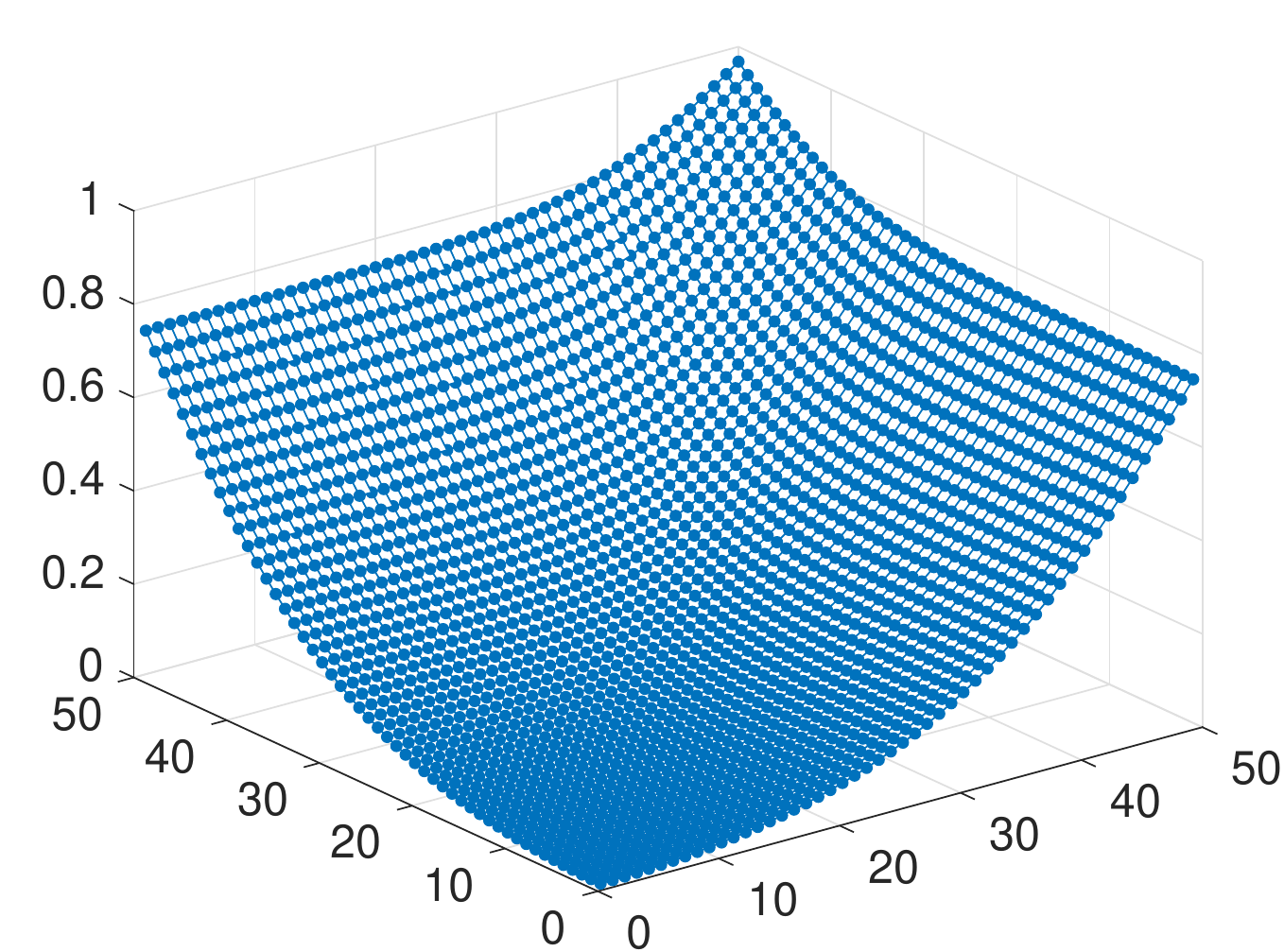}}\hspace*{0.1cm}
\subfigure[$e_{u,\indeigkone,\indeigktwo}$ in $\mathbb{S}_{3,50,0}^\opt\otimes\mathbb{S}_{3,50,0}^\opt$]{\includegraphics[height=4.1cm]{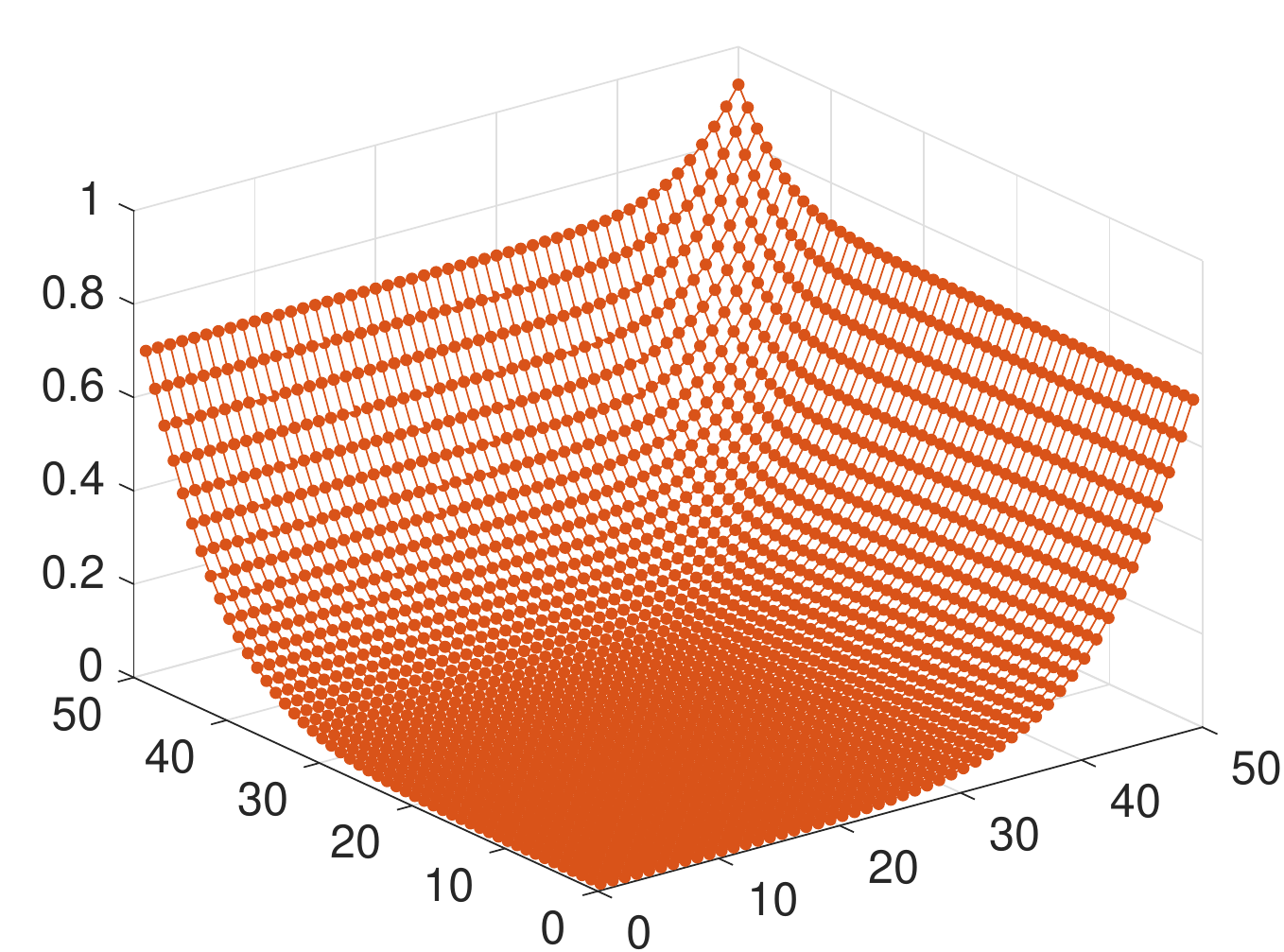}}\hspace*{0.1cm}
\subfigure[$e_{u,\indeigkone,\indeigktwo}$ in $\mathbb{S}_{5,50,0}^\opt\otimes\mathbb{S}_{5,50,0}^\opt$]{\includegraphics[height=4.1cm]{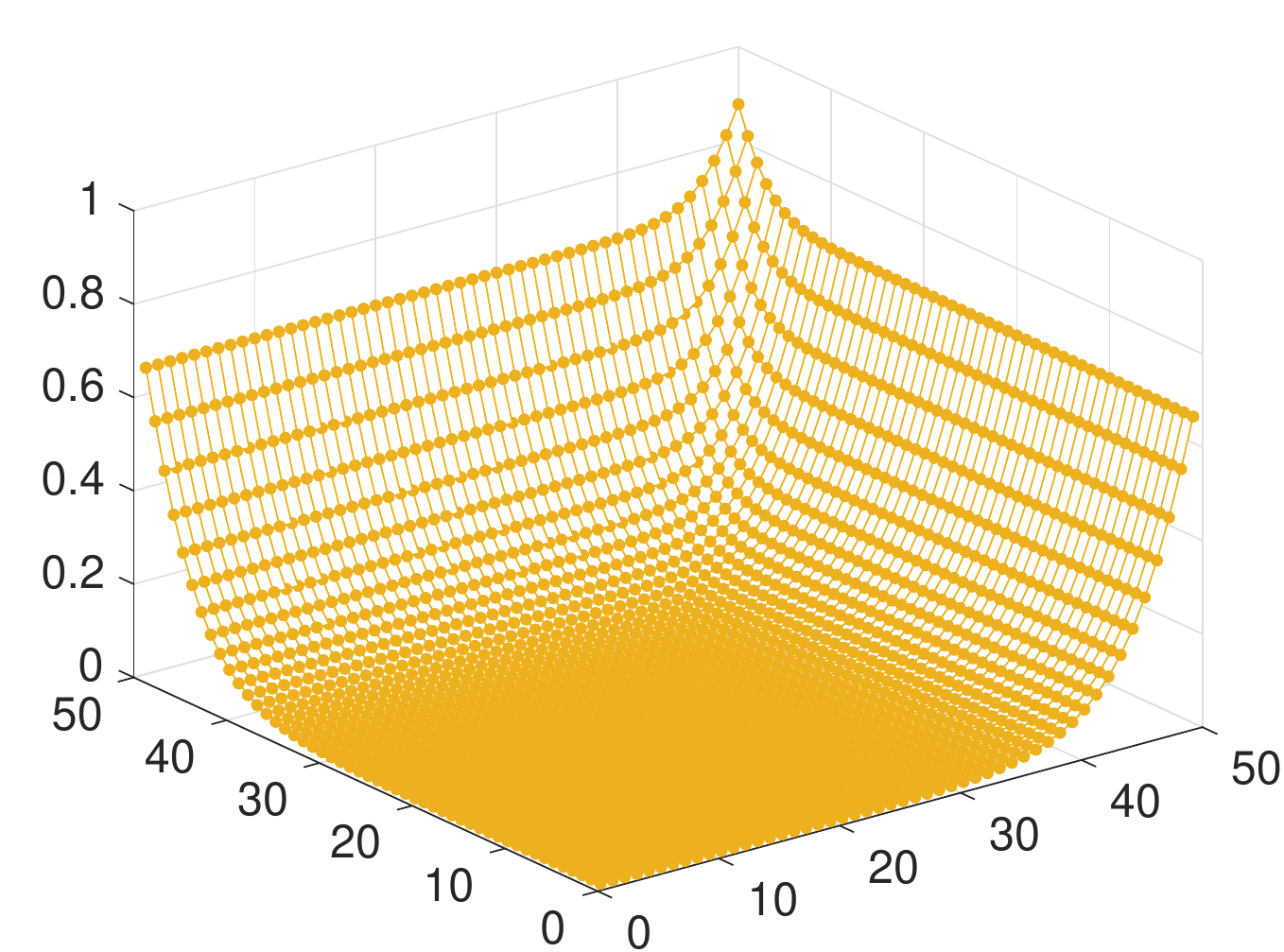}}
\caption{Example~\ref{ex:eigenvalues2D}: Relative frequency errors $e_{\omega,\indeigkone,\indeigktwo}$ in \eqref{eq:error-eigenvalues2D} and $L^2$ relative eigenfunction errors $e_{u,\indeigkone,\indeigktwo}$ in \eqref{eq:error-eigenfunctions2D} corresponding to the spline space $\mathbb{S}_{p,n,0}^\opt\otimes\mathbb{S}_{p,n,0}^\opt$ for odd degrees $p$ and $n=50$. The errors are ordered according to the decomposition of the exact frequencies in  \eqref{eq:eig-Laplace2D} in terms of increasing values of their univariate counterparts. No outliers are observed.} \label{fig:eigenvalues2D.odd.surface}
\end{figure}
\begin{figure}[t!]
\centering
\subfigure[$e_{\omega,\indeigkone,\indeigktwo}$ in $\mathbb{S}_{2,50,0}^\opt\otimes\mathbb{S}_{2,50,0}^\opt$]{\includegraphics[height=4.1cm]{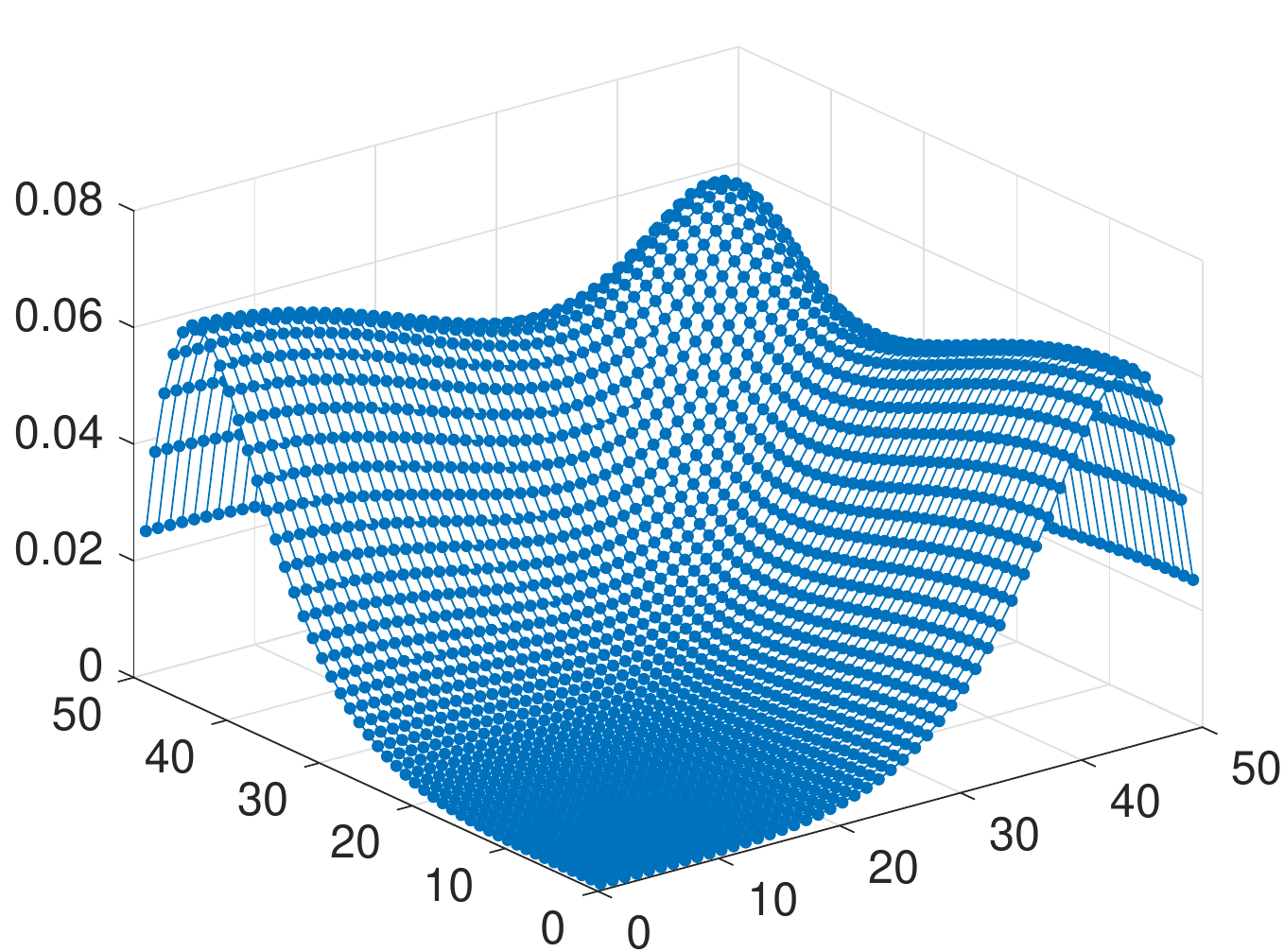}}\hspace*{0.1cm}
\subfigure[$e_{\omega,\indeigkone,\indeigktwo}$ in $\mathbb{S}_{4,50,0}^\opt\otimes\mathbb{S}_{4,50,0}^\opt$]{\includegraphics[height=4.1cm]{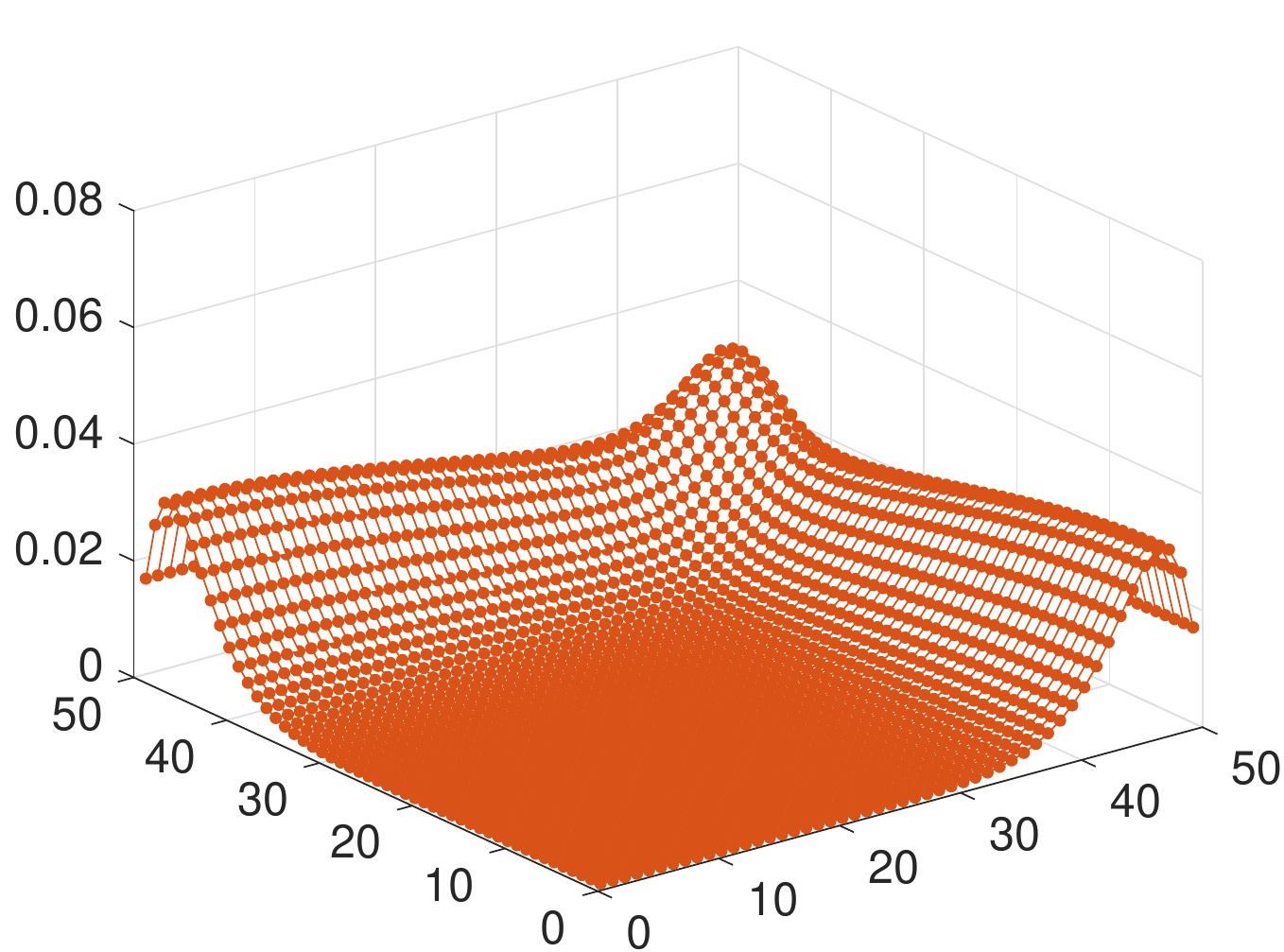}}\hspace*{0.1cm}
\subfigure[$e_{\omega,\indeigkone,\indeigktwo}$ in $\mathbb{S}_{6,50,0}^\opt\otimes\mathbb{S}_{6,50,0}^\opt$]{\includegraphics[height=4.1cm]{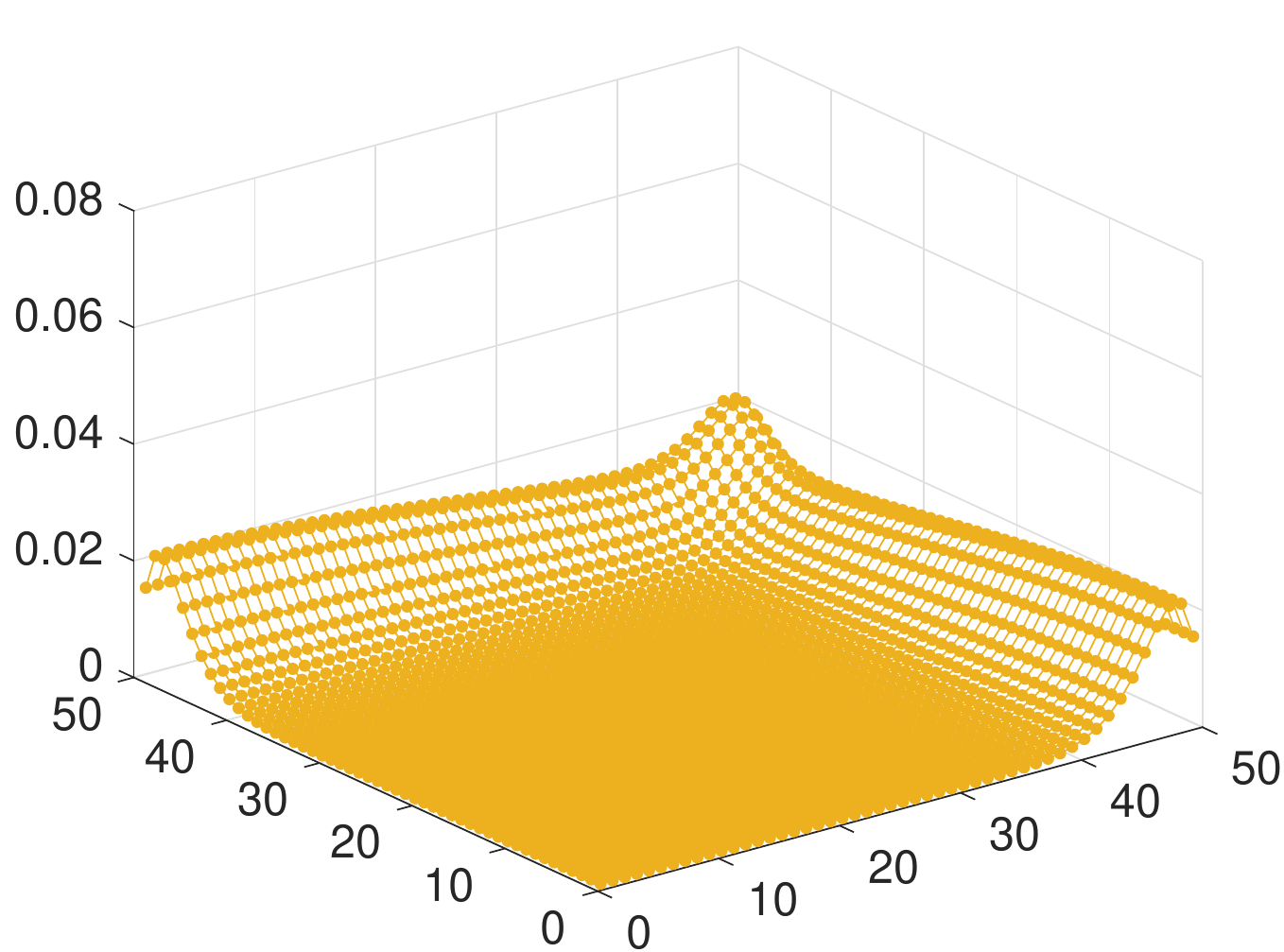}}\\
\subfigure[$e_{u,\indeigkone,\indeigktwo}$ in $\mathbb{S}_{2,50,0}^\opt\otimes\mathbb{S}_{2,50,0}^\opt$]{\includegraphics[height=4.1cm]{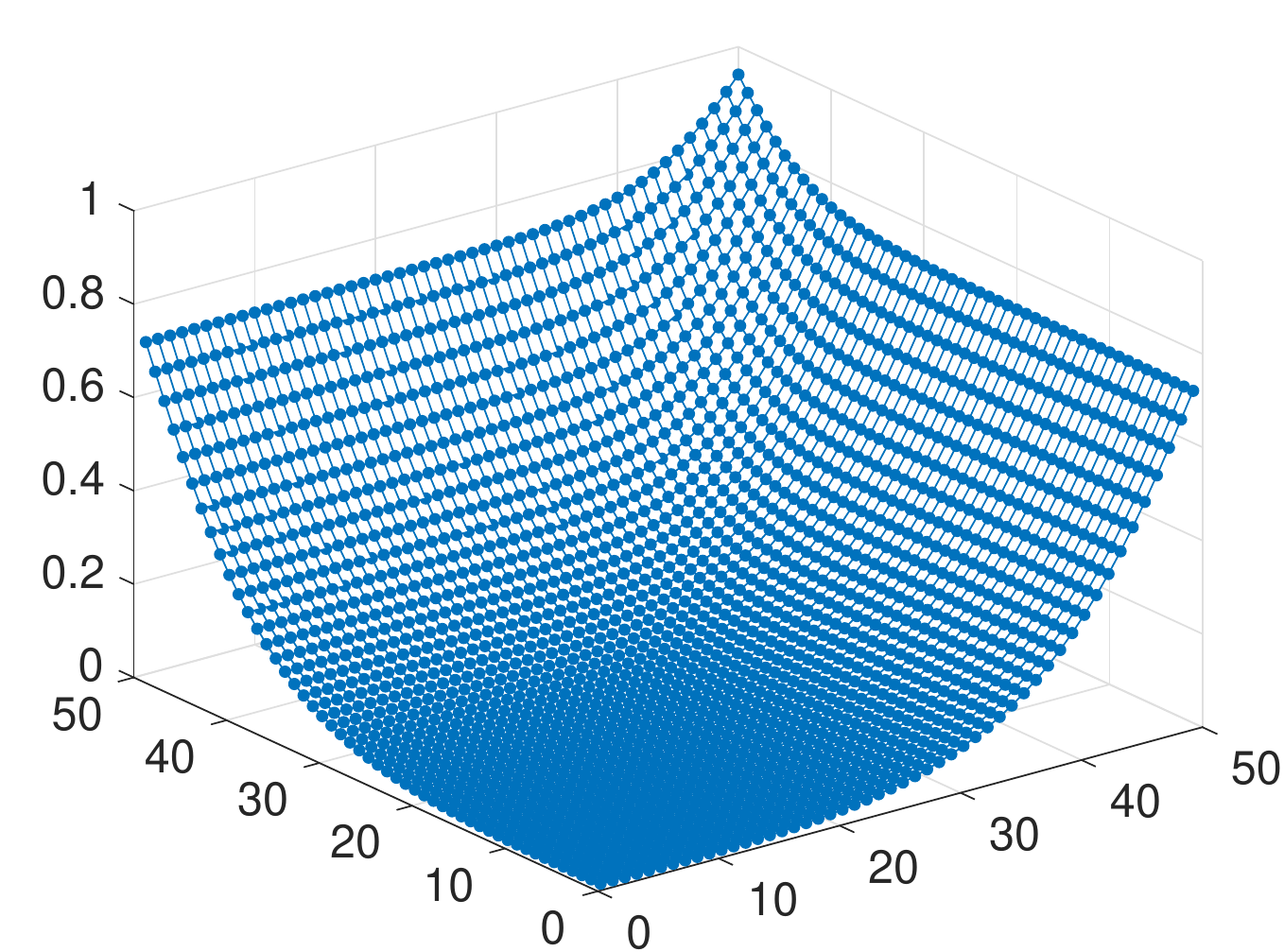}}\hspace*{0.1cm}
\subfigure[$e_{u,\indeigkone,\indeigktwo}$ in $\mathbb{S}_{4,50,0}^\opt\otimes\mathbb{S}_{4,50,0}^\opt$]{\includegraphics[height=4.1cm]{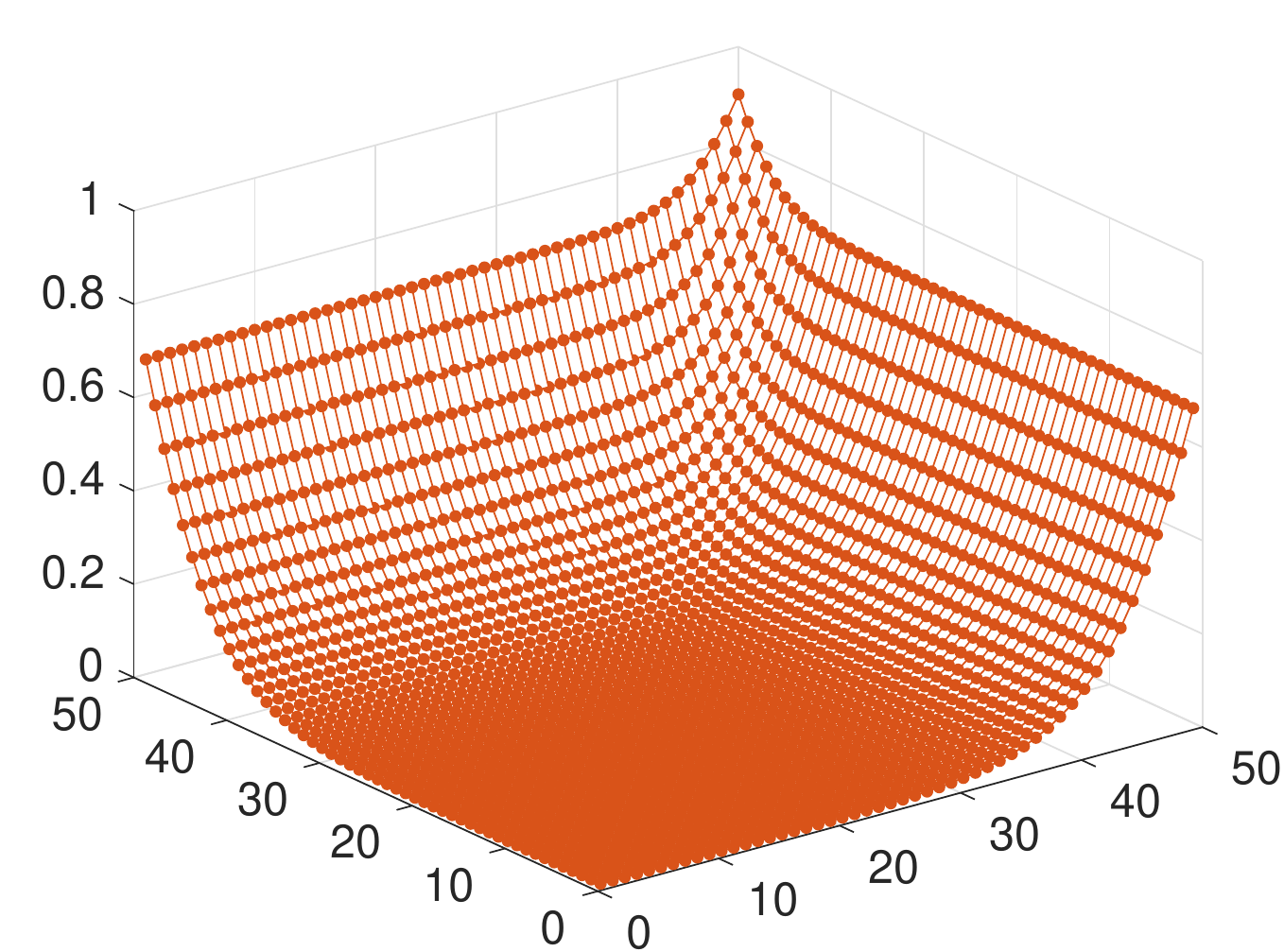}}\hspace*{0.1cm}
\subfigure[$e_{u,\indeigkone,\indeigktwo}$ in $\mathbb{S}_{6,50,0}^\opt\otimes\mathbb{S}_{6,50,0}^\opt$]{\includegraphics[height=4.1cm]{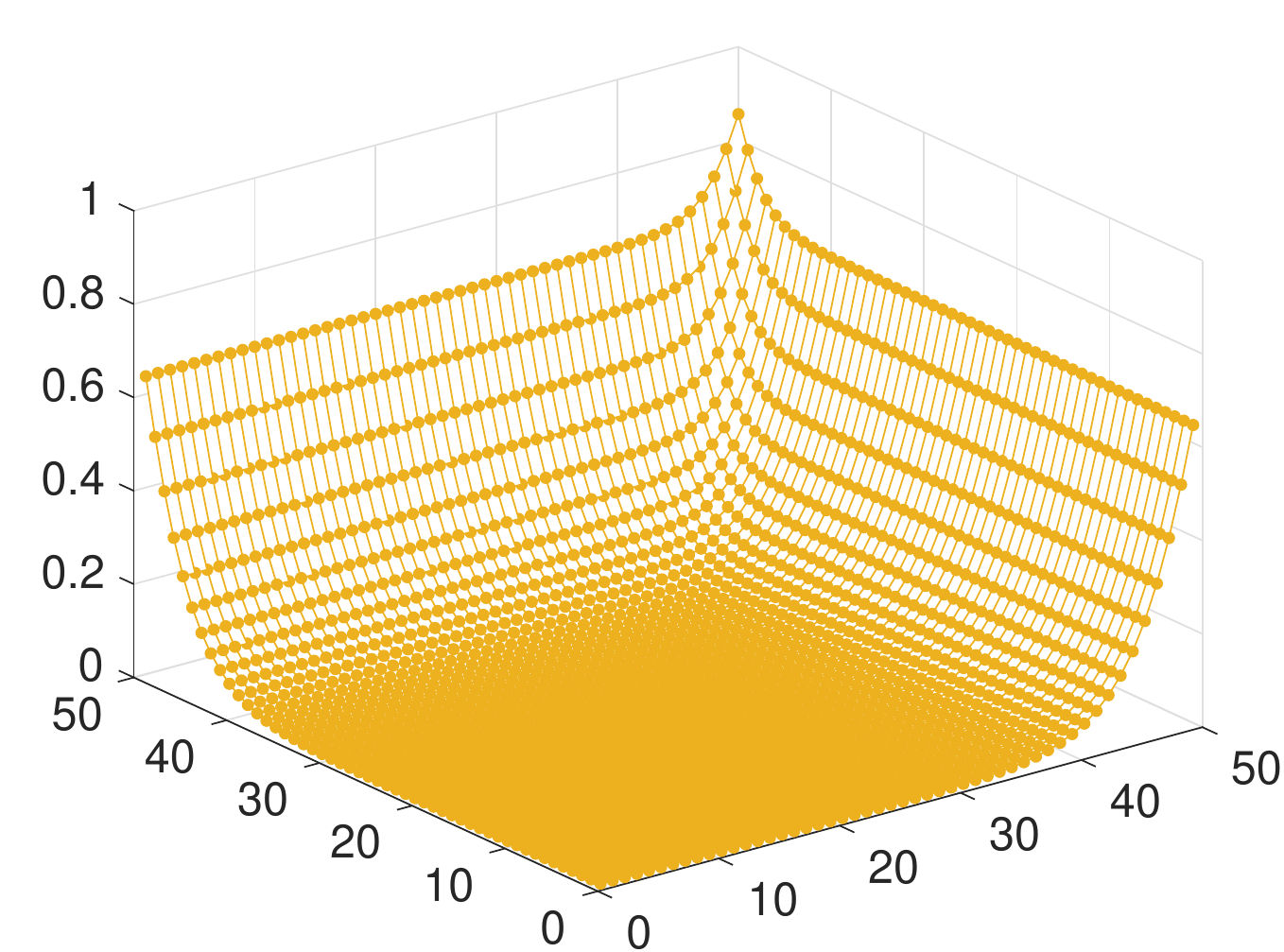}}
\caption{Example~\ref{ex:eigenvalues2D}: Relative frequency errors $e_{\omega,\indeigkone,\indeigktwo}$ in \eqref{eq:error-eigenvalues2D} and $L^2$ relative eigenfunction errors $e_{u,\indeigkone,\indeigktwo}$ in \eqref{eq:error-eigenfunctions2D} corresponding to the spline space $\mathbb{S}_{p,n,0}^\opt\otimes\mathbb{S}_{p,n,0}^\opt$ for even degrees $p$ and $n=50$. The errors are ordered according to the decomposition of the exact frequencies in \eqref{eq:eig-Laplace2D} in terms of increasing values of their univariate counterparts. No outliers are observed.} \label{fig:eigenvalues2D.even.surface}
\end{figure}

\begin{example}\label{ex:eigenvalues2D}
In this example we show the performance of the discretizations for approximating the eigenvalue problem \eqref{eq:prob-eigenv}.
Let $\omega_{h_1,h_2,\indeigkone,\indeigktwo}$ be the approximate value of the frequency $\omega_{\indeigkone,\indeigktwo}$; see \eqref{eq:eig-Laplace2D-type-0-BC}. Here we assume that the matching between the exact frequencies and the approximate ones is retrieved by considering their decomposition in \eqref{eq:eig-Laplace2D} and \eqref{eq:approx-eig-Laplace2D}, respectively, both in terms of increasing values of their univariate counterparts.
In Figures~\ref{fig:eigenvalues2D.odd} and~\ref{fig:eigenvalues2D.even} we depict the relative frequency errors 
\begin{equation}\label{eq:error-eigenvalues2D}
e_{\omega,\indeigkone,\indeigktwo}:=
\frac{\omega_{h_1,h_2,\indeigkone,\indeigktwo}-\omega_{\indeigkone,\indeigktwo}}{\omega_{\indeigkone,\indeigktwo}}, \quad \indeigkone,\indeigktwo=1,\ldots,n,
\end{equation}
obtained by the Galerkin approximations in $\mathbb{S}_{p,n,0}^\opt\otimes\mathbb{S}_{p,n,0}^\opt$, $\overline{\mathbb{S}}_{p,n,0}\otimes \overline{\mathbb{S}}_{p,n,0}$ and $\mathbb{S}_{p,n,0}\otimes\mathbb{S}_{p,n,0}$ for various degrees and $n=50$. 
We clearly notice that the reduced spaces $\mathbb{S}_{p,n,0}^\opt\otimes\mathbb{S}_{p,n,0}^\opt$ and $\overline{\mathbb{S}}_{p,n,0}\otimes \overline{\mathbb{S}}_{p,n,0}$ capture all the eigenvalues without any outlier still maintaining the accuracy of the full tensor-product spline space. 
In the same figures we also report the corresponding $L^2$ relative errors for the eigenfunction approximations
\begin{equation}\label{eq:error-eigenfunctions2D}
e_{u,\indeigkone,\indeigktwo}:=\frac{\|u_{\indeigkone,\indeigktwo}-u_{h_1,h_2,\indeigkone,\indeigktwo}\|}{\|u_{\indeigkone,\indeigktwo}\|}, \quad \indeigkone,\indeigktwo=1,\ldots,n.
\end{equation}
Finally, we remark that the frequency errors and eigenfunction errors seem to approximately lie on several smooth subcurves instead of a single smooth curve as in the univariate case (see Example~\ref{ex:eigenvalues1D}). This can be simply explained by the fact that they actually lie on a bivariate smooth surface due to the bivariate nature of the problem; see \eqref{eq:eig-Laplace2D} and \eqref{eq:approx-eig-Laplace2D}. This is illustrated in Figures~\ref{fig:eigenvalues2D.odd.surface} and~\ref{fig:eigenvalues2D.even.surface} for the space $\mathbb{S}_{p,n,0}^\opt\otimes\mathbb{S}_{p,n,0}^\opt$ considering various degrees.
\end{example}

\begin{remark}\label{rmk:matching_frequencies}
A natural matching between the exact frequencies and the approximate ones is obtained by sorting them both in ascending order. In the univariate case, this leads to satisfactory results both for the relative frequency errors and for the relative eigenfunction errors; see Example~\ref{ex:eigenvalues1D}. Moreover, these errors approximately lie on a smooth curve. In the multivariate setting, the situation is more complicated. Such matching seems to give good results for the relative frequency errors, even if they do not resemble a smooth curve anymore; we refer the reader to \cite{Cottrell:2006,Deng:2021} for examples. The corresponding relative eigenfunction errors, however, indicate that there could be a substantial mismatch. This problem has been resolved by the matching strategy adopted in \cite{Hiemstra:2021}, and followed here, which exploits the decompositions in \eqref{eq:eig-Laplace2D} and \eqref{eq:approx-eig-Laplace2D}. It gives satisfactory results both for the relative frequency errors and for the relative eigenfunction errors; see Example~\ref{ex:eigenvalues2D}. Similarly to the univariate case, the corresponding errors approximately lie on a smooth surface.
Although the matching strategy is the same as in \cite{Hiemstra:2021}, we have visualized the errors in a different way. In \cite{Hiemstra:2021} the plots are obtained by a sorting of both the relative frequency errors and the relative eigenfunction errors. We believe, on the other hand, that it is important to keep track of the magnitude of the frequencies in order to show and understand what is the behavior of the error for low and for high frequencies (and the corresponding eigenfunctions) and also for consistency with the univariate plots.
\end{remark}

\begin{figure}[t!]
\centering
\subfigure[$\mathbb{S}_{p,n,0}^\opt\otimes\mathbb{S}_{p,n,0}^\opt$]{\includegraphics[height=4.1cm]{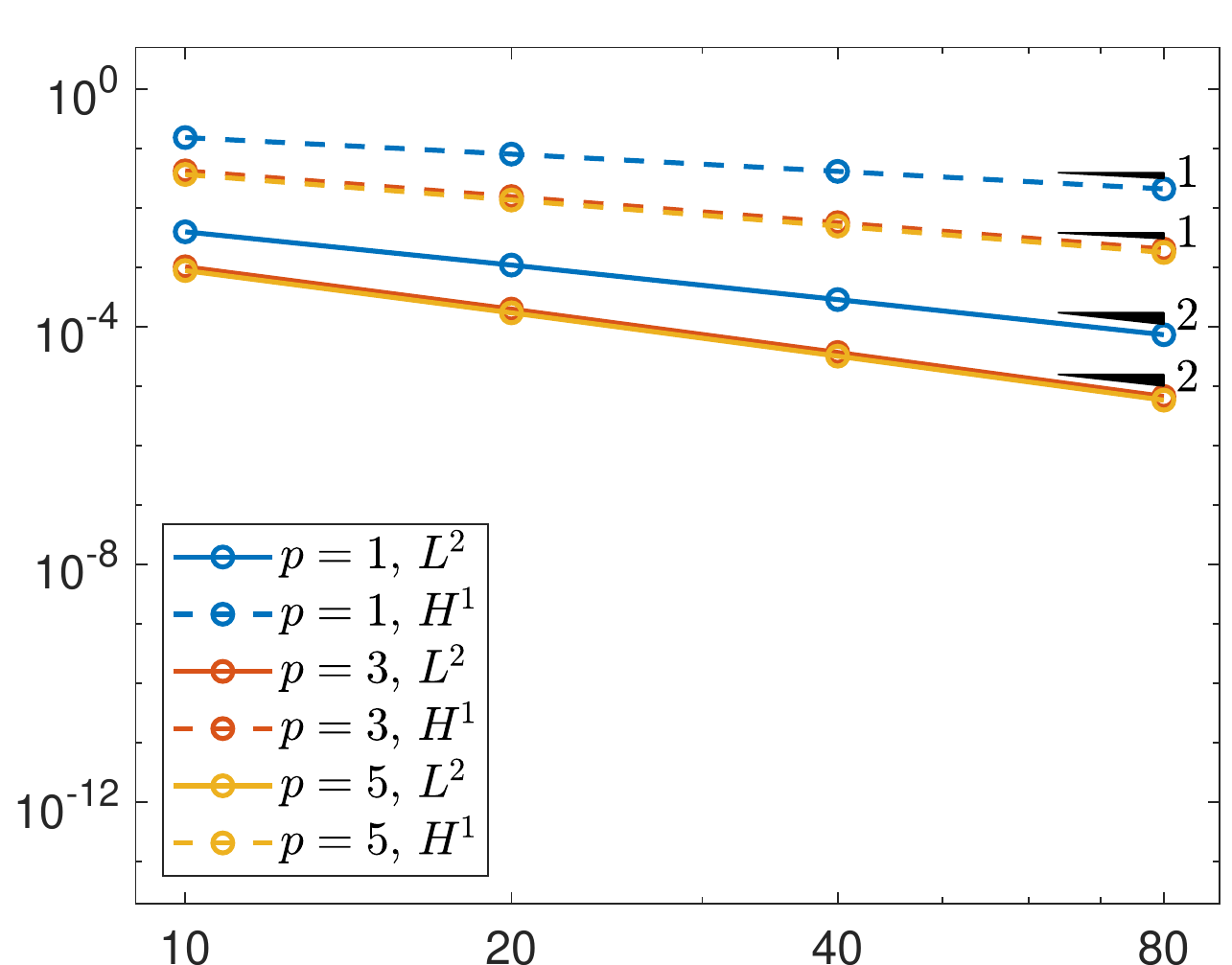}}\hspace*{0.1cm}
\subfigure[$\mathbb{S}_{p,n,0}^\opt\otimes\mathbb{S}_{p,n,0}^\opt$ with correction]{\includegraphics[height=4.1cm]{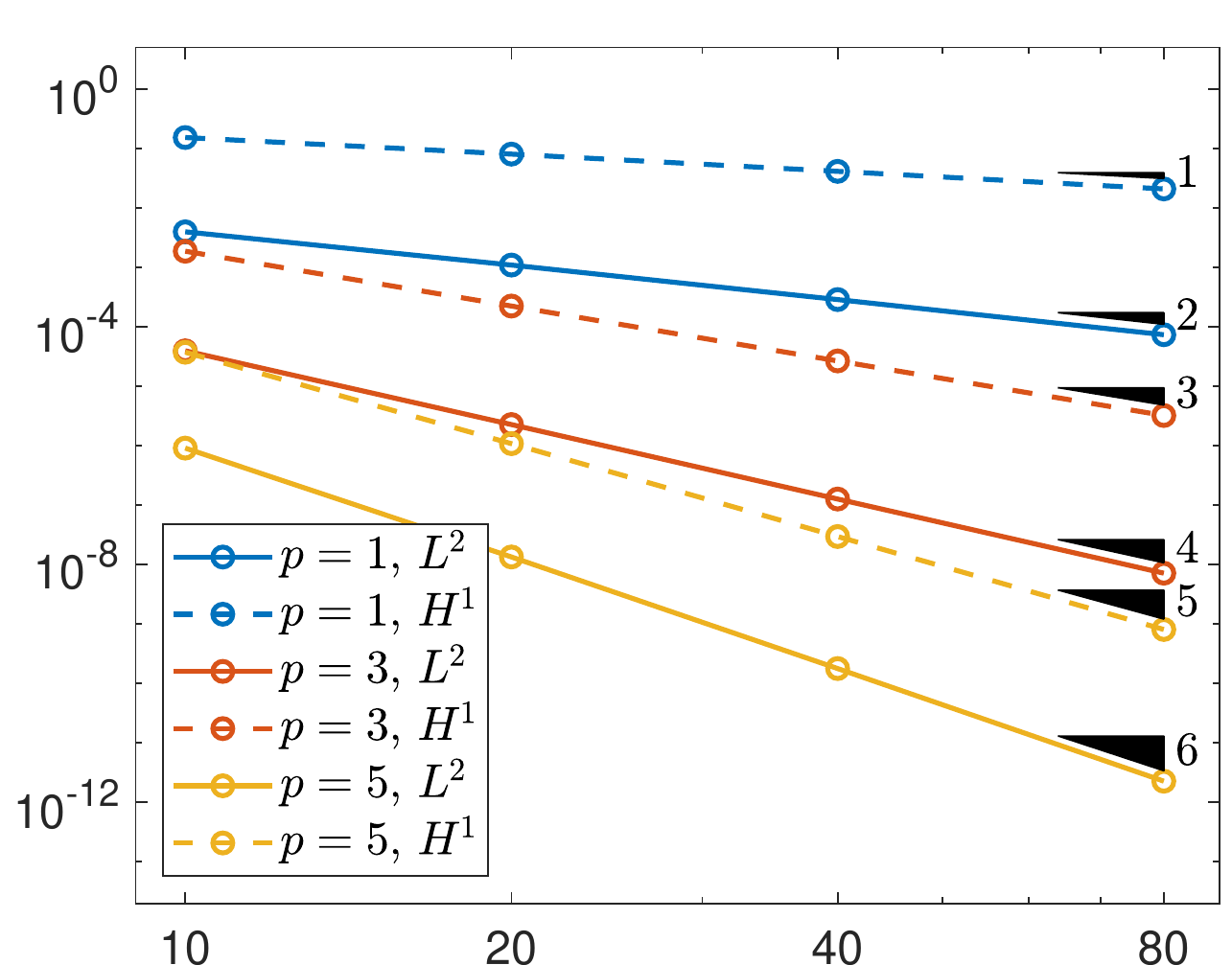}}\hspace*{0.1cm}
\subfigure[$\mathbb{S}_{p,n,0}\otimes\mathbb{S}_{p,n,0}$]{\includegraphics[height=4.1cm]{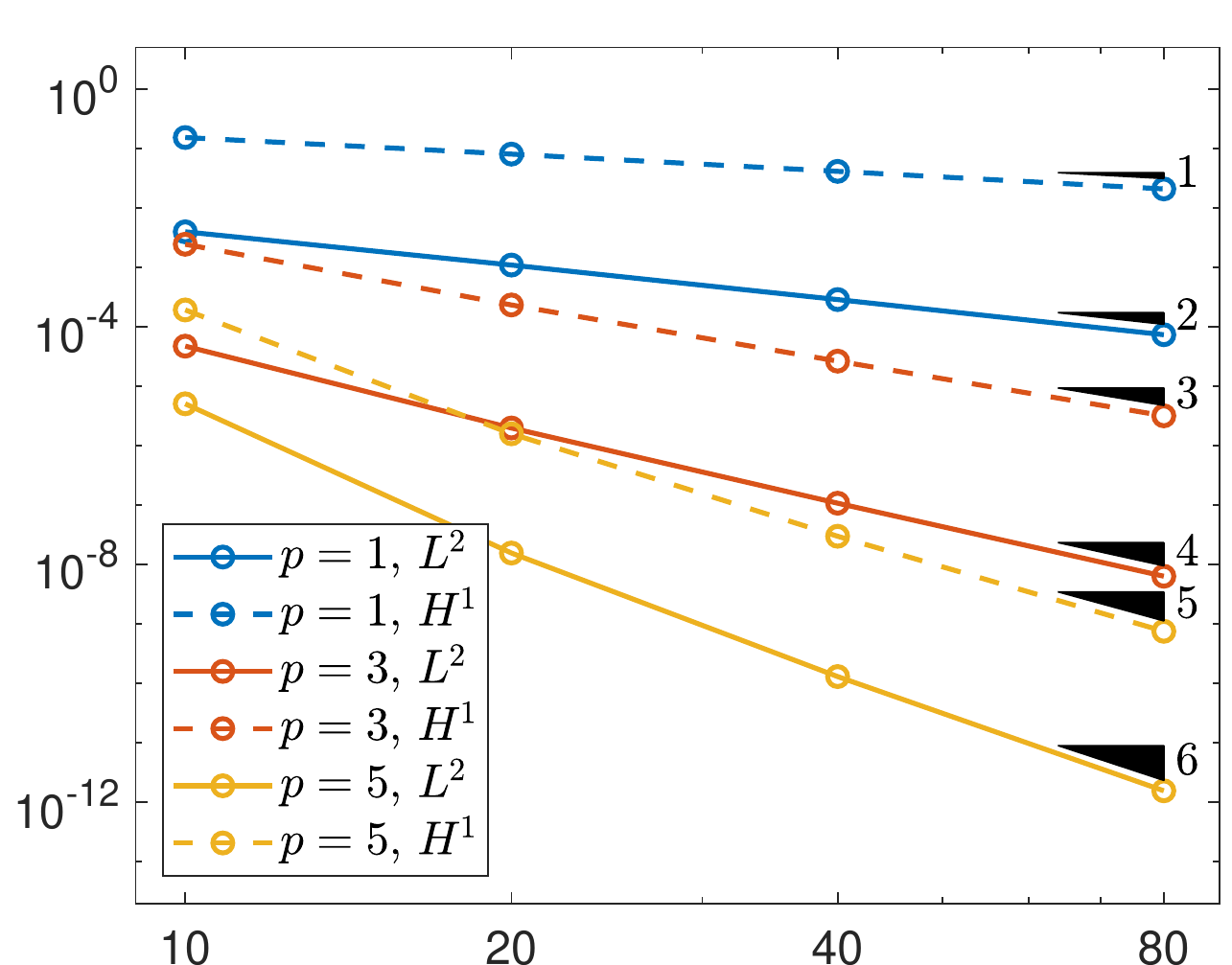}}
\caption{Example~\ref{ex:convergence2D-boundary}: $L^2$ and $H^1$ error convergence in the spline spaces $\mathbb{S}_{p,n,0}^\opt\otimes\mathbb{S}_{p,n,0}^\opt$ and $\mathbb{S}_{p,n,0}\otimes\mathbb{S}_{p,n,0}$ in terms of $n$, for odd degrees $p$. Both without and with boundary data correction are considered for the space $\mathbb{S}_{p,n,0}^\opt\otimes\mathbb{S}_{p,n,0}^\opt$. The reference convergence order in $n$ is indicated by black triangles.} \label{ex:convergence2D-boundary:a}
\bigskip
\centering
\subfigure[$\mathbb{S}_{p,n,0}^\opt\otimes\mathbb{S}_{p,n,0}^\opt$]{\includegraphics[height=4.1cm]{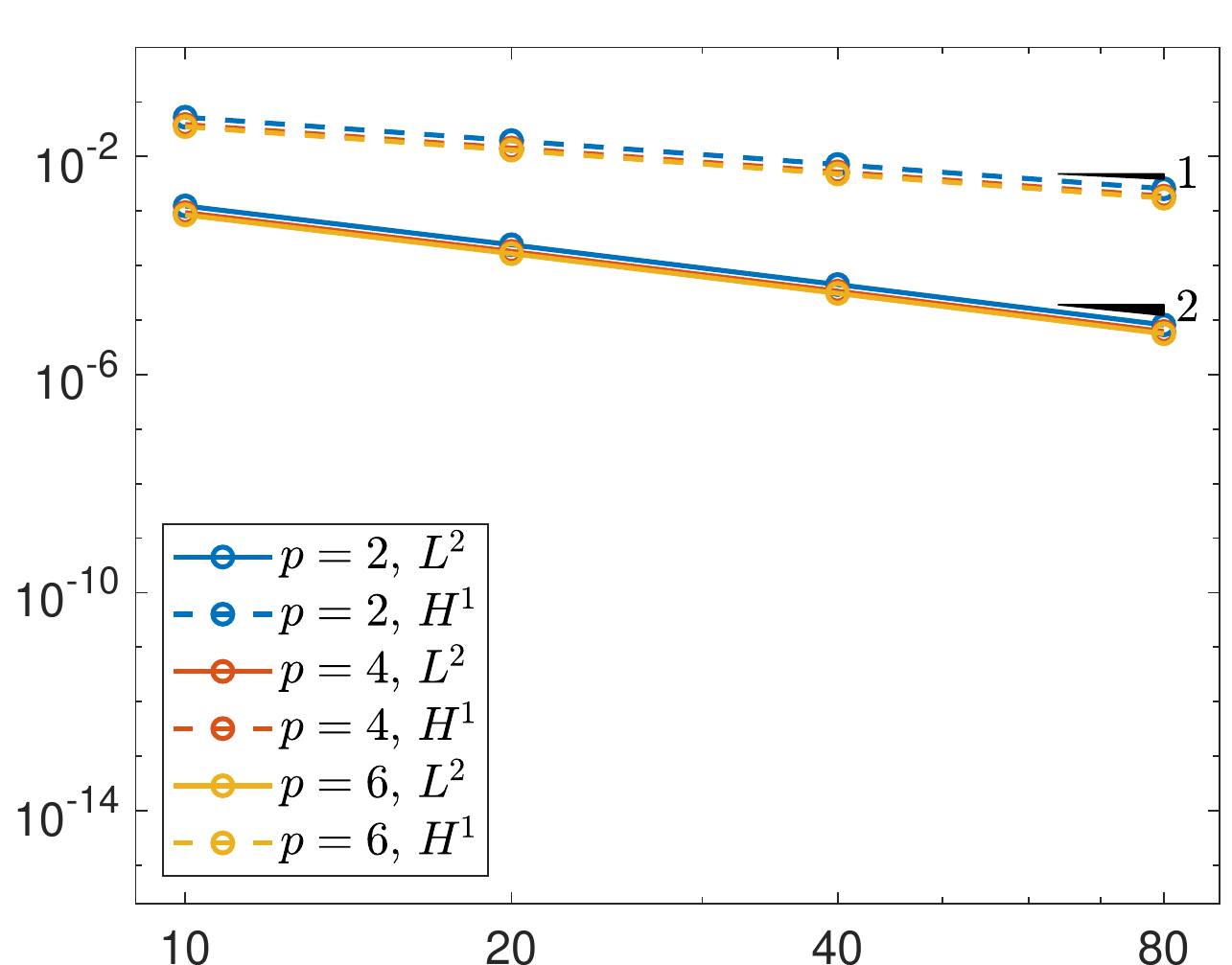}}\hspace*{0.1cm}
\subfigure[$\overline{\mathbb{S}}_{p,n,0}\otimes \overline{\mathbb{S}}_{p,n,0}$]{\includegraphics[height=4.1cm]{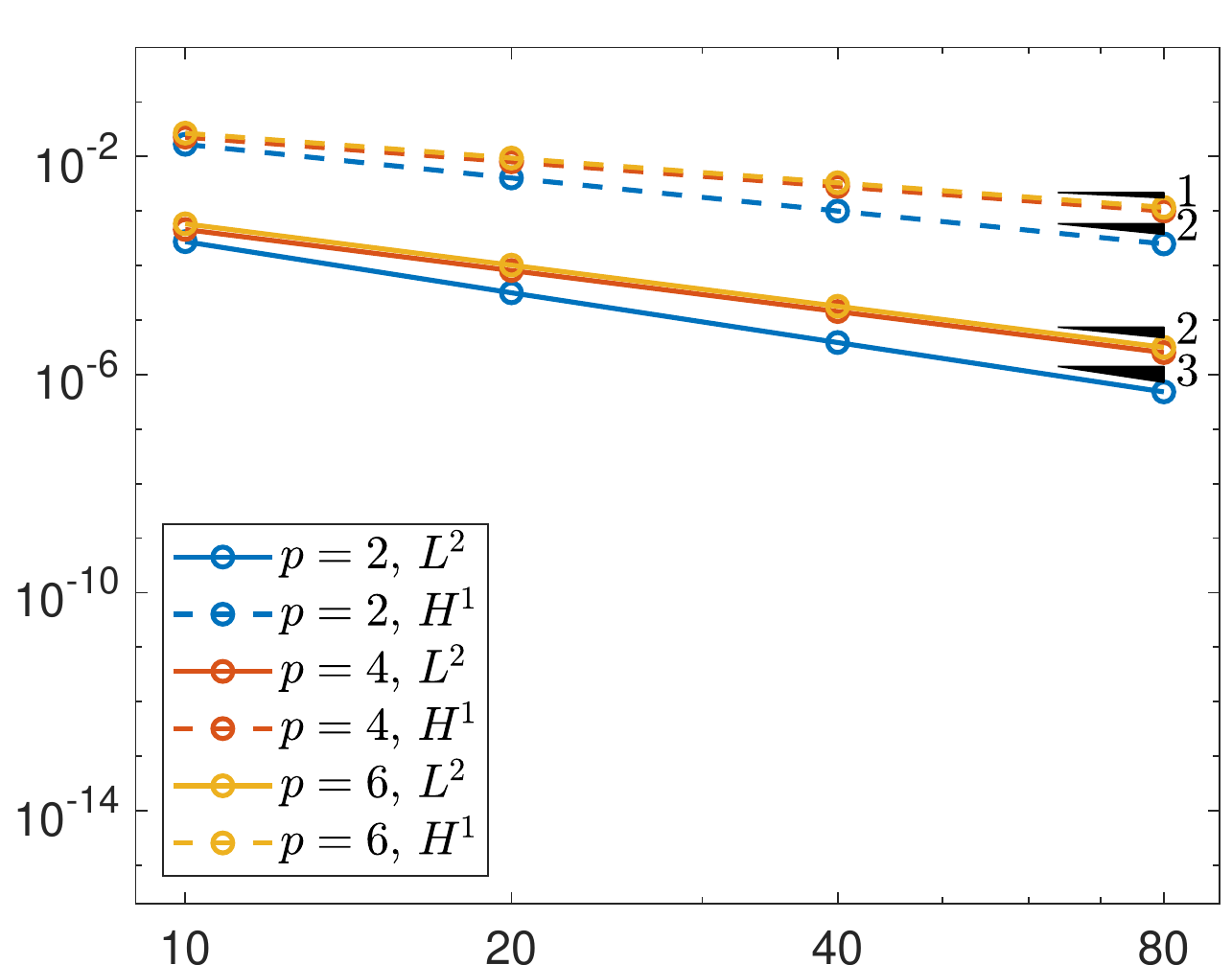}} \\
\subfigure[$\mathbb{S}_{p,n,0}^\opt\otimes\mathbb{S}_{p,n,0}^\opt$ with correction]{\includegraphics[height=4.1cm]{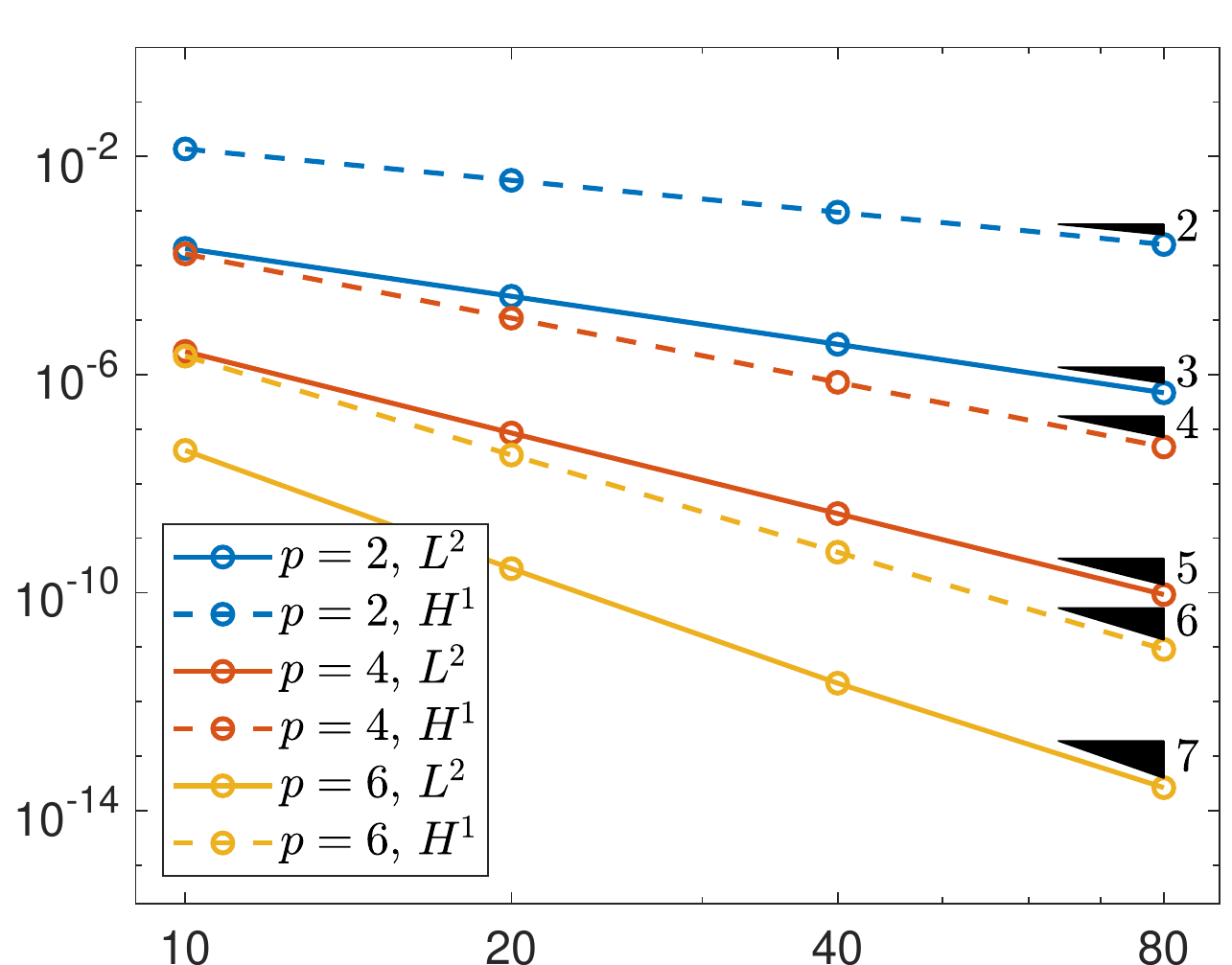}}\hspace*{0.1cm}
\subfigure[$\overline{\mathbb{S}}_{p,n,0}\otimes \overline{\mathbb{S}}_{p,n,0}$ with correction]{\includegraphics[height=4.1cm]{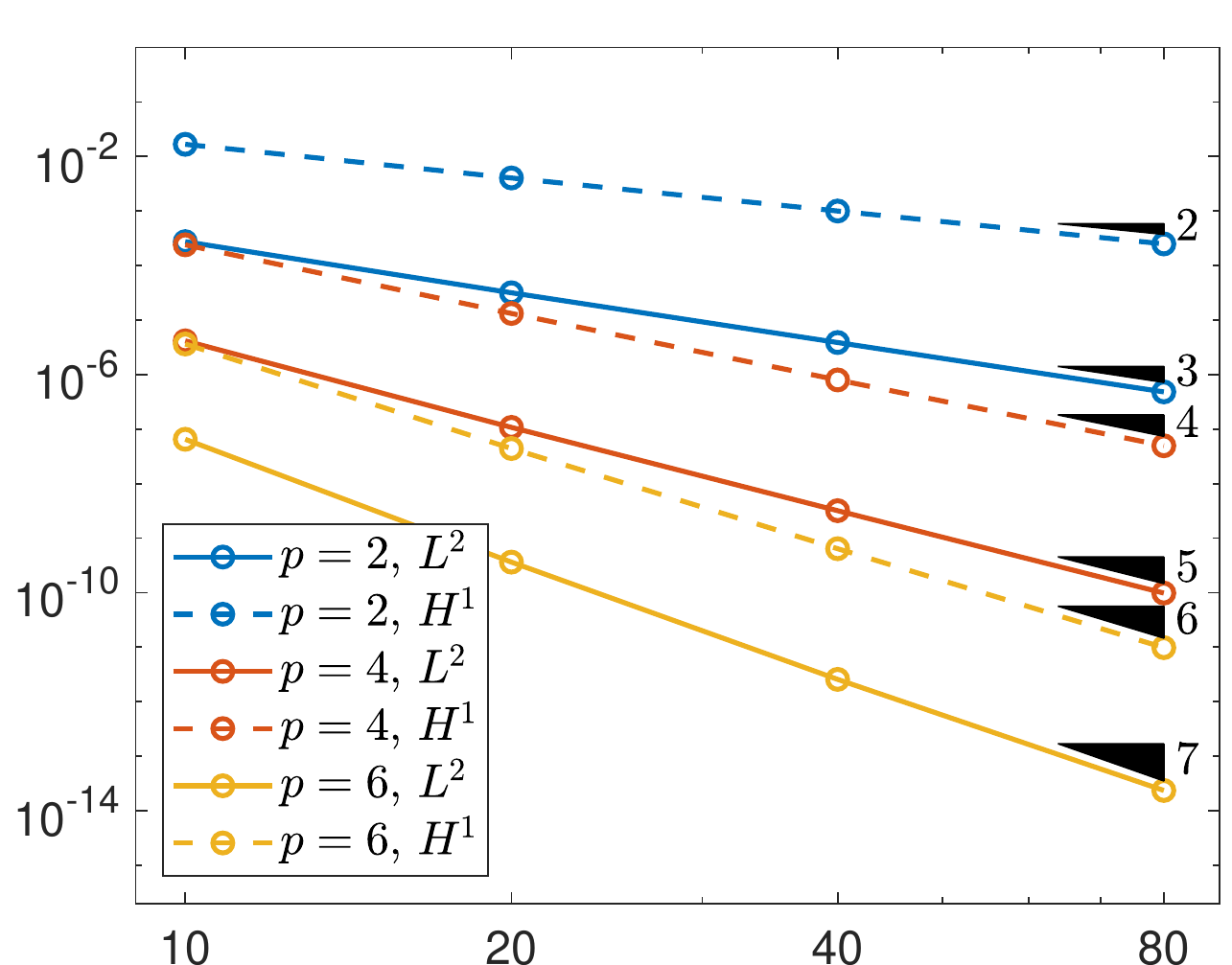}}\hspace*{0.1cm}
\subfigure[$\mathbb{S}_{p,n,0}\otimes\mathbb{S}_{p,n,0}$]{\includegraphics[height=4.1cm]{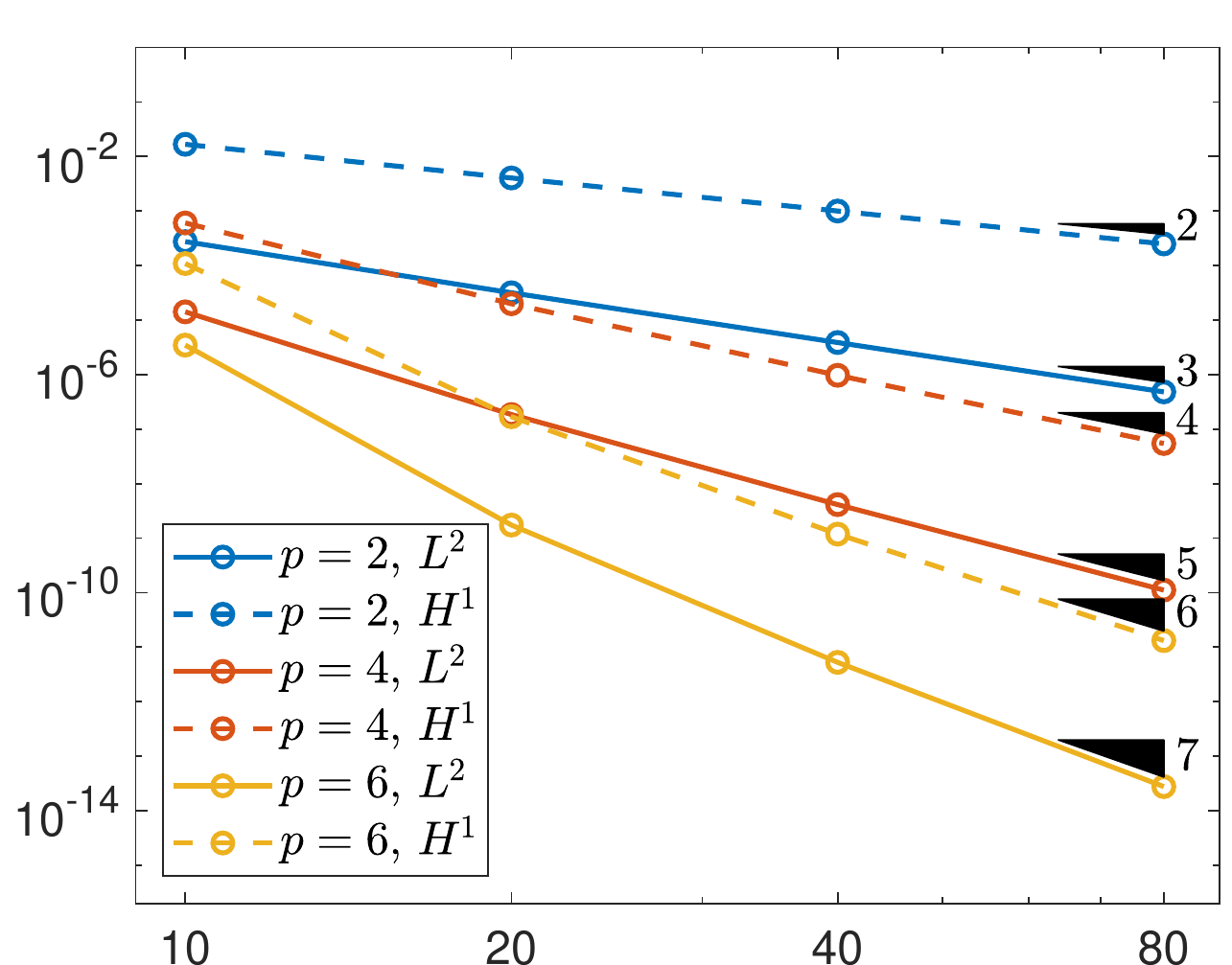}}
\caption{Example~\ref{ex:convergence2D-boundary}: $L^2$ and $H^1$ error convergence in the spline spaces $\mathbb{S}_{p,n,0}^\opt\otimes\mathbb{S}_{p,n,0}^\opt$, $\overline{\mathbb{S}}_{p,n,0}\otimes \overline{\mathbb{S}}_{p,n,0}$ and $\mathbb{S}_{p,n,0}\otimes\mathbb{S}_{p,n,0}$ in terms of $n$, for even degrees $p$. Both without and with boundary data correction are considered for the spaces $\mathbb{S}_{p,n,0}^\opt\otimes\mathbb{S}_{p,n,0}^\opt$ and $\overline{\mathbb{S}}_{p,n,0}\otimes \overline{\mathbb{S}}_{p,n,0}$. The reference convergence order in $n$ is indicated by black triangles.} \label{ex:convergence2D-boundary:b}
\end{figure}

\begin{example}\label{ex:convergence2D-boundary}
As a test for the strategy presented in Section~\ref{sec:general-BC-2D} in the bivariate setting, we consider problem \eqref{eq:Laplace} with the manufactured solution 
$$u(x_1,x_2)=x_1(1-\cos(2\pi x_1))(1-e^{x_2})(1-e^{1-x_2}).$$
The exact solution does not satisfy the additional conditions on high-order derivatives defining the spaces $\mathbb{S}_{p,n,0}^\opt\otimes\mathbb{S}_{p,n,0}^\opt$ and $\overline{\mathbb{S}}_{p,n,0}\otimes \overline{\mathbb{S}}_{p,n,0}$. In Figures~\ref{ex:convergence2D-boundary:a} and~\ref{ex:convergence2D-boundary:b} we depict the convergence of the approximate solutions in the spline spaces $\mathbb{S}_{p,n,0}^\opt\otimes\mathbb{S}_{p,n,0}^\opt$, $\overline{\mathbb{S}}_{p,n,0}\otimes \overline{\mathbb{S}}_{p,n,0}$ and $\mathbb{S}_{p,n,0}\otimes\mathbb{S}_{p,n,0}$ in terms of $n$, for various values of $p$. Like in the univariate setting, we see a substantial loss of accuracy for $p>2$ when approximating the solution in the reduced spline spaces. However, the full convergence order is recovered by applying the boundary data correction described in Section~\ref{sec:general-BC-2D}.
\end{example}

\section{Conclusion}
\label{sec:conclusion}

In this paper we have presented and theoretically analyzed the use of optimal spline subspaces in isogeometric Galerkin discretizations of eigenvalue problems related to the Laplace operator subject to standard homogeneous boundary conditions (Dirichlet/Neumann/mixed) in the univariate and in the multivariate tensor-product case.
By completing the theory started in \cite{Sande:2019} for periodic boundary conditions, we have proved that in the optimal spline subspaces described in \cite{Floater:2017,Floater:2018}, as suggested in \cite{Sande:2019,Sande:2020}, the whole spectrum is well approximated and no spurious values appear, thus resulting in accurate outlier-free discretizations.

The main contribution of the paper is twofold:
\begin{itemize}
\item we have provided explicit error estimates for Ritz projectors in the considered optimal spline subspaces;
\item by exploiting the above estimates, we have derived explicit error estimates for the approximated eigenfunctions and frequencies in terms of the exact ones.
\end{itemize}
The first item shows that the optimal spline subspaces possess full approximation power, while the second item implies that, for a fixed number of degrees of freedom, there is convergence in $p$ of the whole discrete spectrum (thus no outliers).
The optimal spaces we are dealing with are subspaces of the standard (maximally smooth) spline space defined on certain uniform knot sequences (whose structure depends on the parity of the degree $p$). The subspaces are identified and simply described by mimicking the behavior of vanishing derivatives (up to order $p$) of the exact eigenfunctions at the boundary. 

It turns out that the optimal spline subspaces analyzed here are very similar (and actually identical in several cases) to those proposed as trial spaces for outlier removal in \cite{Hiemstra:2021}; see also \cite{Deng:2021}.
However, for a fixed type of boundary condition (Dirichlet/Neumann/mixed) the subspaces in \cite{Hiemstra:2021} can slightly differ from ours in the partition and in the maximum order of vanishing derivatives at the boundary, depending on the parity of the degree $p$; see Section~\ref{sec:outlier-free}.
The subspaces in \cite{Hiemstra:2021} were previously introduced for uniform knot sequences in \cite{Sogn:2018,Takacs:2016}, and further analyzed for general knot sequences in \cite[Section~5.2]{Sande:2020}. A complete error analysis
for such subspaces --- when different from the optimal ones addressed in the present paper --- is worth to be subject of future research. Nevertheless, their strong similarity with the optimal subspaces discussed here suggests that they are ``almost optimal'' and motivates their effectiveness for outlier removal, which was already clearly numerically demonstrated in \cite{Hiemstra:2021} and also in Section~\ref{sec:numerics}.

As a side result, we have also illustrated that the outlier-free optimal spline subspaces can be exploited to obtain discretization schemes for general (non-homogeneous) problems with full approximation power by providing a suitable compensation of the boundary conditions.

Moreover, the discussed optimal spline subspaces can be equipped with B-spline-like bases. We have provided an explicit expression of these bases in terms of cardinal B-splines and almost trivial extraction operators; see Section~\ref{sec:bsplines}. This makes the optimal subspaces completely equivalent to the full spline space from the computational point of view.

Summarizing, the results of this paper fully uncover and theoretically explain the relation between optimal spline subspaces and eigenvalue/eigenfunction convergence in isogeometric Galerkin discretizations, and fix the outlier issue both from the theoretical and the computational point of view.
Finally, we remark that the paper focuses on Galerkin discretizations but similar results can be expected when dealing with collocation methods.

\section*{Acknowledgements}
This work was supported 
by the Beyond Borders Programme of the University of Rome Tor Vergata through the project ASTRID (CUP E84I19002250005)
and
by the MIUR Excellence Department Project awarded to the Department of Mathematics, University of Rome Tor Vergata (CUP E83C18000100006).
The authors are members of Gruppo Nazionale per il Calcolo Scientifico, Istituto Nazionale di Alta Matematica.

\bibliography{nwidths}

\appendix
\section{Convergence to eigenvalues and eigenfunctions} 
\label{Appendix:A}

We extend the results of \cite[Section~7]{Sande:2019} in order to prove that there are no outliers in the Galerkin eigenvalue approximation for abstract eigenvalue problems when using optimal subspaces. The first part of this appendix is very similar to \cite[Section~7]{Sande:2019}.

For $f\in L^2$, let $K$ be the integral operator
$$ K f(x) := \int_a^b K(x,y) f(y)\,\d y. $$
As in \cite{Pinkus:85}, we use the notation $K(x,y)$ for the kernel of~$K$. We will only consider kernels that are continuous or piecewise continuous.
We denote by $K^*$ the adjoint, or dual, of the operator $K$,
defined by
$$ (f,K^\ast g) = (Kf, g). $$
The operators $K^\ast K$ and $K K^\ast$ are self-adjoint, positive semi-definite,
and have eigenvalues
\begin{equation*}
\lambda_1 \ge \lambda_2 \ge \cdots \ge \lambda_\indeigk \ge \cdots \ge 0,
\end{equation*}
with corresponding orthonormal eigenfunctions
\begin{align}
K^\ast K \phi_\indeigk = \lambda_\indeigk \phi_\indeigk, \quad \indeigk=1,2,\ldots,\label{eq:phi}
\\
K K^\ast \psi_\indeigk = \lambda_\indeigk \psi_\indeigk, \quad \indeigk=1,2,\ldots.\label{eq:psi}
\end{align}
For $\lambda_\indeigk>0$ we have
\begin{equation} \label{eq:phi-psi}
\psi_\indeigk=\left(\lambda_\indeigk\right)^{-1/2}K\phi_\indeigk, \quad 
\phi_\indeigk=\left(\lambda_\indeigk\right)^{-1/2}K^*\psi_\indeigk.
\end{equation}

We are interested in the function classes $A^r$ and $A^r_*$ studied in \cite{Floater:2018,Sande:2019}.
For an arbitrary integral operator $K$, they can be defined as 
$$A^1:=A:=K(B), \quad A^1_*:=A_*:=K^*(B),$$
and, for $r\geq 2$,
\begin{equation}\label{eq:Ar}
A^r:=K(A^{r-1}_*),\quad A^r_*:=K^*(A^{r-1}),
\end{equation}
where $B$ denotes the unit ball in $L^2(a,b)$.
As stated in \cite[Theorem~3]{Floater:2018}, we have that
\begin{equation}
\label{eq:nwidth-r}
d_n(A^r_*)=d_n(A^r)=d_n(A)^r=(\lambda_{n+1})^{r/2},
\end{equation}
and the space $[\psi_1,\ldots,\psi_n]$ is optimal for $A^r$, while the space $[\phi_1,\ldots,\phi_n]$ is optimal for $A^r_*$, for all $r\geq 1$. Moreover, let $\mathbb{X}_0$ and $\mathbb{Y}_0$ be any finite-dimensional subspaces of $L^2$ and define the subspaces $\mathbb{X}_\prec$ and $\mathbb{Y}_\prec$ in an analogous way to \eqref{eq:Ar}, by
\begin{equation}\label{eq:Xs}
\mathbb{X}_\prec:=K(\mathbb{Y}_{\prec-1}), \quad \mathbb{Y}_\prec:=K^*(\mathbb{X}_{\prec-1}),
\end{equation}
for $\prec \geq 1$. Finally, for any $\prec\geq0$, let $X_\prec$ be the $L^2$-projector onto $\mathbb{X}_\prec$ and let $Y_\prec$ be the $L^2$-projector onto $\mathbb{Y}_\prec$.
It was then shown in \cite[Theorem~4]{Floater:2018} that if $\mathbb{X}_0$ is optimal for the $n$-width of $A$ and $\mathbb{Y}_0$ is optimal for the $n$-width of $A_*$ then, for $r\geq 1$,
\begin{itemize}
	\item the subspaces $\mathbb{X}_\prec$ are optimal for the $n$-width of $A^{r}$, and
	\item the subspaces $\mathbb{Y}_\prec$ are optimal for the $n$-width of $A^{r}_*$,
\end{itemize}
for all $\prec\geq r-1$.
In other words,
if $\mathbb{X}_0$ is optimal for the $n$-width of $A$ and $\mathbb{Y}_0$ is optimal for the $n$-width of $A_*$ then from the definition of optimal space (see Section~\ref{sec:outlier-free}) and \eqref{eq:nwidth-r}
we immediately get for all $\prec\geq r-1$,
\begin{align}
\label{eq:err-optimal}
\|(I-X_\prec)u\|&\leq d_n(A)^r=(\lambda_{n+1})^{r/2}, \quad \forall u\in A^r,
\\
\label{eq:err-optimal*}
\|(I-Y_\prec)u\|&\leq d_n(A_*)^r=(\lambda_{n+1})^{r/2}, \quad \forall u\in A_*^r.
\end{align}

In the following we consider the non-trivial case $\lambda_{n+1}>0$.
The relations \eqref{eq:phi-psi} imply that
\begin{equation}\label{eq:A-elem}
\left(\lambda_\indeigk\right)^{r/2}\psi_\indeigk\in  A^r, \quad  \left(\lambda_\indeigk\right)^{r/2}\phi_\indeigk\in  A_*^r,\quad 
\forall r\geq 1.
\end{equation}
Thus, \eqref{eq:err-optimal}--\eqref{eq:err-optimal*} give us the following result \cite[Theorem~7.1]{Sande:2019}.
\begin{theorem}\label{thm:eig}
	Suppose $\mathbb{X}_0$ is optimal for the $n$-width of $A$ and $\mathbb{Y}_0$ is optimal for the $n$-width of $A_*$. Let $X_\prec$ be the $L^2$-projector onto $\mathbb{X}_\prec$ and $Y_\prec$ be the $L^2$-projector onto $\mathbb{Y}_\prec$. Then, 
	\begin{align*}
	\|(I-X_\prec)\psi_\indeigk\| \leq \left(\frac{\lambda_{n+1}}{\lambda_\indeigk}\right)^{(\prec+1)/2}, 
	\quad\quad
\|(I-Y_\prec)\phi_\indeigk\| \leq \left(\frac{\lambda_{n+1}}{\lambda_\indeigk}\right)^{(\prec+1)/2}.
	\end{align*}
\end{theorem}

Let $\partial$ denote the derivative operator. From now on, we consider integral operators $K$ such that 
\begin{equation}
\label{eq:assumption-K}
\partial K\big|_{A_*^r}=I, \quad \partial K^*\big|_{A^r}=-I, \quad r\geq 0,
\end{equation}
where we let $A^0:=\{u\in B : u\perp \ker(K^*)\}$ and $A^0_*:=\{u\in B : u\perp \ker(K)\}$. We further assume
\begin{equation}\label{assumption-X0}
\mathbb{X}_{0}\perp\ker K^*,\quad \mathbb{Y}_{0}\perp\ker{K}.
\end{equation}
This assumption is true in all non-trivial cases when dealing with optimal spaces.

Then, for $\prec\geq1$ we define the Ritz projectors
\begin{align}\label{def:genRitz}
R_{\mathbb{X}_\prec}:K(L^2(a,b))\to\mathbb{X}_\prec,\quad R_{\mathbb{X}_\prec} &:= KY_{\prec-1}\partial, \\
\label{def:genRitz*}
R_{\mathbb{Y}_\prec}:K^*(L^2(a,b))\to\mathbb{Y}_\prec,\quad R_{\mathbb{Y}_\prec} &:= -K^*X_{\prec-1}\partial.
\end{align}
Note that from \eqref{eq:assumption-K}, for any $u\in K(L^2(a,b))$ the projector $R_{\mathbb{X}_\prec}$ satisfies
\begin{equation*}
(\partial R_{\mathbb{X}_\prec} u, \partial v) = (\partial u,\partial v), \quad \forall v\in \mathbb{X}_\prec,
\end{equation*}
and similarly for any $u\in K^*(L^2(a,b))$,
\begin{equation*}
(\partial R_{\mathbb{Y}_\prec} u, \partial v) = (\partial u,\partial v), \quad \forall v\in \mathbb{Y}_\prec.
\end{equation*}
The projections in \eqref{def:genRitz} and \eqref{def:genRitz*} satisfy the stability estimates
\begin{equation}\label{ineq:stab:a}
\begin{aligned}
\|\partial R_{\mathbb{X}_\prec} u\|&= \|Y_{\prec -1}\partial u\|\leq  \|\partial u\|, 
\\
\|\partial R_{\mathbb{Y}_\prec} v\|&= \|X_{\prec -1}\partial v\|\leq  \|\partial v\|, 
\end{aligned}
\end{equation}
for all $u\in K(L^2(a,b))$ and $v\in K^*(L^2(a,b))$. These stability estimates will be useful in Appendix~\ref{Appendix:B}.

Using Theorem~\ref{thm:eig} we can now prove the following proposition. The method of proof is based on those found in \cite{Floater:2018,Sande:2019,Sande:2020}.

\begin{proposition}\label{prop:err}
	Suppose $\mathbb{X}_0$ is optimal for the $n$-width of $A$ and $\mathbb{Y}_0$ is optimal for the $n$-width of $A_*$. Let $R_{\mathbb{X}_\prec}$ and $R_{\mathbb{Y}_\prec}$ be the projectors in \eqref{def:genRitz} and \eqref{def:genRitz*}. Then, for $\prec\geq\max\{ 1,r-1\}$,
\begin{equation*}
	\begin{aligned}
	\|u-R_{\mathbb{X}_\prec}u\| &\leq (\lambda_{n+1})^{r/2},
	\quad 
	\|\partial u-\partial R_{\mathbb{X}_\prec}u \| \leq (\lambda_{n+1})^{(r-1)/2},
	\quad \forall u\in A^r,
	\\
	\|u-R_{\mathbb{Y}_\prec}u\| &\leq (\lambda_{n+1})^{r/2},
	\quad 
	\|\partial u-\partial R_{\mathbb{Y}_\prec} u\| \leq (\lambda_{n+1})^{(r-1)/2},
	\quad \forall u\in A_*^r.
	\end{aligned}
\end{equation*}	
\end{proposition}
\begin{proof}
	The two cases are analogous and so we only look at $u\in A^r$. For  $u\in A^r$ we have $u=Kf$  for some $f\in A_*^{r-1}$. 
	From assumption \eqref{eq:assumption-K} we get $f=\partial u$. Using the definition of $R_{\mathbb{X}_\prec}$ and $(I-Y_{\prec-1})^2=(I-Y_{\prec-1})$ we have
	\begin{align*}
	\|(I-R_{\mathbb{X}_\prec})u\| &= \|(K-KY_{\prec-1})f\| = \|K(I-Y_{\prec-1})^2f\|
	\\
	&\leq \|K(I-Y_{\prec-1})\|\,\|(I-Y_{\prec-1})f\| 
	\\
	&= \|(I-Y_{\prec-1})K^*\|\,\|(I-Y_{\prec-1})f\|,
\end{align*}	
where $\|(I-Y_{\prec-1})K^*\|$ is the $L^2$-norm operator of $(I-Y_{\prec-1})K^*$.
From \cite[Theorem~4]{Floater:2018} we know that $\mathbb{Y}_{\prec-1}$ is optimal for $A^1_*=K^*(B)$  and so $\|(I-Y_{\prec-1})K^*\|=d_n(A^1_*)=\lambda_{n+1}^{1/2}$. The result follows by applying \eqref{eq:err-optimal*} to $\|(I-Y_{\prec-1})f\|$.	

Let us now focus on the error bound for the derivative. For $u\in A^r$  we have $\partial u\in A^{r-1}_*$ and from \eqref{eq:assumption-K}, \eqref{def:genRitz} and \eqref{eq:err-optimal*}, 
$$
\|\partial u-\partial R_{\mathbb{X}_\prec}u \|=\|\partial u-Y_{\prec-1}\partial u \|\leq(\lambda_{n+1})^{(r-1)/2}.
$$
This completes the proof.
\end{proof}

Observe that the case $r=1$ of the above proposition implies the stability estimates
\begin{equation}\label{ineq:stab:b}
\begin{aligned}
\|R_{\mathbb{X}_\prec} u\|&\leq \|u\|+ \|R_{\mathbb{X}_\prec} u-u\|\leq \|u\| + \lambda_{n+1}^{1/2}\|\partial u\|,
\\
\|R_{\mathbb{Y}_\prec} v\|&\leq \|v\|+ \|R_{\mathbb{Y}_\prec} v-v\|\leq \|v\| + \lambda_{n+1}^{1/2}\|\partial v\|,
\end{aligned}
\end{equation}
for all $u\in K(L^2(a,b))$ and $v\in K^*(L^2(a,b))$. These stability estimates will be useful in Appendix~\ref{Appendix:B}.

Using \eqref{eq:phi-psi} we get
\begin{equation*}
\partial\psi_\indeigk=\lambda_\indeigk^{-1/2}\phi_\indeigk, \quad  \partial\phi_\indeigk=-\lambda_\indeigk^{-1/2}\psi_\indeigk, \ 
\end{equation*}
so that
$$
\|\partial\psi_\indeigk\|=\|\partial\phi_\indeigk\|=\lambda_\indeigk^{-1/2}.
$$
From \eqref{eq:A-elem} and Proposition~\ref{prop:err} we obtain the following result.

\begin{corollary}\label{cor:eig}
Suppose $\mathbb{X}_0$ is optimal for the $n$-width of $A$ and $\mathbb{Y}_0$ is optimal for the $n$-width of $A_*$. Let $R_{\mathbb{X}_\prec}$ and $R_{\mathbb{Y}_\prec}$ be the projectors in \eqref{def:genRitz} and \eqref{def:genRitz*}. Then, for $\prec\geq 1$, 
\begin{align*}
\frac{\|(I-R_{\mathbb{X}_\prec})\psi_\indeigk\|}{\|\psi_\indeigk\|} &\leq \left(\frac{\lambda_{n+1}}{\lambda_\indeigk}\right)^{(\prec+1)/2},
\quad
\frac{\|\partial(I-R_{\mathbb{X}_\prec})\psi_\indeigk\|}{\|\partial\psi_\indeigk\|} \leq \left(\frac{\lambda_{n+1}}{\lambda_\indeigk}\right)^{\prec/2},
\\
\frac{\|(I-R_{\mathbb{Y}_\prec})\phi_\indeigk\|}{\|\phi_\indeigk\|} &\leq \left(\frac{\lambda_{n+1}}{\lambda_\indeigk}\right)^{(\prec+1)/2},
\quad
\frac{\|\partial(I-R_{\mathbb{Y}_\prec})\phi_\indeigk\|}{\|\partial\phi_\indeigk\|} \leq \left(\frac{\lambda_{n+1}}{\lambda_\indeigk}\right)^{\prec/2}.
\end{align*}
\end{corollary}

\begin{example}\label{ex:our-classes}
It was shown in \cite{Floater:2018} that the function classes $A^r_0$ and $A^r_2$ in Section~\ref{sec:outlier-free} are examples of the function classes $A^r_*$ and $A^r$ in \eqref{eq:Ar}, while $A^r_1$ is of the form $\PP_0\oplus A^r$. 	
More precisely,
the classes $A^r_2$ are the function classes $A^r$ in \eqref{eq:Ar} obtained by using the integral operator 
\begin{equation}
\label{eq:K}
Kf(x):=\int_0^xf(y)\,\d y,
\end{equation}
thus implying
$$
K^*f(x)=\int_x^1f(y)\,\d y.
$$
To construct the remaining function classes $A_0^r$ and $A_1^r$ we consider the operator
\begin{equation}
\label{eq:K1}
K_1:=(I-P_0)K,
\end{equation}
where $P_0$ denotes the $L^2$-projector onto constant functions and $K$ is given in \eqref{eq:K}. Then, see \cite[Section~8]{Floater:2018},
$$A_1^r=\PP_0\oplus A^r,\quad A_0^r=A^r_*,$$ 
where $A^r, A^r_*$ are as in \eqref{eq:Ar} with $K$ replaced by $K_1$.
On this concern it is important to note that
\begin{equation*}
A_0^1= \{u\in H_0^1:\|u'\|\leq 1, \ u(0)=u(1)=0\}
=\{K_1^*f: \ \|f\|\leq 1, \ f\perp 1\}=K_1^*(B).
\end{equation*}
From \cite[Theorem~2]{Floater:2018} we know that the spaces $ \mathbb{S}_{p,n,i}^\opt$ in \eqref{eq:opt-spaces} are optimal spaces of the form \eqref{eq:Xs}, with $\prec=p$, for $A_i^r$, $i=0,2,$ for $r\geq 1$ and $p\geq r-1$. On the other hand, the space $\mathbb{S}_{p,n,1}^\opt$ is, similarly to $A_1^r$, of the form $\mathbb{P}_0\oplus \mathbb{X}_{p}$; it is optimal for $A_1^r$ for $p\geq r-1$.
Note that the integral operators in \eqref{eq:K} and \eqref{eq:K1} satisfy the conditions in \eqref{eq:assumption-K} and the spline spaces satisfy assumption \eqref{assumption-X0}. 
Therefore, Proposition~\ref{prop:err} holds for the function classes in Section~\ref{sec:outlier-free} equipped with the optimal spaces \eqref{eq:opt-spaces}.
\end{example}	

Let us now consider the following eigenvalue problem: find $\psi\in K(L^2(a,b))$ and $\mu\in\RR$ such that
\begin{equation}\label{abstr-eig-prob}
(\partial \psi, \partial v) = \mu (\psi, v), \quad \forall v\in K(L^2(a,b)),
\end{equation}
and its discretization given by: find $\psi_\prec\in \mathbb{X}_\prec$ and $\mu_\prec\in\RR$ such that
\begin{equation}\label{abstr-eig-prob-disc}
(\partial \psi_\prec, \partial v) = \mu_\prec (\psi_\prec, v), \quad \forall v\in \mathbb{X}_\prec.
\end{equation}
Note that, if $\phi_\indeigk$, $\lambda_\indeigk$ are as in \eqref{eq:phi} then $K\phi_\indeigk$ and $\mu_\indeigk:=\frac{1}{\lambda_\indeigk}$ solve the eigenvalue problem \eqref{abstr-eig-prob}. We order the eigenvalues in \eqref{abstr-eig-prob} and \eqref{abstr-eig-prob-disc} increasingly.

Combining Corollary~\ref{cor:eig} with a classical argument from \cite{Strang:2008} we can prove an error estimate for the eigenvalues, but first we need the following lemma.
\begin{lemma}\label{lem:dim}
	Let $\mathbb{X}$ be any $n$-dimensional subspace such that the error estimate 
	\begin{equation}\label{weird-ineq}
	\|u-Pu\| \leq c \|\partial u\|, \quad \forall u\in[\psi_1,\ldots,\psi_n],
	\end{equation}
	holds for some projector $P$ onto $\mathbb{X}$. Then, for any $\indeigk\leq n$ such that $\mu_\indeigk^{-1/2}>c$, the dimension of the range-space $P([\psi_1,\ldots,\psi_\indeigk])$ is equal to $\indeigk$.
\end{lemma}
\begin{proof}
	Assume that $\dim P([\psi_1,\ldots,\psi_\indeigk])<\indeigk$. Then there exists a non-zero $$u=\alpha_1\psi_1+\ldots+\alpha_\indeigk\psi_\indeigk$$ such that $Pu=0$. If we let $j$ denote the largest index such that $\alpha_j\neq 0$ then
	\begin{equation*}
	\|\partial u\|^2 = \alpha_1^2\mu_1+\ldots+\alpha_\indeigk^2\mu_\indeigk 
	\leq \mu_j(\alpha_1^2+\ldots+\alpha_j^2)
	= \mu_j\|u\|^2 = \mu_j\|u-Pu\|^2,
	\end{equation*}
	which contradicts \eqref{weird-ineq} since $\mu_j^{-1/2}\geq \mu_\indeigk^{-1/2}>c$.
\end{proof}

\begin{corollary}\label{cor:Strang}
	For any $\indeigk=1,\ldots,n$, let $\mu_\indeigk$ be the $\indeigk$-th eigenvalue of \eqref{abstr-eig-prob} and let $\mu_{\prec,\indeigk}$ be the $\indeigk$-th eigenvalue of \eqref{abstr-eig-prob-disc}. Assume $\mu_1\leq\dots\leq \mu_n<\mu_{n+1}$. If $\mathbb{X}_0$ is optimal for the $n$-width of $A$ and $\mathbb{Y}_0$ is optimal for the $n$-width of $A_*$, then
	\begin{equation*}
	\mu_\indeigk \leq \mu_{\prec,\indeigk} \leq \frac{\mu_\indeigk}{\left(1-\left(\frac{\mu_\indeigk}{\mu_{n+1}}\right)^{(\prec+1)/2}\right)^2}.
	\end{equation*}
\end{corollary}

\begin{proof}
	As a consequence of the min-max theorem we have 
	\begin{equation}\label{min-max}
	\mu_\indeigk = \min_{\mathbb{E}\in \mathcal{V}_\indeigk}\max_{v\in \mathbb{E}} \frac{\|\partial v\|^2}{\|v\|^2},
	\quad
	\mu_{\prec,\indeigk} = \min_{\mathbb{E}\in \mathcal{V}_{\prec,\indeigk}}\max_{v\in \mathbb{E}} \frac{\|\partial v\|^2}{\|v\|^2},
	\end{equation}
	where $\mathcal{V}_\indeigk$ is the set of all subspaces of $K(L^2(a,b))$ of dimension $\indeigk$, and $\mathcal{V}_{\prec,\indeigk}$ is the set of all subspaces of $\mathbb{X}_\prec$ of dimension $\indeigk$. It immediately follows that $\mu_\indeigk \leq \mu_{\prec,\indeigk}$. We note that for $\mu_\indeigk$ the minimum in \eqref{min-max} is achieved for $\mathbb{E}=[\psi_1,\ldots,\psi_\indeigk]$.
	
	Next we define the range-space $\mathbb{I}:=R_{\mathbb{X}_\prec}([\psi_1,\ldots,\psi_\indeigk])$. Since $\mathbb{X}_\prec$ is an optimal subspace we deduce from Lemma \ref{lem:dim} that $\dim\mathbb{I}=\indeigk$. By choosing $\mathbb{I}$ in the formula for $\mu_{\prec,\indeigk}$ in \eqref{min-max} we obtain
	\begin{equation}\label{ineq:eig-proof}
	\begin{aligned}
	\mu_{\prec,\indeigk} &\leq \max_{w\in \mathbb{I}}  \frac{\|\partial w\|^2}{\|w\|^2} = 
	\max_{v\in [\psi_1,\ldots,\psi_\indeigk]}  \frac{\|\partial R_{\mathbb{X}_\prec} v\|^2}{\|R_{\mathbb{X}_\prec}v\|^2}
	\leq \max_{v\in [\psi_1,\ldots,\psi_\indeigk]}  \frac{\|\partial v\|^2}{\|R_{\mathbb{X}_\prec}v\|^2} 
	= \max_{v\in [\psi_1,\ldots,\psi_\indeigk]}  \frac{\|\partial v\|^2}{\|v\|^2} \frac{\|v\|^2}{\|R_{\mathbb{X}_\prec}v\|^2}
	\\
	&\leq \mu_\indeigk \max_{v\in [\psi_1,\ldots,\psi_\indeigk]} \frac{\|v\|^2}{\|R_{\mathbb{X}_\prec}v\|^2}.
	\end{aligned}
\end{equation}
	By Corollary \ref{cor:eig} we have that $\|v-R_{\mathbb{X}_\prec}v\|\leq  \left({\mu_\indeigk}/{\mu_{n+1}}\right)^{(\prec+1)/2}\|v\|$  for all $v\in [\psi_1,\ldots,\psi_\indeigk]$, and so
	\begin{equation*}
	\|R_{\mathbb{X}_\prec}v\|\geq \|v\|\left(1-\left(\frac{\mu_\indeigk}{\mu_{n+1}}\right)^{(\prec+1)/2}\right).
	\end{equation*}
	Finally, we arrive at
	\begin{equation*}
	\mu_{\prec,\indeigk}\leq  \mu_\indeigk\frac{1}{\left(1-\left(\frac{\mu_\indeigk}{\mu_{n+1}}\right)^{(\prec+1)/2}\right)^2},
	\end{equation*}
	which completes the proof.
\end{proof}

\begin{remark}\label{rmk:stab}
Any projector $P$ that satisfies the error estimate in Lemma~\ref{lem:dim} and the stability estimate $\|\partial P v\|\leq \|\partial v\|$ for all $v\in [\psi_1,\ldots,\psi_n]$, can be used in the above argument to provide an error estimate for the eigenvalues.
\end{remark}

Using the argument of \cite[Lemma~6-4.2]{Raviart:83} we can also obtain the following error estimate that provides a sharper asymptotic rate of convergence for the eigenvalues, but which is not useful in the interesting case of large $\indeigk$ and small $\prec$.

\begin{corollary}\label{Cor:RT}
For any $\indeigk=1,\ldots,n$,
let $\mu_\indeigk$ be the $\indeigk$-th eigenvalue of \eqref{abstr-eig-prob} and let $\mu_{\prec,\indeigk}$ be the $\indeigk$-th eigenvalue of \eqref{abstr-eig-prob-disc}. Assume $\mu_1\leq\dots\leq \mu_n<\mu_{n+1}$. If $\mathbb{X}_0$ is optimal for the $n$-width of $A$ and $\mathbb{Y}_0$ is optimal for the $n$-width of $A_*$, then for any choice of $\indeigk$ and $\prec$ such that
	\begin{equation*}
	\sqrt{\indeigk}\frac{\mu_\indeigk}{\mu_1}\left(\frac{\mu_\indeigk}{\mu_{n+1}}\right)^{\prec}< \frac{1}{2},
	\end{equation*}
	we have
	\begin{equation*}
	\mu_\indeigk \leq \mu_{\prec,\indeigk} \leq \frac{\mu_\indeigk}{1-2\sqrt{\indeigk}\frac{\mu_\indeigk}{\mu_1}\left(\frac{\mu_\indeigk}{\mu_{n+1}}\right)^{\prec}}.
	\end{equation*}
\end{corollary}
\begin{proof}
Observe that for any $v\in [\psi_1,\ldots,\psi_\indeigk]$ such that $\|v\|=1$ we have
\begin{align*}
1-\|R_{\mathbb{X}_\prec}v\|^2 = (v-R_{\mathbb{X}_\prec}v, v+R_{\mathbb{X}_\prec}v) = -\|v-R_{\mathbb{X}_\prec}v\|^2 + 2(v-R_{\mathbb{X}_\prec}v,v),
\end{align*}
and so
\begin{equation}\label{eq:R-proof}
\|R_{\mathbb{X}_\prec}v\|^2 = 1 + \|v-R_{\mathbb{X}_\prec}v\|^2 - 2(v-R_{\mathbb{X}_\prec}v,v) \geq 1 - 2(v-R_{\mathbb{X}_\prec}v,v).
\end{equation}
Now, letting $v=\sum_{i=1}^\indeigk \alpha_i\psi_i$ we find that
\begin{align*}
(v-R_{\mathbb{X}_\prec}v,v) &=\sum_{i=1}^\indeigk \alpha_i(v-R_{\mathbb{X}_\prec}v,\psi_i) 
= \sum_{i=1}^\indeigk \alpha_i\mu_i^{-1}(\partial(v-R_{\mathbb{X}_\prec}v),\partial\psi_i)
\\
&=\sum_{i=1}^\indeigk \alpha_i\mu_i^{-1}(\partial(v-R_{\mathbb{X}_\prec}v),\partial(\psi_i-R_{\mathbb{X}_\prec}\psi_i)) ,
\end{align*}
where we have used the definition of $R_{\mathbb{X}_\prec}$ in the last step. Using the Cauchy-Schwarz inequality we get
\begin{equation}\label{ineq:R-proof}
\begin{aligned}
(v-R_{\mathbb{X}_\prec}v,v) &\leq \|\partial(v-R_{\mathbb{X}_\prec}v)\| \left\|\sum_{i=1}^\indeigk \alpha_i\mu_i^{-1}\partial(\psi_i-R_{\mathbb{X}_\prec}\psi_i)\right\|
\\
&\leq \|\partial(v-R_{\mathbb{X}_\prec}v)\| \left(\sum_{i=1}^\indeigk \alpha_i^2\mu_i^{-2}\right)^{1/2}\left(\sum_{i=1}^\indeigk\|\partial(\psi_i-R_{\mathbb{X}_\prec}\psi_i)\|^2\right)^{1/2}
\\
&\leq \mu_1^{-1}\sqrt{\indeigk} \max_{w \in [\psi_1,\ldots,\psi_\indeigk]} \frac{\|\partial(w-R_{\mathbb{X}_\prec}w)\|^2}{\|w\|^2}
\\
&\leq \sqrt{\indeigk}\frac{\mu_\indeigk}{\mu_1} \left(\frac{\mu_\indeigk}{\mu_{n+1}}\right)^{\prec},
\end{aligned}
\end{equation}
where the last inequality follows from Corollary~\ref{cor:eig}. Combining \eqref{ineq:eig-proof}, \eqref{eq:R-proof} and \eqref{ineq:R-proof} now completes the proof.
\end{proof}

For any $\indeigk=1,\ldots,n$, we define the $\indeigk$-th separation constant by
\begin{equation*}
\rho_\indeigk := \max_{\substack{i=1,\dots,n \\ i\neq \indeigk}} \frac{\mu_\indeigk}{|\mu_\indeigk - \mu_{\prec,i}|}.
\end{equation*}
Then, from \cite{Strang:2008} we obtain the following error estimate for the eigenfunctions. For simplicity we state the result in the case of distinct eigenvalues; see \cite[Chapter~6]{Strang:2008} for a discussion on the general case.
\begin{corollary}
	\label{cor:galerkin-A}
	For any $\indeigk=1,\ldots,n$, let $\psi_\indeigk$ be the $\indeigk$-th eigenfunction of \eqref{abstr-eig-prob} and let $\psi_{\prec,\indeigk}$ be the $\indeigk$-th eigenfunction of \eqref{abstr-eig-prob-disc}. Assume $\mu_1<\dots< \mu_n<\mu_{n+1}$. If $\mathbb{X}_0$ is optimal for the $n$-width of $A$ and $\mathbb{Y}_0$ is optimal for the $n$-width of $A_*$, then
	\begin{equation*}
	\|\psi_\indeigk-\psi_{\prec,\indeigk}\|\leq 2(1+\rho_\indeigk) \|\psi_\indeigk-R_{\mathbb{X}_\prec}\psi_\indeigk\| \leq 2(1+\rho_\indeigk) \left(\frac{\mu_\indeigk}{\mu_{n+1}}\right)^{(\prec+1)/2}.
	\end{equation*}
\end{corollary}
\begin{proof}
See \cite[Theorem~6.2]{Strang:2008}.
\end{proof}

Results analogous to Corollaries~\ref{cor:Strang} and~\ref{cor:galerkin-A} hold for the following eigenvalue problem: find $\phi\in K^*(L^2(a,b))$ and $\mu\in\RR$ such that
\begin{equation*}
(\partial \phi, \partial v) = \mu (\phi, v), \quad \forall v\in K^*(L^2(a,b)),
\end{equation*}
and its discretization given by: find $\phi_\prec\in \mathbb{Y}_\prec$ and $\mu_\prec\in\RR$ such that
\begin{equation*}
(\partial \phi_\prec, \partial v) = \mu_\prec (\phi_\prec, v), \quad \forall v\in \mathbb{Y}_\prec.
\end{equation*}

\section{Error estimates for tensor-product Ritz projections onto optimal subspaces} 
\label{Appendix:B}
In this appendix we prove error estimates for tensor-product Ritz projections onto optimal subspaces. The arguments are taken from \cite[Section~6]{Sande:2020} but we repeat them here for the sake of completeness.
For simplicity of notation, we will only consider the case $d=2$ and optimal spaces of the form $\mathbb{X}_{\prec_1}\otimes \mathbb{X}_{\prec_2}$. The arguments for tensor-product optimal spaces of the form $\mathbb{X}_{\prec_1}\otimes \mathbb{Y}_{\prec_2}$ and $\mathbb{Y}_{\prec_1}\otimes \mathbb{Y}_{\prec_2}$ are completely analogous. We assume throughout this appendix that for any $i=1,2$, the space $\mathbb{X}_{\prec_i}$ is optimal for $d_n(A^r)$ for all $r\leq \prec_i+1$.

Define the tensor-product projector $R_{\bfprec}:K(L^2(a,b))\otimes K(L^2(a,b))\to \mathbb{X}_{\prec_1}\otimes \mathbb{X}_{\prec_2}$ by 
\begin{equation*}
R_{\bfprec}:=R_{\mathbb{X}_{\prec_1}}\otimes R_{\mathbb{X}_{\prec_2}}.
\end{equation*}
We let $\Omega:=(a,b)^2$ and denote the $L^2$-norm on $\Omega$ by $\|\cdot\|_{\Omega}$.
Then, similar to \cite[Lemma~10]{Sande:2020} we have the following result.
\begin{lemma}\label{lem:tensorRitz}
Let $u\in K(L^2(a,b))\otimes K(L^2(a,b))$ be given. Then, for all $\prec_1,\prec_2\geq 1$ we have
\begin{align*}
\|u-R_{\bfprec}u\|_{\Omega} &\leq\|u-R_{\mathbb{X}_{\prec_1}}u\|_{\Omega}+\|u-R_{\mathbb{X}_{\prec_2}}u\|_{\Omega}
\\
&\quad +\lambda_{n+1}^{1/2}\min\left\{\|\partial_{1}u-R_{\mathbb{X}_{\prec_2}}\partial_1u\|_{\Omega},\,\|\partial_{2}u-R_{\mathbb{X}_{\prec_1}}\partial_2u\|_{\Omega}\right\},
\\
\|\partial_1(u-R_{\bfprec}u)\|_{\Omega} &\leq\|\partial_1(u-R_{\mathbb{X}_{\prec_1}}u)\|_{\Omega}+\|\partial_{1}u-R_{\mathbb{X}_{\prec_2}}\partial_1u\|_{\Omega}.
\end{align*}
\end{lemma}
\begin{proof}
From \eqref{ineq:stab:b} and by adding and subtracting $R_{\mathbb{X}_{\prec_1}}u$ we obtain
\begin{align*}
\|u-R_{\bfprec}u\|_{\Omega} &\leq \|u-R_{\mathbb{X}_{\prec_1}}u\|_{\Omega} + \|R_{\mathbb{X}_{\prec_1}}(u-R_{\mathbb{X}_{\prec_2}}u)\|_{\Omega}
\\
&\leq \|u-R_{\mathbb{X}_{\prec_1}}u\|_{\Omega} + \|u-R_{\mathbb{X}_{\prec_2}}u\|_{\Omega} + \lambda_{n+1}^{1/2}\|\partial_1(u-R_{\mathbb{X}_{\prec_2}}u)\|_{\Omega},
\end{align*}
and similarly for $R_{\mathbb{X}_{\prec_2}}u$. The first result now follows since $\partial_i$ commutes with $R_{\mathbb{X}_{\prec_j}}$ for $i\neq j$.
Analogously, using \eqref{ineq:stab:a} we obtain
\begin{align*}
\|\partial_1(u-R_{\bfprec}u)\|_{\Omega} &\leq \|\partial_1(u-R_{\mathbb{X}_{\prec_1}}u)\|_{\Omega} + \|\partial_1R_{\mathbb{X}_{\prec_1}}(u-R_{\mathbb{X}_{\prec_2}}u)\|_{\Omega}
\\
&\leq \|\partial_1(u-R_{\mathbb{X}_{\prec_1}}u)\|_{\Omega} + \|\partial_1(u-R_{\mathbb{X}_{\prec_2}}u)\|_{\Omega},
\end{align*}
and the second result follows.
\end{proof}

By using Proposition~\ref{prop:err} we can now achieve error estimates for the tensor-product Ritz projection. The argument is taken from \cite[Theorem~7]{Sande:2020}.
To state the next proposition we define the general spaces
$$\mathcal{K}^1:=K(L^2(a,b)), \quad \mathcal{K}^1_*:=K^*(L^2(a,b)),$$
and, for $r\geq 2$,
\begin{equation*}
\mathcal{K}^r:=K(\mathcal{K}^{r-1}_*),\quad \mathcal{K}^r_*:=K^*(\mathcal{K}^{r-1}).
\end{equation*}
Observe that the spaces $H^r_i$ in \eqref{eq:Hr} for $i=0,1,2$ are examples of these general spaces, as explained in Example~\ref{ex:our-classes}.

\begin{proposition}\label{prop:tensorRitz}
Let $u\in \mathcal{K}^r(L^2(a,b))\otimes \mathcal{K}^r(L^2(a,b))$ for $r\geq2$ be given. If $\prec_1,\prec_2\geq r-1$, then
\begin{align*}
\|u-R_{\bfprec}u\|_{\Omega}&\leq \lambda_{n+1}^{r/2}\left(\|\partial_1^ru\|_{\Omega}+\|\partial_2^ru\|_{\Omega} 
 +\min\left\{\|\partial_1\partial_2^{r-1}u\|_{\Omega}, \,\|\partial_1^{r-1}\partial_2u\|_{\Omega}\right\}\right),
\end{align*}
and
\begin{align*}
\|\partial_1(u-R_{\bfprec}u)\|_{\Omega}&\leq \lambda_{n+1}^{(r-1)/2}\left(\|\partial_1^ru\|_{\Omega}+\|\partial_1\partial_2^{r-1}u\|_{\Omega}\right), 
\\
\|\partial_2(u-R_{\bfprec}u)\|_{\Omega}&\leq \lambda_{n+1}^{(r-1)/2}\left(\|\partial_1^{r-1}\partial_2u\|_{\Omega}+\|\partial_2^ru\|_{\Omega}\right).
\end{align*}
\end{proposition}
\begin{proof}
Using Lemma~\ref{lem:tensorRitz} and Proposition~\ref{prop:err} we find that
\begin{align*}
&\|u-R_{\bfprec}u\|_{\Omega} \\
&\quad\leq\|u-R_{\mathbb{X}_{\prec_1}}^1u\|_{\Omega}+\|u-R_{\mathbb{X}_{\prec_2}}^1u\|_{\Omega}
 +\lambda_{n+1}^{1/2}\min\left\{\|\partial_{1}u-R_{\mathbb{X}_{\prec_2}}^1\partial_1u\|_{\Omega},\, \|\partial_{2}u-R_{\mathbb{X}_{\prec_1}}^1\partial_2u\|_{\Omega}\right\}
\\
&\quad\leq \lambda_{n+1}^{r/2}\left(\|\partial_1^ru\|_{\Omega}+\|\partial_2^ru\|_{\Omega} 
 +\min\left\{\|\partial_1\partial_2^{r-1}u\|_{\Omega}, \,\|\partial_1^{r-1}\partial_2u\|_{\Omega}\right\}\right),
\end{align*}
which proves the first result. The other results follow by a similar argument.
\end{proof}

\end{document}